\documentclass[12pt]{amsart}
\usepackage[margin=1.3in]{geometry}
\usepackage{amsthm}
\usepackage{amssymb}
\usepackage{mathrsfs}
\usepackage[colorlinks=true,linkcolor=black,anchorcolor=black,citecolor=black,filecolor=black,menucolor=black,runcolor=black,urlcolor=black]{hyperref}
\usepackage{verbatim}
\usepackage{graphicx}
\usepackage{xcolor}
\usepackage{array}

\makeatletter

\let\c@table\c@figure
\makeatother

\usepackage{todonotes}\newcommand{\todoin}[2][]{\todo[inline,#1]{#2}}

\def\incplotscl#1#2#3{\begin{minipage}{#1}\begin{center}\includegraphics[width=#1]{#2}\\[10pt] \sf #3 \end{center}\end{minipage}}

\newtheorem{theorem}{Theorem}[section]
\newtheorem{lemma}[theorem]{Lemma}
\newtheorem{prop}[theorem]{Proposition}

\theoremstyle{definition}
\newtheorem{definition}[theorem]{Definition}
\newtheorem{notation}[theorem]{Notation}

\newtheorem{question}[theorem]{Question}

\newtheorem{remark}[theorem]{Remark}

\newcommand{\boatstar}{{\bigstar}}
\newcommand{\pentagon}{{\rm pent}}
\newcommand{\robinson}{{\blacktriangle}}
\newcommand{\kitedart}{{\spadesuit}}
\newcommand{\rhombus}{{\blacklozenge}}
\newcommand{\AB}{{\blacksquare}}

\newcommand{\subst}{{\operatorname{S}}}

\newcommand{\bbN}{{\mathbb{N}}}
\newcommand{\bbR}{{\mathbb{R}}}
\newcommand{\bbZ}{{\mathbb{Z}}}

\newcommand{\calE}{{\mathcal{E}}}
\newcommand{\calP}{{\mathcal{P}}}
\newcommand{\calR}{{\mathcal{R}}}

\newcommand{\calT}{{\mathcal{T}}}
\newcommand{\calV}{{\mathcal{V}}}

\newcommand{\scrH}{{\mathscr{H}}}

\newcommand{\tr}{{\mathrm{Tr}}}

\definecolor{purple}{rgb}{.5,0,1}
\definecolor{green}{rgb}{0,.5,0}
\definecolor{orange}{rgb}{1,.5,0}

\DeclareMathOperator{\supp}{supp}

\numberwithin{equation}{section}

\begin{document}

\title[Discontinuities of the IDS for Aperiodic Tilings]{Discontinuities of the Integrated Density\\ of States for Laplacians Associated with\\ Penrose and Ammann--Beenker Tilings}

\author[D. Damanik]{David Damanik}
\address{Department of Mathematics, Rice University, Houston, TX~77005, USA}
\email{damanik@rice.edu}

\author[M.\ Embree]{Mark Embree}
\address{Department of Mathematics, Virginia Tech, Blacksburg, VA 24061, USA}
\email{embree@vt.edu}

\author[J. Fillman]{Jake Fillman}
\address{Department of Mathematics, Texas State University, San Marcos, TX 78666, USA}
\email{fillman@txstate.edu}

\author[M.\ Mei]{May Mei}
\address{Department of Mathematics, Denison University, Granville, OH 43023, USA}
\email{meim@denison.edu}

\begin{abstract}
Aperiodic substitution tilings provide popular models for quasicrystals, materials exhibiting aperiodic order.
We study the graph Laplacian associated with four tilings from the mutual local derivability class of the Penrose tiling,  as well as the Ammann--Beenker tiling.  
In each case we exhibit locally-supported eigenfunctions, which necessarily cause jump discontinuities in the integrated density of states for these models.  
By bounding the multiplicities of these locally-supported modes, in several cases we provide concrete lower bounds on this jump.
These results suggest a host of questions about spectral properties of the Laplacian on aperiodic tilings,
which we collect at the end of the paper.
\end{abstract}

\maketitle

\setcounter{tocdepth}{1}
\tableofcontents

\noindent

\hypersetup{
	linkcolor={black!30!blue},
	citecolor={red},
	urlcolor={black!30!blue}
}


\section{Introduction}
\subsection{Prologue}
The structure of ordered materials such as crystals has long been a topic of fascination in mathematics and science. The discovery of quasicrystals in the 1980s ushered in new techniques and motivations for investigating aperiodic structures with underlying symmetries. 
Following this discovery, the study of electronic transport properties of particles in quasi\-crystalline media has become a fundamental question in mathematics and physics.

Since their discovery in the 1980s by Shechtman \emph{et al.}\ \cite{SBGC1984PRL}, quasicrystals have generated substantial interest in mathematical physics. For a sample of the mathematical literature devoted to quasicrystals and the mathematics of aperiodic order, see \cite{BaakeGrimm2013:AOVol1, BaakeGrimm2013:AOVol2, BaakeMoody2000CRM, KellLenzSav2015, Moody1997NATO, Patera1998FIM} and references therein. Given the physical origins of these models, there has naturally been interest in the analysis of quantum mechanical systems associated with quasicrystals. As such, many researchers have studied spectral problems associated with self-adjoint operators that inherit their structure from a mathematical model of a quasicrystal. From this perspective, one-dimensional quasicrystal models have been discussed extensively, since those models enjoy the largest variety of tools in the spectral toolbox.  Particularly refined results have been obtained for the Fibonacci Hamiltonian, the most prominent one-dimensional quasicrystal model; see, e.g., \cite{DamGorYes2016Invent, KohSutTan1987, OPRSS1983, Suto1987CMP, Suto1989JSP} and references therein. 

The Penrose tiling is a two-dimensional structure that shares many features with quasicrystals discovered in nature, such as five-fold rotational symmetry and the pure point nature of suitable diffraction measures associated with the tiling \cite{BaakeGrimm2013:AOVol1, TilingsEncyclopedia,Gardner1997, GrunbaumShephard1987}. Despite a substantial amount of interest from mathematics and physics, there are relatively few results about the spectral theory of Laplacians on the Penrose tiling, due to the disappearance of some of the crucial tools used in the analysis of one-dimensional quasicrystals. One surprising spectral phenomenon that these operators can exhibit is the presence of locally-supported eigenfunctions. It is known that such locally-supported eigenfunctions can never occur for finite-range operators on $\ell^2(\bbZ^d)$.

One can construct Laplacians from tilings in two different ways: hopping between tiles and hopping between vertices. We distinguish these paradigms as the ``tile model'' and the ``vertex model,'' respectively. For both the tile and the vertex model associated with the rhombus tiling, the presence of locally-supported eigenfunctions was observed in the 1980s \cite{ATFK1988PRB, FATK1988PRB, KohSut1986PRL}. Several other prominent tilings (Robinson triangle, boat--star, and kite--dart) are equivalent to the Penrose tiling, in the sense of mutual local derivability (MLD).\ \ On one hand, the presence of finitely supported eigenfunctions depends very sensitively on the local structure of the tiling on which one studies the Laplacian, and hence one would not expect the existence of such eigenfunctions to hold universally in a given MLD class, since the MLD relation can alter the local structure of a tiling. Nevertheless, we study four tilings which are in the MLD class of the Penrose tiling and show that all of them exhibit locally-supported eigenfunctions (and hence exhibit a discontinuous IDS).\ \ More recently, vertex models associated with the Penrose and Ammann--Beenker tiling were studied in \cite{MirzOkt2020PRB, Oktel2021, Oktel2022}. We study the tile model for the Ammann--Beenker tiling (also called the octagonal tiling), showing that it too exhibits locally-supported eigenfunctions.

\subsection{Setting and Results}\label{ss:settings}

Let us now define the relevant objects and state our results. We will work with Laplacians on graphs associated with tilings. A \emph{graph} $\Gamma = (\calV,\calE)$ consists of a nonempty set $\calV$ of \emph{vertices} and a set $\calE$ comprised of unordered pairs of elements of $\calV$.
We write $u\sim v$ if $(u,v) \in \calE$ and say that $u$ and $v$ are \emph{connected} by an edge. The \emph{degree} of $v$ is the number of neighbors of $v$: $\deg(v) = \#\{u : u \sim v\}$.

The \emph{Laplacian} on the graph $\Gamma = (\calV,\calE)$ is the operator
\[\Delta=\Delta_\Gamma:\ell^2(\calV) \to \ell^2(\calV), \quad
[\Delta\psi](v) = \sum_{u \sim v}\big(\psi(v)-\psi(u)\big).\]
Equivalently, one can define $\Delta = \mathcal{D}-\mathcal{A}$, where $\mathcal{D}$ is the degree operator and $\mathcal{A}$ is the adjacency operator:
\begin{align*}
[\mathcal{D}\psi](v) & = \deg(v)\psi(v) \\
[\mathcal{A}\psi](v) & = \sum_{u \sim v} \psi(u).
\end{align*}

We are interested in infinite graphs that arise from substitution tilings of the plane by polygons.
Namely (once a tiling of the plane by polygons has been constructed), the associated graph has one vertex for each polygon of the tiling, and two vertices are connected if the associated polygons share at least one edge.

One fruitful way to study such infinite graphs is to analyze finite truncations.
Namely, one may consider finite subsets $\calV_n \subseteq \calV$ and $\calE_n = \{(u,v) : u,v \in \calV_n\}$ with $\calV_n \uparrow \calV$ in a suitable sense, and let $\Delta_n := \Delta_{\Gamma_n}$ denote the Laplacian on the finite graph $\Gamma_n = (\calV_n,\calE_n)$.
The normalized eigenvalue counting measure is given by
\begin{equation} \label{eq:intro:NuNDef}
 \nu_n(B) = \frac{1}{\#\calV_n} \tr \, \chi_B(\Delta_n), \quad B \subseteq \bbR \text{ measurable}.
 \end{equation}
Under suitable assumptions (which are met in all of the cases under consideration in the present work),
$\nu_{n}$ converges in the weak$^*$ sense to a limiting measure
\begin{equation} \label{eq:intro:NuDef}\nu = \nu_\Gamma,
\end{equation} which we call the \emph{density of states measure} (DOSM) of the graph $\Gamma$, and the limit is indepedent on the choice of $\{\calV_n\}_{n=1}^\infty$; this is described in more detail in \cite{LenzStollmann2003MPAG, LenzStollmann2001, LenzStollmann2005}.

The \emph{integrated density of states} of $\Gamma$ is the accumulation function of the measure $\nu_{\Gamma}$:
\begin{equation} \label{eq:intro:kGammaDef}
k_\Gamma(E) = \nu_\Gamma\big((-\infty,E]\big).
\end{equation}

One is then naturally interested in regularity properties of this function: is it continuous on suitable intervals, and if so, what can one say about the modulus of continuity there, and so on. 

Here, we study questions of this kind for various versions of the \emph{Penrose tiling}.

\begin{notation}
We use $\Gamma_{\boatstar}$, $\Gamma_{\robinson}$, $\Gamma_{\rhombus}$, $\Gamma_{\kitedart}$, $\Gamma_\AB$ to refer to the graphs of the boat--star tiling, the Robinson triangle tiling, the rhombus tiling, the kite--dart tiling, and the Ammann--Beenker tiling respectively. For ease of notation, we drop the $\Gamma$ when referring to the corresponding integrated density of states as $k_\boatstar$, $k_\robinson$, $k_\rhombus$, $k_\kitedart$, and $k_\AB$.
\end{notation}

\begin{theorem} \label{t:main}
If $\square \in \{\boatstar, \robinson, \rhombus, \kitedart\}$, then the integrated density of states $k_\square$ is discontinuous.
\end{theorem}

A similar result holds for the \emph{Ammann--Beenker tiling}.

\begin{theorem} \label{t:main:ABtiling}
The integrated density of states $k_\AB$ is discontinuous.
\end{theorem}

We direct the reader to later sections for precise definitions of these tilings.

The tilings that we discuss are linearly repetitive (see, e.g., \cite{BaakeGrimm2013:AOVol1} for the definition and a discussion of this concept): in particular, any pattern that is observed once is observed infinitely often with positive frequency.
In view of \eqref{eq:intro:NuNDef}, \eqref{eq:intro:NuDef}, and \eqref{eq:intro:kGammaDef}, a locally-supported eigenfunction necessarily produces a discontinuity of the IDS at the corresponding eigenvalue. More precisely, if $\Delta_\Gamma$ enjoys an eigenvalue $E$ with an eigenfunction having local support, then (a one-tile neighborhood of) the support of the eigenfunction occurs with positive frequency, and hence one observes a jump discontinuity in $k_\Gamma$.\ \ 
In fact, it is known that (under suitable assumptions on the underlying graph) a discontinuity of the IDS at energy $E$ is \emph{equivalent} to the presence of a locally-supported eigenfunction with eigenvalue~$E$~\cite{KlasLenzStol2003CMP}.

Furthermore, with this picture, one can estimate the size of the jump discontinuity by estimating the frequency with which the support of the eigenfunction occurs. Concretely, we can sharpen the conclusions of Theorem~\ref{t:main} in some individual cases. Here is a representative selection of theorems that one can prove.

\begin{theorem} \label{t:boatstarquant}
If $\Gamma$ is a graph associated with the boat--star tiling,
\begin{equation}
k_\boatstar(4+)-k_\boatstar(4-) \geq \frac{65-29\sqrt{5}}{10} \approx 0.01540 \ldots \ . 
\end{equation}
\end{theorem}

\begin{theorem}\label{t:robinsonquant}
If $\Gamma$ is a graph associated with the Robinson triangle tiling,
\begin{equation}
k_\robinson(E+)-k_\robinson(E-) \geq \frac{65-29\sqrt{5}}{20} \approx 0.007701\ldots ,\quad
E \in\{ 2,  4\}.
\end{equation}
\end{theorem}

\begin{theorem} \label{t:ABquant}
If $\Gamma$ is a graph associated with the Ammann--Beenker tiling,
\begin{align}
k_\AB(4+)-k_\AB(4-) & \geq  1270-898\sqrt{2} \approx 0.036221\ldots \ ,\\
k_\AB(6+)-k_\AB(6-) & \geq  116 - 82\sqrt{2}\kern12pt  \approx 0.0344879 \ldots \ . 
\end{align}
\end{theorem}

\begin{remark}Let us make some remarks about these theorems.
\begin{enumerate}
\item[(a)] Since the frequency calculations are somewhat similar in the different examples, we do not discuss quantitative estimates in all cases, but rather focus on a representative subset of examples. An estimate for the lower bound on the jump in the IDS for the rhombus tiling is discussed in \cite{FATK1988PRB}. The jump discontinuity for the kite--dart tiling may be estimated similarly to the others.
\item[(b)] One may naturally be interested in whether the bounds are sharp, that is, whether the jump in the IDS is precisely given by the enumerated expressions.
Let us comment on the difficulties associated with ``the other direction.''
The lower bounds are computed by (1)~identifying patterns in a given tiling that can support a finitely-supported eigenfunction and (2)~finding combinatorial mechanisms in the substitution structure generating the tiling that enable us to estimate the frequency with which the desired pattern(s) occur.
Thus, if one wishes to prove that the estimates are sharp, one must overcome two obstacles:
\begin{enumerate}
\item[(1)] One must show that one has identified all pattern(s) in the tiling that permit a finitely-supported eigenfunction with the desired energy.

\item[(2)] One must show that the pattern(s) that one has identified can only arise via the combinatorial mechanisms that one used to estimate the frequency.
\end{enumerate}
The second obstacle can likely be overcome with a sufficiently careful analysis of suitably large supertiles.
However, the first obstacle appears to be genuinely intractable with current technology.
(Indeed, Figures~\ref{fig:rhombus_more_ev} and~\ref{fig:ABlevel3_bigmodes} 
show eigenfunctions with large-but-finite support that emerge on larger tilings, 
and cannot be expressed as linear combinations of our simpler eigenfunctions supported on small patches.)
\end{enumerate}
\end{remark}

One crucial point that we want to emphasize is the synergy between the numerical and spectral analyses. The eigenfunctions discussed in this paper were first discovered via numerical spectral computations on finite graph Laplacians $\Delta_n$.  
Given the finite nature of the sought-after eigenfunctions, such numerical calculations (once carried out on a sufficiently large finite patch) suffice to demonstrate the existence of finitely supported eigenfunctions and discontinuities of the IDS.\ \ 
Once found, simple locally-supported modes can readily be verified by hand. 
However for some tilings, larger graphs reveal additional eigenfunctions whose 
local support extends to several hundred tiles, making manual calculations inadvisable.  

In addition to suggesting theorems, numerics can also provide evidence for new conjectures. In that spirit, we will conclude the paper with numerical plots of large finite-volume approximations of the integrated densities of states associated with these tilings, and pose some interesting open problems suggested by this work.

\subsection*{Acknowledgements} The authors thank Michael Baake, Semyon Dyatlov, and Anton Gorodetski for many helpful conversations and the American Institute of Mathematics for hospitality and support through the SQuaRE program during a remote meeting in January~2021 and a January~2022 visit, during which part of this work was completed. D.D.\ was supported in part by NSF grants DMS--1700131 and DMS--2054752, and Simons Fellowship $\# 669836$. M.E.\ was supported in part by NSF grant DMS-1720257.  J.F.\ was supported in part by NSF grant DMS--2213196 and Simons Foundation Collaboration grant $\# 711663$.



\section{Preliminaries} \label{sec:tilingDefs}

\subsection{Tilings and Associated Laplacians}
To set the stage and fix notation, let us recall some notation, conventions, and definitions largely following Baake--Grimm~\cite{BaakeGrimm2013:AOVol1}. 
\begin{definition}[Patterns, Fragments, and Tiles]
A \emph{pattern} $\calT = \{T_i : i \in I\}$ in $\bbR^2$ is a nonempty set whose elements $T_i$ are nonempty subsets of $\bbR^2$. We write $\calT \sqsubset \bbR^2$ to denote that $\calT$ is a pattern in $\bbR^2$ and say $\calT$ is a \emph{tiling} if $I$ is countable, the $T_i$ are closed and nonempty sets, $\bigcup_{i \in \bbZ} T_i=\bbR^2$, and $T_i^\circ \cap T_j^\circ = \emptyset$ for all $i \neq j$.  The elements $T_i$ of $\calT$ are called \emph{tiles} or \emph{fragments} of $\calT$. 

In the sequel, we will occasionally want to distinguish tiles that are the same as subsets of $\bbR^2$ but that nevertheless have different behavior under substitution rules.
For instance, the reader may consider the example below in Definition~\ref{def:triangleRules}, in which there are two basic tile shapes (acute and obtuse triangles), but two different colors of each shape (each of which behaves as the mirror image of the other under substitutions). One often uses colors or decorations to distinguish between different types of the same shape. By abuse of notation, we will still refer to tilings with colors or decorations as \emph{tilings} rather than \emph{decorated tilings}.

For $\calT \sqsubset \bbR^2$ and $K \subseteq \bbR^2$, $\calT \sqcap K$ is the pattern consisting of all fragments of $\calT$ that intersect $K$ nontrivially:
\[\calT \sqcap K := \{T_i : T_i \in \calT \text{ and }   T_i \cap K \neq \emptyset\}.\]
Naturally, for $\calT \sqsubset \bbR^2$ and $t \in \bbR^2$, the \emph{translation} of $\calT$ by $t$ is given by
\[t+\calT = \{t+T_i : T_i \in \calT\}.\] We refer to the equivalence class of tiles up to translation as \emph{prototiles}. Finally, given patterns $\calT_0 \subseteq \calT \sqsubset \bbR^2$, an \emph{occurrence} of $\calT_0$ in $\calT$ is any translation of $\calT_0$ that is also a subset of $\calT$; in other words, an occurrence of $\calT_0$ is any arrangement of tiles in $\calT$ that looks the same as $\calT_0$, up to translation.

The pattern $\calT' \sqsubset \bbR^2$ is said to be \emph{locally derivable} from $\calT \sqsubset \bbR^2$ (denoted $\calT \overset{\rm LD}\rightsquigarrow \calT'$) if for some $R>0$ one has
\[(-x+\calT)\sqcap B_R = (-y+\calT)\sqcap B_R \implies (-x+\calT')\sqcap \{0\} = (-y+\calT')\sqcap \{0\}, \]
where $B_R$ denotes the open ball of radius $R$ centered at the origin (note that equality of patterns includes equality of colors as well). If $\calT \overset{\rm LD}\rightsquigarrow \calT'$ and $\calT' \overset{\rm LD}\rightsquigarrow \calT$, we say that $\calT$ and $\calT'$ are \emph{mutually locally derivable} (MLD) and denote this by $\calT \overset{\rm MLD} \leftrightsquigarrow \calT'$.
\end{definition}

\begin{definition}
A \emph{polygon} is a nonempty compact subset $\bbR^2$ with dense interior obtained by intersecting finitely many closed half-planes. From this point onward, all tiles are assumed to be polygonal. Let $\calT = \{T_i : i \in I\}$ be such a tiling of $\bbR^2$. The \emph{induced graph} $\Gamma = \Gamma_\calT = (\calV, \calE)$ has $\calV = I$ and one has $u \sim v$ if and only if $T_u$ and $T_v$ share at least one edge.
The associated \emph{Laplace operator} acts on the space $\scrH = \ell^2(\calV)$ via
\begin{equation}\label{eq:laplacian}
[\Delta\psi](v) = \sum_{u \sim v} (\psi(u)-\psi(v)) \, \quad \psi \in \ell^2(\calV).
\end{equation}
\end{definition}

\subsection{Substitution Tilings}

Let us now describe the main setting in which we work: tilings that are generated by a substitution rule.

\begin{definition}[Substitution Tilings]
Let $\calP = \{P_1,\ldots,P_n\}$ denote a finite \emph{protoset}, or collection of prototiles in $\bbR^2$. Denote by $\calP^*$ the collection of finite patterns $\calT \sqsubset \bbR^2$ whose elements are images of elements of $\calP$ under translation and rotation. A \emph{substitution} is a map $\subst:\calP \to\calP^*$.
One can extend $\subst$ to $\calP^*$ in a natural manner, so we can speak of iterates of $\subst$.

A \emph{substitution tiling} associated with $\subst$ is a polygonal tiling $\mathcal{T}$ such that any finite patch of $\mathcal{T}$ occurs in $\subst^n(P)$ for some $P \in \mathcal{P}$ and some $n \in \mathbb{N}$. The collection $\mathbb{X}_\subst$ of all such tilings is called the \emph{hull} of $\subst$ and is a compact set in a suitable tiling metric. Since it is not central to our work, we will not specify the tiling metric precisely, but we simply say that two tilings are close in the tiling metric if after a small shift they coincide on a large ball centered at the origin. Clearly $\mathbb{R}^2$ acts on $\mathbb{X}_\subst$ by translations. It is known that for suitable substitutions, this translation action is \emph{minimal} (i.e., the translation orbit of any element of $\mathbb{X}_\subst$ is dense in $\mathbb{X}_\subst$).

Given a substitution $\subst$ on a set $\calP$ as above, the associated \emph{substitution matrix} is the $n \times n$ matrix $M$ whose entry in row $i$ and column $j$ is the number of occurrences of tile $P_i$ in $\subst(P_j)$.
\end{definition}

In what follows, we will consider five tilings generated by substitution rules. Let us start with an example.
\begin{definition} \label{def:triangleRules}
The \emph{Robinson triangle substitution} has four basic tiles:

\begin{center}
\includegraphics[width=1.25in]{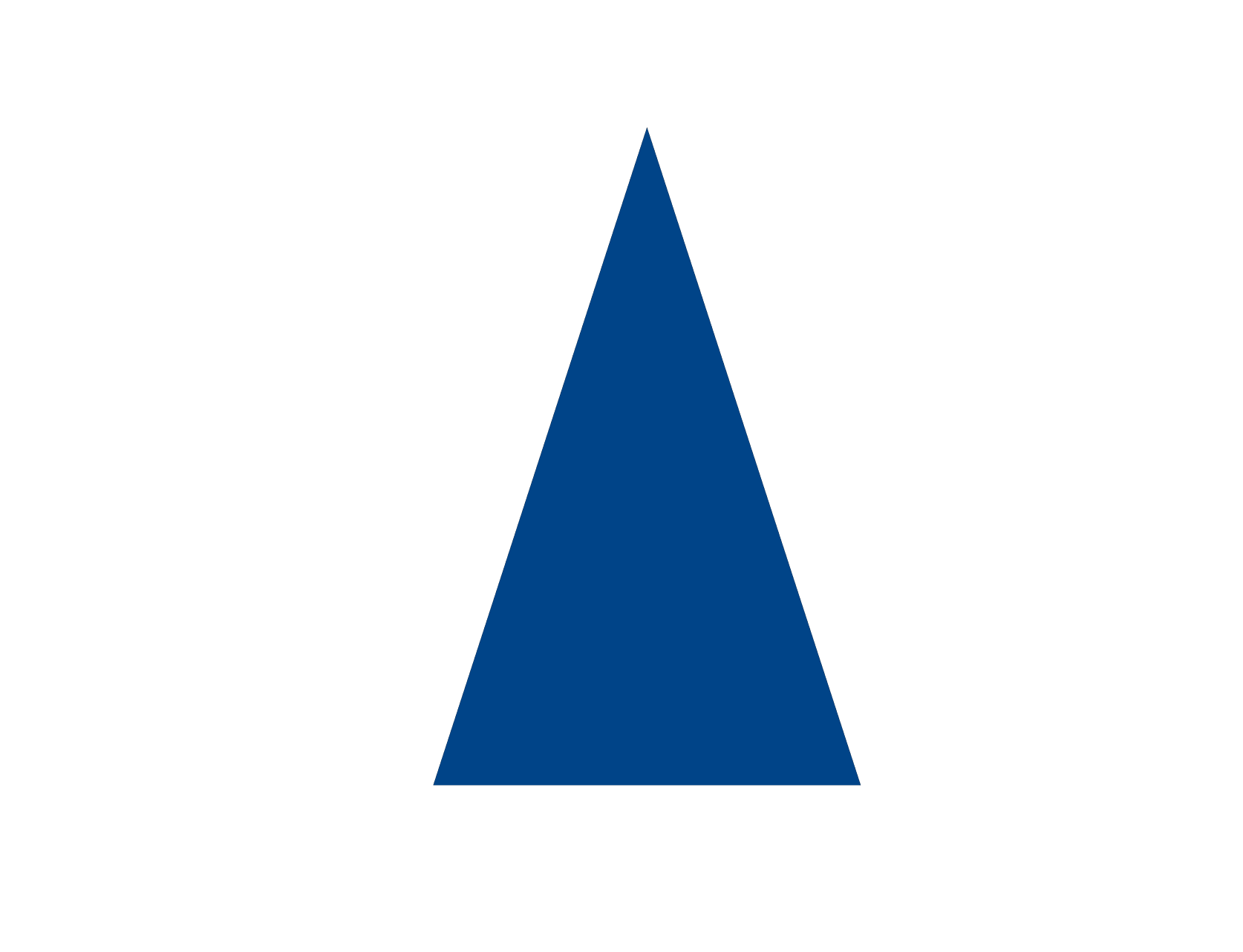}\quad
\raisebox{8pt}{\includegraphics[width=1.25in]{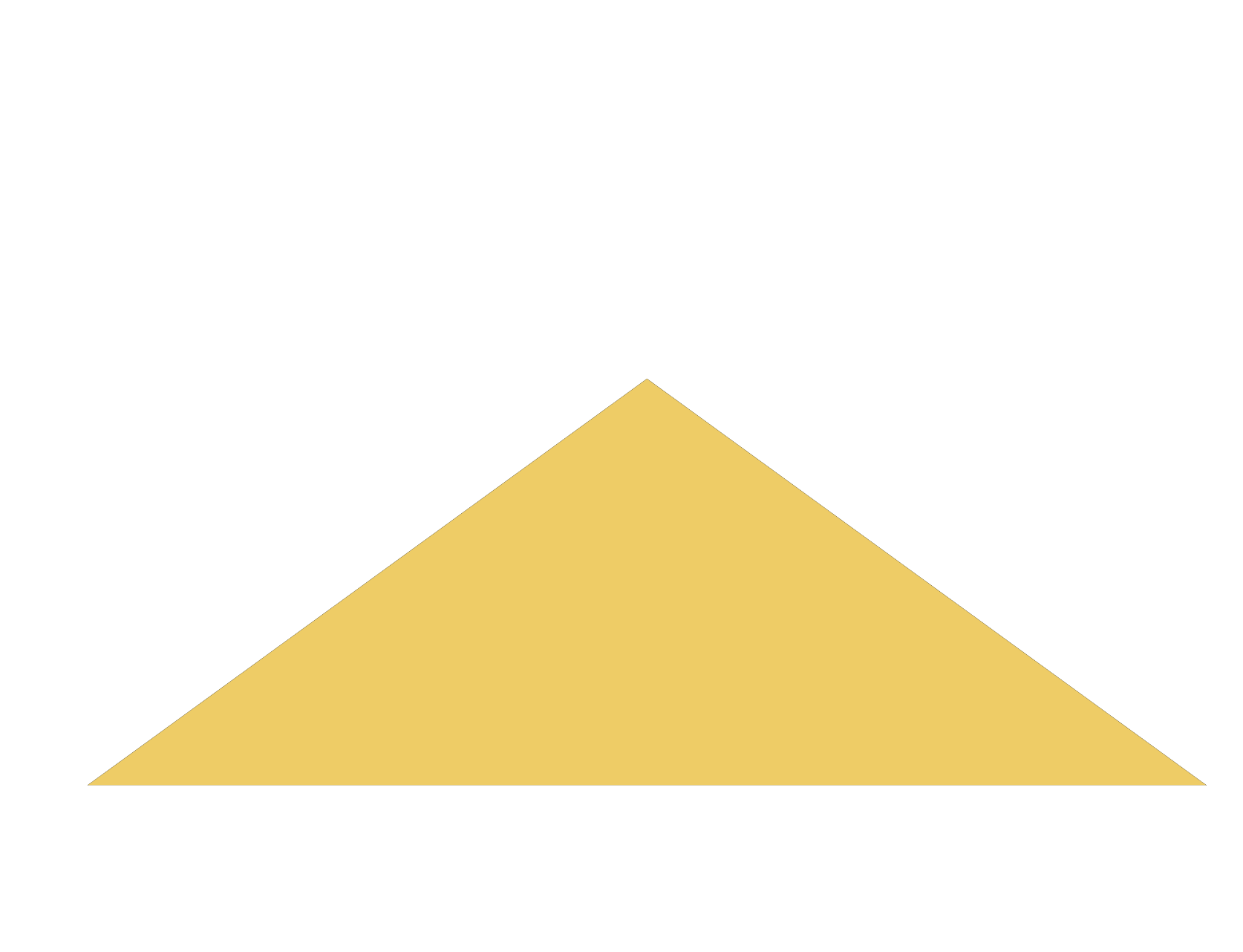}}\quad
\includegraphics[width=1.25in]{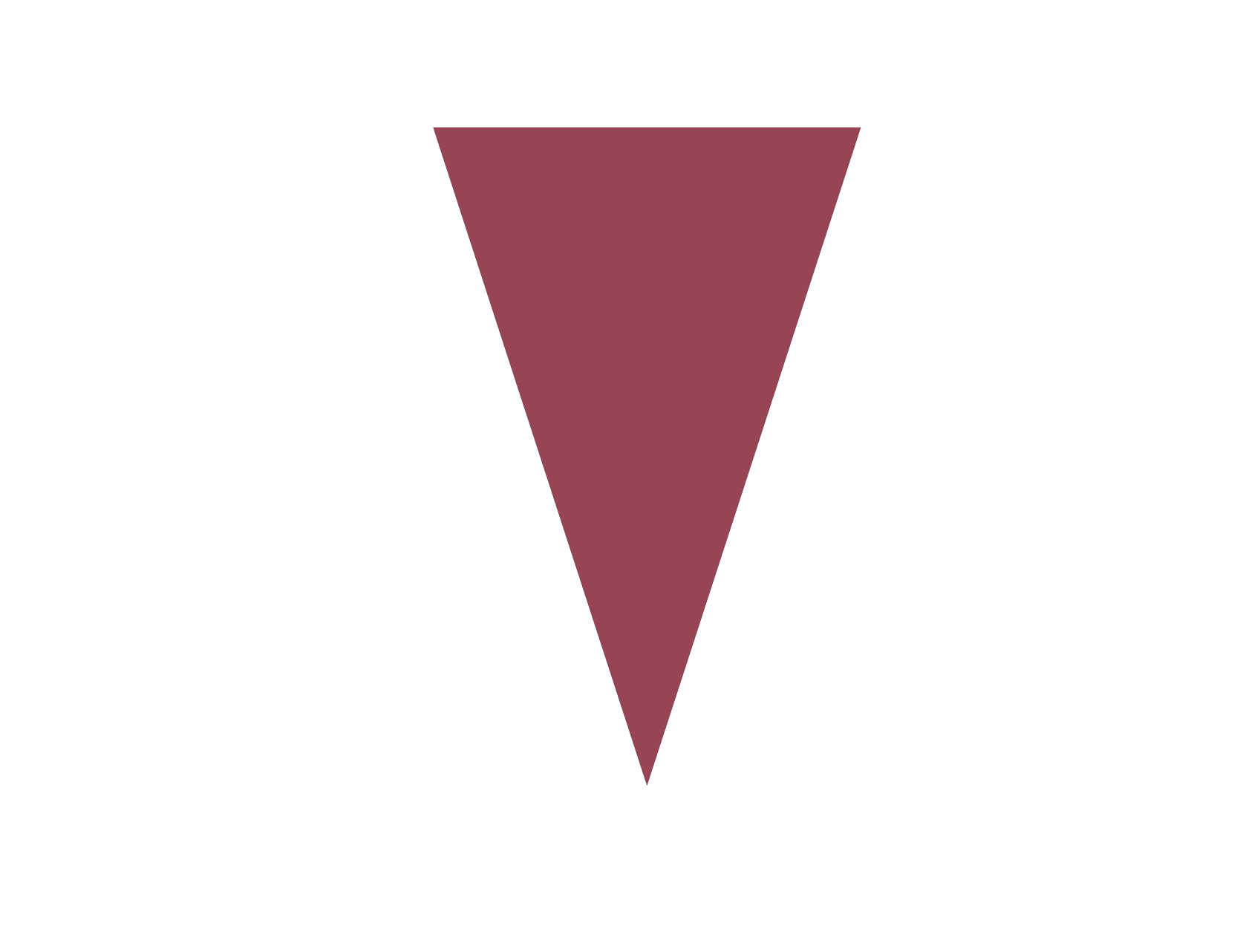}\quad
\raisebox{8pt}{\includegraphics[width=1.25in]{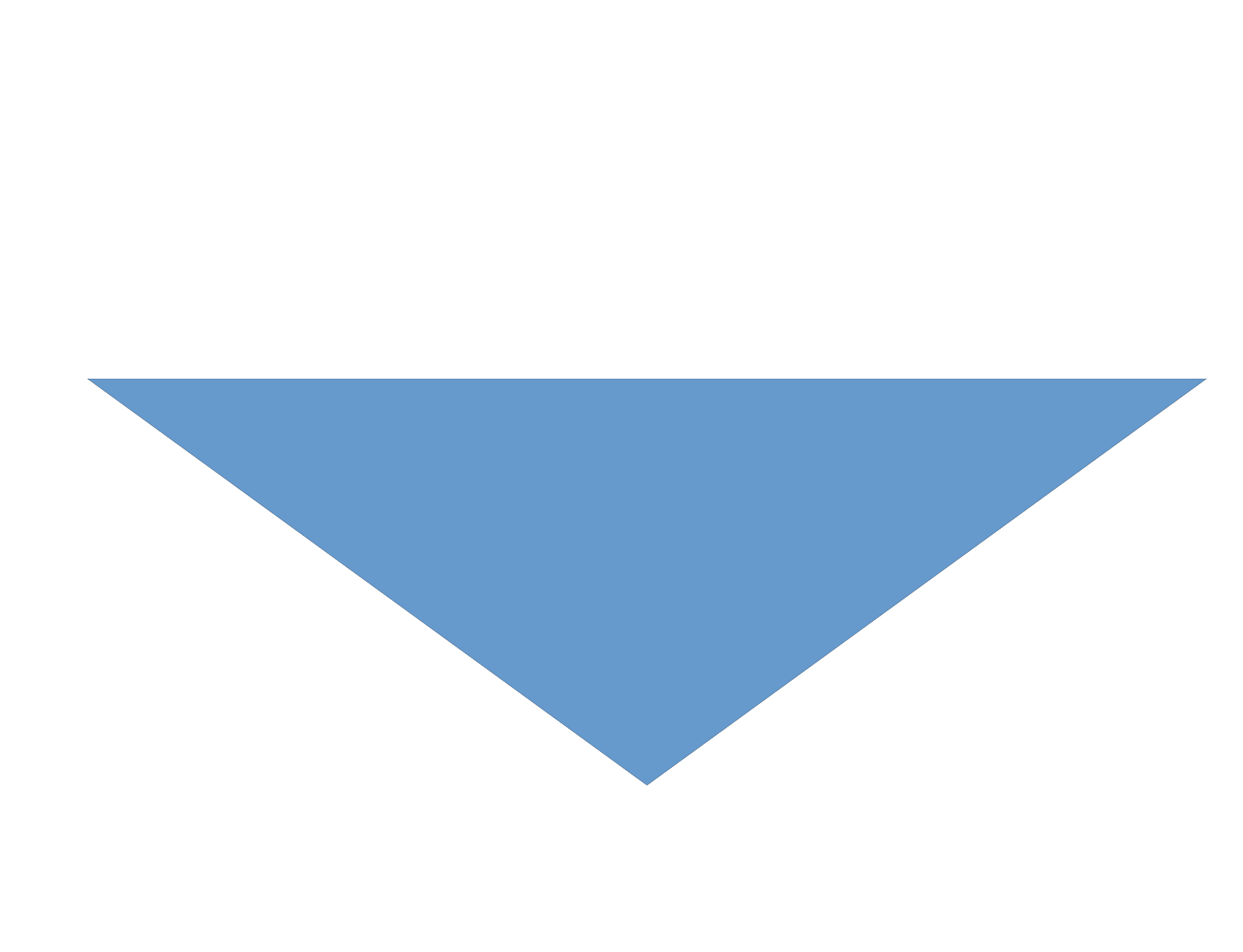}}
\end{center}

The substitution rules are (redrawing \cite[Fig.~13]{Frank2008Expo}):

\begin{center}
\includegraphics[width=1.25in]{Images/triangle_sub_4a}\hspace{-2em}
\includegraphics[width=1.25in]{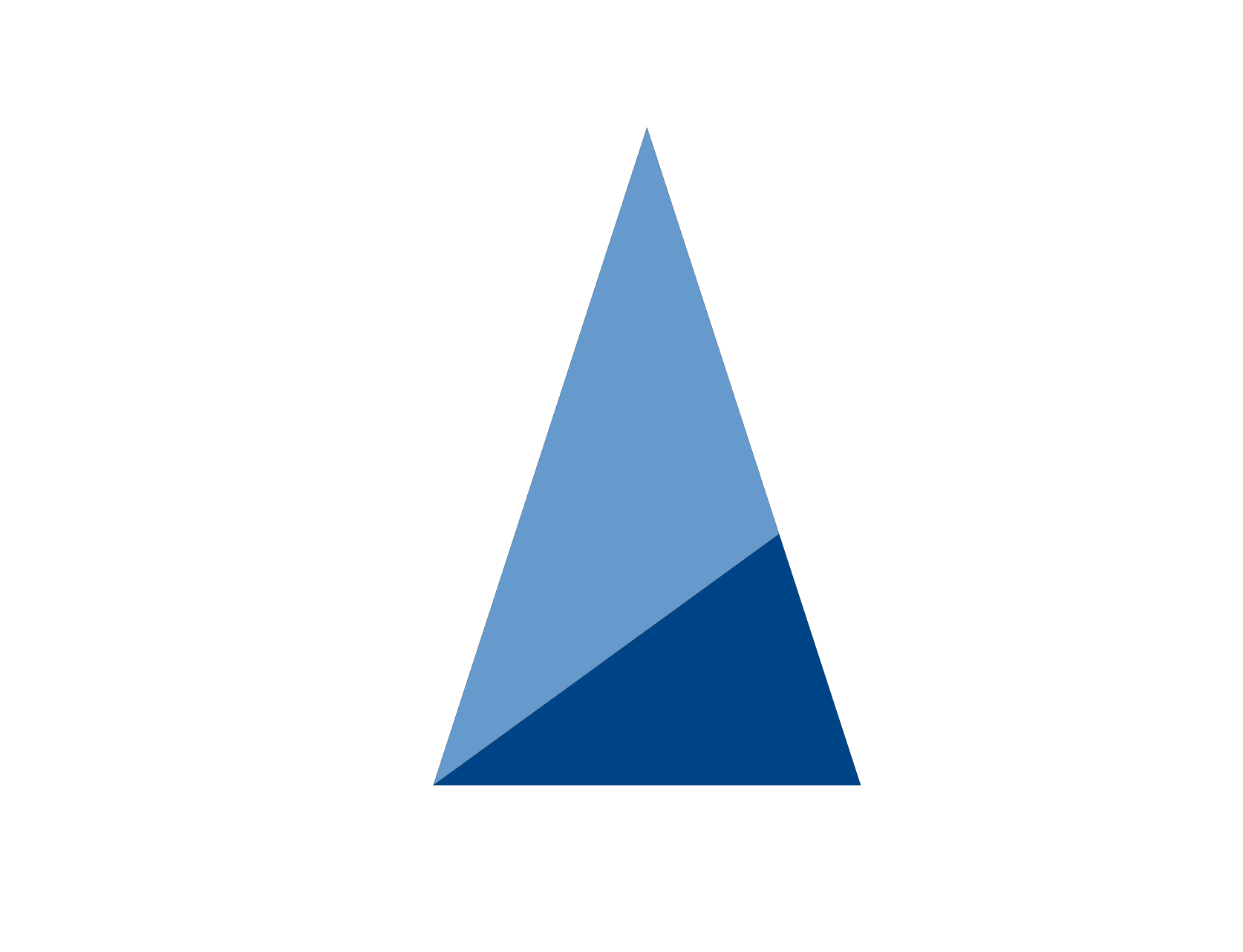}
\qquad
\includegraphics[width=1.25in]{Images/triangle_sub_1a}\hspace*{2em} 
\includegraphics[width=1.25in]{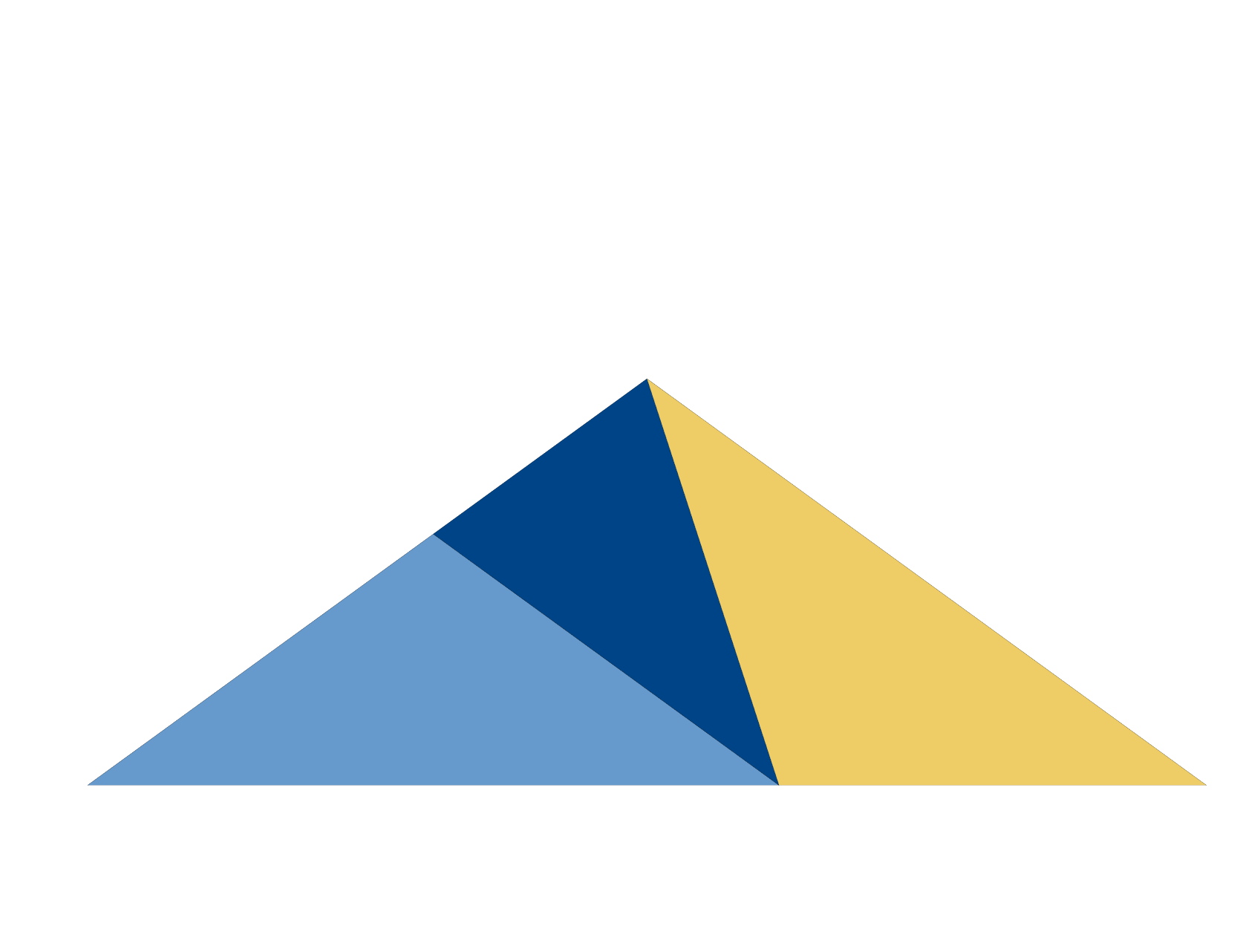}

\begin{picture}(0,0)
\put(-122,40){\large $\to$}
\put( 92,40){\large $\to$}
\end{picture}

\vspace*{-1em}
\includegraphics[width=1.25in]{Images/triangle_sub_3a}\hspace{-2em}
\includegraphics[width=1.25in]{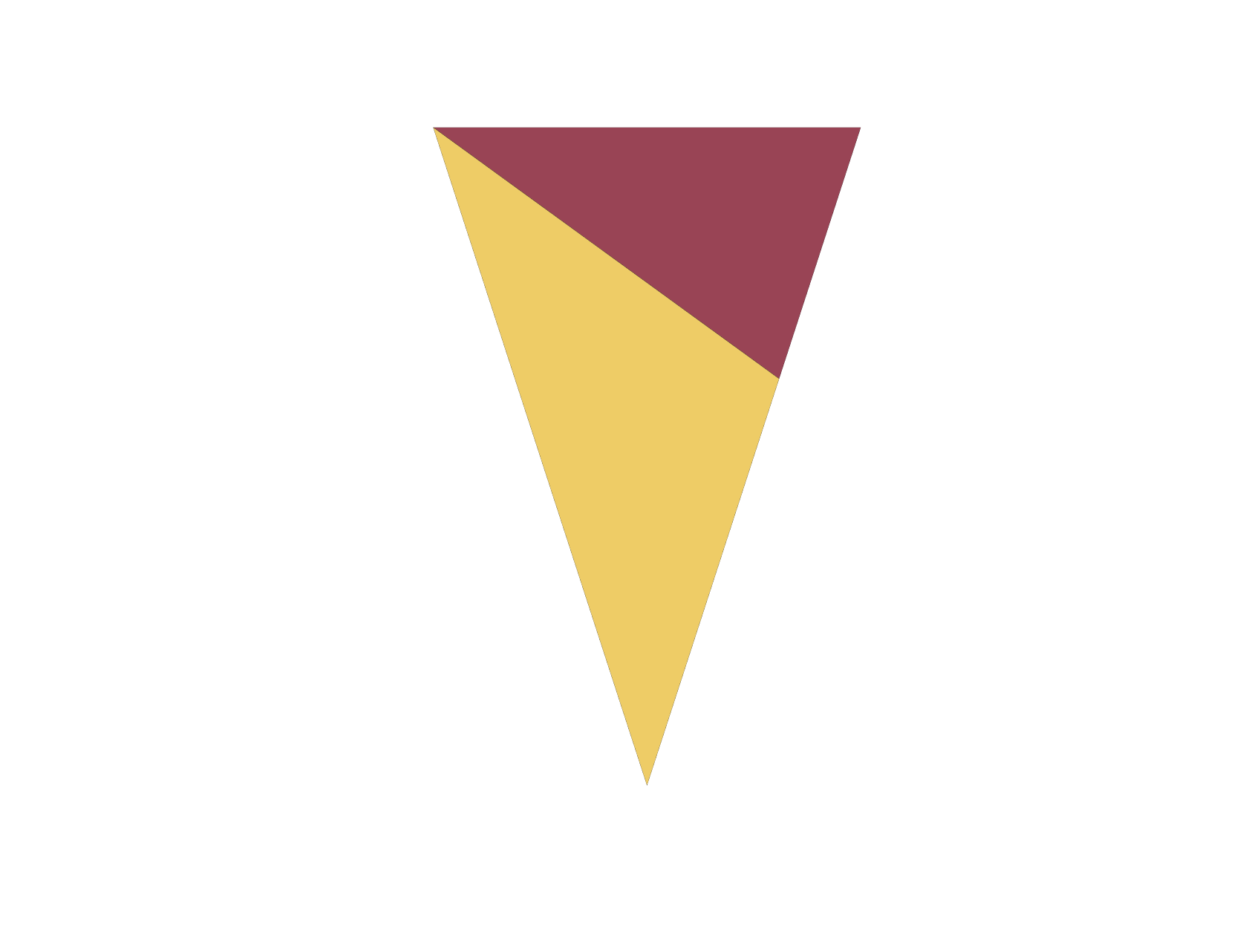}
\qquad
\includegraphics[width=1.25in]{Images/triangle_sub_2a}\hspace*{2em} 
\includegraphics[width=1.25in]{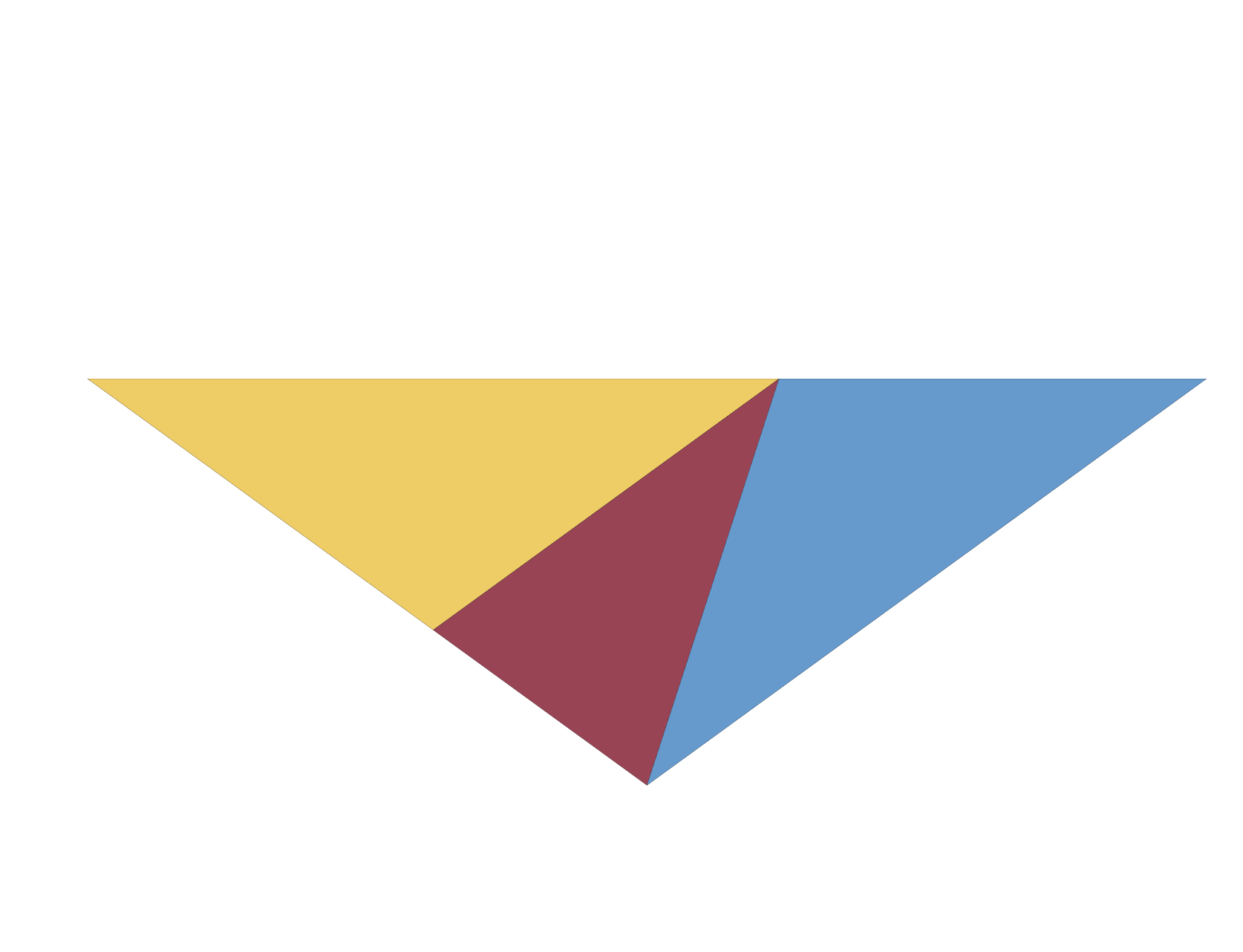}

\begin{picture}(0,0)
\put(-122,40){\large $\to$}
\put( 92,40){\large $\to$}
\end{picture}
\end{center}

We denote this substitution by $\subst_\robinson$. Abusing notation somewhat, we write $M_\robinson := M_{\subst_{\robinson}}$ for its substitution matrix. Ordering the tiles from left-to-right as above, 
\[ M_\robinson = \begin{bmatrix} \ 1 & 1 & 0 & 0\  \\ \ 0 & 1 & 1 & 1\  \\ \ 0 & 0 & 1 & 1\  \\ \ 1 & 1 & 0 & 1\ \end{bmatrix}.\]
\end{definition}

Recall that the graph $\Gamma = (\calV,\calE)$ associated with a tiling $\calT = \{T_i : i \in I\}$ has vertex set $\calV = \{T_i : i \in I\}$ and edges between two vertices if and only if the corresponding tiles share an edge. 

Let us now describe more precisely the mechanism that enables one to estimate the discontinuity in the IDS.\ \ The following result is well known, but we make it explicit for the reader's benefit. Throughout this discussion, fix a substitution tiling $\calT$ and associated graph $\Gamma = (\calV,\calE)$.\ \ If $\calV_0 \subseteq \calV$ is a finite patch, we denote its boundary by $\partial \calV_0$, which consists of all the tiles in $\calV_0$ that share an edge with a tile in $\calV \setminus \calV_0$. 
A priori, one may be concerned that the degree of the tiles in $\partial \calV_0$ are ill-defined. In the specific patches we consider in this paper, one may verify directly that this is not the case.

\begin{definition}
We say $P \subseteq \calV$ is a \emph{good eigenfunction support} at energy $E$ if 
\begin{enumerate}
\item $P$ is finite; 
\item there is a nontrivial eigenfunction $\psi$ of $\Delta_\Gamma$ with $\Delta_\Gamma \psi = E\psi$ and $\supp(\psi) = P$;
\item no proper subset of $P$ enjoys the previous property;
\item every occurrence of $P$ in $\calV$ supports an eigenfunction.
\end{enumerate}
\end{definition}

\begin{prop}  Suppose $P \subseteq \calV$ is a good eigenfunction support at energy $E$. For any finite patch $\Gamma_0 = (\calV_0,\calE_0)$, $\calV_0 \subseteq \calV$, the multiplicity of $E$ for $\Delta_{\Gamma_0}$ is bounded from below by the largest cardinality of a set of occurrences of $P$ in $\calV_0$ with the following properties: no occurrence intersects $\partial\calV_0$ and no occurrence is contained in the union of other occurrences.
\end{prop}

\begin{proof}
Choose a collection of occurrences of $P$ having the enumerated properties. The definitions ensure that each occurrence of $P$ yields an eigenfunction and that the collection of these eigenfunctions is linearly independent.
\end{proof}




\medskip
\section{Boats and Stars}

\subsection{Basics}

\begin{definition} \label{def:boatStarRules}
Following \cite{GrunbaumShephard1987}, the \emph{boat--star substitution} has six basic tiles:

\begin{center}
\includegraphics[width=1.25in]{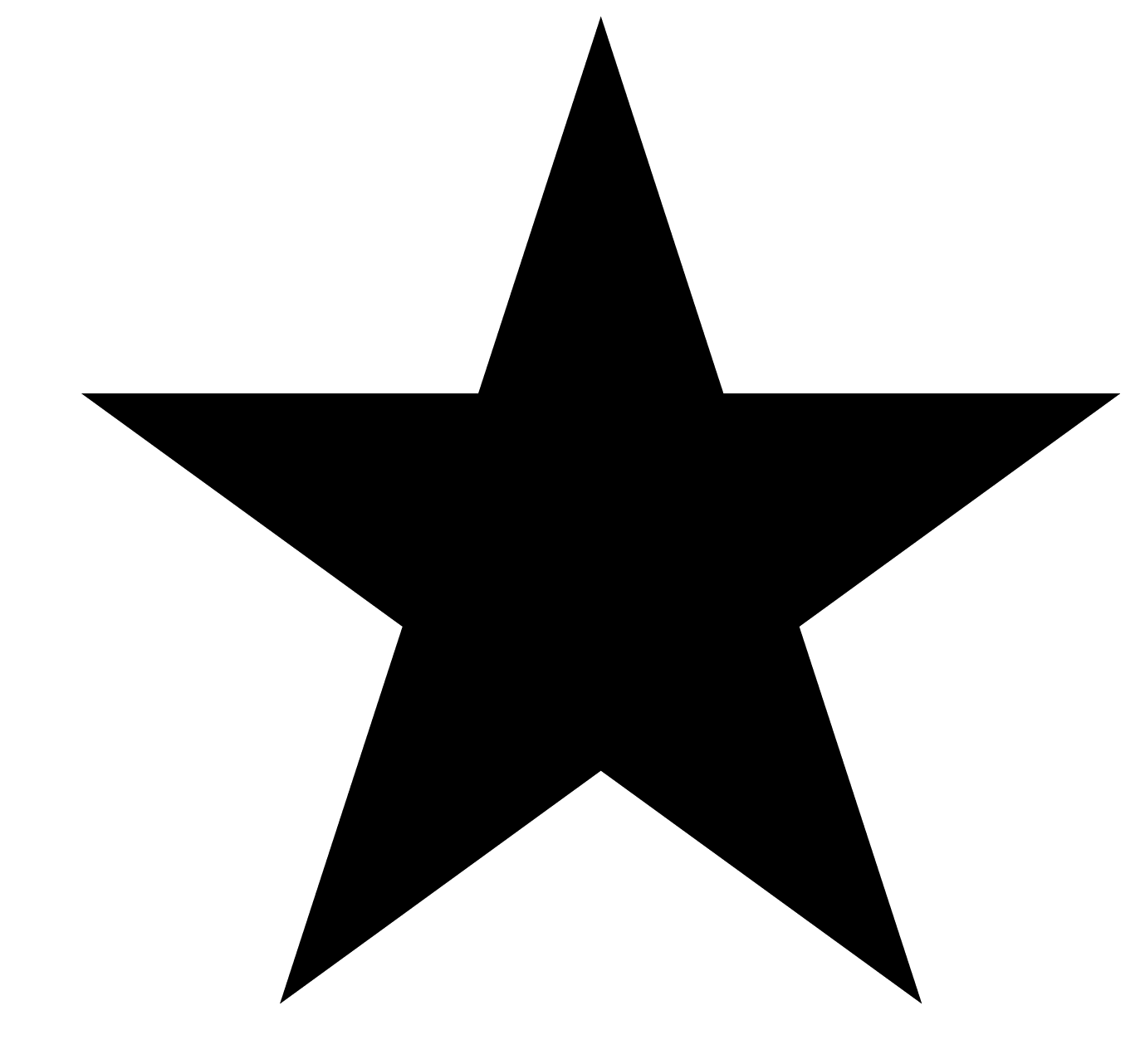}\quad
\includegraphics[width=1.25in]{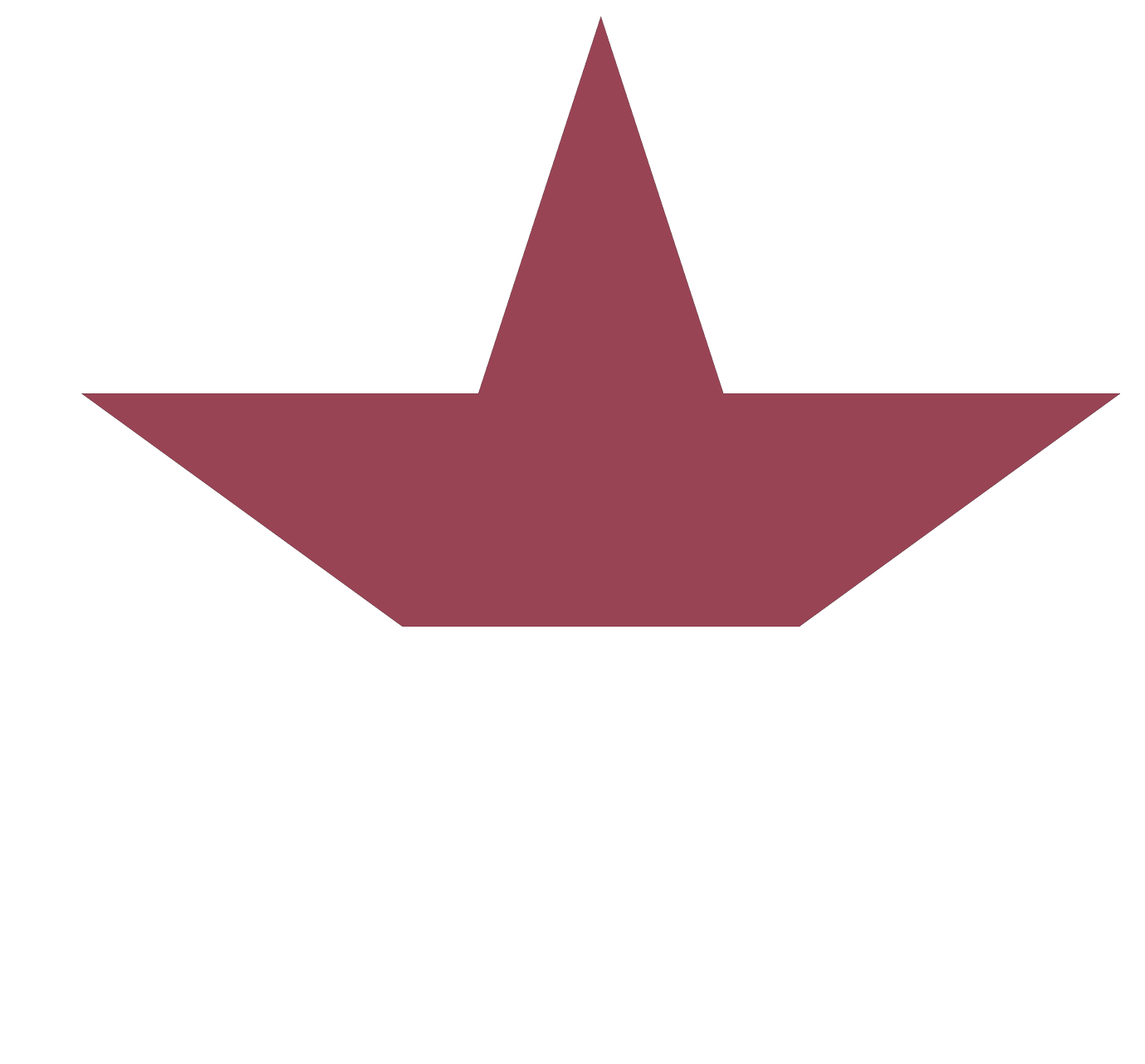}\quad
\includegraphics[width=1.25in]{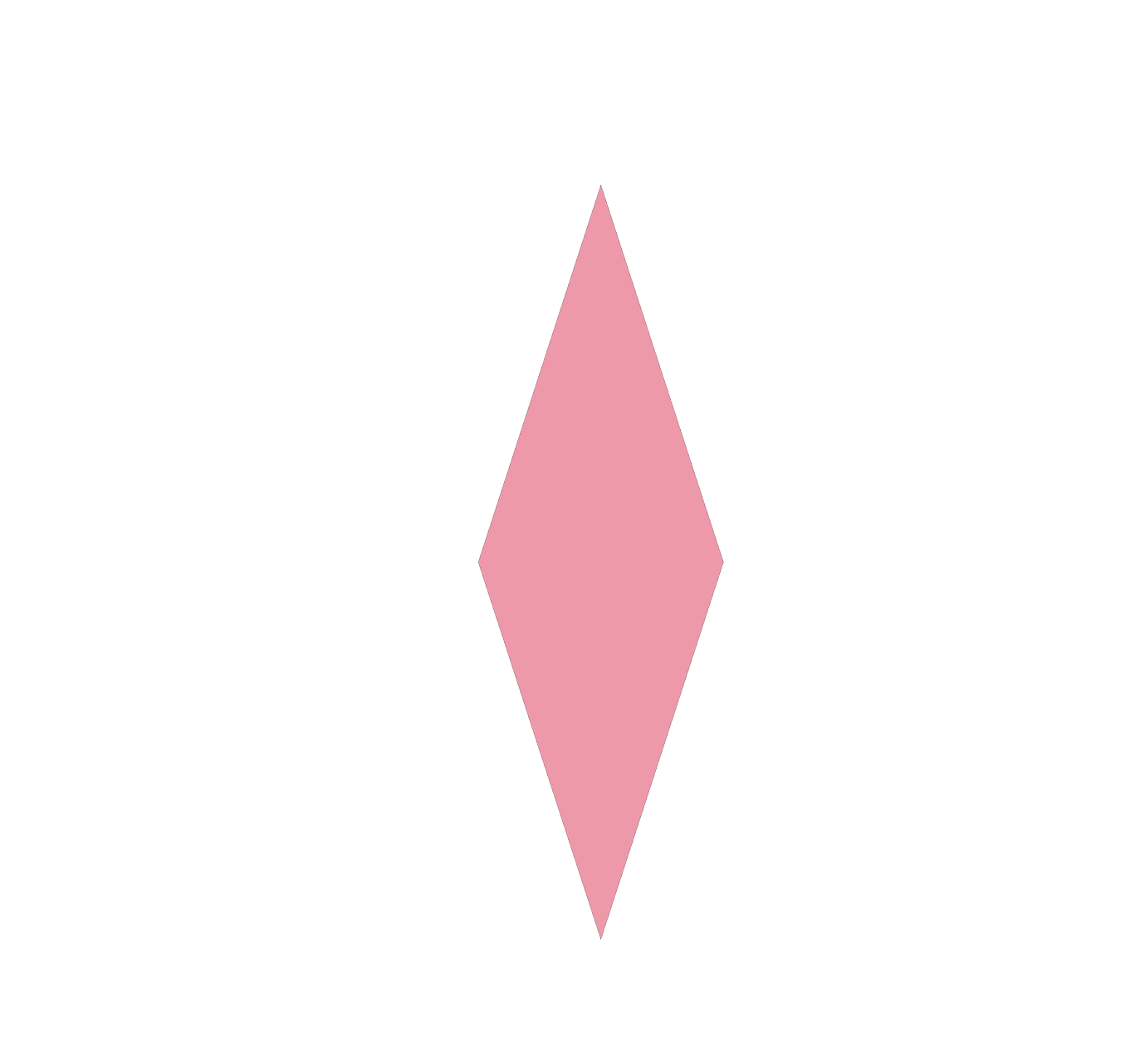}

\includegraphics[width=1.25in]{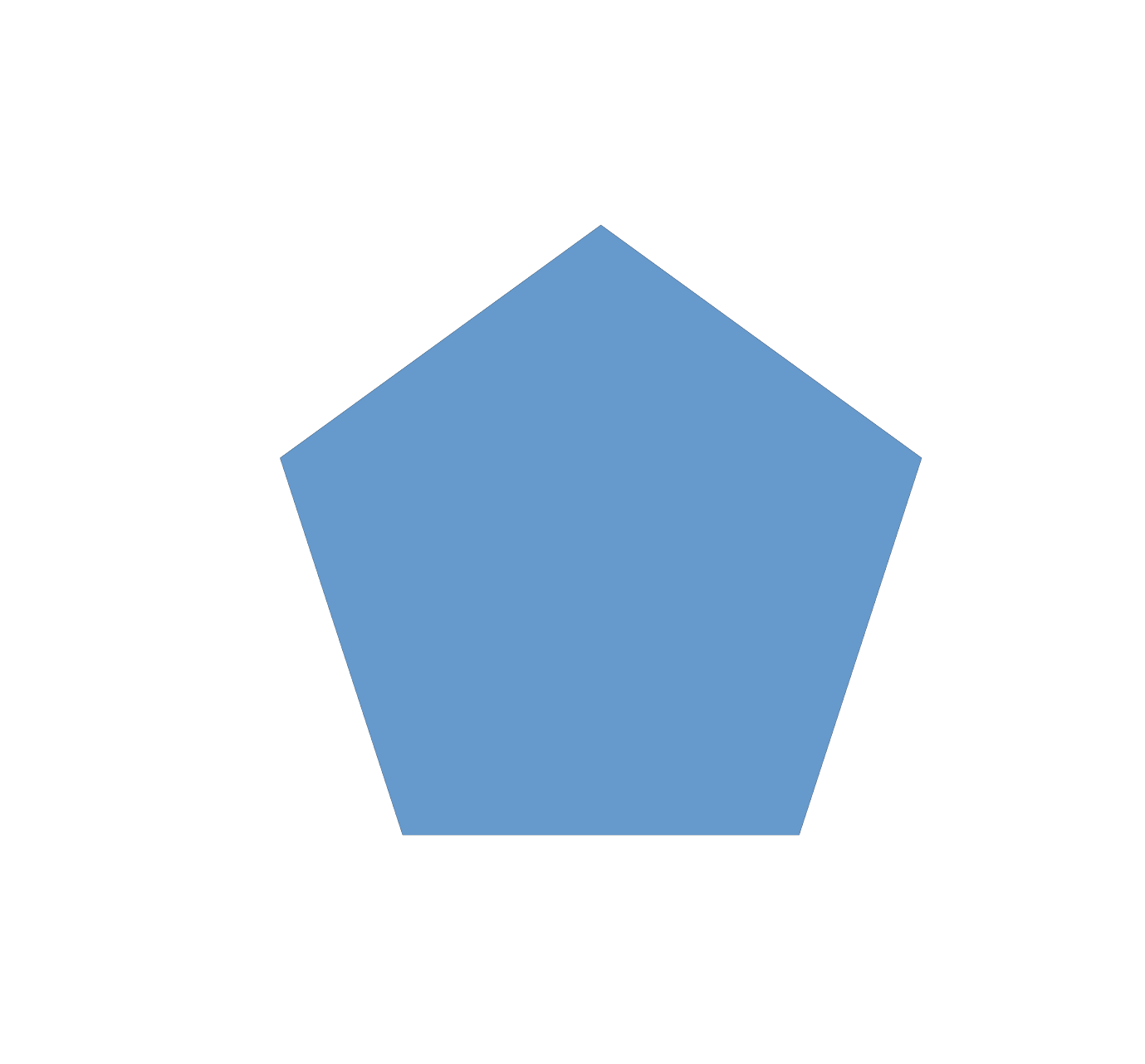}\quad
\includegraphics[width=1.25in]{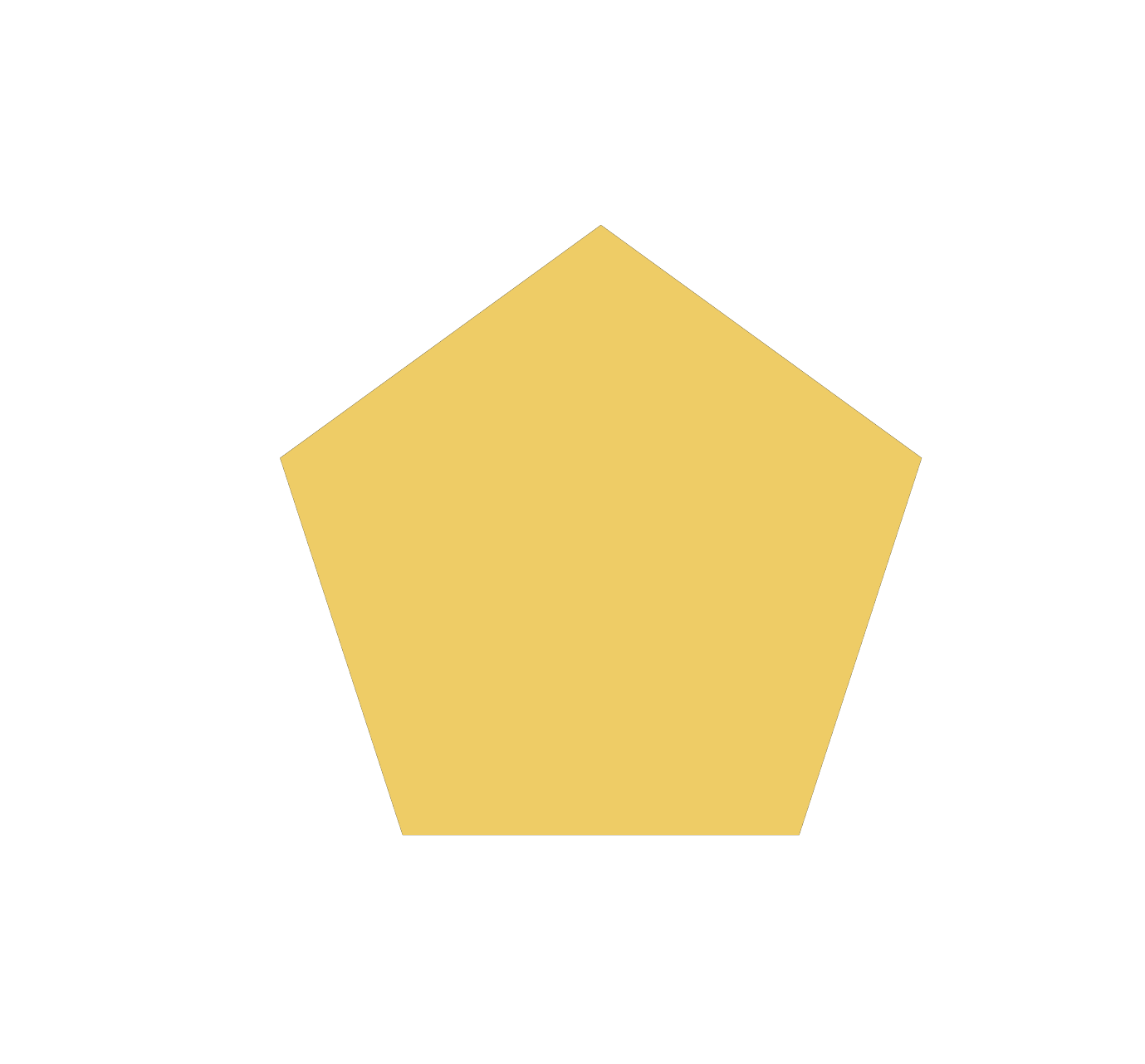}\quad
\includegraphics[width=1.25in]{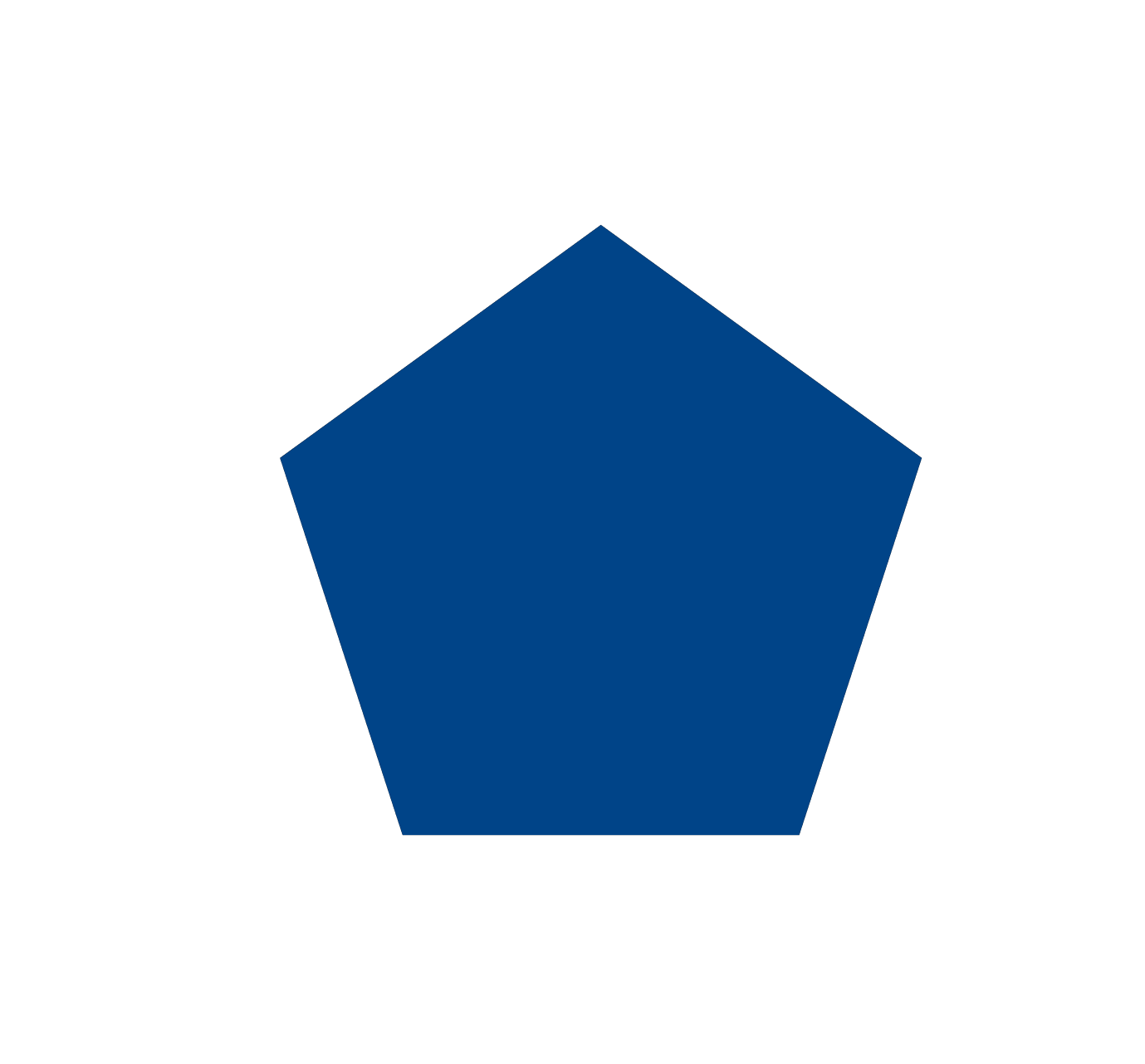}
\end{center}

The substitution rules are:

\begin{center}
\includegraphics[width=1.in]{Images/boat_star_sub1a}
\hspace*{0.25em}\raisebox{30pt}{\large $\to$} \hspace*{0.25em}
\includegraphics[width=1.in]{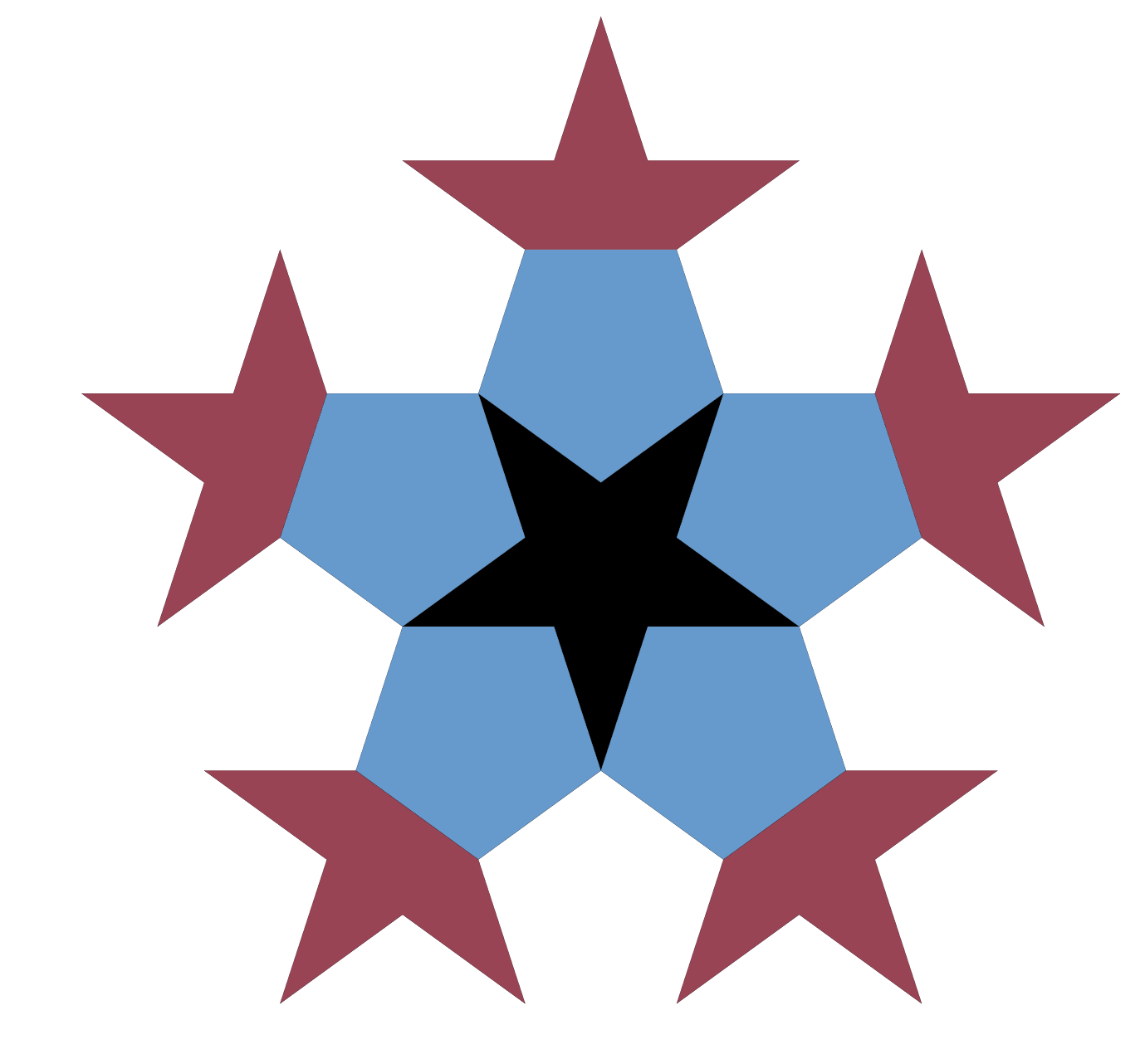}
\qquad
\includegraphics[width=1.in]{Images/boat_star_sub2a}
\hspace*{0.25em}\raisebox{30pt}{\large $\to$} \hspace*{0.25em}
\includegraphics[width=1.in]{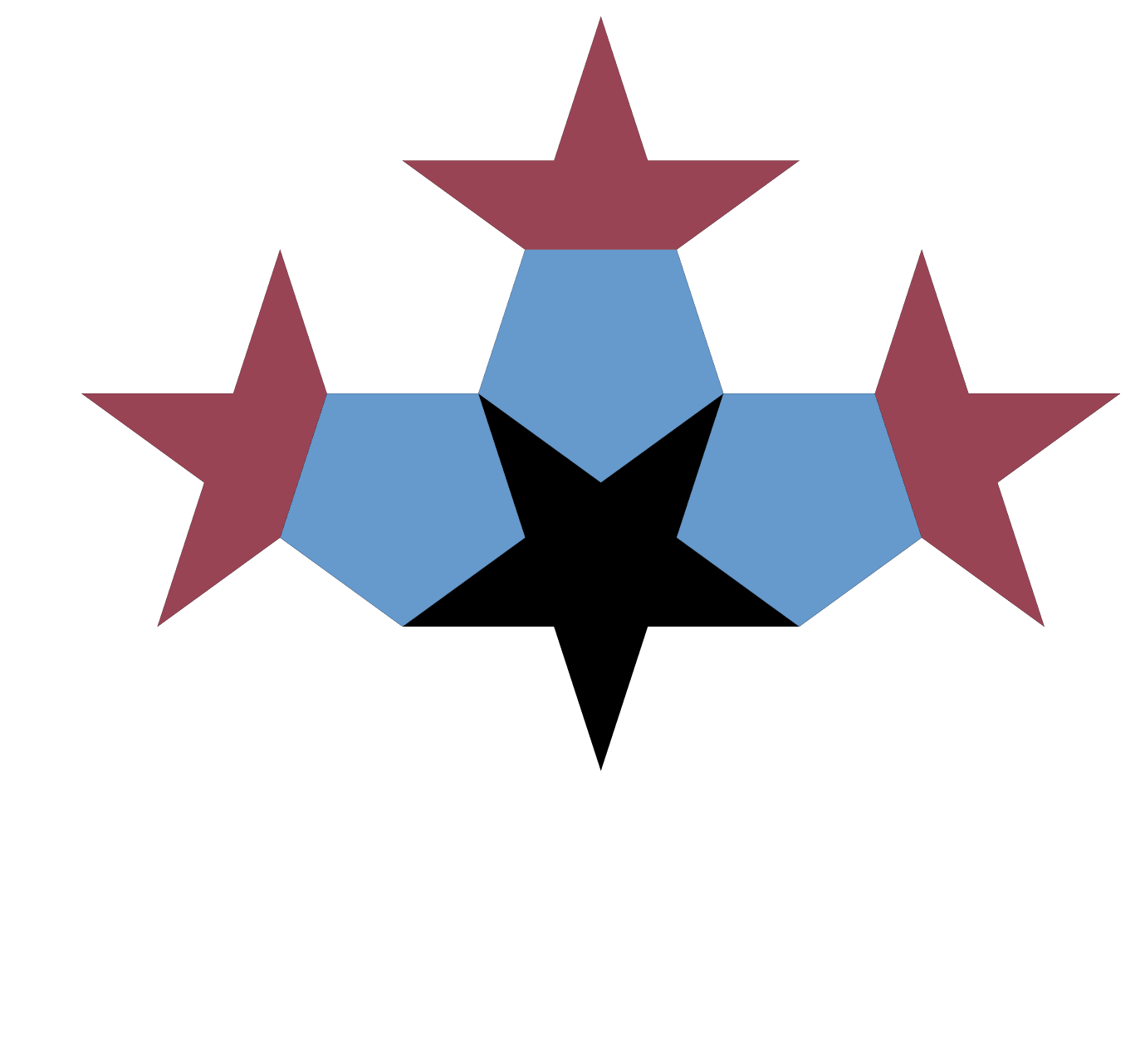}

\includegraphics[width=1.in]{Images/boat_star_sub3a}
\hspace*{0.25em}\raisebox{30pt}{\large $\to$} \hspace*{0.25em}
\includegraphics[width=1.in]{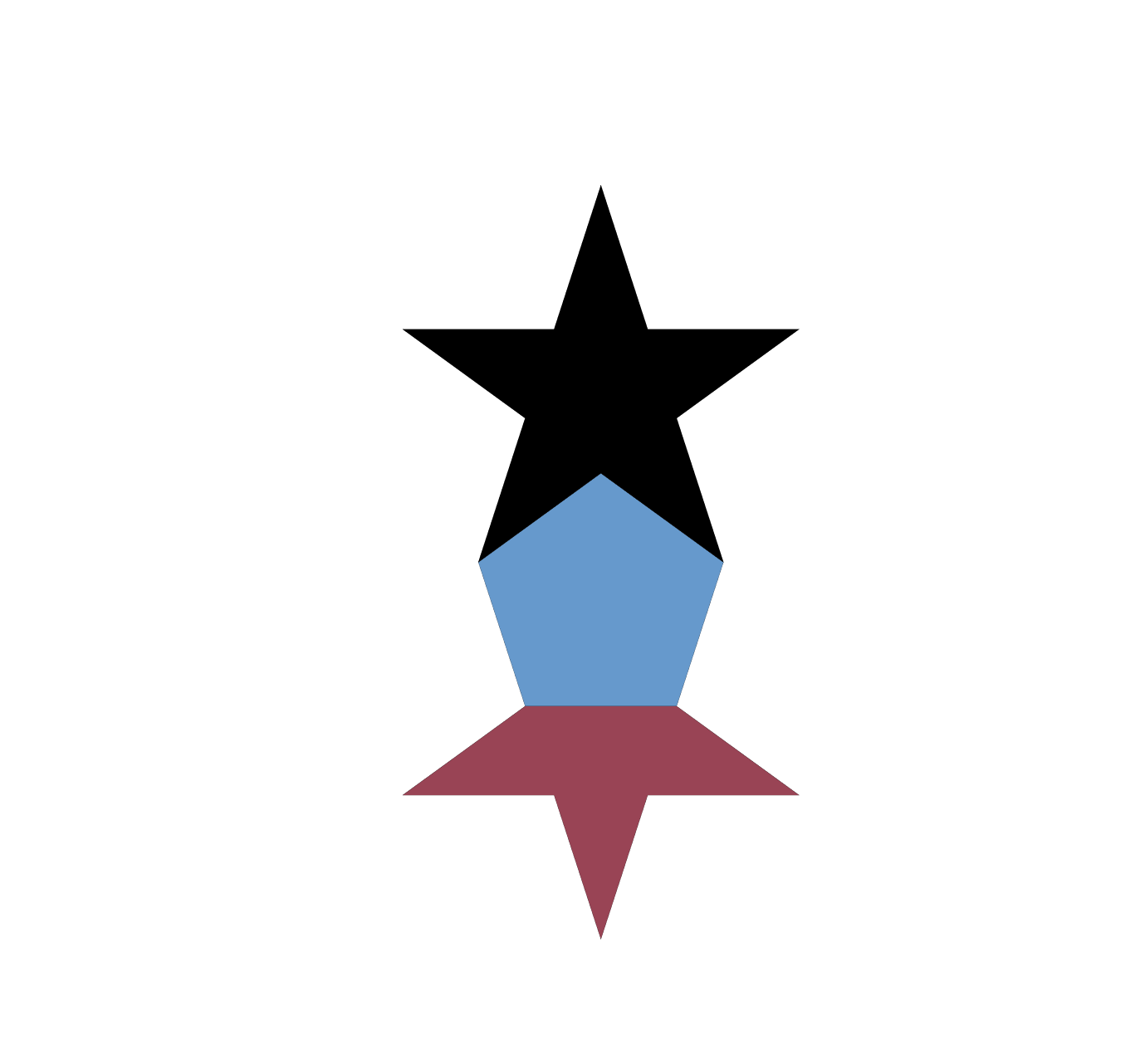}
\qquad
\includegraphics[width=1.in]{Images/boat_star_sub4a}
\hspace*{0.25em}\raisebox{30pt}{\large $\to$} \hspace*{0.25em}
\includegraphics[width=1.in]{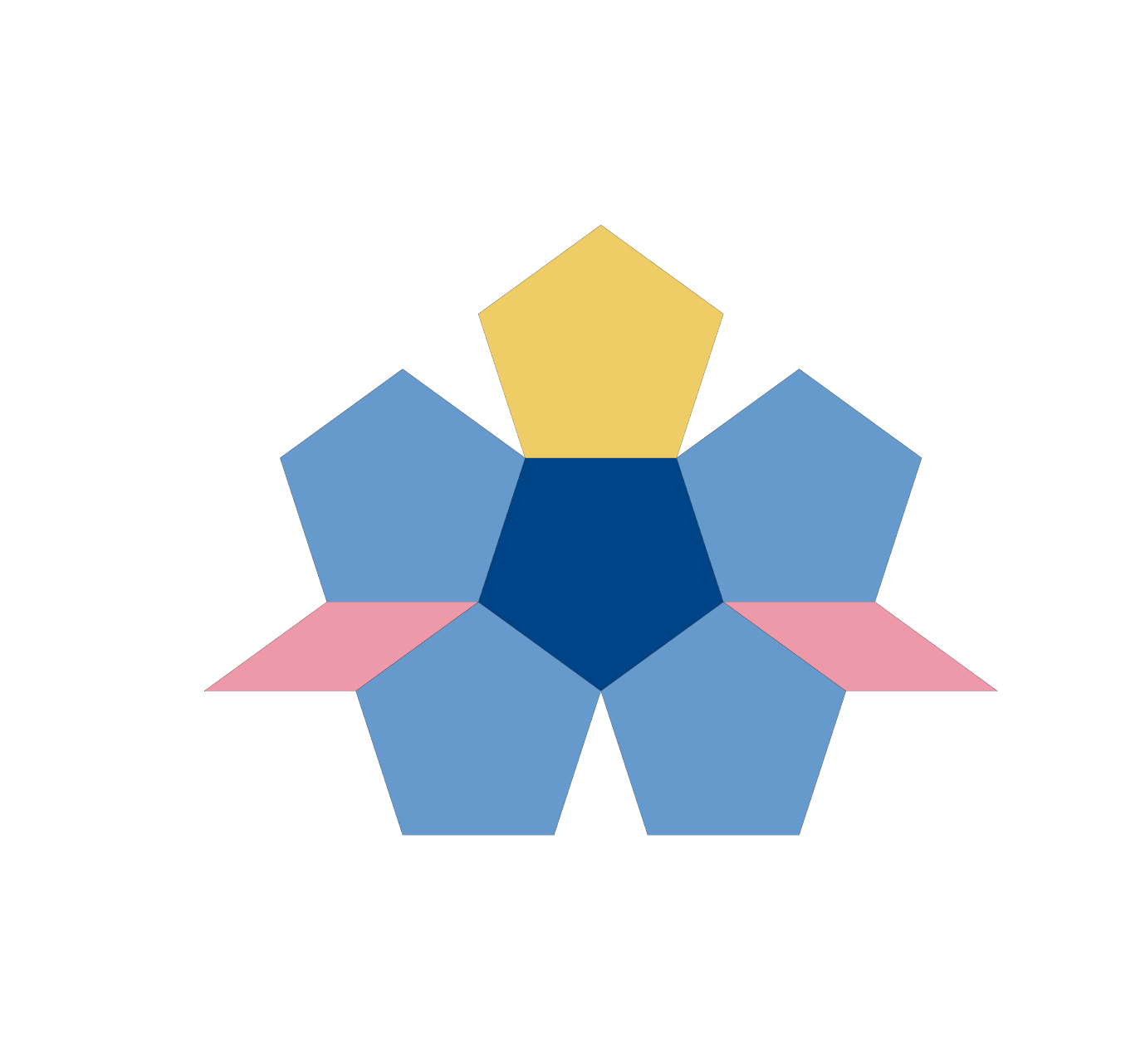}

\includegraphics[width=1.in]{Images/boat_star_sub5a}
\hspace*{0.25em}\raisebox{30pt}{\large $\to$} \hspace*{0.25em}
\includegraphics[width=1.in]{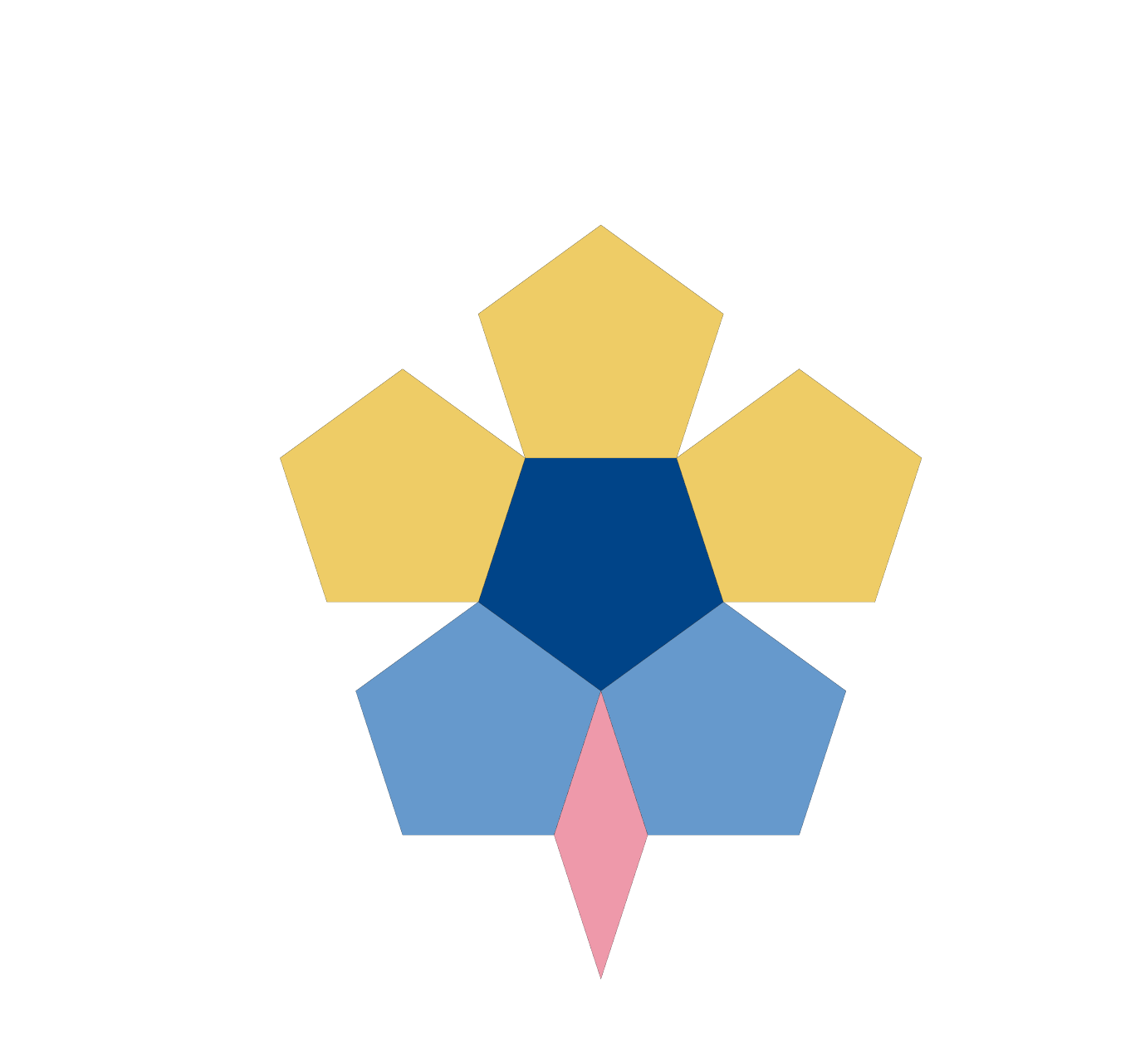}
\qquad
\includegraphics[width=1.in]{Images/boat_star_sub6a}
\hspace*{0.25em}\raisebox{30pt}{\large $\to$} \hspace*{0.25em}
\includegraphics[width=1.in]{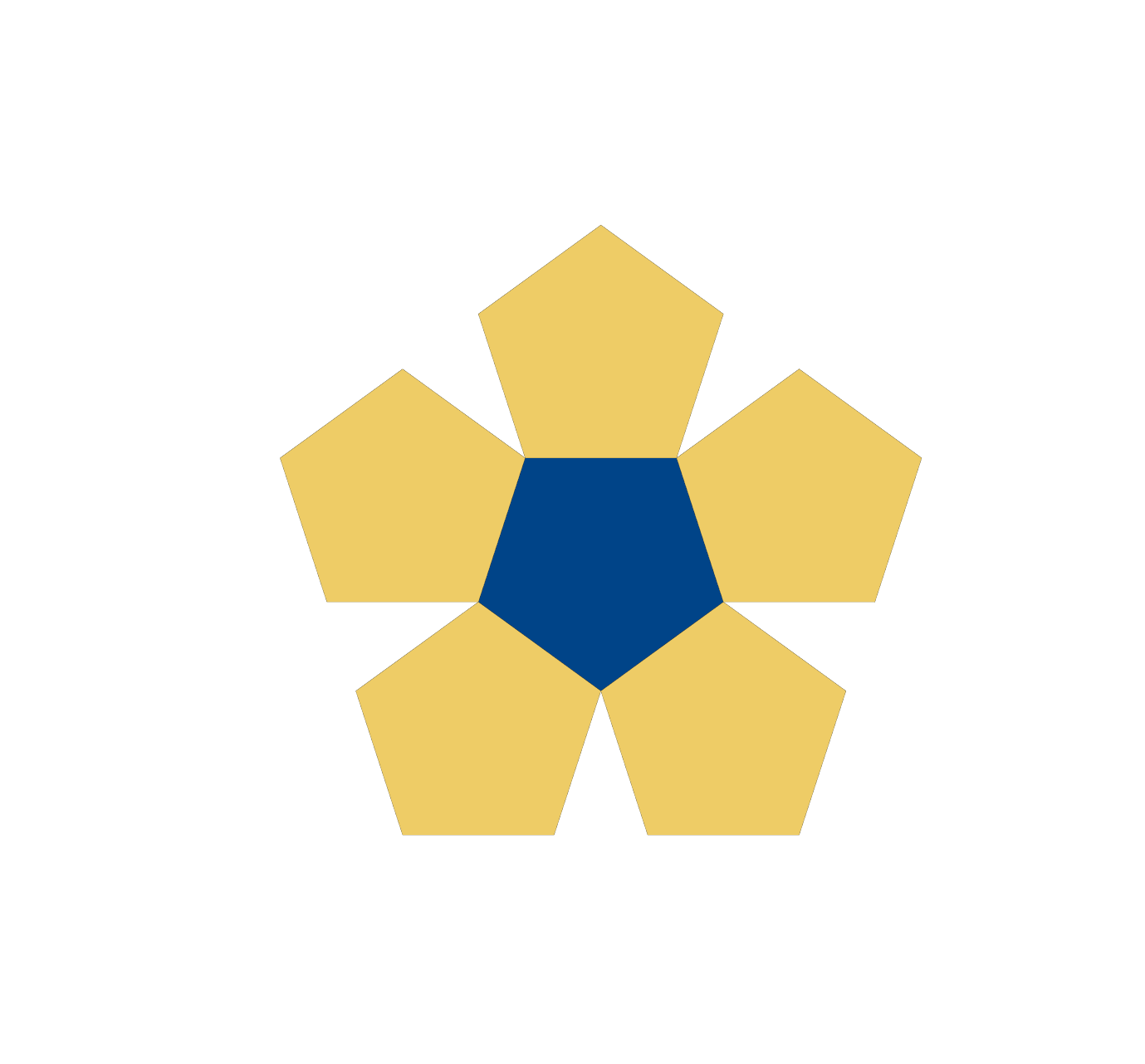}
\end{center}
We denote this substitution by $\subst_\boatstar$.
\end{definition}

There are six tiles: stars, boats, diamonds, and three types of pentagons. Ordering the tiles as in Definition~\ref{def:boatStarRules}, one can see that the substitution matrix for the boat--star substitution is
\[M_\boatstar = 
\begin{bmatrix} 
1 & 1 & 1 & 0 & 0 & 0 \\
5 & 3 & 1 & 0 & 0 & 0 \\
0 & 0 & 0 & 2 & 1 & 0 \\
5 & 3 & 1 & 4 & 2 & 0 \\
0 & 0 & 0 & 1 & 3 & 5 \\
0 & 0 & 0 & 1 & 1 & 1 \end{bmatrix}.\]

Let $\calT_0^\boatstar$ denote the pattern consisting of a single star tile and $\calT_n^\boatstar =\subst^n_{\boatstar}(\calT_0^\boatstar)$. Denote the golden ratio by
\[\varphi = \frac{\sqrt{5}+1}{2}.\]

\begin{lemma} \label{lem:boatStar:tilecount}
For each $n \geq 0$, the total number of tiles in $\calT_n^\boatstar$ is
\[P_\boatstar(n) = \frac{1}{22} (25+9\sqrt{5}) \varphi^{4n}
 - \frac{5}{33} 4^{n+1} - \frac{2}{3} + \frac{1}{22} (25-9\sqrt{5}) \varphi^{-4n},\]
 of which precisely
 \[P_{\textup{\pentagon}}(n) = \frac{1}{22} (17+7\sqrt{5}) \varphi^{4n}
 - \frac{40}{33} 4^{n} - \frac{1}{3} + \frac{1}{22}(17-7\sqrt{5}) \varphi^{-4n}\]
 are pentagons.
\end{lemma}

\begin{proof}
Observe that the substitution matrix $M_\boatstar$ has eigenvalues $\varphi^{ 4},4,1,\varphi^{-4},0,0$ with corresponding eigenfunctions (listed in the same order)
\[v_1 = \frac{1}{2} \begin{bmatrix} 3-\sqrt{5} \\ -5 + 3\sqrt{5} \\ 5 - \sqrt{5} \\ 5+\sqrt{5} \\ 2\sqrt{5} \\ 2 \end{bmatrix}, \quad v_2 = \begin{bmatrix} -1 \\ -4 \\ 1 \\ -2\\ 8 \\ 2 \end{bmatrix}, \quad
v_3 = \begin{bmatrix} 1 \\ -5 \\ 5 \\5 \\-5 \\ 1 \end{bmatrix},\]
\[v_4 = \frac{1}{2} \begin{bmatrix} 3+\sqrt{5} \\ -5 - 3\sqrt{5} \\ 5 + \sqrt{5} \\ 5 - \sqrt{5} \\ -2\sqrt{5} \\ 2 \end{bmatrix}, \quad v_5 = \begin{bmatrix} 1 \\ -2 \\ 1 \\ 0\\ 0 \\ 0 \end{bmatrix}, \quad
v_6 = \begin{bmatrix}0\\ 0 \\ 0 \\  1 \\ -2 \\ 1 \end{bmatrix}.\]
Since the tiling begins with a single star, the total number of tiles at stage $n$ is precisely $\langle w, M_\boatstar^n e_1\rangle$ where $w = \begin{bmatrix} 1 & 1 & 1 & 1 & 1 & 1 \end{bmatrix}^\top$.
Decomposing $e_1$ in the basis of eigenfunctions of $M_\boatstar$, one observes
\[e_1 = \frac{1}{22}(7-\sqrt{5})v_1 - \frac{5}{33}v_2 - \frac{1}{3} v_3 + \frac{1}{22} (7+\sqrt{5})v_4.\]
Calculate
\[\langle w,v_1 \rangle = 5 + 2\sqrt{5},
\quad \langle w,v_2 \rangle = 4,
\quad \langle w,v_3 \rangle = 2,
\quad \langle w,v_4 \rangle = 5 - 2\sqrt{5}. \]
Thus, the total number of tiles at stage $n$ is
\begin{align*}
\langle w, M_\boatstar^n e_1 \rangle
&= \frac{1}{22} (7-\sqrt{5})(5+2\sqrt{5}) \varphi^{4n}
 - \frac{5}{33} 4^{n+1} - \frac{2}{3} + \frac{1}{22} (7+\sqrt{5})(5-2\sqrt{5}) \varphi^{-4n} \\
& = \frac{1}{22} (25+9\sqrt{5}) \varphi^{4n}
 - \frac{5}{33} 4^{n+1} - \frac{2}{3} + \frac{1}{22} (25-9\sqrt{5}) \varphi^{-4n}.
\end{align*}
Similarly, to count pentagons let $u= \begin{bmatrix} 0 & 0 & 0 & 1 & 1 & 1 \end{bmatrix}^\top$ and compute
\[\langle u,v_1 \rangle = \frac{1}{2}(7 + 3\sqrt{5}),
\quad \langle u,v_2 \rangle = 8,
\quad \langle u,v_3 \rangle = 1,
\quad \langle u,v_4 \rangle = 7 - 3\sqrt{5}. \]
Thus, the total number of pentagons at stage $n$ is
\begin{align*}
\langle u, M_\boatstar^n e_1 \rangle
&= \frac{1}{22} (7-\sqrt{5})\frac{1}{2}(7 + 3\sqrt{5}) \varphi^{4n}
 - \frac{40}{33} 4^{n} - \frac{1}{3} + \frac{1}{22} (7+\sqrt{5})\frac{1}{2}(7 - 3\sqrt{5}) \varphi^{-4n} \\
& = \frac{1}{22} (17+7\sqrt{5}) \varphi^{4n}
 - \frac{40}{33} 4^{n} - \frac{1}{3} + \frac{1}{22}(17 - 7\sqrt{5}) \varphi^{-4n},
\end{align*}
as desired.
\end{proof}

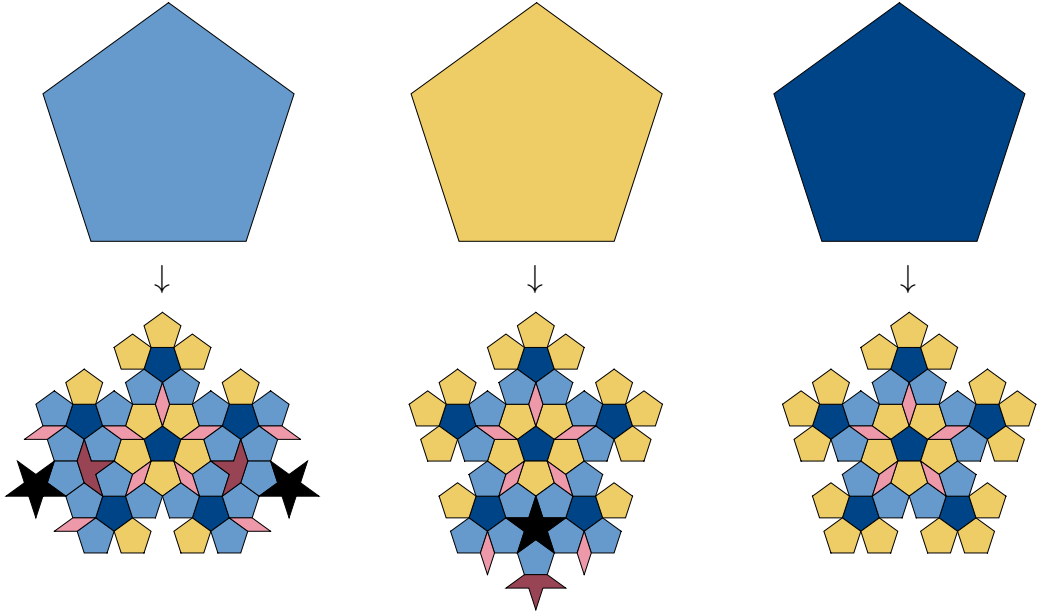
\begin{figure}[t!]
\begin{center}
\includegraphics[width=1.5in]{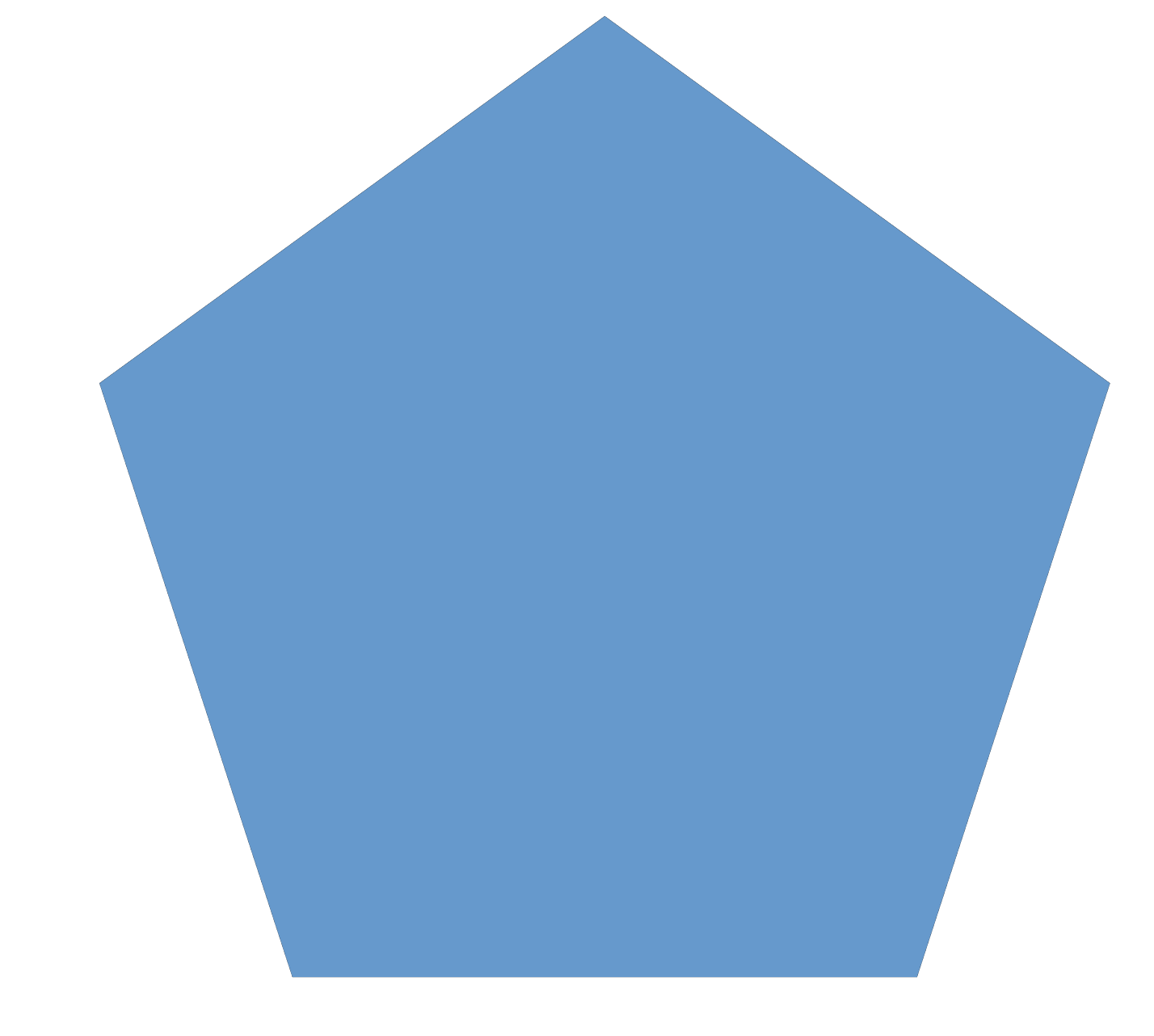}
\hspace*{23pt}
\includegraphics[width=1.5in]{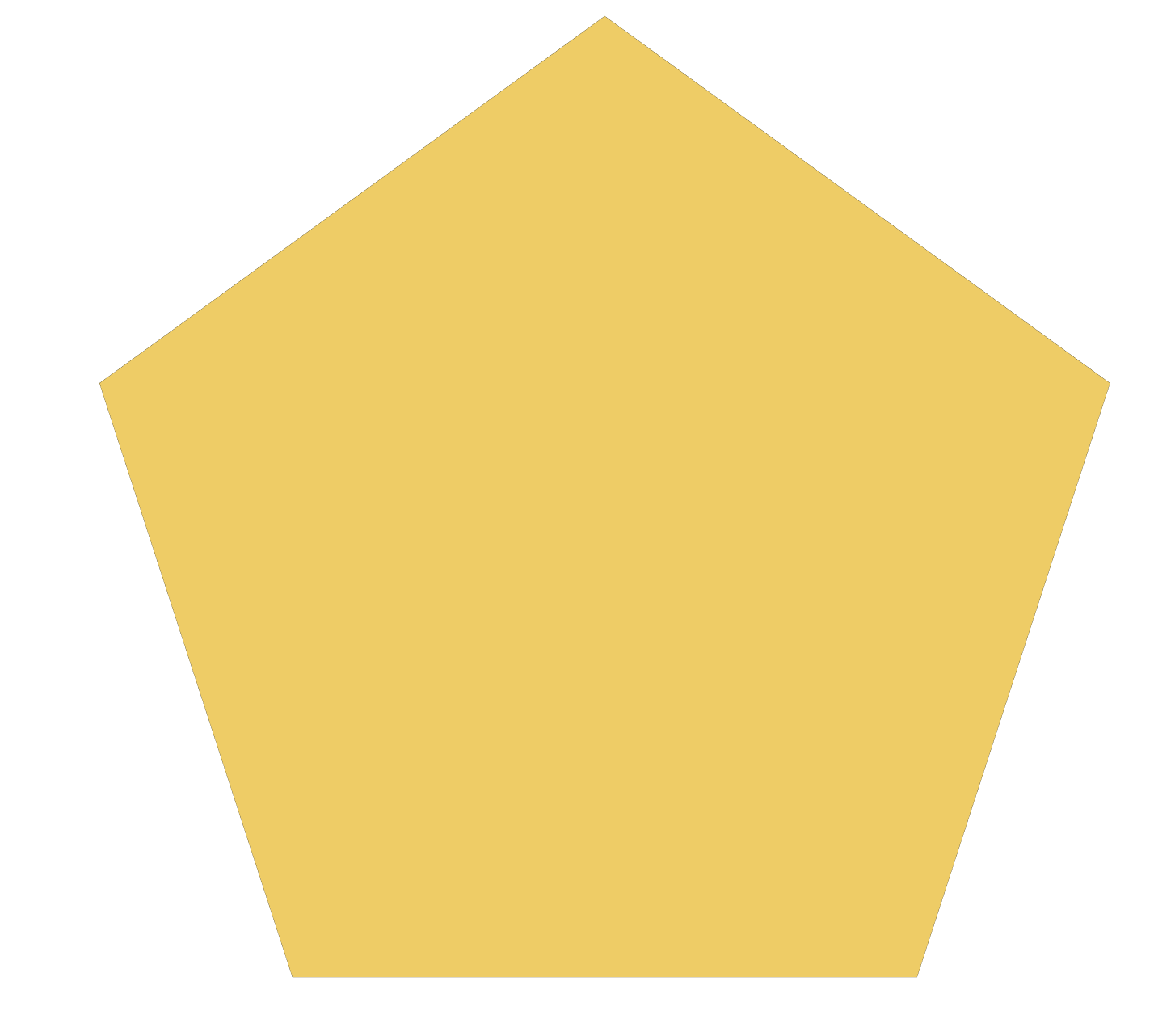}
\hspace*{21pt}
\includegraphics[width=1.5in]{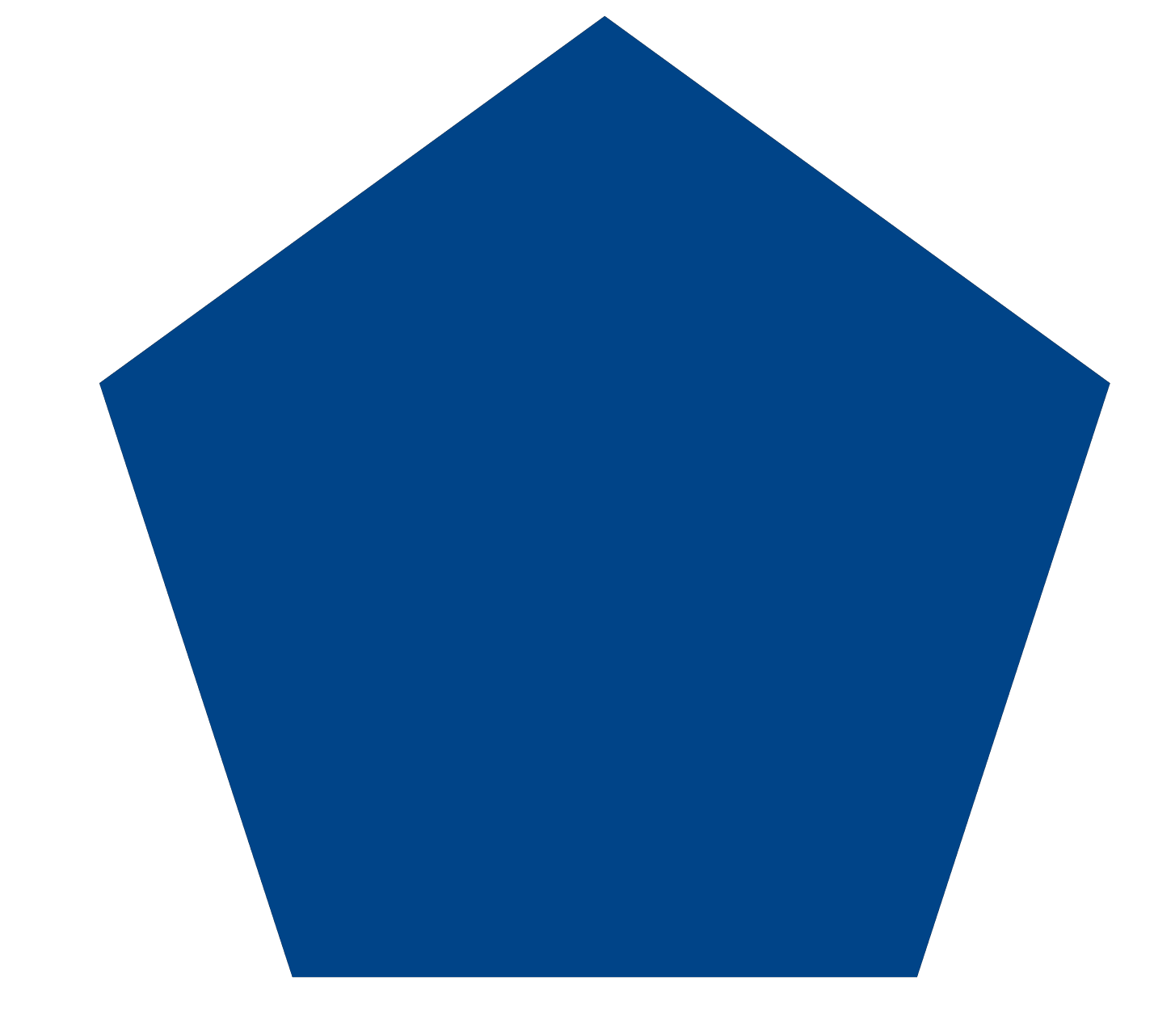}

\begin{picture}(0,0)
\put(-141,2){$\downarrow$}
\put(0,2){$\downarrow$}
\put(141.2,2){$\downarrow$}
\end{picture}

\vspace*{4pt}
\includegraphics[width=1.9in]{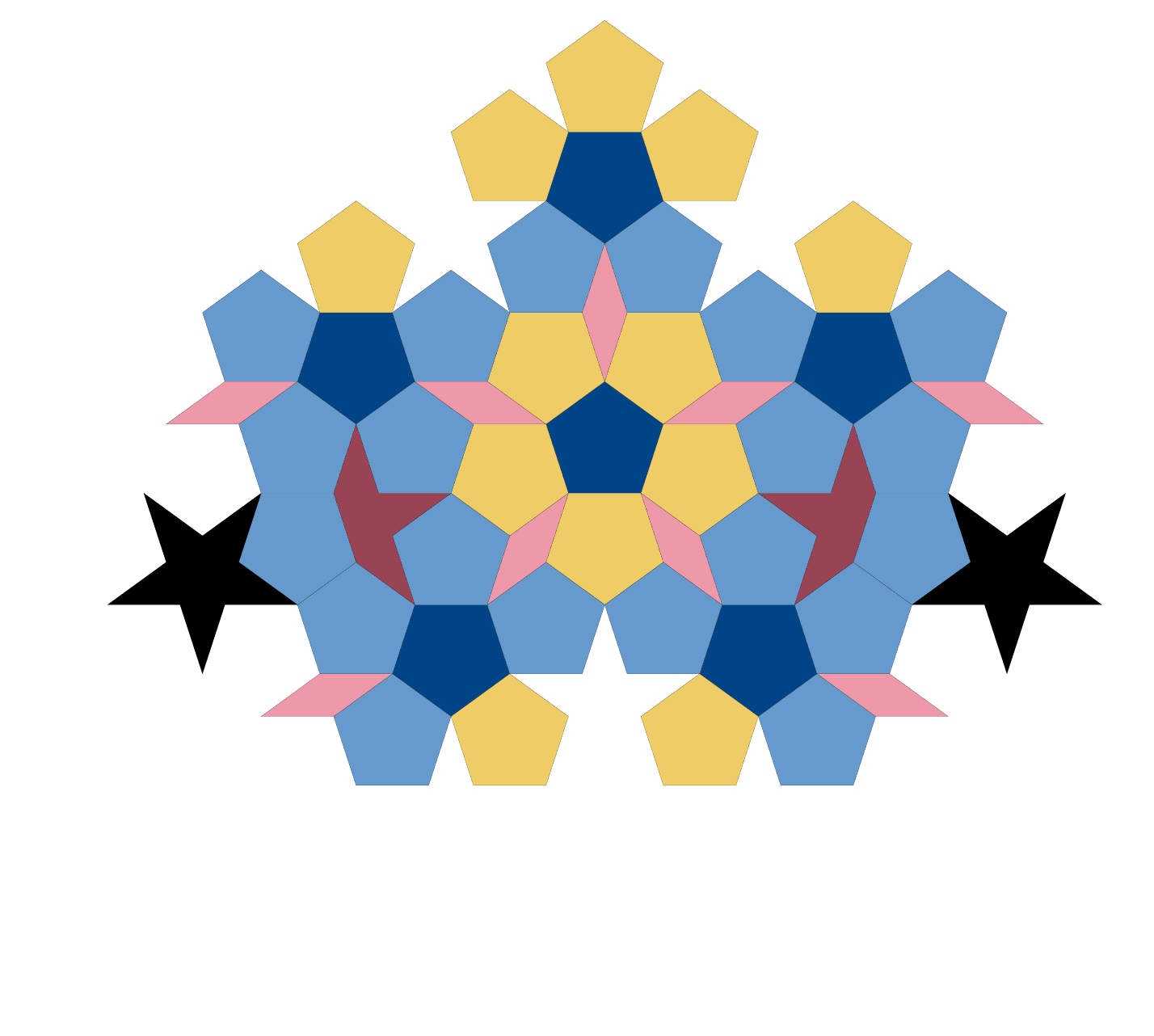}
\includegraphics[width=1.9in]{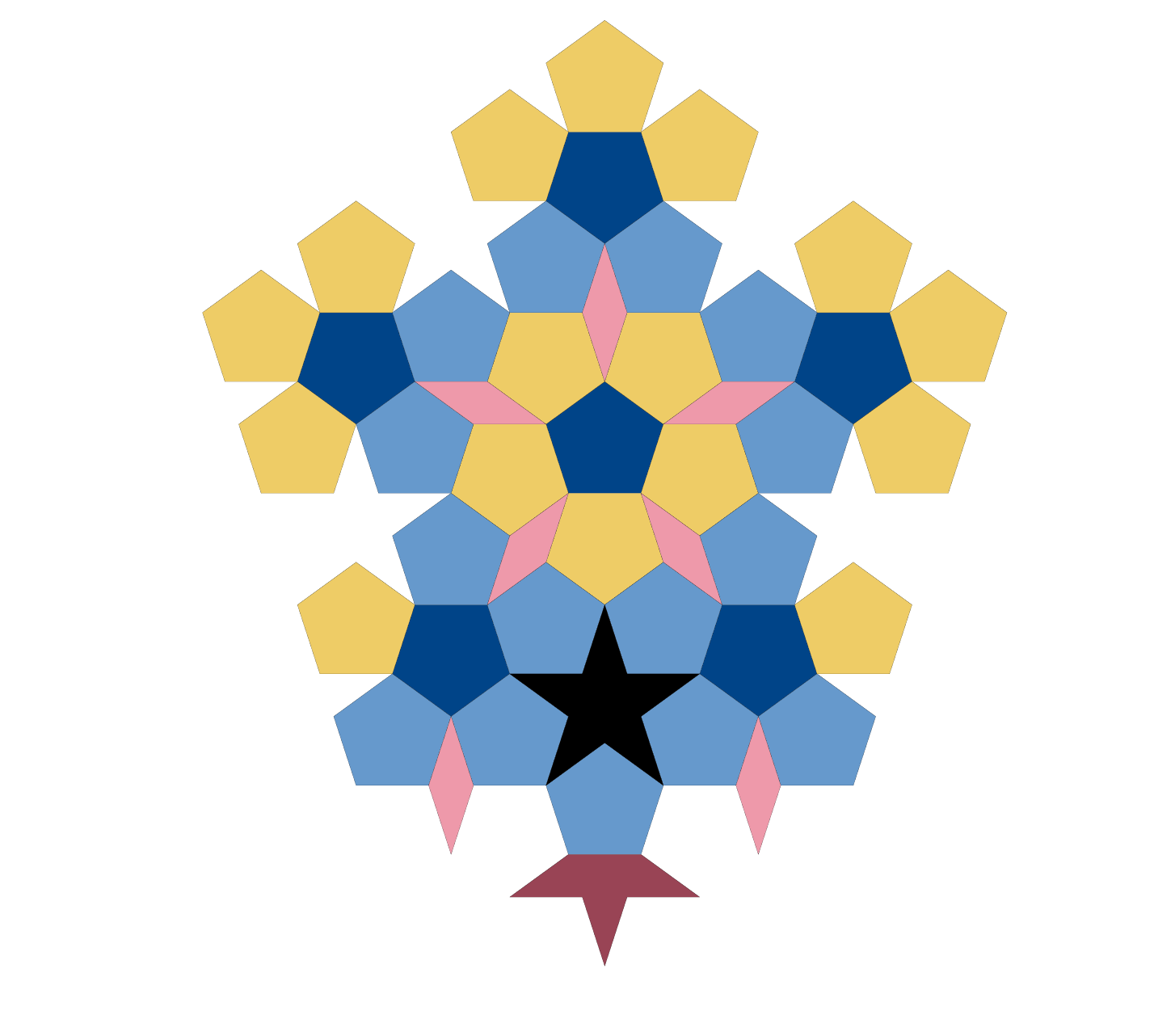}
\includegraphics[width=1.9in]{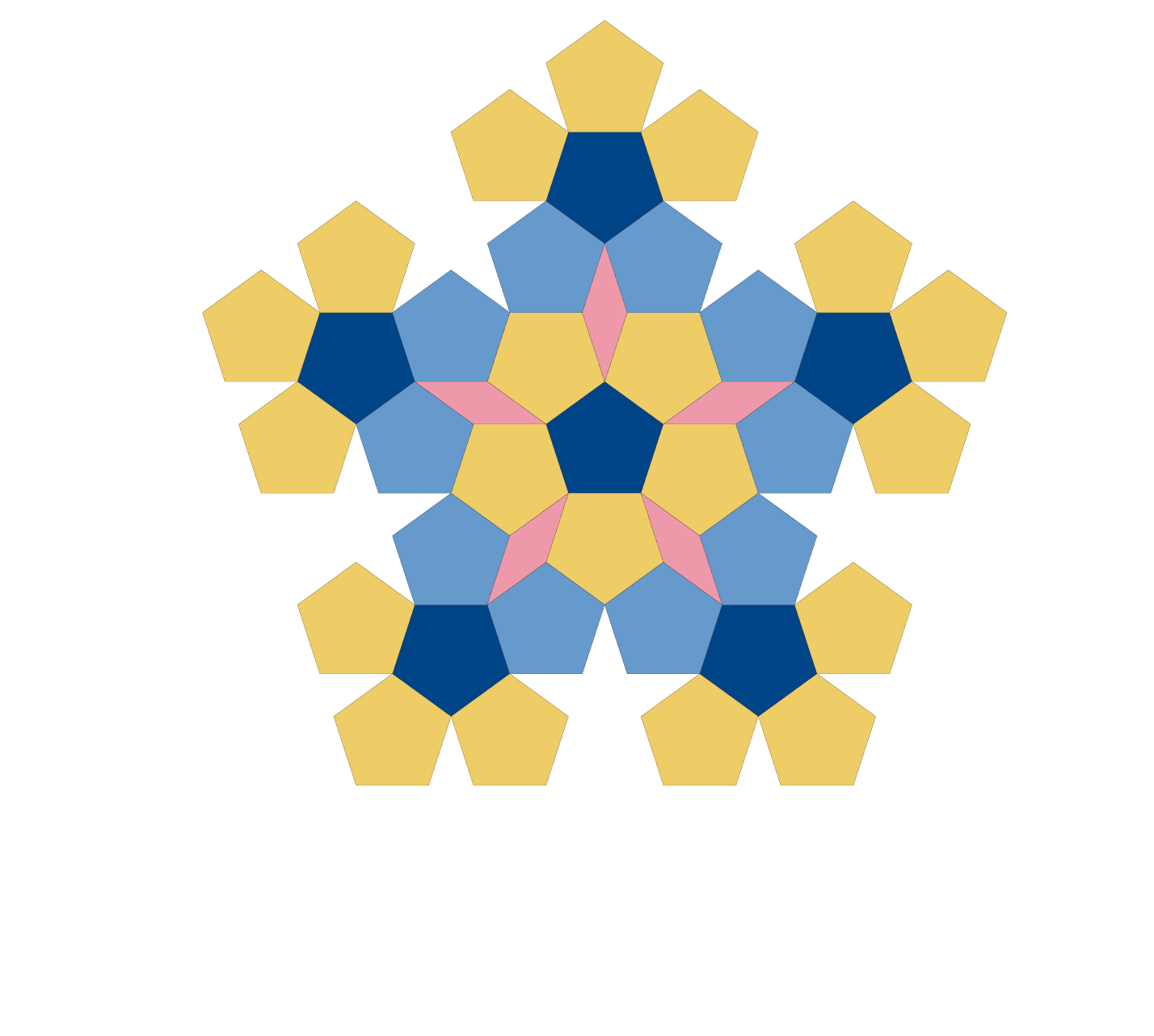}
\end{center}

\vspace*{-12pt}
\caption{Pentagonal supertiles at level $2$ (i.e., $\calT_2^\kappa$ for each color $\kappa$).} \label{fig:boatStar:pent2}
\end{figure}

\subsection{Ring Modes}

We now explain how the locally-supported eigenfunctions arise and how to estimate their frequency.
For each of the three colors $\kappa$, let $\calT_0^{\kappa}$ denote the pattern that consists of a single pentagon with color $\kappa$,  let $\calT_n^{\kappa} = \subst_\boatstar^n(\calT_0^\kappa)$ denote the result of substituting $n$ times, let $\Gamma_n^\kappa$ denote the induced finite graph, and denote the corresponding graph Laplacian by $\Delta_n^\kappa$. We will refer to $\calT_n^\kappa$ as a \emph{level-$n$ pentagonal supertile}. See Figure~\ref{fig:boatStar:pent2} for the three level-two pentagonal supertiles. 

\begin{figure}[b!]
\includegraphics[width=2in]{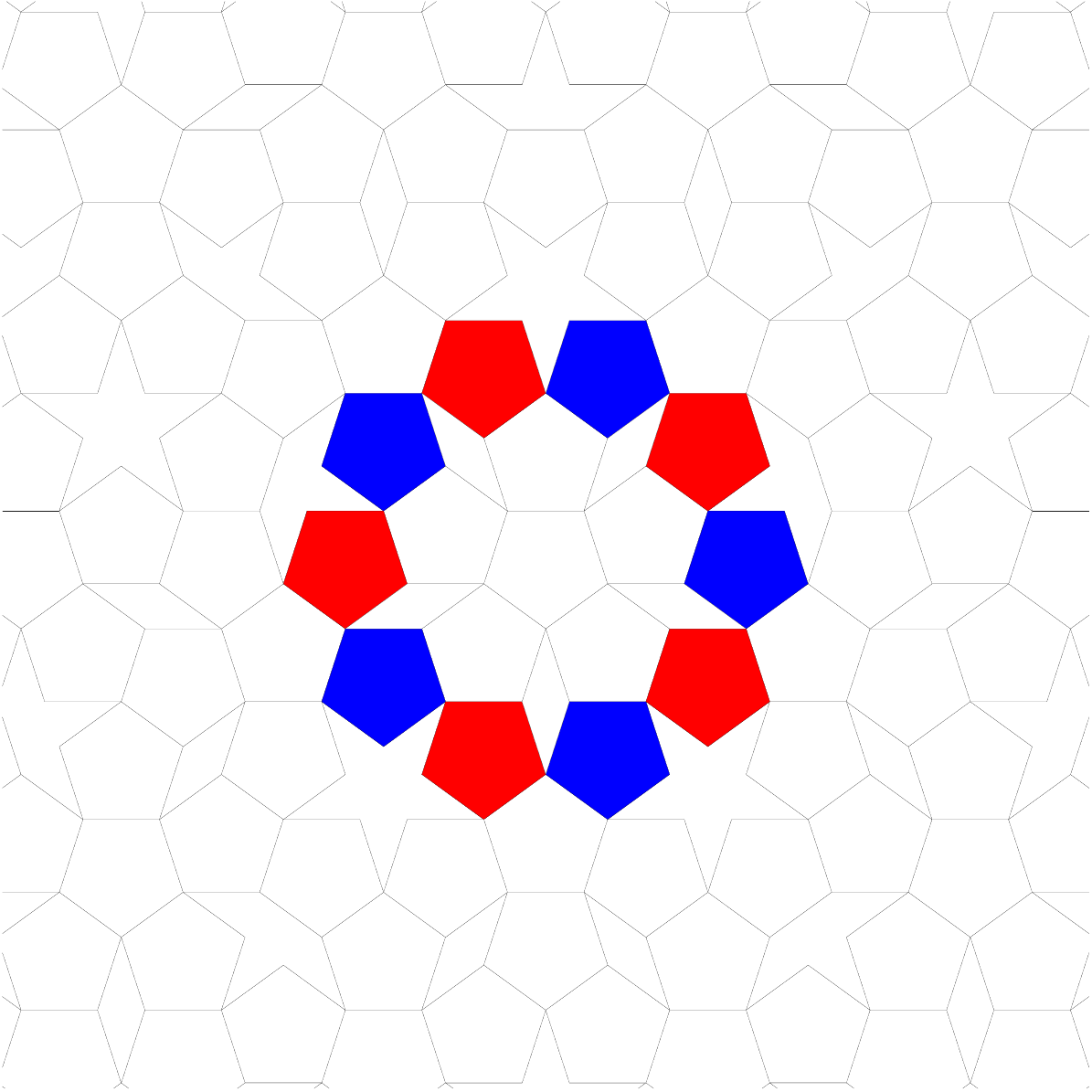}
\caption{A locally-supported eigenfunction for $E=4$ on the boat--star tiling.  
The function takes the value $+1$ on blue tiles, $-1$ on red tiles, and zero elsewhere.} 
\label{fig:boatStar:pentSupport}
\end{figure}

The crucial observation is that each level-two pentagonal supertile contains a pattern that supports a locally-supported eigenfunction.
Namely, the ring of ten pentagons encircling the center is precisely the tile set that can be used to support a locally-supported eigenfunction.
One can locate \emph{fifty} level-two pentagonal supertiles in the level-four supertile shown in Figure~\ref{fig:boatStar:E=4lev4}.

\begin{figure}[t!]
\begin{center}
\includegraphics[width=4in]{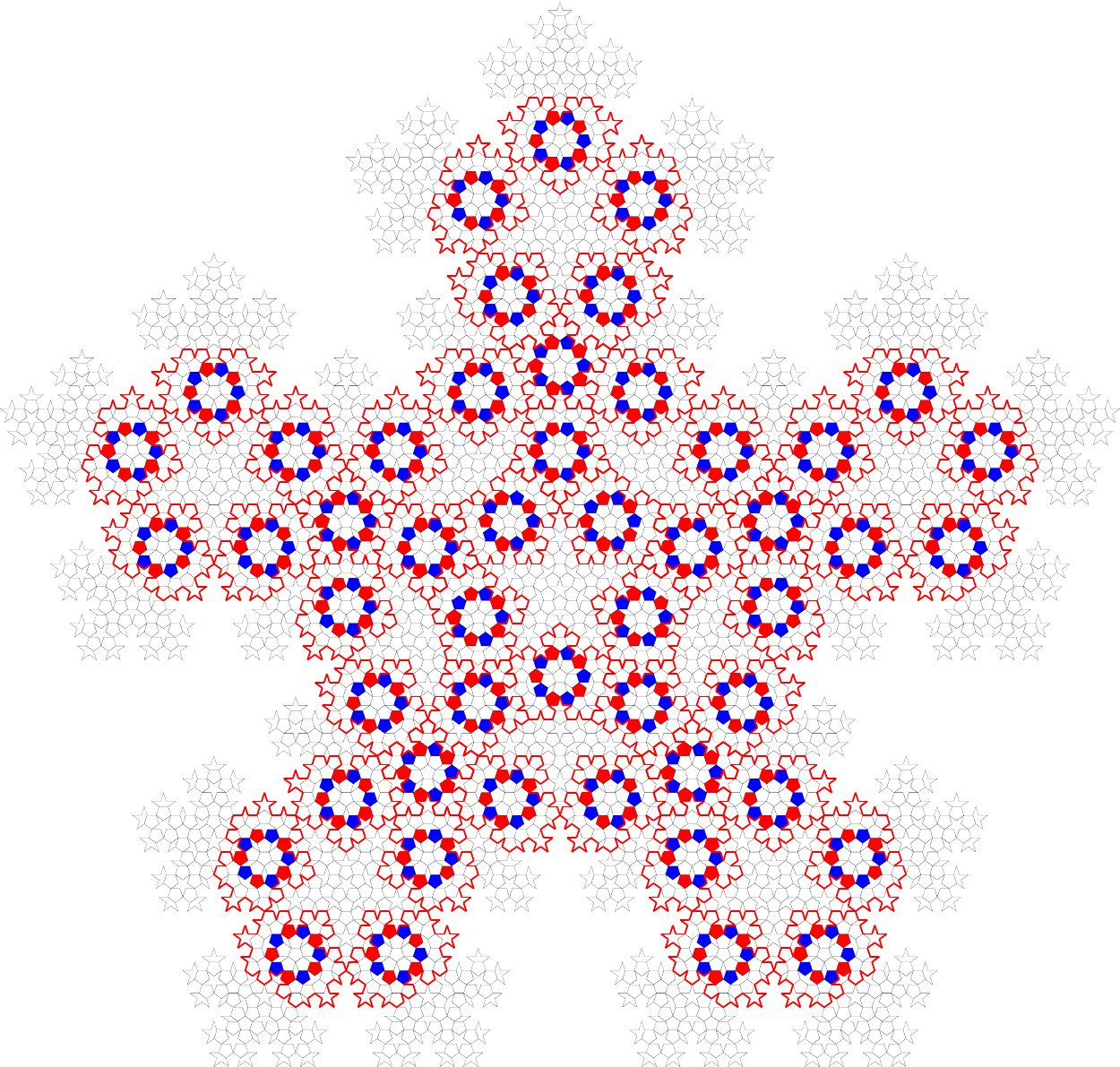}
\end{center}

\caption{Boat--star tiling, level~4, $E=4$.  Fifty locally-supported ring modes: 
the colored tiles correspond to mode entries equal to $\pm1$.  
The red lines show the boundaries of the fifty pentagonal supertiles 
(each of which is generated by a pentagon on the level~2 tiling).
\label{fig:boatStar:E=4lev4}}
\end{figure}

\begin{lemma} \label{lem:boatStar:ringmodes}
For all $\kappa$, $4$ is an eigenvalue of $\Delta_2^\kappa$.
\end{lemma}

\begin{proof}Let us begin by explaining how the eigenfunction arises.
Consider Figure~\ref{fig:boatStar:pent2}, which shows $\calT_2^\kappa$ for each color $\kappa$.
In each $\calT_2^\kappa$, one observes a ring of ten pentagons encircling the center, highlighted in Figure~\ref{fig:boatStar:pentSupport}.
Denote this pattern by $\calR$.
Define a vector $\psi$ by assigning the value $+1$ to each red pentagon, $-1$ to each blue pentagon, and $0$ to all other tiles.

A brief calculation shows $\Delta_2^\kappa \psi = 4\psi$.
Indeed, when $u$ corresponds to a face with combinatorial distance $2$ from $\calR$, then
\[[\Delta_2^\kappa\psi](u) = 0 = 4 \psi(u).\]
Similarly, one checks $[\Delta_2^\kappa\psi](u)=4\psi(u)$ for any $u$ coming from a face of $\calR$.
Each face with combinatorial distance one from $\calR$ has precisely two neighbors in $\calR$, so, due to the alternating pattern, one observes 
\[[\Delta_2^\kappa\psi](u) = 1-1=0 = 4 \psi(u),\]
hence showing that $\psi$ is an eigenfunction of eigenvalue $4$, as desired.\end{proof}

\begin{proof}[Proof of Theorem~\ref{t:boatstarquant}]
By Lemma~\ref{lem:boatStar:ringmodes}, the number of occurrences of the pattern $\calR$ at level $n$ may be bounded from below by the number of pentagons that appear in level $n-2$.

For instance, in level $3$, there are five occurrences of the pattern, each of which is precipitated by a pentagon from $\calT_1$; compare Figure~\ref{fig:boatStar:E=4lev4}.

We now make two observations.
First, each of these occurrences will be separated by all other occurrences by a tiling distance of at least two.

Second, we need to address a minor technicality.
Namely: some of the ring patterns from Lemma~\ref{lem:boatStar:ringmodes} may appear on the interior of the tiling, while others may occur on the boundary.
As can be seen from Figure~\ref{fig:boatStar:pent2}, either occurrence leads to an eigenfunction.

Thus, we see that the multiplicity of the eigenvalue $4$ at level $n$ is bounded from below by the number of pentagons that occur in level $n-2$.
Denoting the IDS by $k_\boatstar$, Lemma~\ref{lem:boatStar:tilecount} gives 
\begin{align*}
k_\boatstar(4+)- k_\boatstar(4-) & \geq \lim_{n\to\infty} \frac{P_{\text{\pentagon}}(n)}{P(n+2)} \\[1mm]
& = \lim_{n \to \infty} \frac{\frac{1}{22} (17+7\sqrt{5}) \varphi^{4n}+ O(4^n)}{\frac{1}{22} (25+9\sqrt{5}) \varphi^{4(n+2)} +O(4^n)} \\[1mm]
& = \frac{17+7\sqrt{5}}{(25+9\sqrt{5})\varphi^8} \\[1mm]
& = \frac{65-29\sqrt{5}}{10},
\end{align*}
as claimed.\end{proof}

One might naturally question whether this estimate on the multiplicity is sharp.

\begin{question} \label{quest:boatStar:4mult}
For $n=1,2,\ldots,8$, the multiplicity of the eigenvalue $4$ for $\calT_n$ is  given by
\begin{equation}
m(n) = \begin{cases}
0, & n=0,1;\\
1, & n = 2 ;\\
\frac{1}{22} (17+7\sqrt{5}) \varphi^{4(n-2)}
 - \frac{40}{33} 4^{n-2} - \frac{1}{3} + \frac{1}{22}(17-7\sqrt{5}) \varphi^{-4(n-2)}, & n \ge 3 .
\end{cases}
\end{equation}
Does this pattern persist? That is, is it true that the multiplicity of $E=4$ at level $n$ is given by $m(n)$ for all $n \geq 1$?
\end{question}

\begin{remark}
Question~\ref{quest:boatStar:4mult} has been answered in the affirmative (numerically) for all $n \leq 8$; compare Table~\ref{fig:boatStar:multTable}.
(The single eigenfunction that appears at energy $4$ at level~2 is not a ring consisting of ten pentagons, but is qualitatively different; its support comprises thirty tiles, all on the boundary; in contrast to the ring modes, this pattern does \emph{not} extend to a locally-supported eigenfunction of larger patches of the tiling.) 
\end{remark}

\vspace*{1em}
\begin{table}[b!]

\caption{Boat--star tiling: Multiplicities of different eigenvalues at levels $2$--$8$, with a comparison to the conjectured multiplicity of $E=4$. The final column shows the numerical approximation to the IDS jump at $E=4$. The theoretically obtained lower bound on $k_\boatstar(4+) - k_\boatstar(4-)$ from Theorem~\ref{t:boatstarquant} is $(65-29\sqrt{5})/10 \approx 0.01540286525\ldots$.}\label{fig:boatStar:multTable}

\begin{center}
\begin{tabular}{crrrrrc}
\emph{level} &
\multicolumn{1}{c}{\emph{tiles}} &
\multicolumn{1}{c}{$E = 1/\varphi^2$} &
\multicolumn{1}{c}{$E = \varphi^2$} &
\multicolumn{1}{c}{$E = 4$}  &
\multicolumn{1}{c}{$m(n)$} &
\multicolumn{1}{c}{$k_{\boatstar,n}(4+) - k_{\boatstar,n}(4-)$}
  \\ \hline
 1 & 16 & 0 & 0 & 0 & 0 & 0.0000000\ldots\\
 2 & 86      & 10 & 10 & 1 & 1           & 0.0116279\ldots \\
 3 & 621     & 30 & 30 & 5 & 5           & 0.0080515\ldots \\
 4 & 4\,371    & 110 & 110 & 50 & 50       & 0.0114390\ldots \\
 5 & 30\,406   & 430 & 430 & 400  & 400    & 0.0131552\ldots \\
 6 & 210\,181  & 1\,710 & 1\,710 & 2\,965 & 2\,965 & 0.0141068\ldots \\
 7 & 1\,447\,691 & 6\,830 & 6\,830 & 21\,210 & 21\,210         & 0.0146509\ldots \\
 8 & 9\,950\,966 & 27\,310& 27\,310 & 148\,920 & 148\,920       & 0.0149653\ldots
\end{tabular}
\end{center}
\end{table}

\begin{figure}[b!]
\begin{center}
\incplotscl{2.85in}{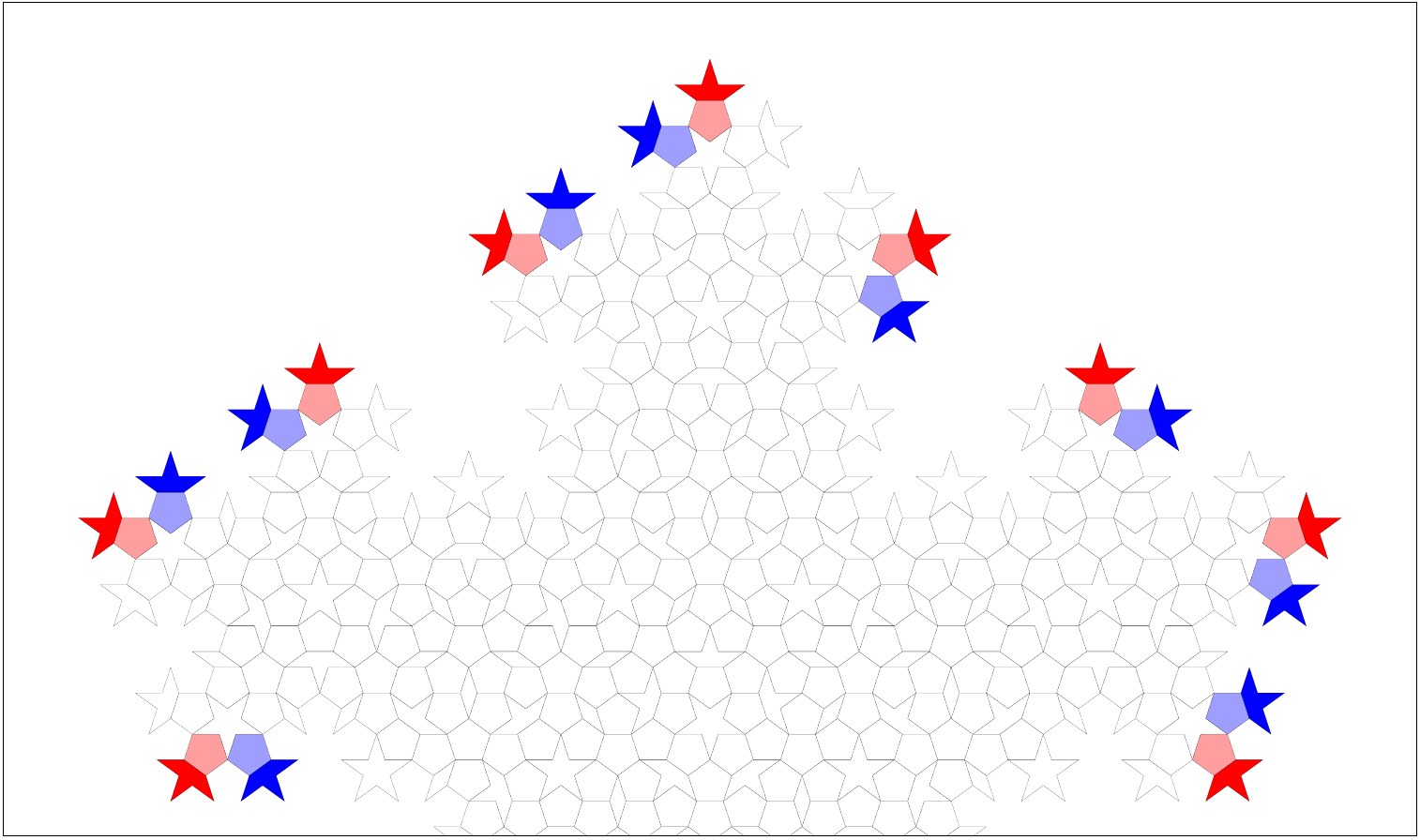}{}\quad
\incplotscl{2.85in}{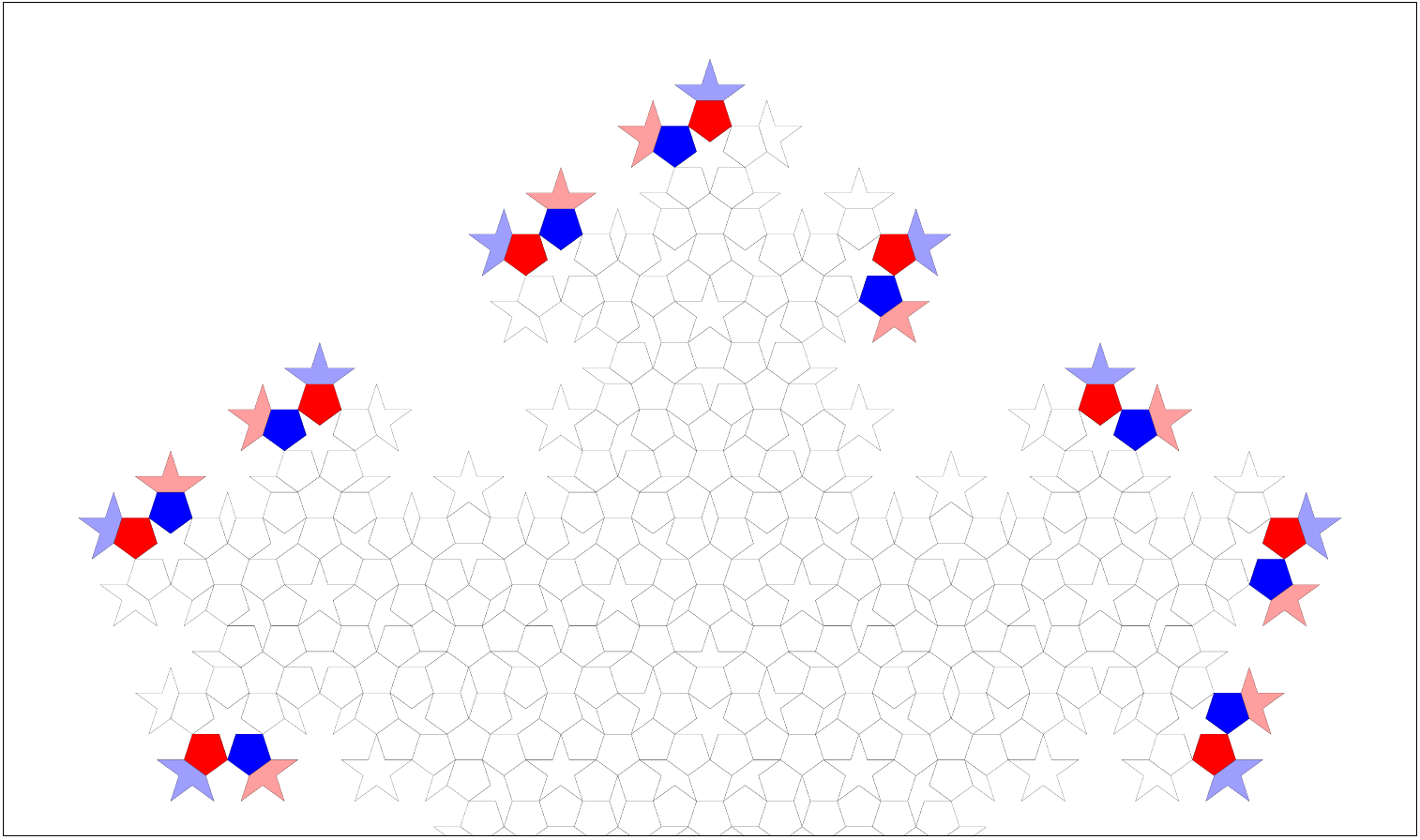}{}

\begin{picture}(0,0)
\put(-206,105){$E=1/\varphi^2$}
\put(12,105){$E=\varphi^2$}
\end{picture}
\end{center}
\caption{Locally-supported boundary eigenfunctions for the boat--star tiling for
$E=1/\varphi^2$ (left) and $E=\varphi^2$ (right) at level~3.
Each plot shows nine linearly independent eigenfunctions, each supported on four tiles.
The nonzero entries of the eigenfunctions are 
$\pm 1$ (dark blue and red) and $\pm 1/\varphi$ (light blue and red).
\label{fig:boatStar:bdy}}
\end{figure}

Table~\ref{fig:boatStar:multTable} summarizes some results of our computations. 
In addition to the ring modes at $E=4$, Table~\ref{fig:boatStar:multTable} also 
contains counts for high-multiplicity eigenvalues at $E=1/\varphi^2$ and $E=\varphi^2$
associated with modes that are locally-supported, but only on the boundary of the finite
patch $\calT_n^\boatstar$.  Figure~\ref{fig:boatStar:bdy} shows a few of these boundary modes.
Denote by $k_{\boatstar,n}$ the IDS associated with $\calT_n^\boatstar$.  The boundary modes
outnumber the ring modes at early levels, and will cause a jump in $k_{\boatstar,n}$
at $E=1/\varphi^2$ and $E=\varphi^2$ that diminishes as the level increases;
The jump at $E=4$ grows with the level, as quantified in Table~\ref{fig:boatStar:multTable}. 
(Peek ahead to Figure~\ref{fig:ids_boatstar} for an illustration.)





\section{Triangles}

This section discusses the Robinson triangle substitution, including proofs of the relevant portion of Theorem~\ref{t:main}, as well as Theorem~\ref{t:robinsonquant}. The Robinson triangle part of Theorem~\ref{t:main} follows immediately from the observation of a single locally-supported eigenfunction~\cite{KlasLenzStol2003CMP}. The bulk of this section is then concerned with the proof of Theorem~\ref{t:robinsonquant}, the estimate from below of the discontinuity in the integrated density of states at energies $E = 4$ and $E=6$.

\subsection{Basics}

The \emph{Robinson triangle substitution}, denoted $\subst_\robinson$, was specified in Definition~\ref{def:triangleRules}.

\begin{notation}
We will refer to the acute and obtuse triangles as $A$ and $O$ tiles, respectively. \end{notation}

Let $\calT_0^\robinson$ denote the pattern consisting of 10 $O$ tiles of oscillating color arranged in a star.  We will refer to $\calT_n^\robinson := \subst_\robinson^{n}(\calT^\robinson_0)$ as the level-$n$ tiling; see Figure~\ref{fig:robinsonLvl1-3}.

\begin{figure}[b!]
\includegraphics[width=1.6in]{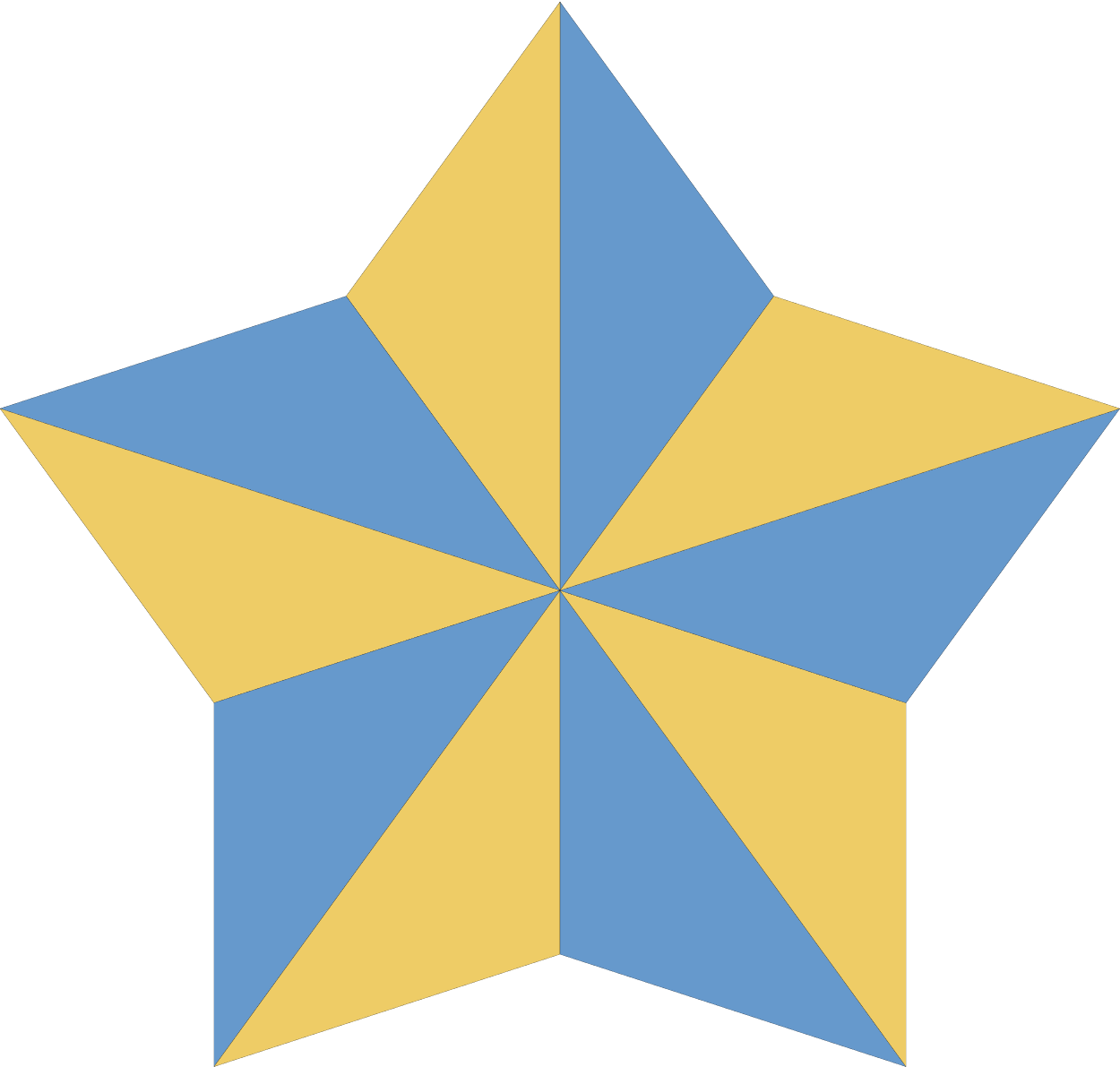}\quad
\includegraphics[width=1.6in]{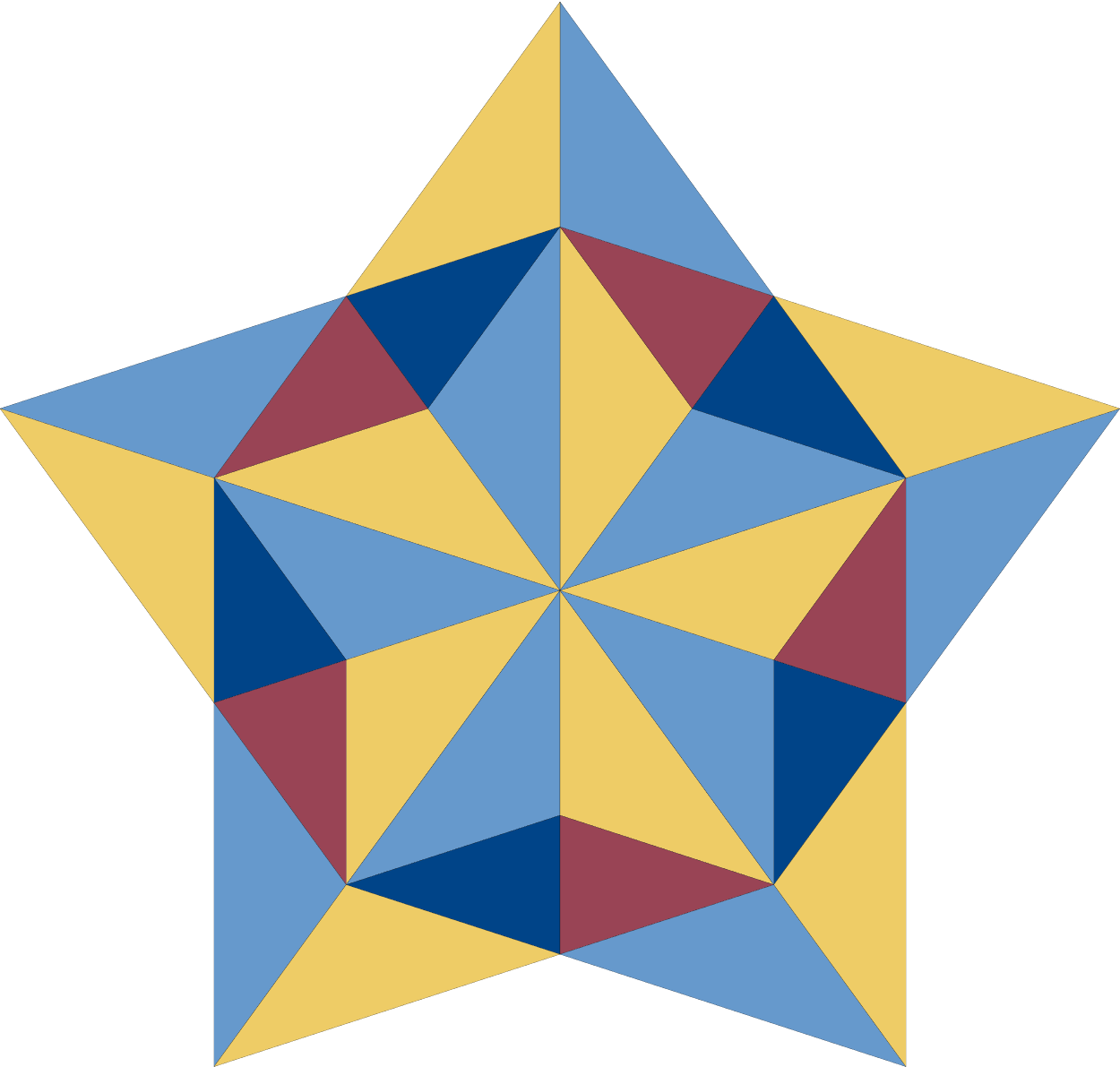}\quad
\includegraphics[width=1.6in]{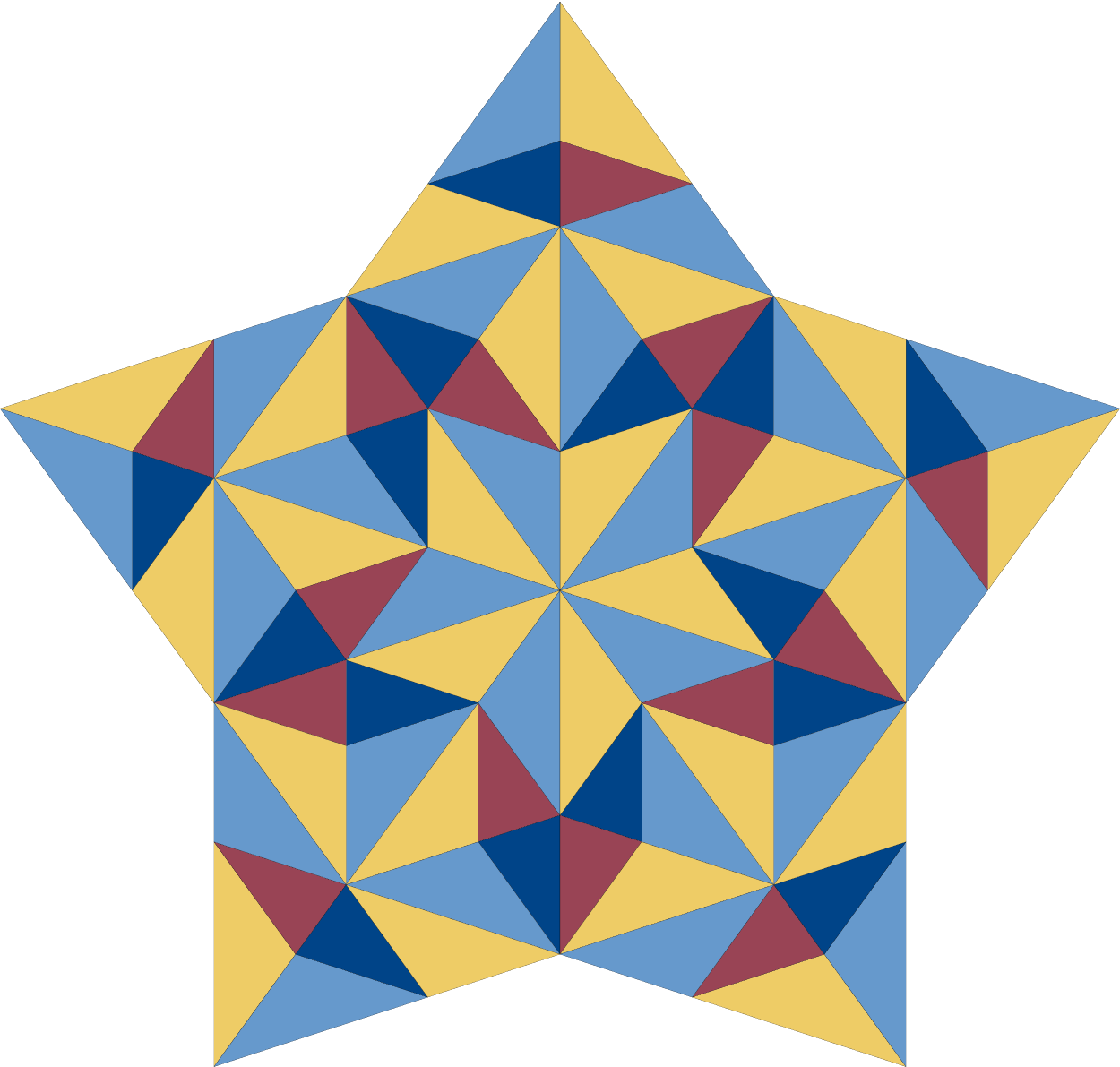}

\includegraphics[width=1.6in]{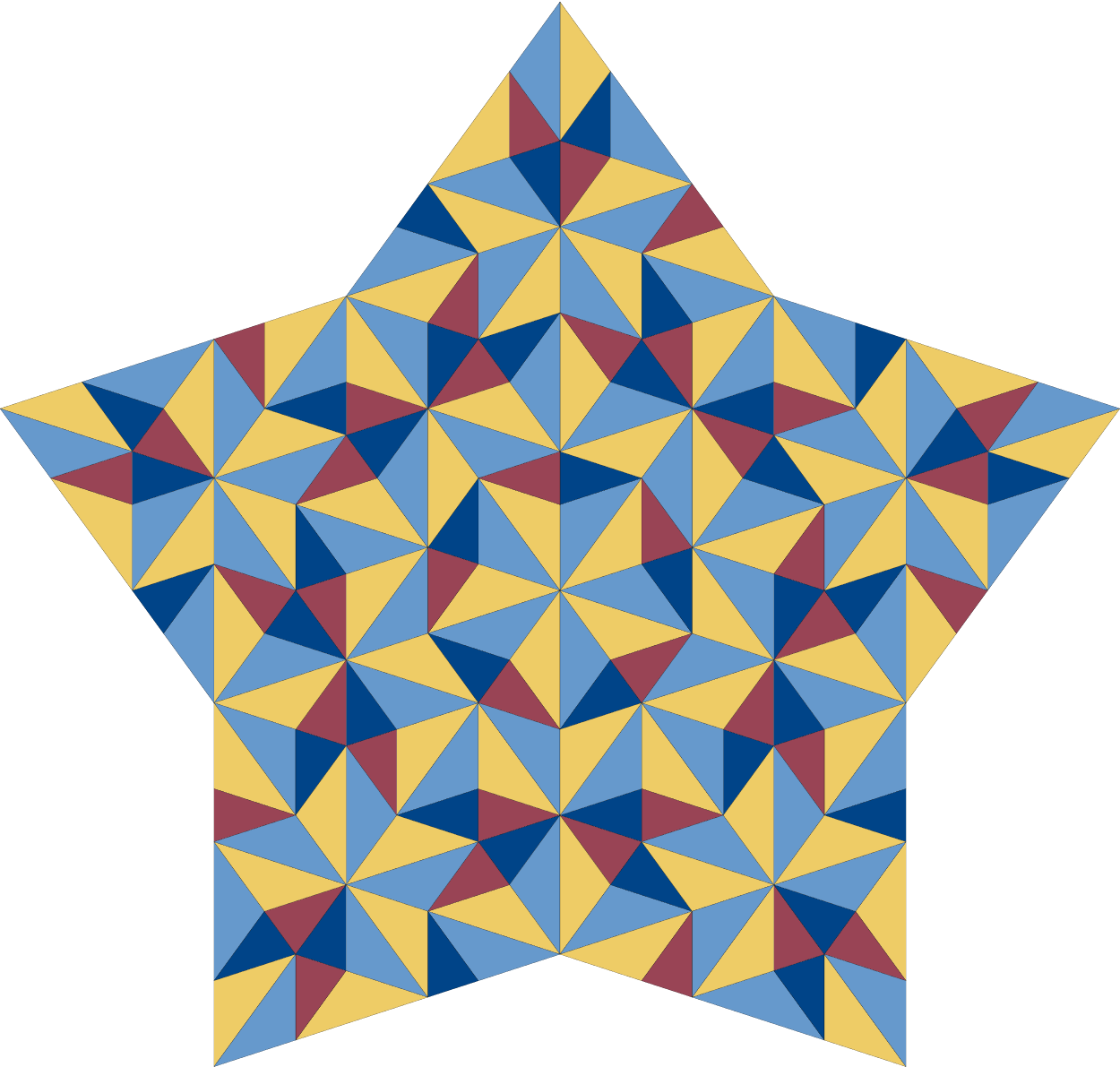}\quad
\includegraphics[width=1.6in]{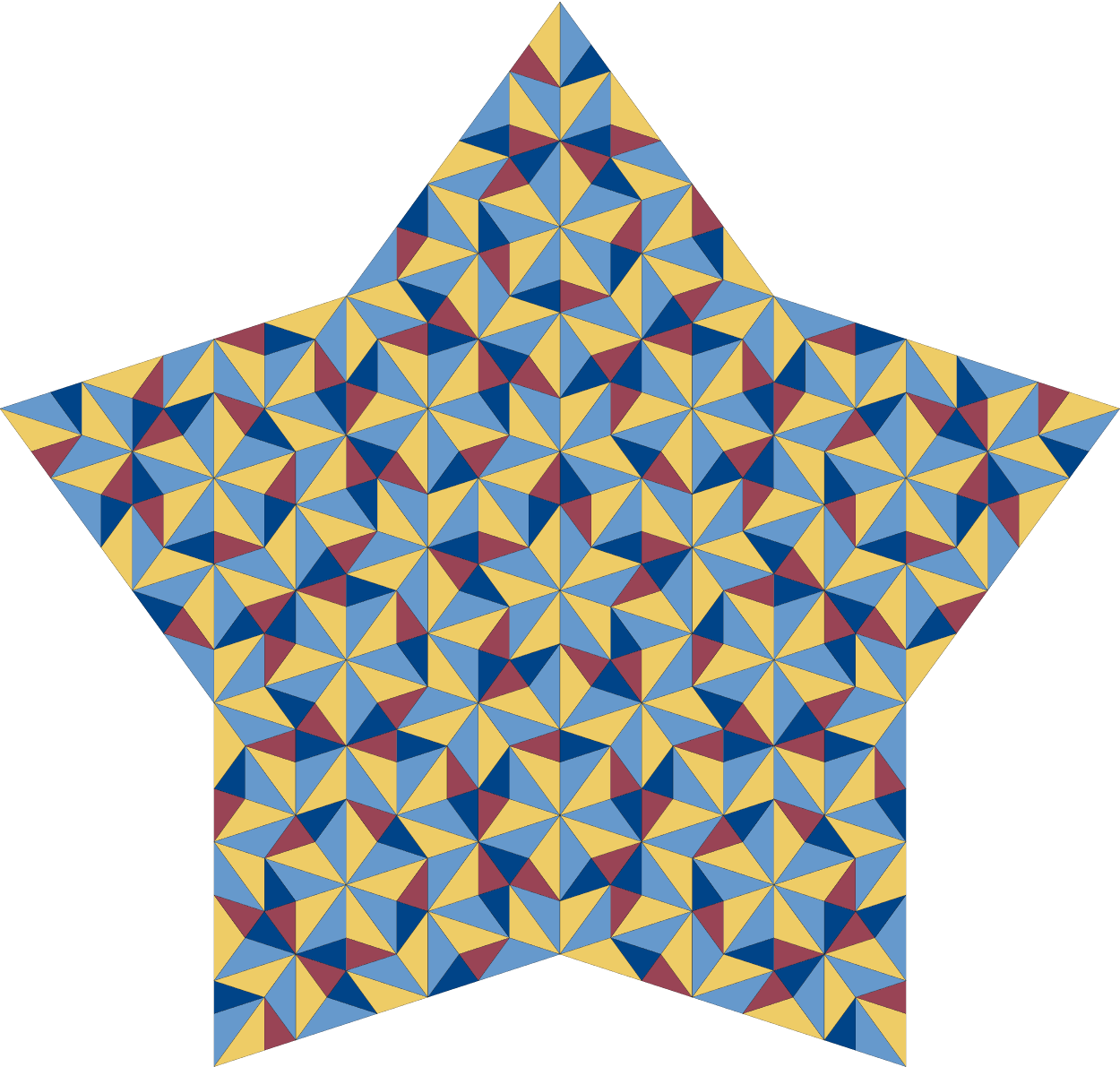}

\caption{The first five levels of the tiling generated by the
triangular substitution rule: $\calT_0^\robinson, \ \calT_1^\robinson, \ldots, \ \calT_4^\robinson$.}
\label{fig:robinsonLvl1-3}
\end{figure}

\medskip
For $\alpha \in \{A,O\}$, let $\alpha_n$ denote the number of $\alpha$-type tiles in $\calT^\robinson_n$.
Additionally, let $\{F_n\}_{n=0}^\infty$ denote the sequence of Fibonacci numbers, given as
\begin{equation} \label{eq:FibDef}
F_0=0, \ F_1=1, \quad F_{n+1} = F_n + F_{n-1}, \ n \geq 1.
\end{equation}
\begin{prop} \label{prop:tri:tilecount}
For every $n \geq 0$,
 \begin{equation} \label{eq:RnDn:rec}
O_{n+1} = 2O_n + A_n, \qquad A_{n+1} = O_n + A_n
\end{equation}
and
\begin{equation} \label{eq:RnDnformulae}
\hspace*{0.5pt}
O_n = 10F_{2n+1},
 \hspace*{43.5pt}
A_n = 10 F_{2n}.
\end{equation}
The total number of tiles at level $n$ is then
\begin{equation} \label{eq:tri:totaltiles}
O_n + A_n = 10F_{2n+2}.
\end{equation}
\end{prop}

\begin{proof}
The recursion \eqref{eq:RnDn:rec} follows immediately from the substitutions in Definition~\ref{def:triangleRules}. One can check that~\eqref{eq:RnDnformulae} holds for $n=0$ and $n=1$ by inspection. Assuming it holds for all $k \leq n$ with $n \geq 1$ \eqref{eq:RnDn:rec} yields
\[
O_{n+1}
= 2O_n + A_n
= 10(2F_{2n+1} + F_{2n})
= 10 F_{2n+3},
\]
where we have applied the recursion of \eqref{eq:FibDef} twice in the final step. Similarly,
\[
A_{n+1}
=
O_n + A_n
=
10(F_{2n+1} + F_{2n})
=
10 F_{2n+2},
\]
which proves \eqref{eq:RnDnformulae} by induction. Combining \eqref{eq:RnDnformulae} and \eqref{eq:FibDef} gives \eqref{eq:tri:totaltiles}.
\end{proof}

\begin{figure}[b!]
\begin{center}
\begin{minipage}{2.1in} \begin{center}
\includegraphics[width=2in]{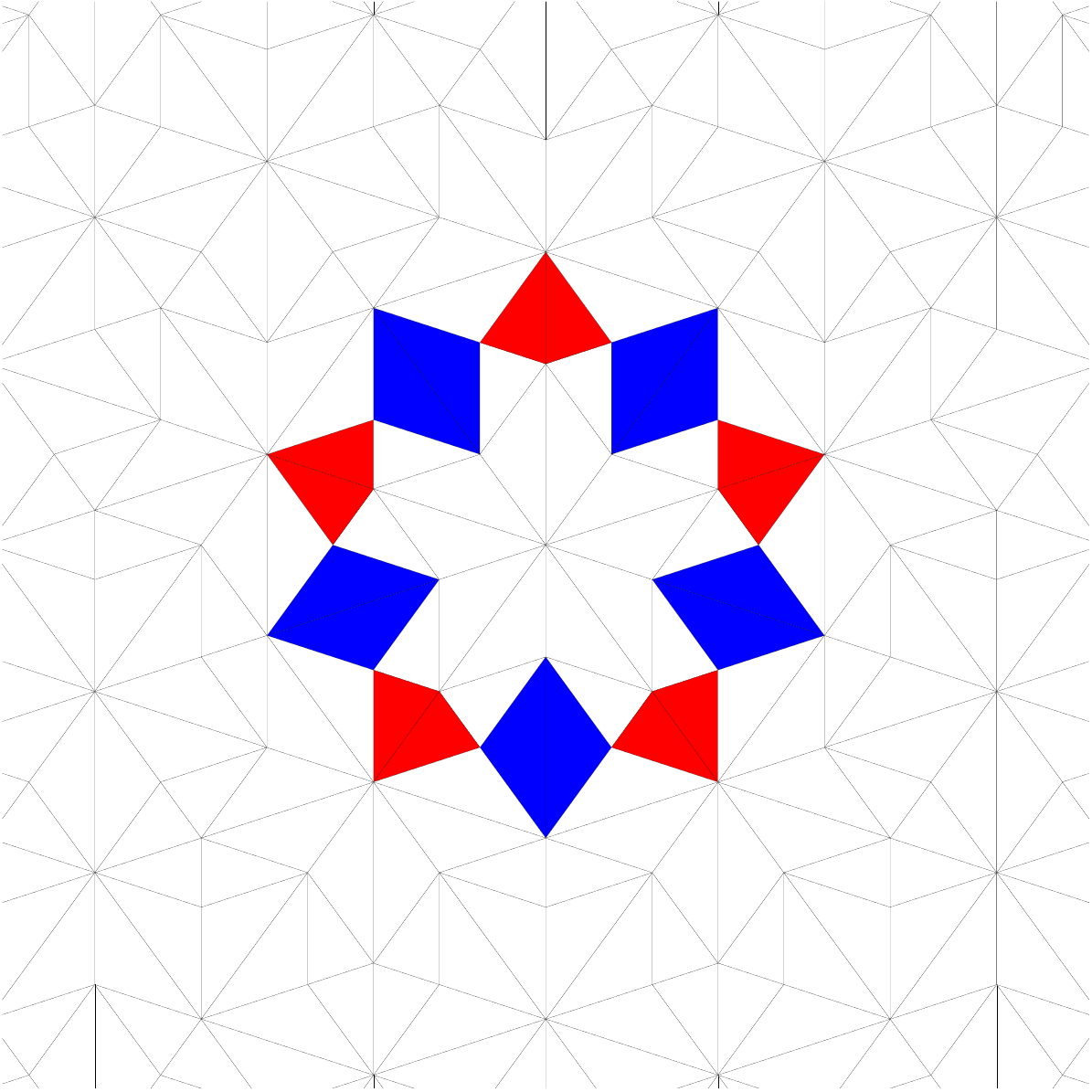}\\ \small $E=2$
\end{center}\end{minipage}
\quad
\begin{minipage}{2.1in} \begin{center}
\includegraphics[width=2in]{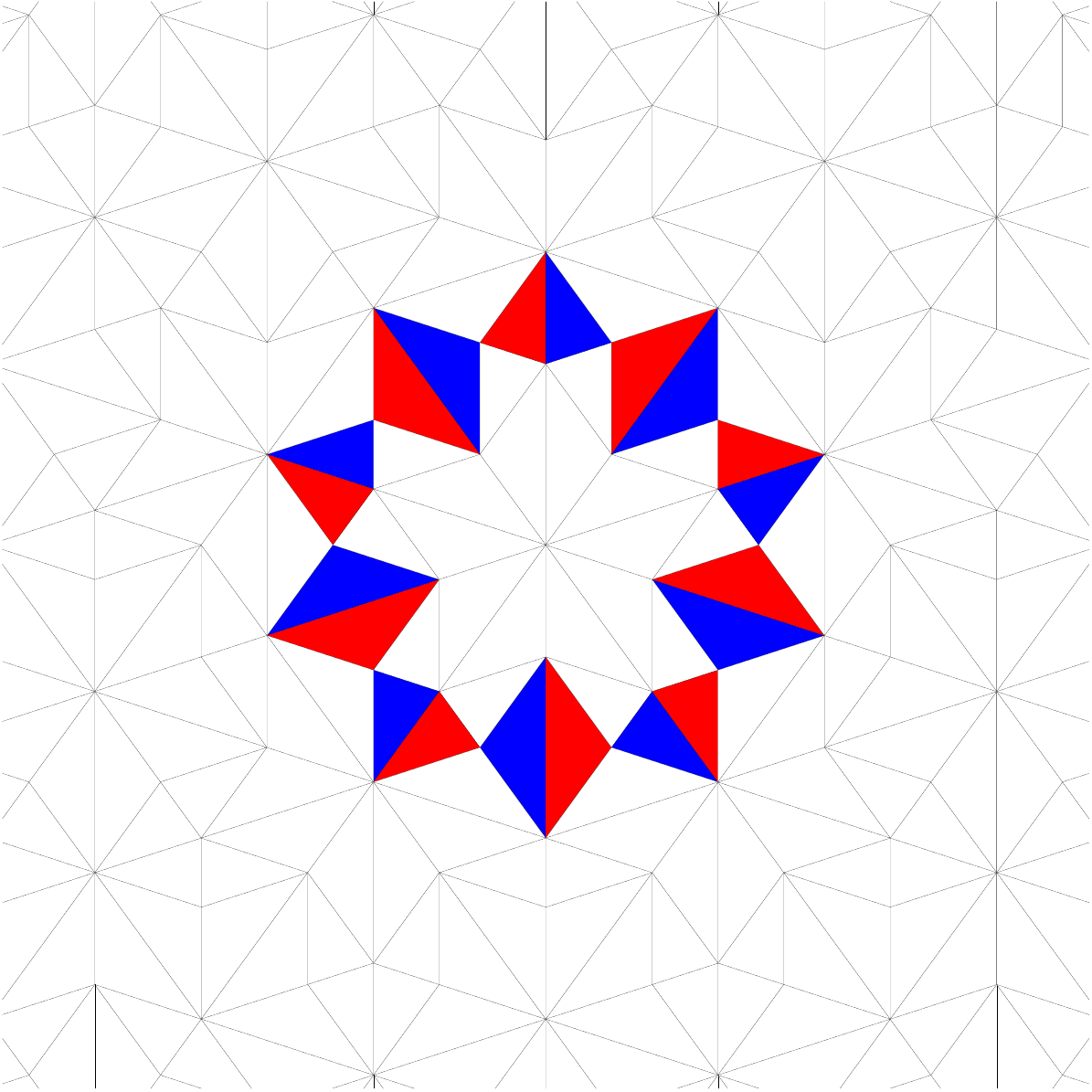}\\ \small $E=4$
\end{center}\end{minipage}
\end{center}
\caption{Locally-supported eigenfunctions for the Robinson triangle substitution,
with $E=2$ and $E=4$.
The nonzero entries of these modes take the values $+1$ (blue) and $-1$ (red).} \label{fig:tri:ringmode}
\end{figure}

\subsection{Ring Modes}

\begin{prop} \label{prop:tri:ring}
Let $\Gamma$ denote the graph associated with the polygonal tiling shown in Figure~\ref{fig:tri:ringmode} and let $\Delta$ denote the corresponding Laplace operator.
Define a vector $\psi$ by
\begin{equation}
\psi(T) = \begin{cases} \phantom{-}1, & T \text{ is blue;} \\
-1, & T \text{ is red;} \\
\phantom{-}0, &  \text{otherwise.}\end{cases}
\end{equation}
Then $\Delta\psi = 2\psi$.
\end{prop}

\begin{proof}The proof follows from a direct calculation.\end{proof}

\begin{figure}[b!]
\begin{center}
\begin{minipage}{2.6in} \begin{center}
\includegraphics[width=2.5in]{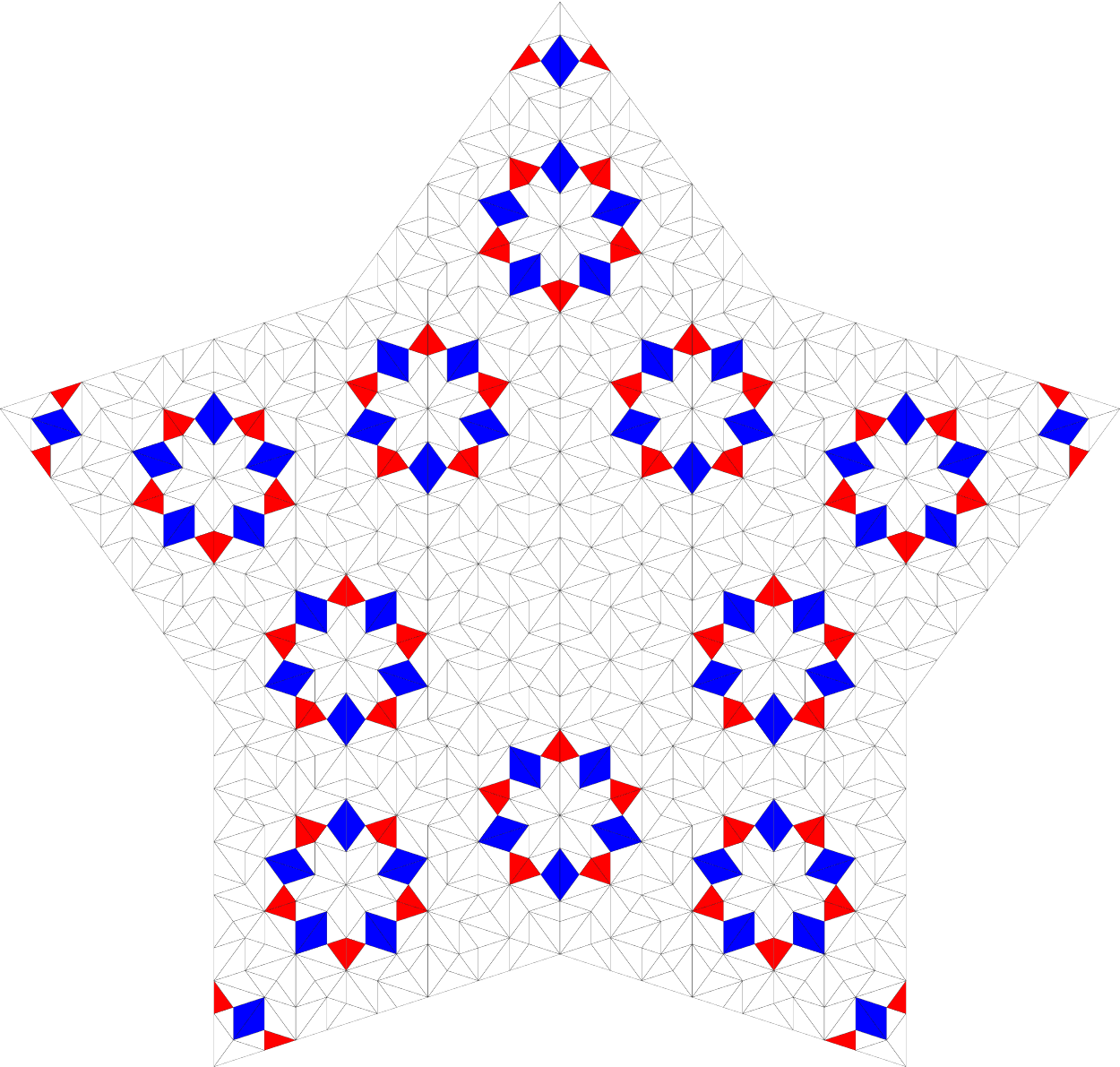}\\ \small $E=2$
\end{center}\end{minipage}
\quad
\begin{minipage}{2.6in} \begin{center}
\includegraphics[width=2.5in]{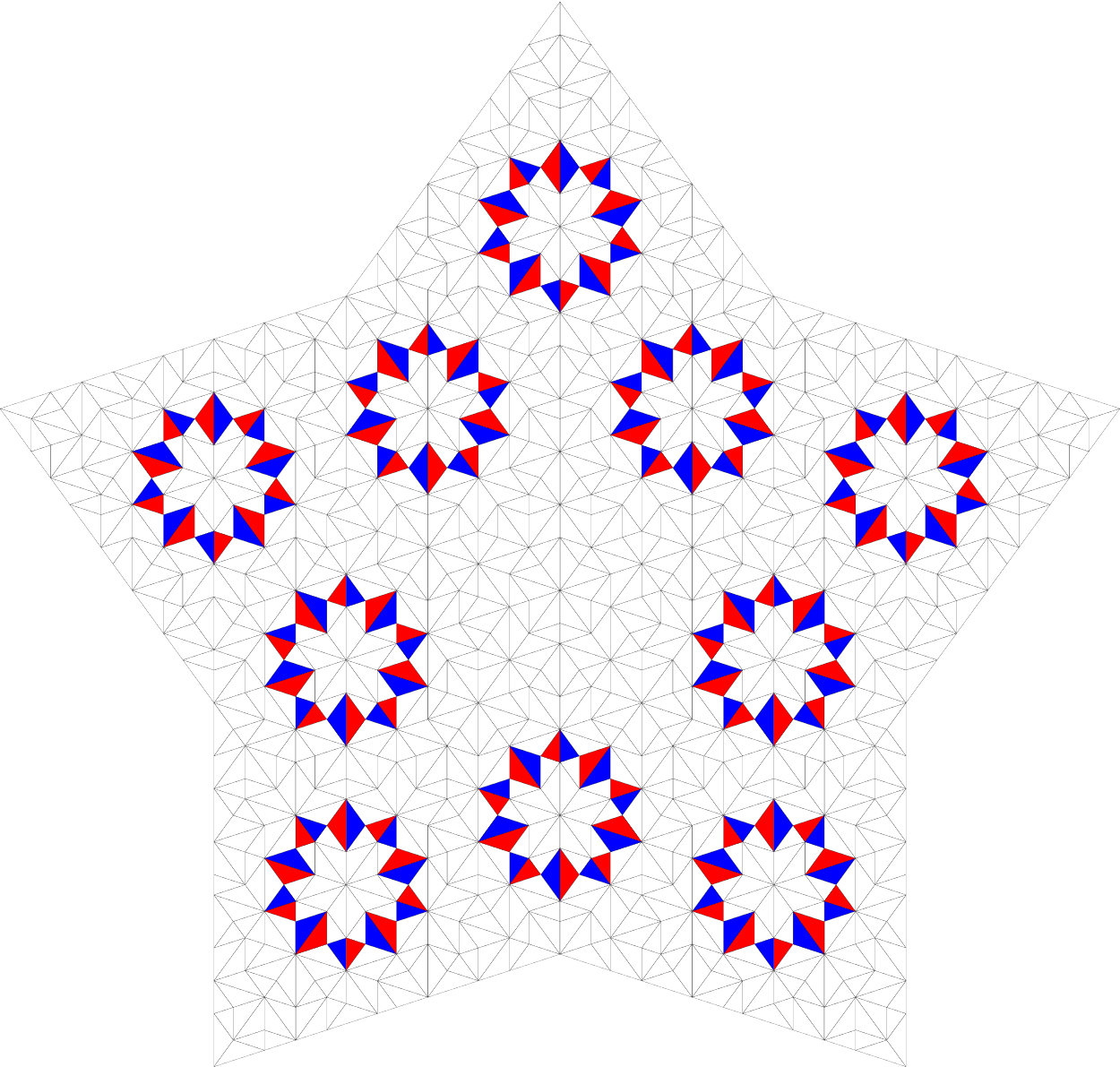}\\ \small $E=4$
\end{center}\end{minipage}
\end{center}
\caption{Localized modes for level~5 of the Robinson triangle substitution:
for both eigenvalues there are ten ``ring modes,'' each supported on 20~tiles;
for $E=2$, five additional modes, each supported on four tiles,
are made possible by the boundary.} \label{fig:tri:ringmode5}
\end{figure}

Let $\calR$ denote the 20-tile pattern corresponding to the ring mode.

\begin{prop} \label{prop:tri:ringcount}
Let $n \geq 5$ be given.
The total number of occurrences of $\calR$ in $\calT_n$ is bounded from below by
\begin{equation}
M_n = \sum_{\substack{\lfloor\frac{n-5}{2}\rfloor \geq k \geq 0 \\  } } O_{n-5-2k} + \frac{1}{2}A_{n-5-2k} \end{equation}
\end{prop}

The overall approach is similar to the corresponding calculation for the boat--star tiling, but the combinatorics are more complicated since the supports of the eigenfunctions may overlap with multiple supertiles.

\begin{proof}[Proof of Theorem~\ref{t:robinsonquant}]
With the help of Propositions~\ref{prop:tri:tilecount}, \ref{prop:tri:ring}, and \ref{prop:tri:ringcount}, we get
\begin{align*}
k_\robinson(2+) - k_\robinson(2-) 
& \geq \lim_{n\to\infty} \frac{(O_{n-5} + O_{n-7} + \cdots) + \frac{1}{2}(A_{n-5} + A_{n-7} + \cdots)}{O_n+A_n} \\[1mm]
& = \lim_{n \to \infty} \frac{(F_{2n-9} + F_{2n-13} + \cdots) + \frac{1}{2} (F_{2n-10} + F_{2n-14} + \cdots )}{F_{2n+2}} \\[1mm]
& = \left(\varphi^{-11} + \frac{1}{2}\varphi^{-12}\right) \left( 1- \varphi^{-4} \right)^{-1} \\[1mm]
& = \frac{65-29\sqrt{5}}{20},
\end{align*}
as desired. Since the same patterns produce the eigenfunctions at energy $E = 4$, the same argument works for that energy as well.
\end{proof}

Figure~\ref{fig:tri:ringmode5} shows the locally-supported eigenfunctions at $E=2$ and $E=4$ for the 
level~5 tiling.  Each of these energies correspond to ten ring modes; however, $E=2$ has greater 
multiplicity as an eigenvalue of $\Delta_5$ because of boundary modes supported on four tiles.
The discrepancy of multiplicities between $E=2$ and $E=4$ grows at additional levels, 
as evident in the numerical calculations presented in Table~\ref{tbl:triangle}.
The convergence of the jump in the IDS at $E=2$ and $E=4$ is painfully slow. 
The largest tiling for which we have data contains 390,881,690 tiles, yet the IDS jump 
at $E=4$ only agrees with the theoretically computed lower bound to about four decimal places.

\begin{table}[t!]
\caption{Robinson triangle tiling: the level of the tiling, the number of tiles, the multiplicity of $E=2$, the multiplicity of $E=4$, the number of boundary eigenfunctions for $E=2$, and the jump in the approximant of the IDS at $E=4$. The theoretically obtained lower bound from Theorem~\ref{t:robinsonquant} is $(65-29\sqrt{5})/20 \approx 0.007701432625$.  \label{tbl:triangle}}
\begin{tabular}{crrrr<{\ \ }c}
\emph{level} & \multicolumn{1}{c}{\emph{tiles}} &		\multicolumn{1}{c}{$E=2$} & 	\multicolumn{1}{c}{$E=4$} &	\multicolumn{1}{c}{\emph{boundary}} & \multicolumn{1}{c}{$k_{\robinson,n}(4+) - k_{\robinson,n}(4-)$} \\
\hline
    1 &             30  &            0  &            1 &       0  &  0.03333333\ldots \\ 
    2 &             80  &            1  &            1 &       0  &  0.01250000\ldots \\ 
    3 &            210  &            5  &            0 &       5  &  0.00000000\ldots \\ 
    4 &            550  &            6  &            1 &       5  &  0.00181818\ldots \\ 
    5 &         1\,440  &           15  &           10 &       5  &  0.00694444\ldots \\ 
    6 &         3\,770  &           36  &           21 &      15  &  0.00557029\ldots \\ 
    7 &         9\,870  &           90  &           65 &      25  &  0.00658561\ldots \\ 
    8 &        25\,840  &          216  &          181 &      35  &  0.00700464\ldots \\ 
    9 &        67\,650  &          550  &          495 &      55  &  0.00731707\ldots \\ 
   10 &       177\,110  &       1\,411  &       1\,316 &      95  &  0.00743041\ldots \\ 
   11 &       463\,680  &       3\,650  &       3\,495 &     155  &  0.00753752\ldots \\ 
   12 &    1\,213\,930  &       9\,471  &       9\,226 &     245  &  0.00760010\ldots \\ 
   13 &    3\,178\,110  &      24\,675  &      24\,280 &     395  &  0.00763976\ldots \\ 
   14 &    8\,320\,400  &      64\,401  &      63\,756 &     645  &  0.00766261\ldots \\ 
   15 &   21\,783\,090  &     168\,285  &     167\,240 &  1\,045  &  0.00767751\ldots \\ 
   16 &   57\,028\,870  &     440\,046  &     438\,361 &  1\,685  &  0.00768665\ldots \\ 
   17 &  149\,303\,520  &  1\,151\,215  &  1\,148\,490 &  2\,725  &  0.00769231\ldots \\ 
   18 &  390\,881\,690  &  3\,012\,556  &  3\,008\,141 &  4\,415  &  0.00769578\ldots \\ 
\hline
\end{tabular}
\end{table}




\section{Rhombi}

\begin{figure}[b!]

\includegraphics[width=1.7in]{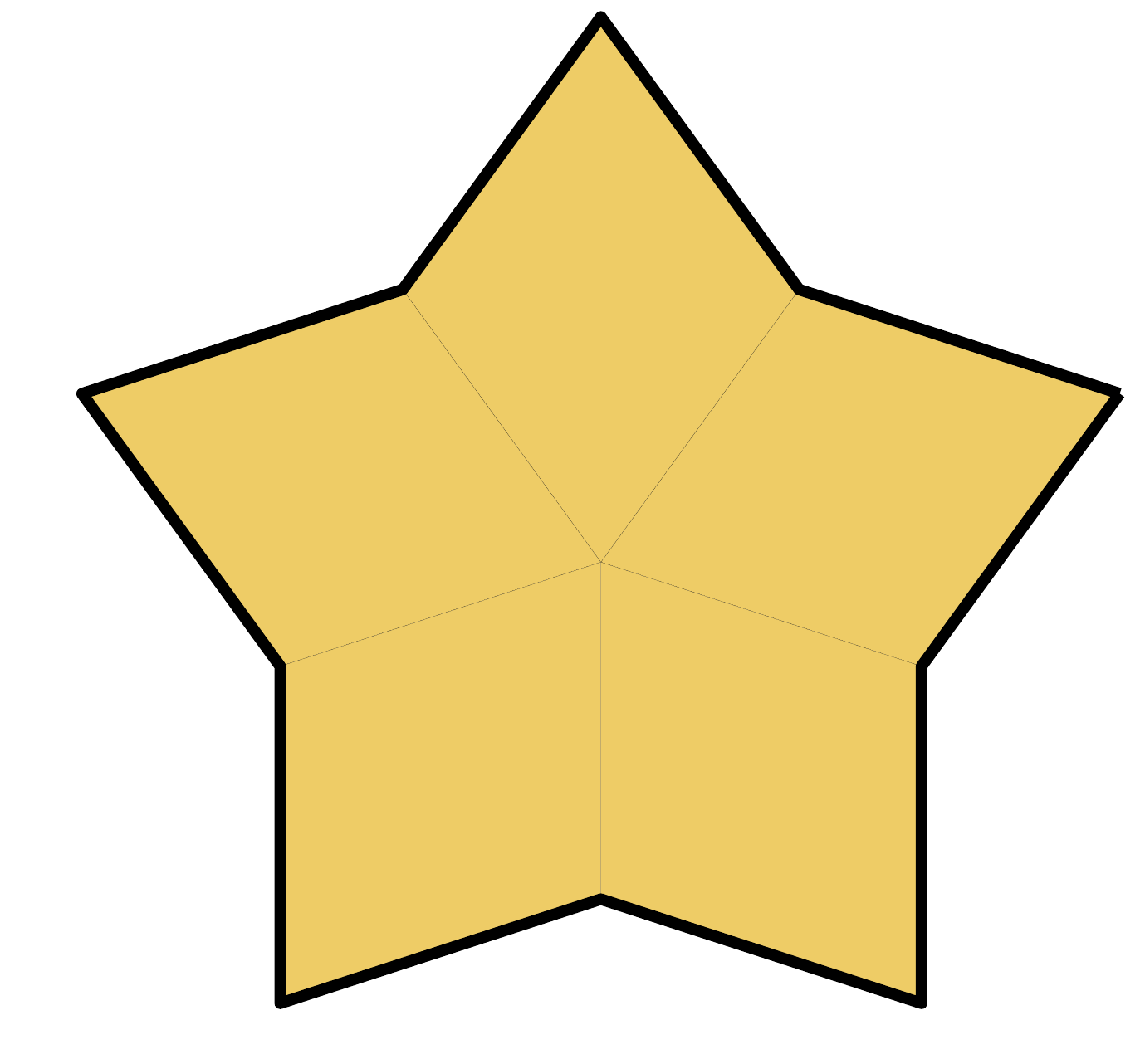}\qquad
\includegraphics[width=1.7in]{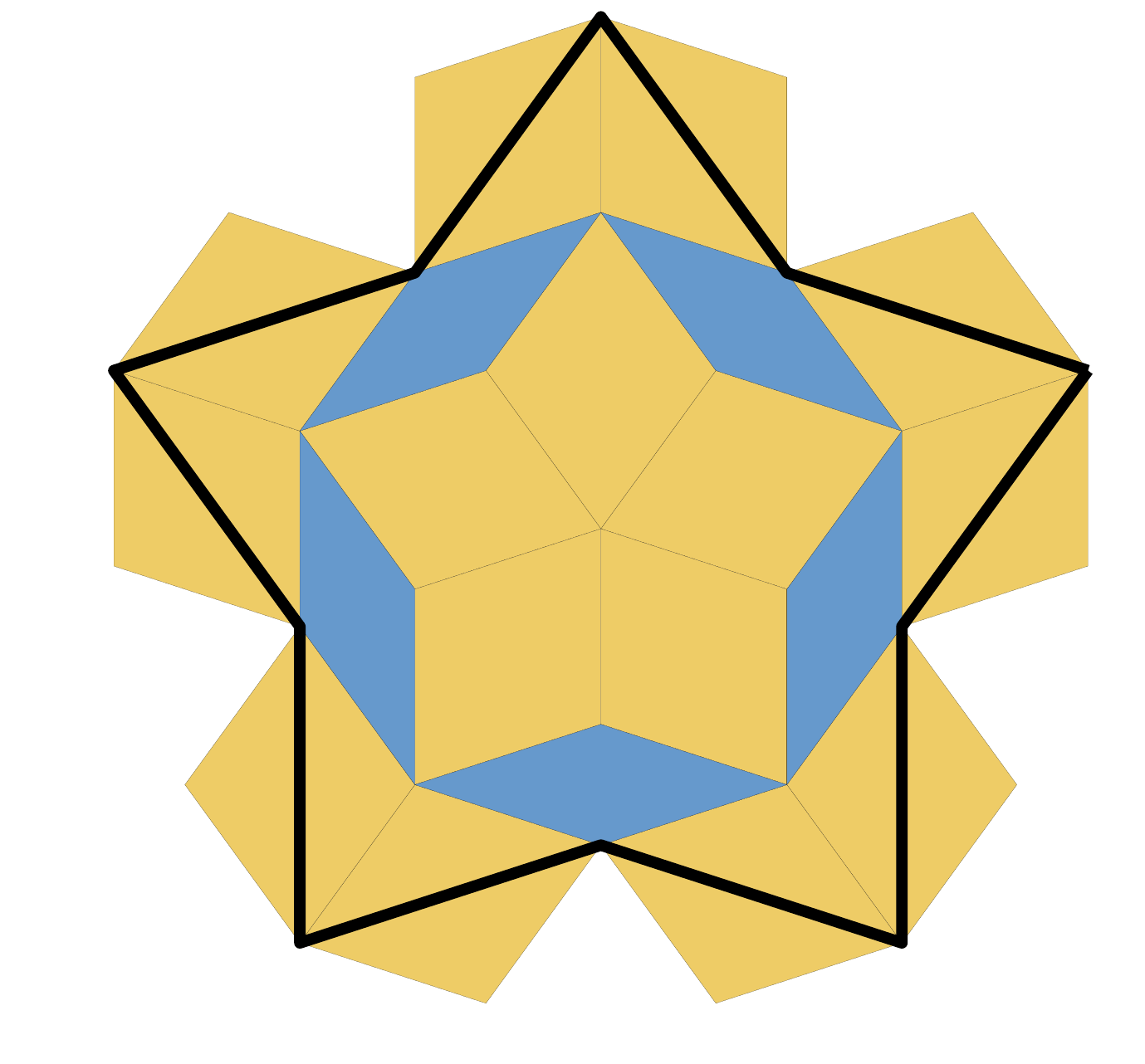}

\begin{picture}(0,0)
\put(-86,4){\small \textsl{Level 0}}
\put( 60,4){\small \textsl{Level 1}}
\end{picture}

\vspace*{0.5em}
\includegraphics[width=1.7in]{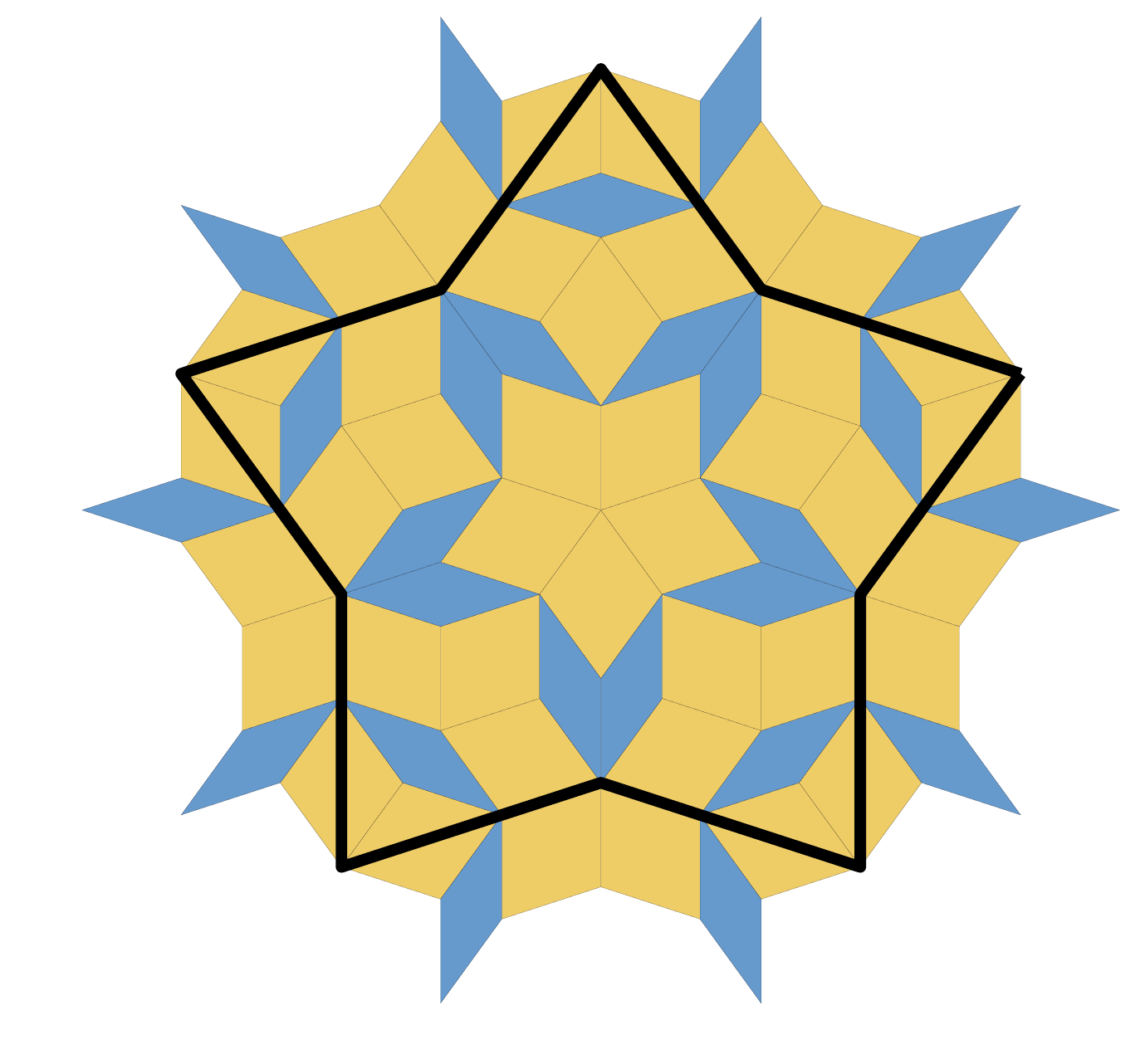}\qquad
\includegraphics[width=1.7in]{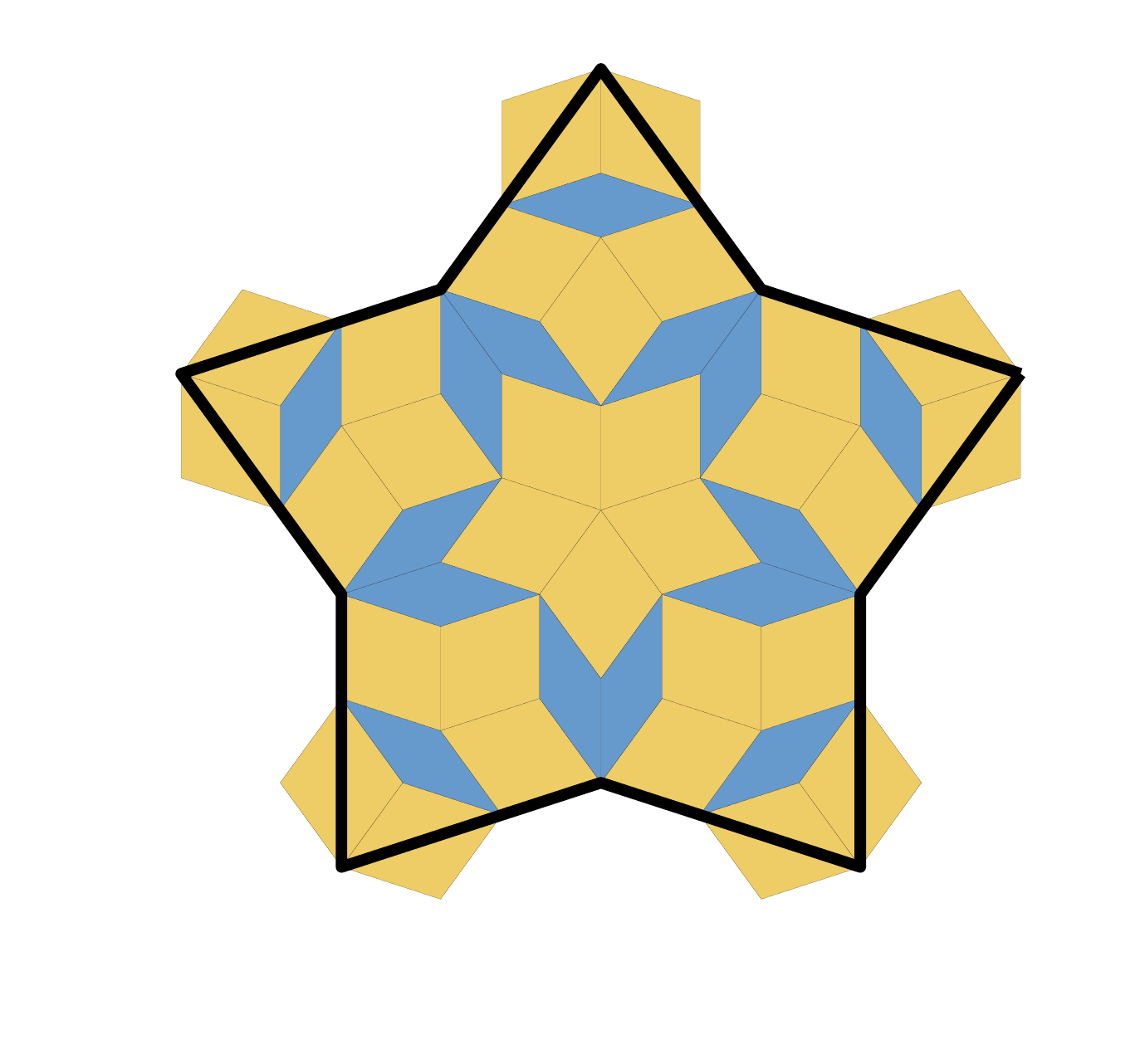}

\begin{picture}(0,0)
\put(-120,7){\small \textsl{Level 2} (untrimmed)}
\put( 32,7){\small \textsl{Level 2} (trimmed)}
\end{picture}

\vspace*{0.5em}
\includegraphics[width=1.7in]{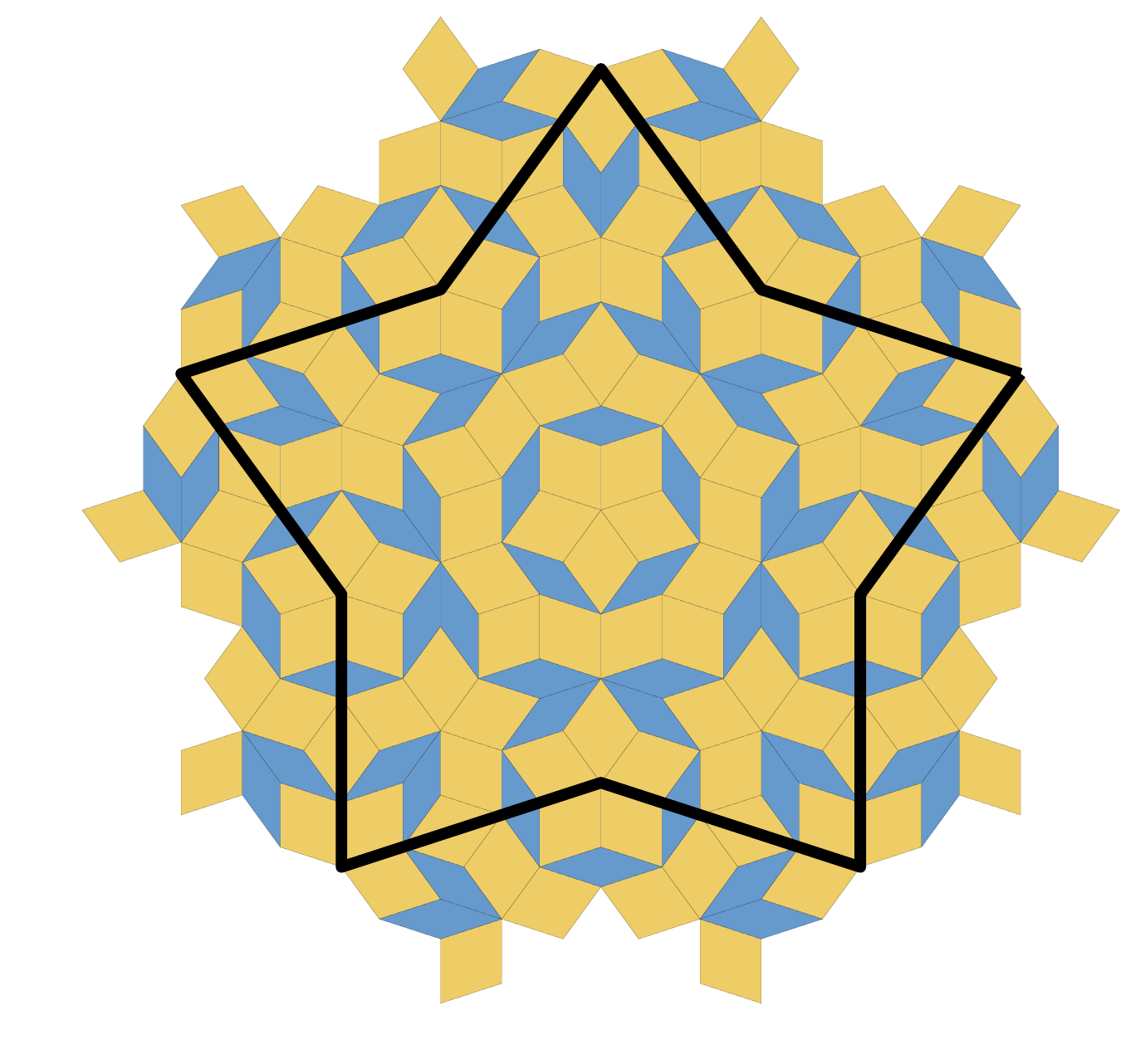}\qquad
\includegraphics[width=1.7in]{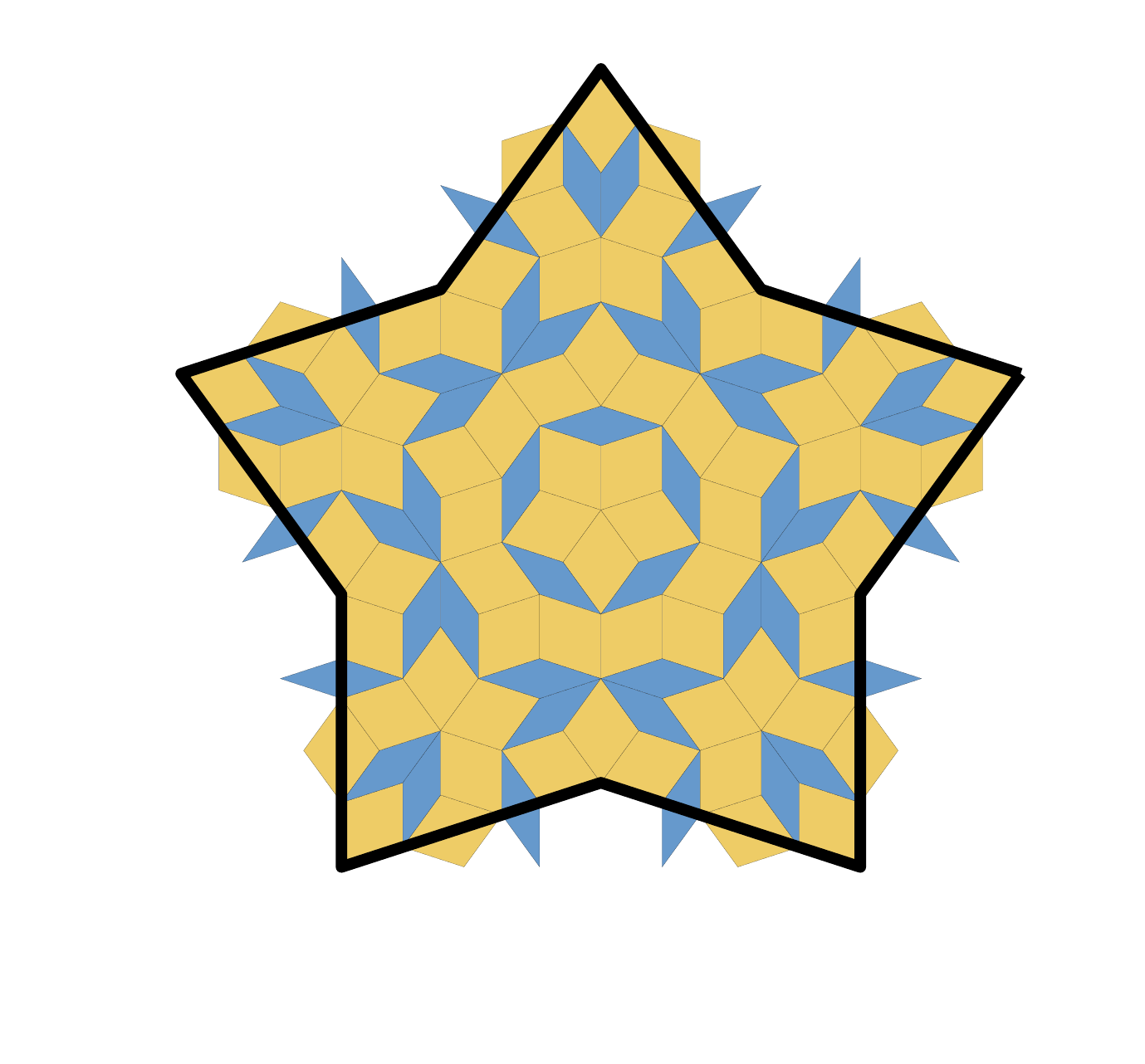}

\begin{picture}(0,0)
\put(-120,7){\small \textsl{Level 3} (untrimmed)}
\put( 32,7){\small \textsl{Level 3} (trimmed)}
\end{picture}

\caption{\label{fig:rhombus_trim}
Rhombus rules applied to five rhombi at level~0, trimming to the original star shape at each iteration.}
\end{figure}

\begin{definition}\label{def:RhombusRules}
The \emph{rhombus} substitution\footnote{Illustration following 
{\tt https://tilings.math.uni-bielefeld.de/substitution/penrose-rhomb/}} is given by

\begin{center}
\hspace*{-2em}
\includegraphics[width=1.25in]{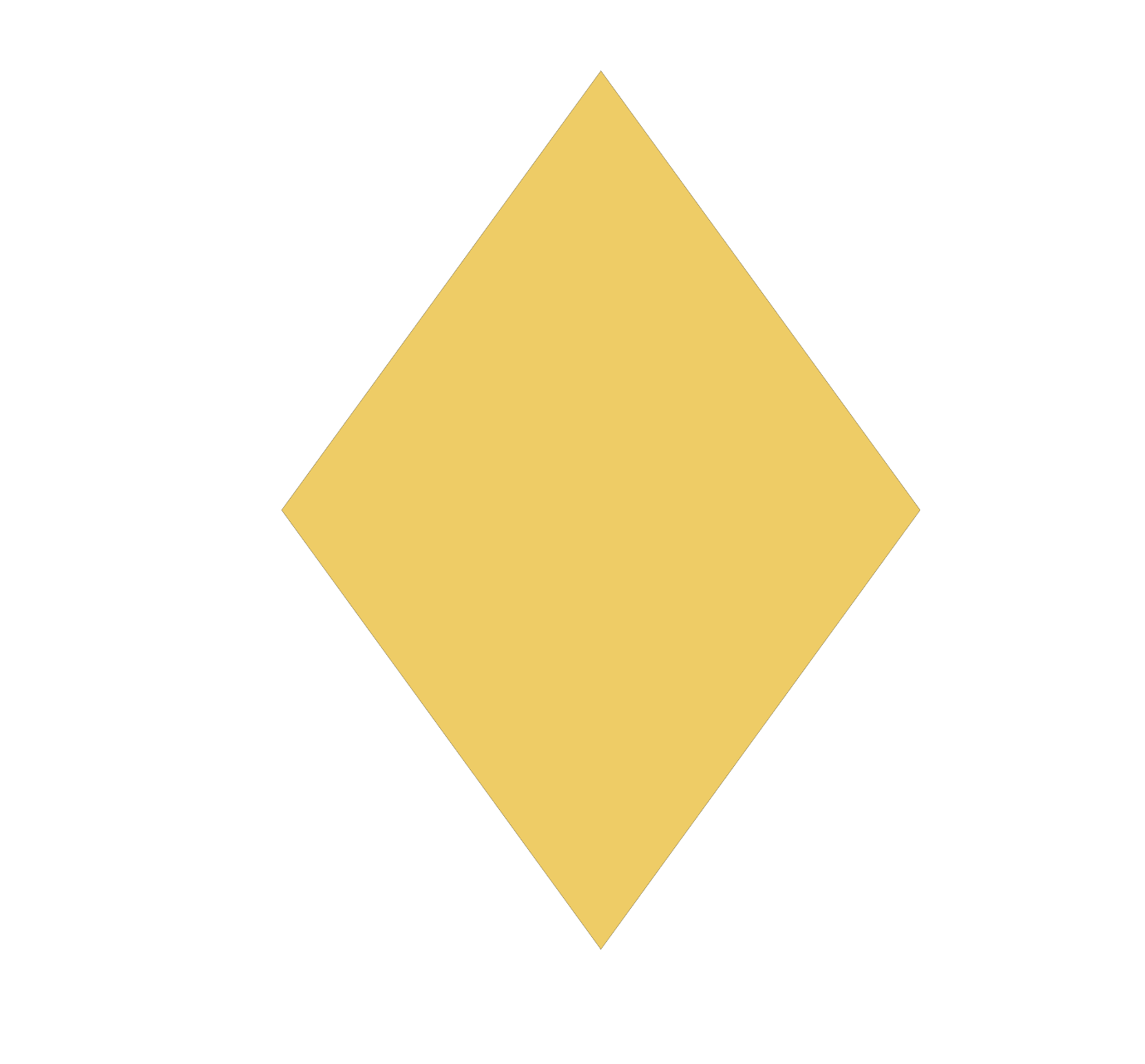}\quad
\includegraphics[width=1.25in]{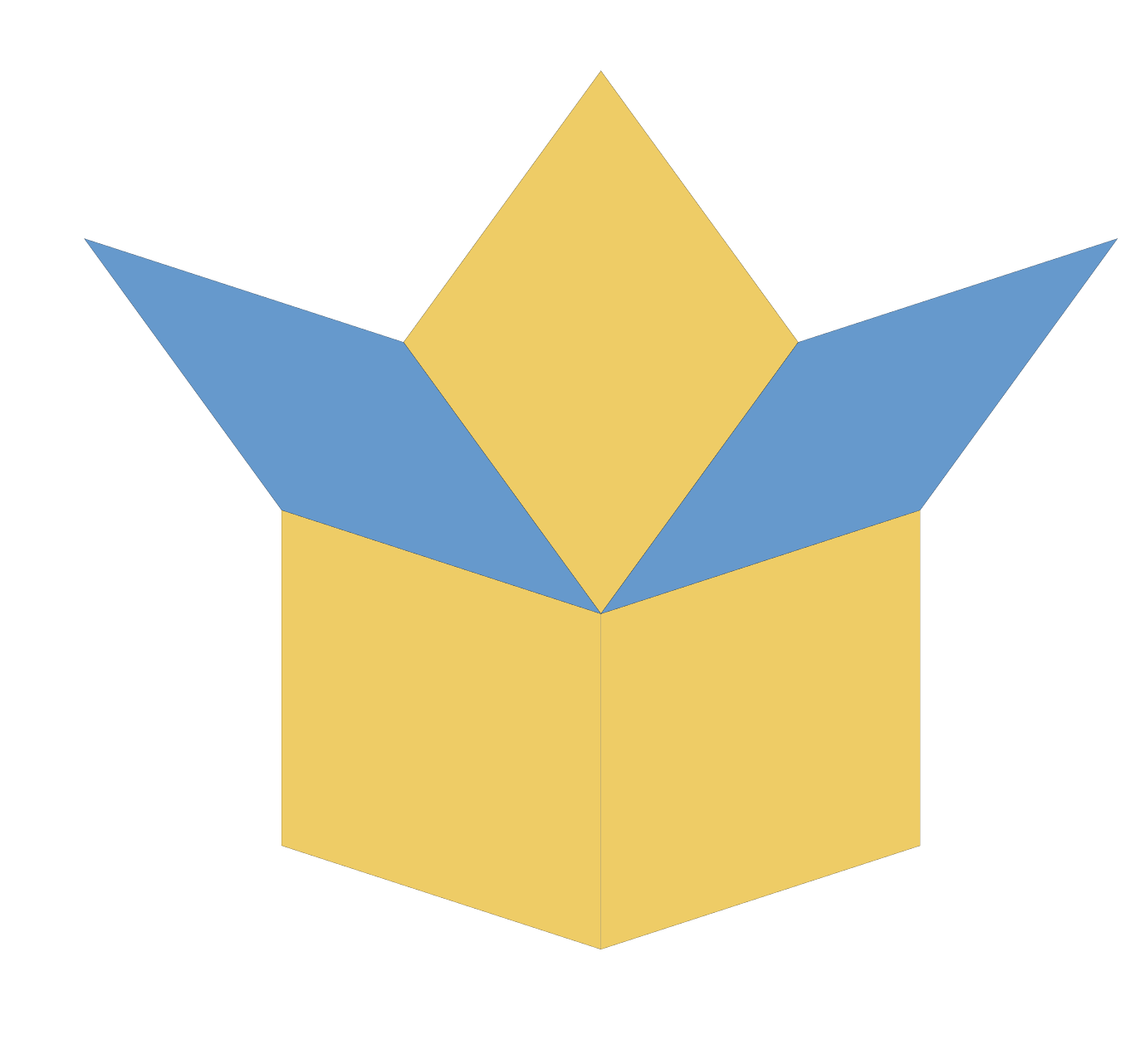}
\qquad\quad
\includegraphics[width=1.25in]{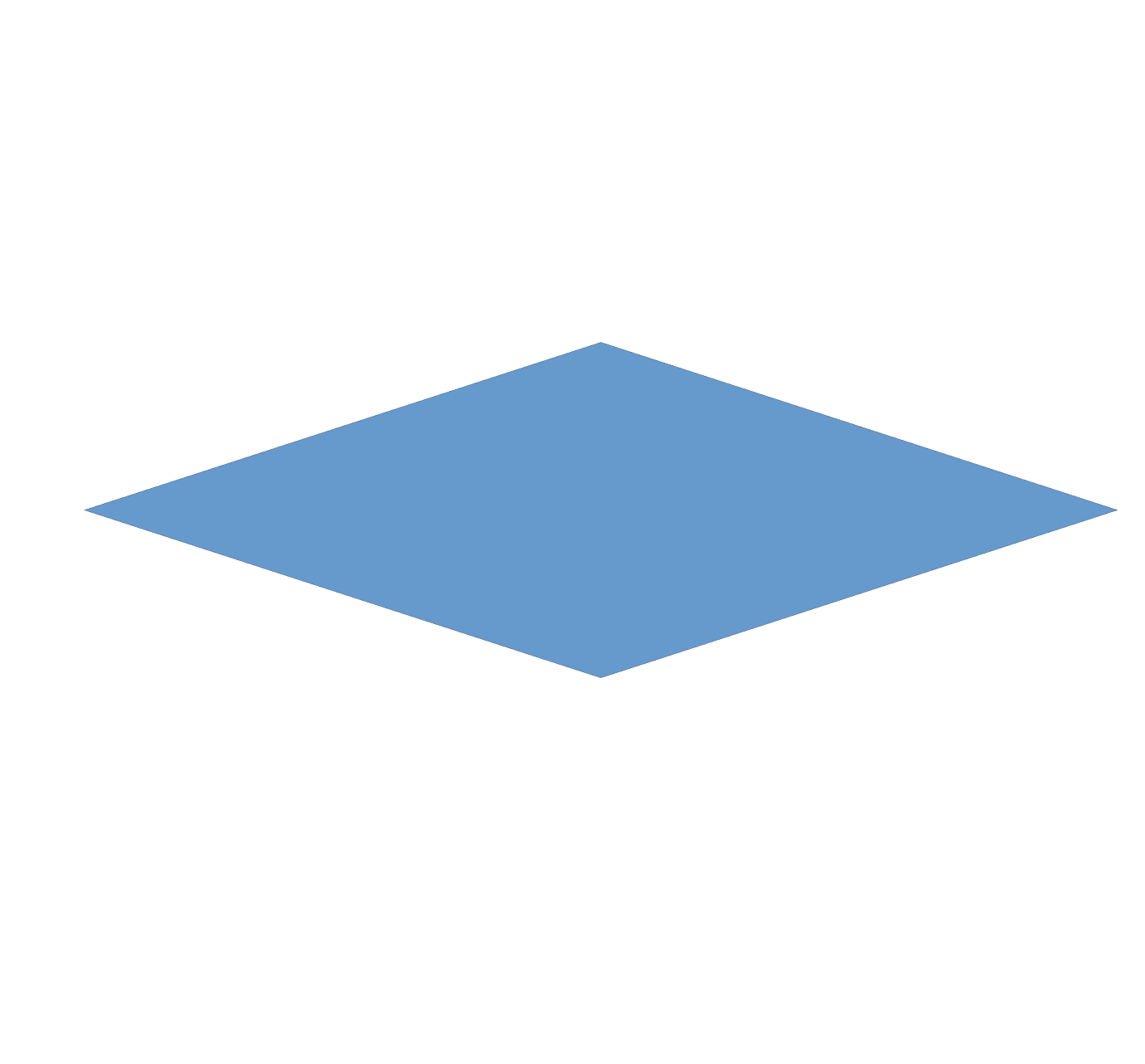}\qquad
\includegraphics[width=1.25in]{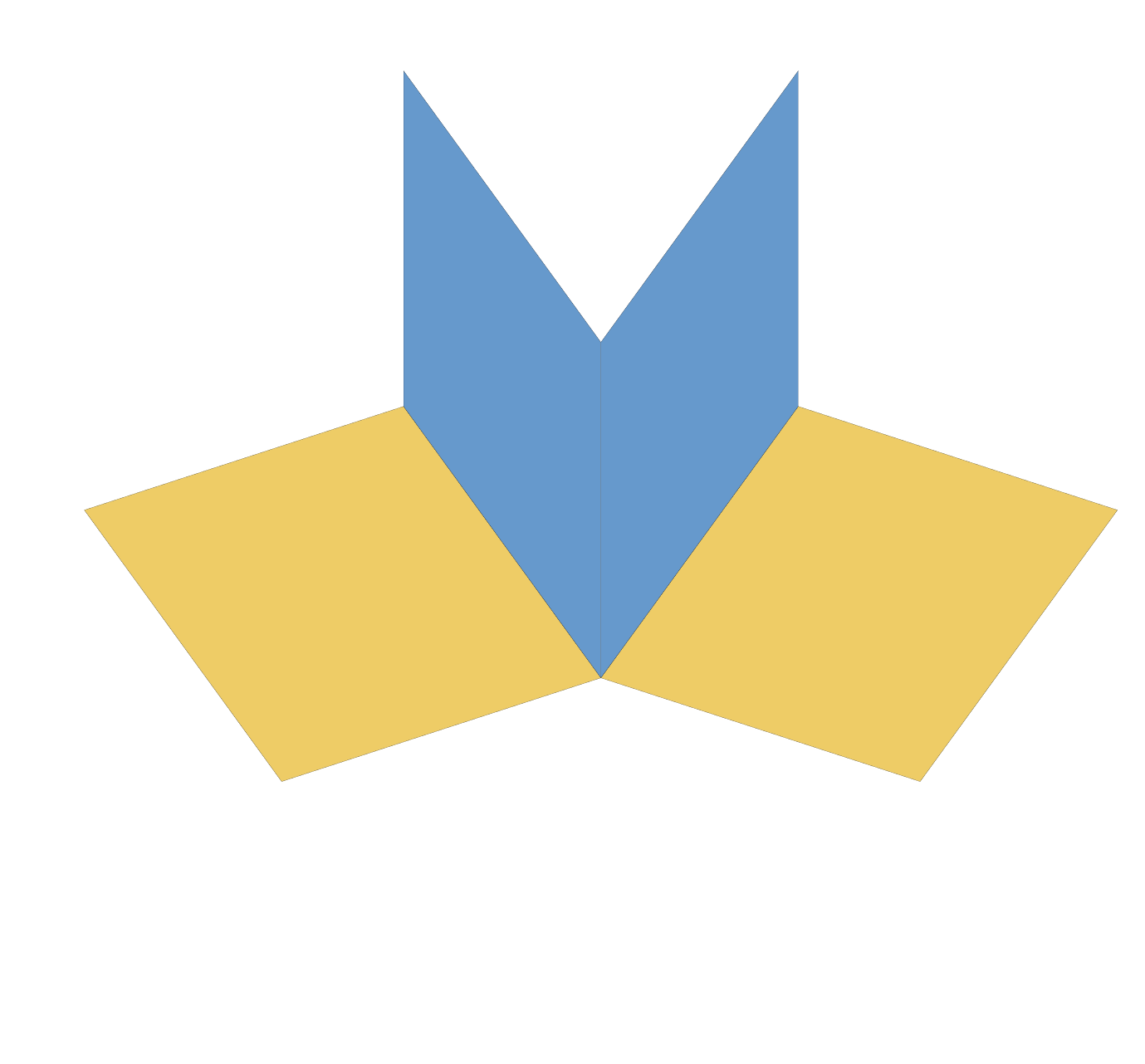}

\begin{picture}(0,0)
\put(-142,53){\large $\to$}
\put(100,53){\large $\to$}
\end{picture}
\end{center}
\end{definition}

\vspace*{-5pt}
In keeping with the star-shaped patterns generated by Definitions~\ref{def:triangleRules} and \ref{def:boatStarRules}, we proceed as in Figure~\ref{fig:rhombus_trim}, alternating between applications of the substitution rule and trimming to a star shape.
We obtain the eigenfunctions in Figures~\ref{fig:rhombus_ev}, \ref{fig:fatk2}, and \ref{fig:rhombus_more_ev}, and numerically compute the values in Table~\ref{fig:rhombus:multTable}.

\begin{figure}[b!]
\begin{center}
\begin{minipage}{2in}
\begin{center} 
\includegraphics[width=1.9in]{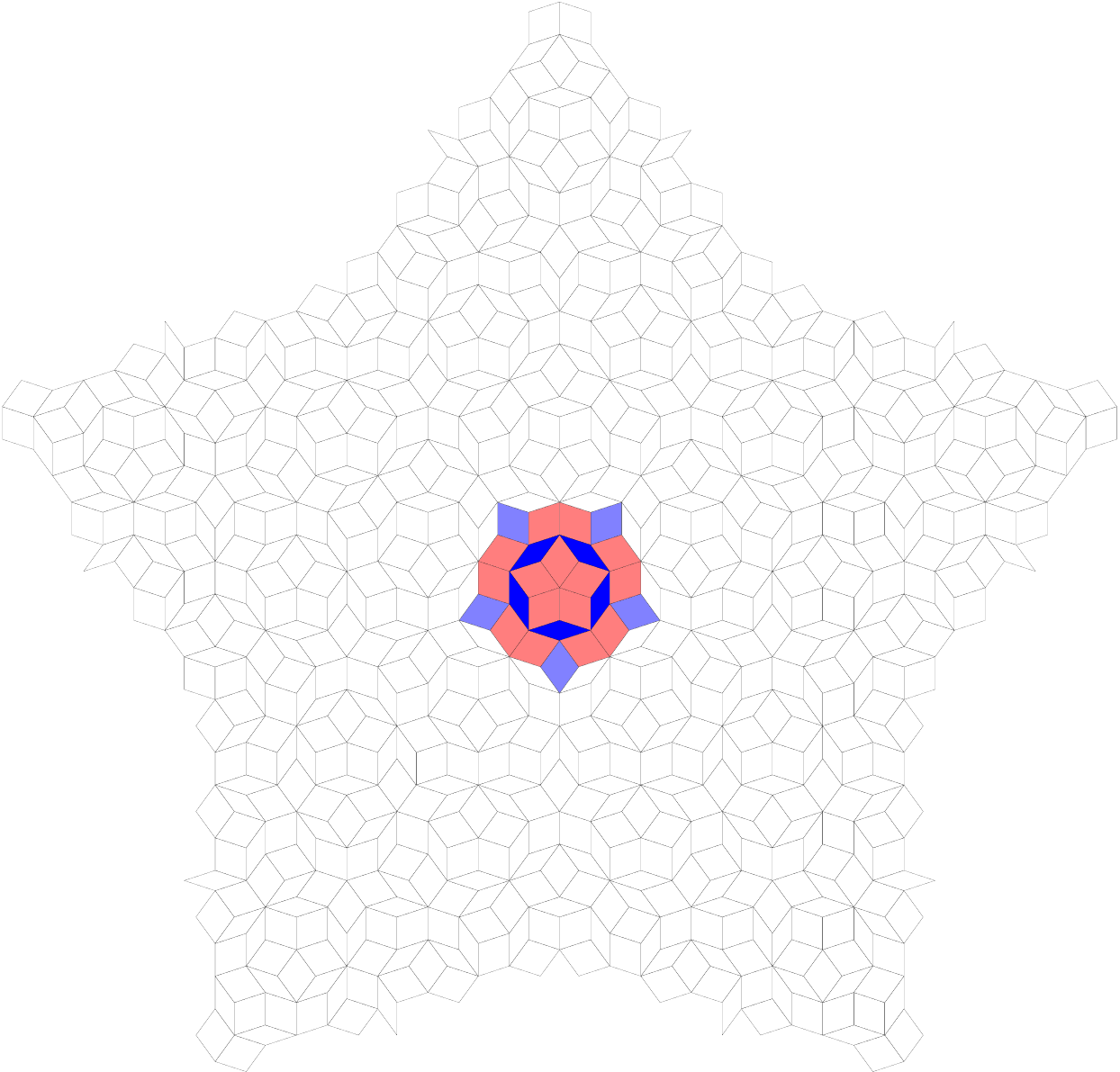}\\
\emph{filled circle}
\end{center}
\end{minipage}
\begin{minipage}{2in}
\begin{center} 
\includegraphics[width=1.9in]{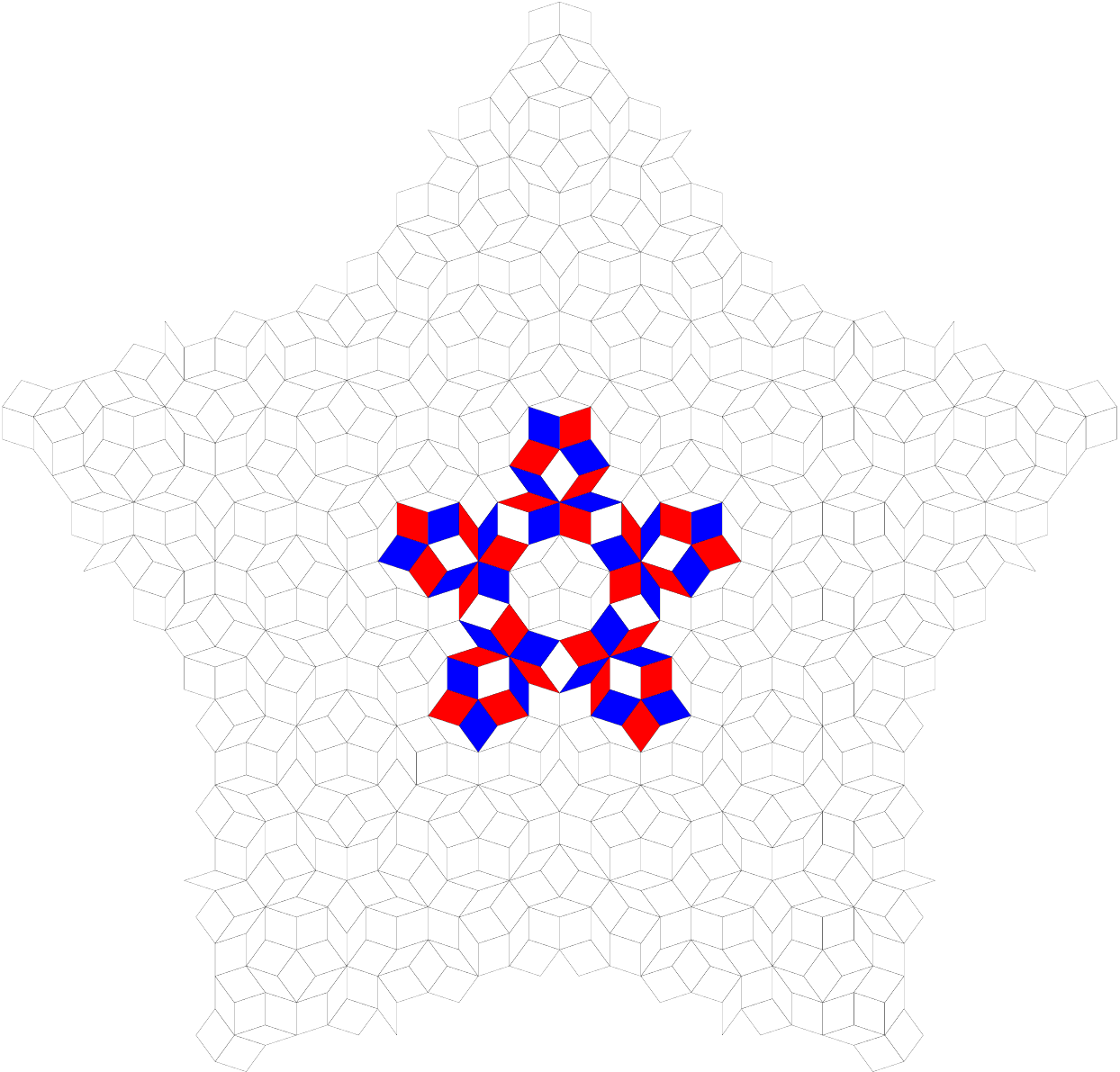}\\
\emph{big star}
\end{center}
\end{minipage}

\vspace*{10pt}
\begin{minipage}{2in}
\begin{center} 
\includegraphics[width=1.9in]{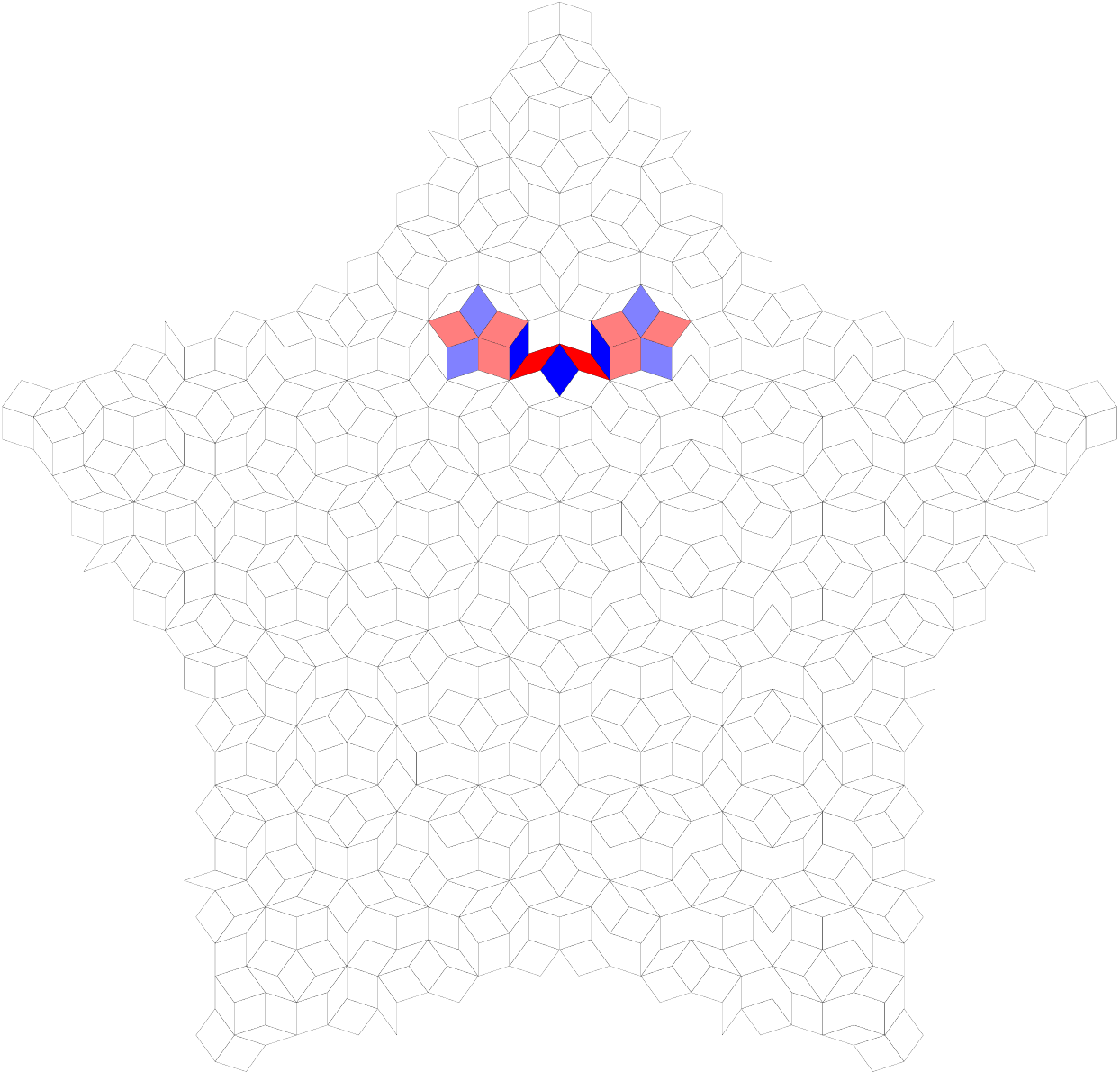}\\
\emph{two star}
\end{center}
\end{minipage}
\begin{minipage}{2in}
\begin{center} 
\includegraphics[width=1.9in]{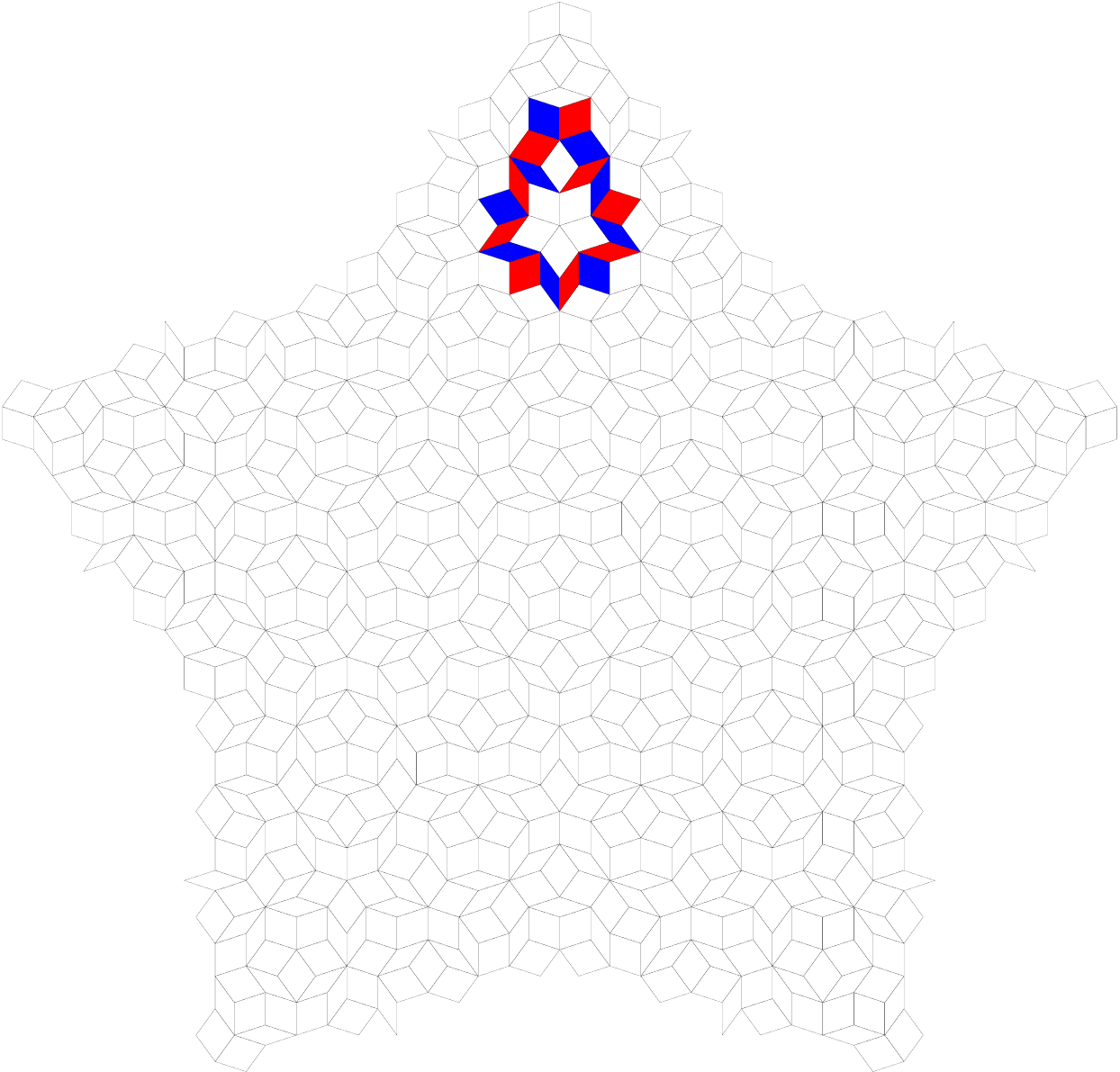}\\
\emph{diamond ring}
\end{center}
\end{minipage}
\end{center}

\caption{\label{fig:rhombus_ev}
For the rhombus tiling, four linearly independent eigenfunctions at energy $E=6$ at level~5. 
We refer to these configurations as the 
\emph{filled circle}, \emph{big star}, \emph{two star}, and \emph{diamond ring}.
(Dark blue and dark red correspond to values $\pm 1$; light blue and light red correspond to $\pm 1/2$.)
Level~5 exhibits 1~filled circle mode, 1~big star mode, 10~diamond ring modes, 
and 15~two star modes, giving a total of 27~linearly independent eigenfunctions.}
\end{figure}

In contrast with the locally-supported eigenfunctions identified for the boat--star and Robinson triangle cases
in the last two sections, for the rhombus tiling a variety of distinct locally-supported eigenfunction configurations
with overlapping local support emerge at low levels, associated with energy $E=6$.
Figure~\ref{fig:rhombus_ev} shows four such mode shapes at level~5.  
The eigenfunctions take values $+1$ on blue tiles and $-1$ on red tiles 
(with intermediate values indicated by a difference in shading) and are zero 
on the uncolored tiles. We classify these mode shapes as:
\begin{itemize}
\item \emph{filled circle}, supported on 25~tiles;
\item \emph{big star}, supported on 50~tiles;
\item \emph{two star}, supported on 15~tiles;
\item \emph{diamond ring}, supported on 18~tiles.
\end{itemize}

Locally-supported eigenfunctions of the Laplacian on the rhombus tiling have been studied before, 
notably by Fujiwara, Arai, Tokihiro, and Kohmoto in~\cite{FATK1988PRB}. 
(It bears mentioning that Equation~(2.1) in \cite{FATK1988PRB} differs from our Equation~(\ref{eq:laplacian}) 
in the first term, so energy $E=2$ in \cite{FATK1988PRB} corresponds to $E=6$ in Table~\ref{fig:rhombus:multTable}.)
Indeed, Fujiwara et al.\ describe five eigenfunctions named A1, A2, B, C, and D, 
and find a cumulative frequency of 0.068189. 
In Figure~\ref{fig:rhombus_ev}, we refer to their B~state as a \emph{two star} mode, and their D~state as a \emph{diamond ring} mode. 
Figure~\ref{fig:fatk2} shows instances where the modes A1, A2, and~C can be constructed as linear combinations of the primitive mode 
shapes in Figure~\ref{fig:rhombus_ev}.  The construction of~C reveals a subtlety that complicates the counting of linearly independent modes: in some cases a \emph{diamond ring} mode can be realized as the combination of a different \emph{diamond ring} and a \emph{two star} mode.

\begin{figure}[b!]
\begin{center}
\begin{picture}(420,160)
\put(0,84){\includegraphics[width=72pt]{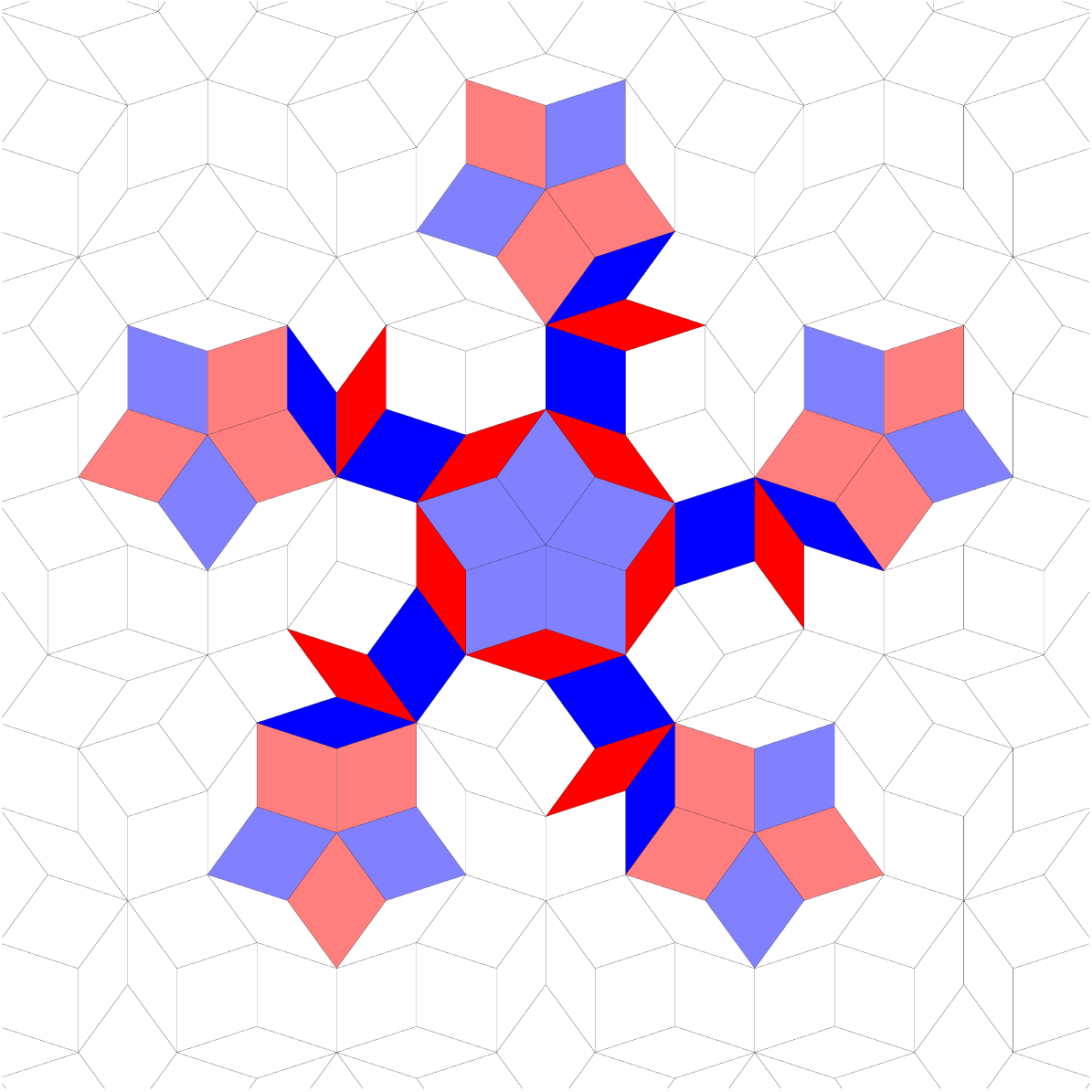}}
\put(0,0){\includegraphics[width=72pt]{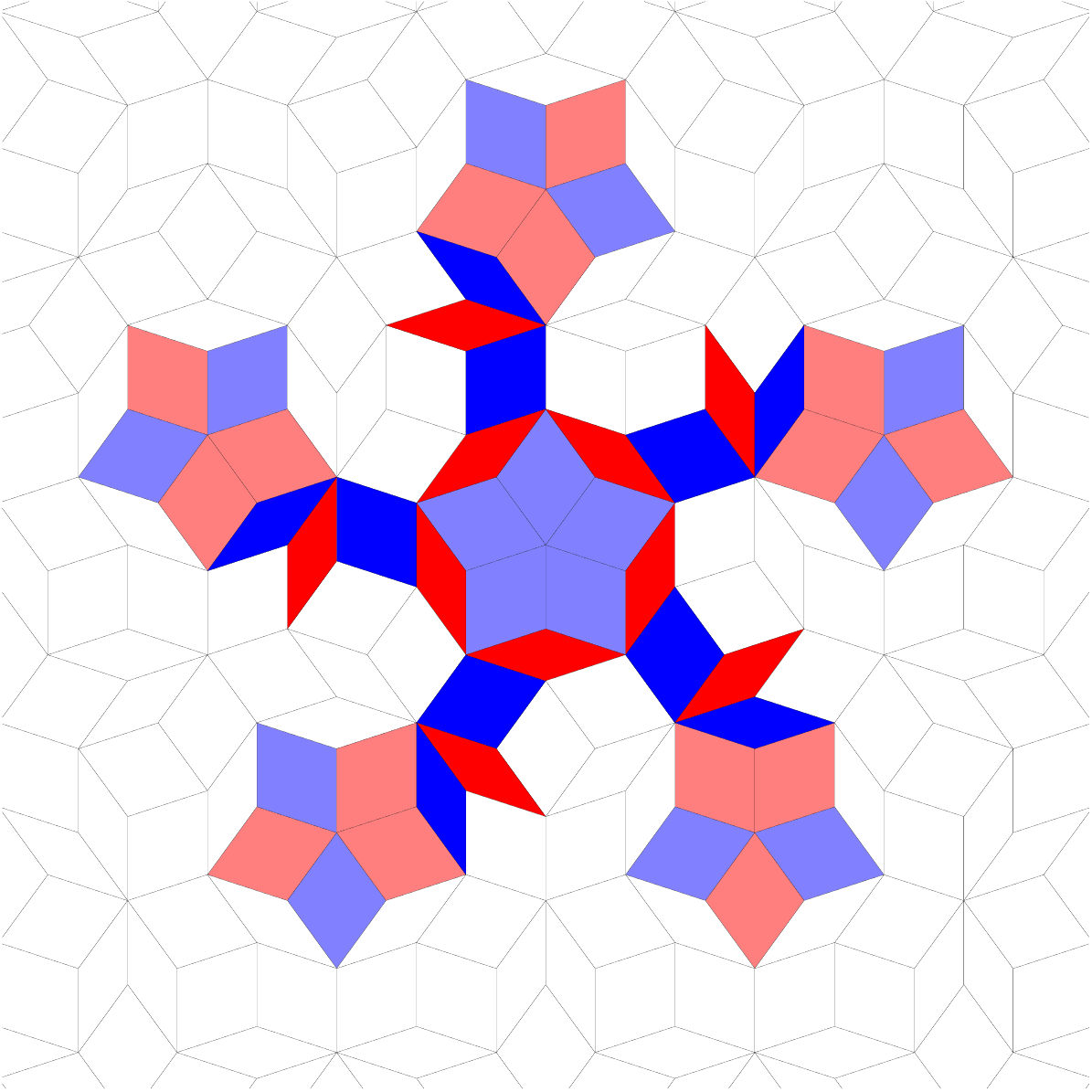}}
\put(102,90){\includegraphics[width=60pt]{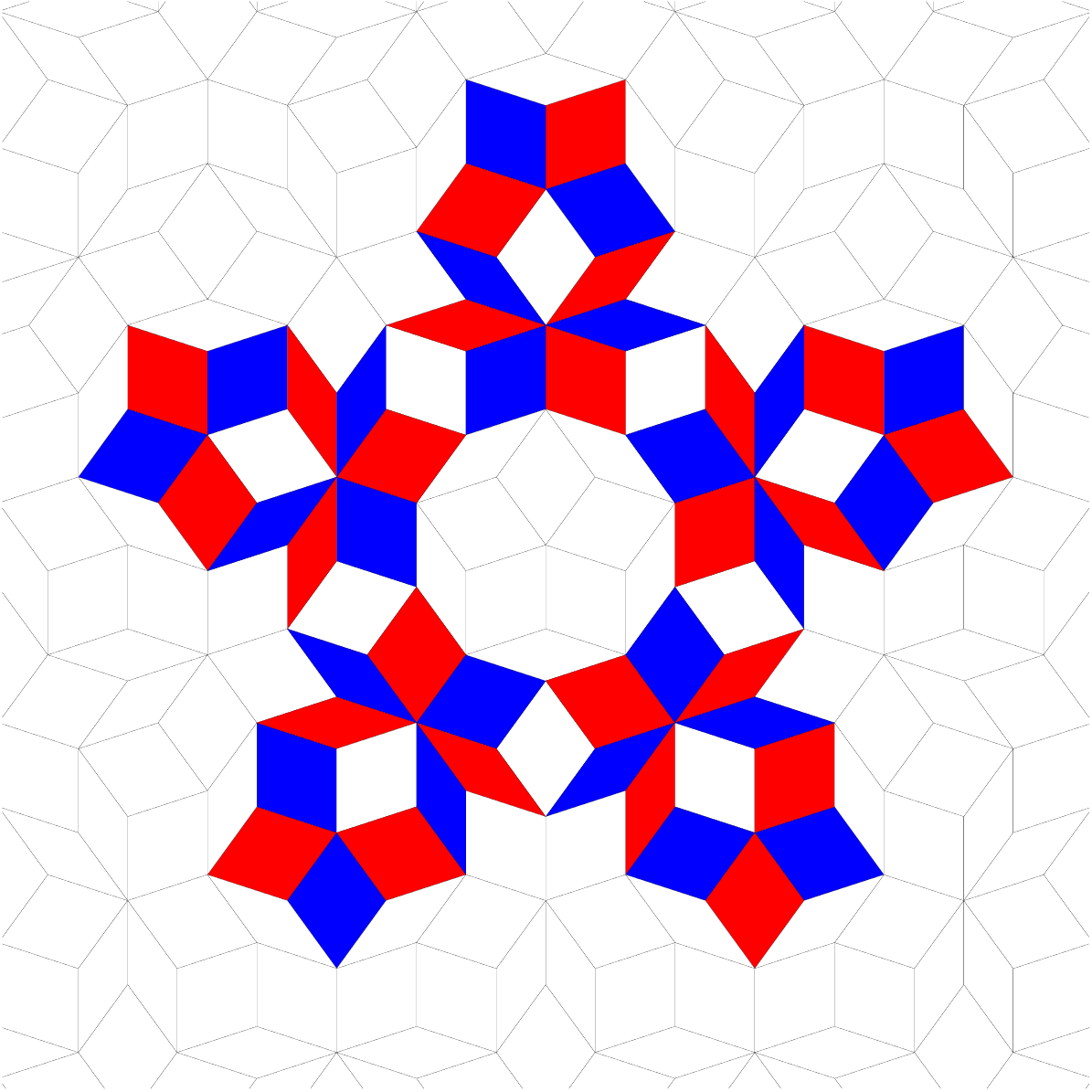}}
\put(102,6){\includegraphics[width=60pt]{Images/rhombus_fatk5_basis7_zoom}}
\put(73,117){\small ${} = -\phantom{1}$}
\put(73,33){\small ${} = +\phantom{1}$}
\put(167,74){\small $+$}
\put(180,74){$\left\{\rule[-80pt]{0pt}{160pt}\right.$}
\put(190,117){\small $-2$}
\put(196,33){\small $+$}
\put(273,117){\small $+$}
\put(273,33){\small $+$}
\put(348,117){\small $+$}
\put(348,33){\small $+$}
\put(210,90){\includegraphics[width=60pt]{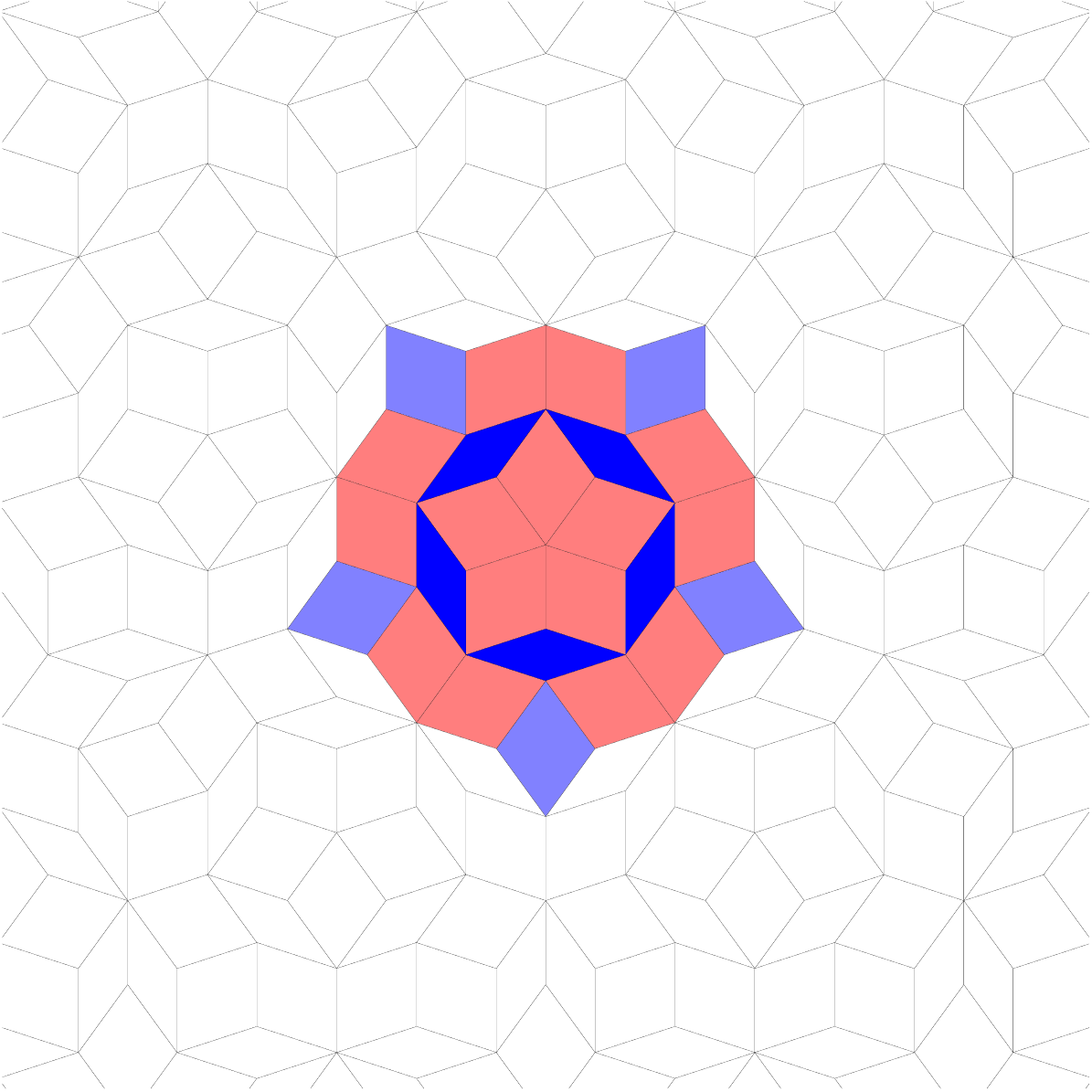}}
\put(285,90){\includegraphics[width=60pt]{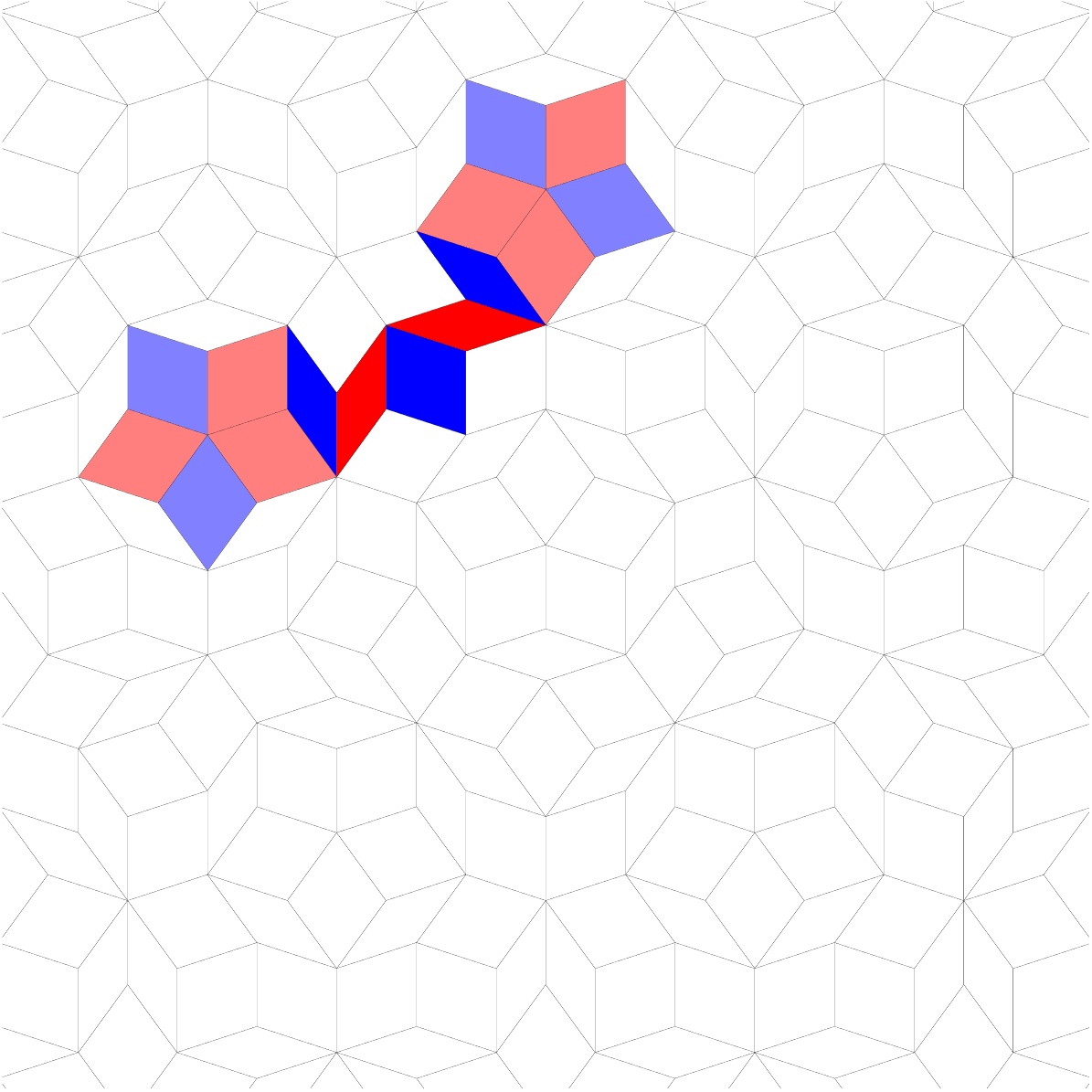}}
\put(360,90){\includegraphics[width=60pt]{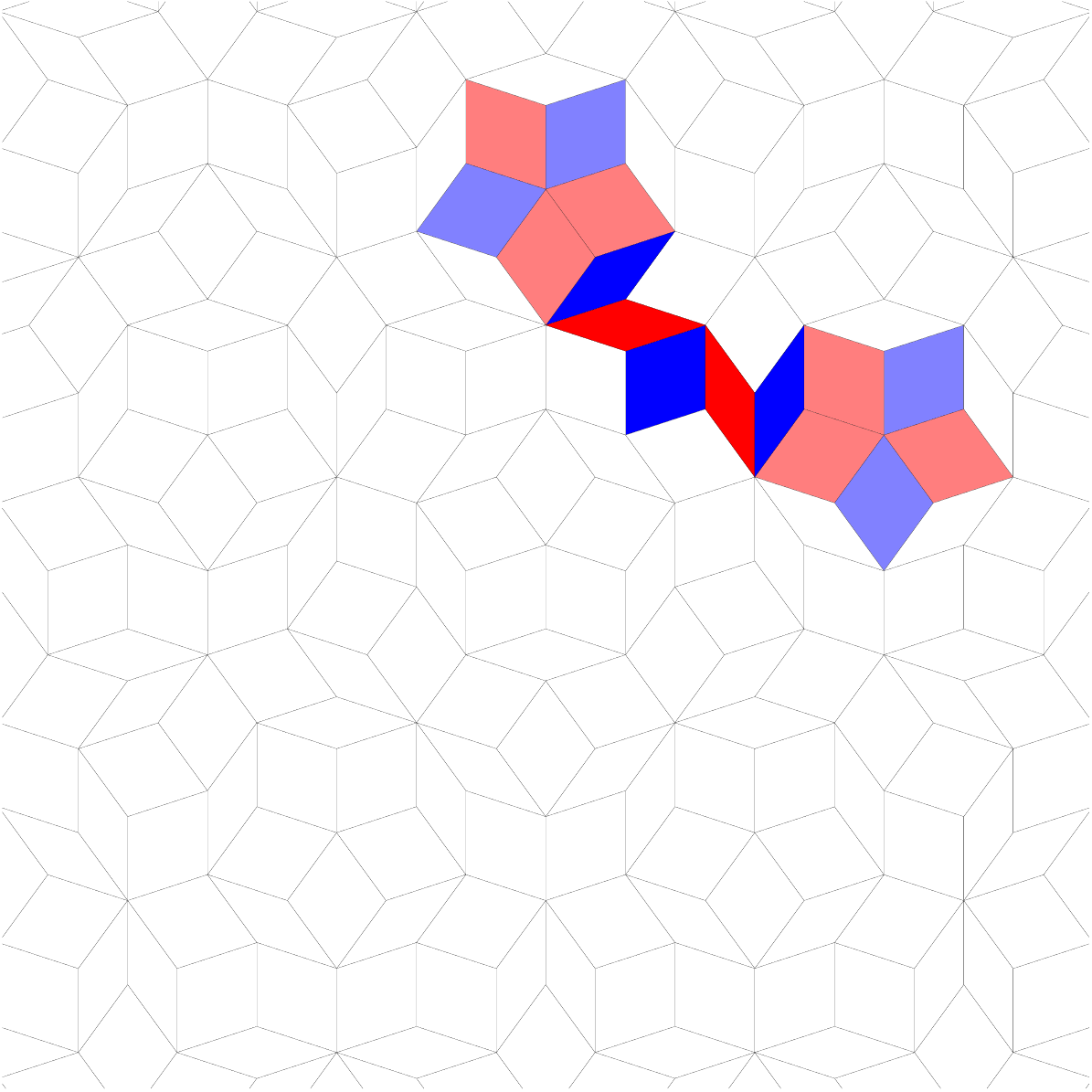}}
\put(210,6){\includegraphics[width=60pt]{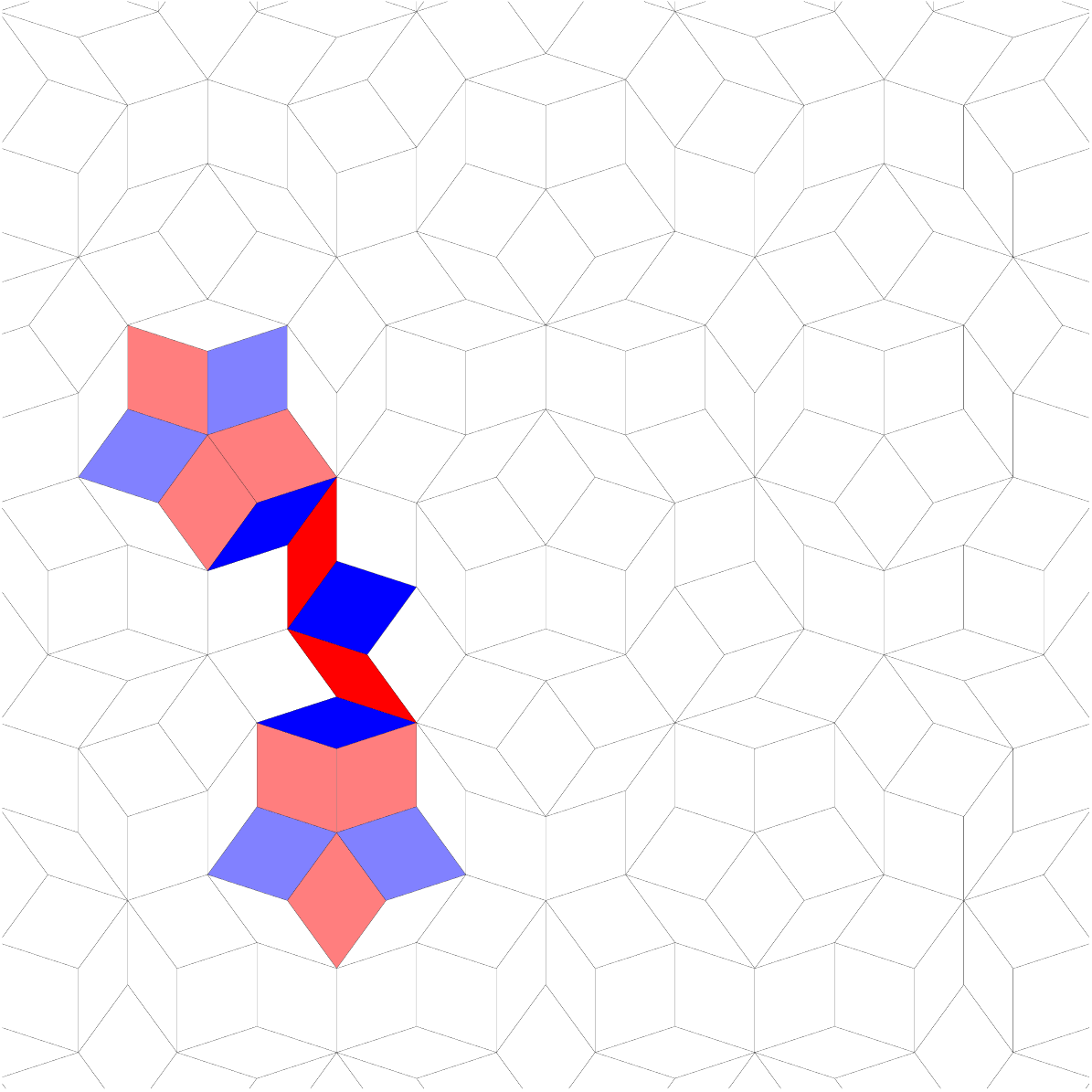}}
\put(285,6){\includegraphics[width=60pt]{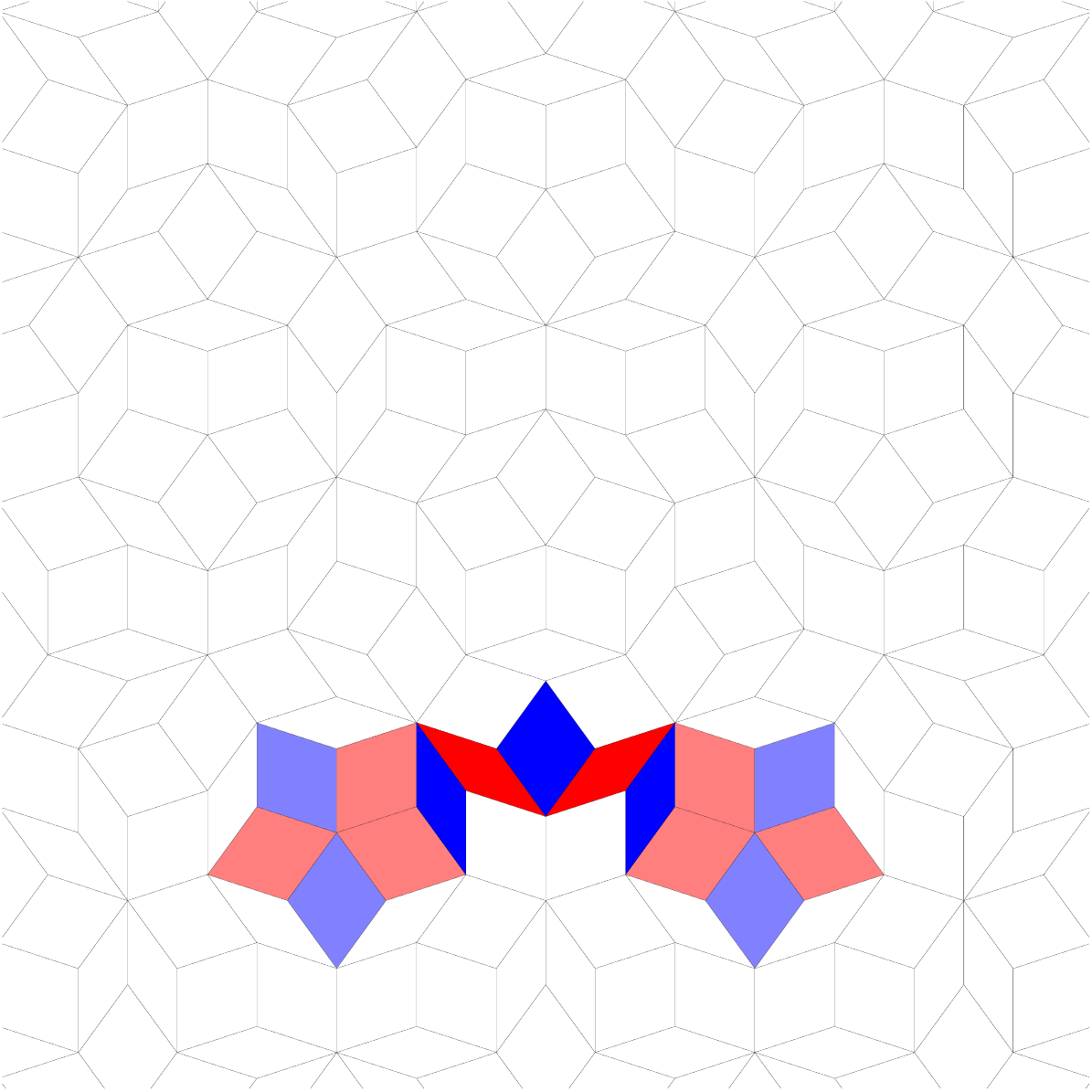}}
\put(360,6){\includegraphics[width=60pt]{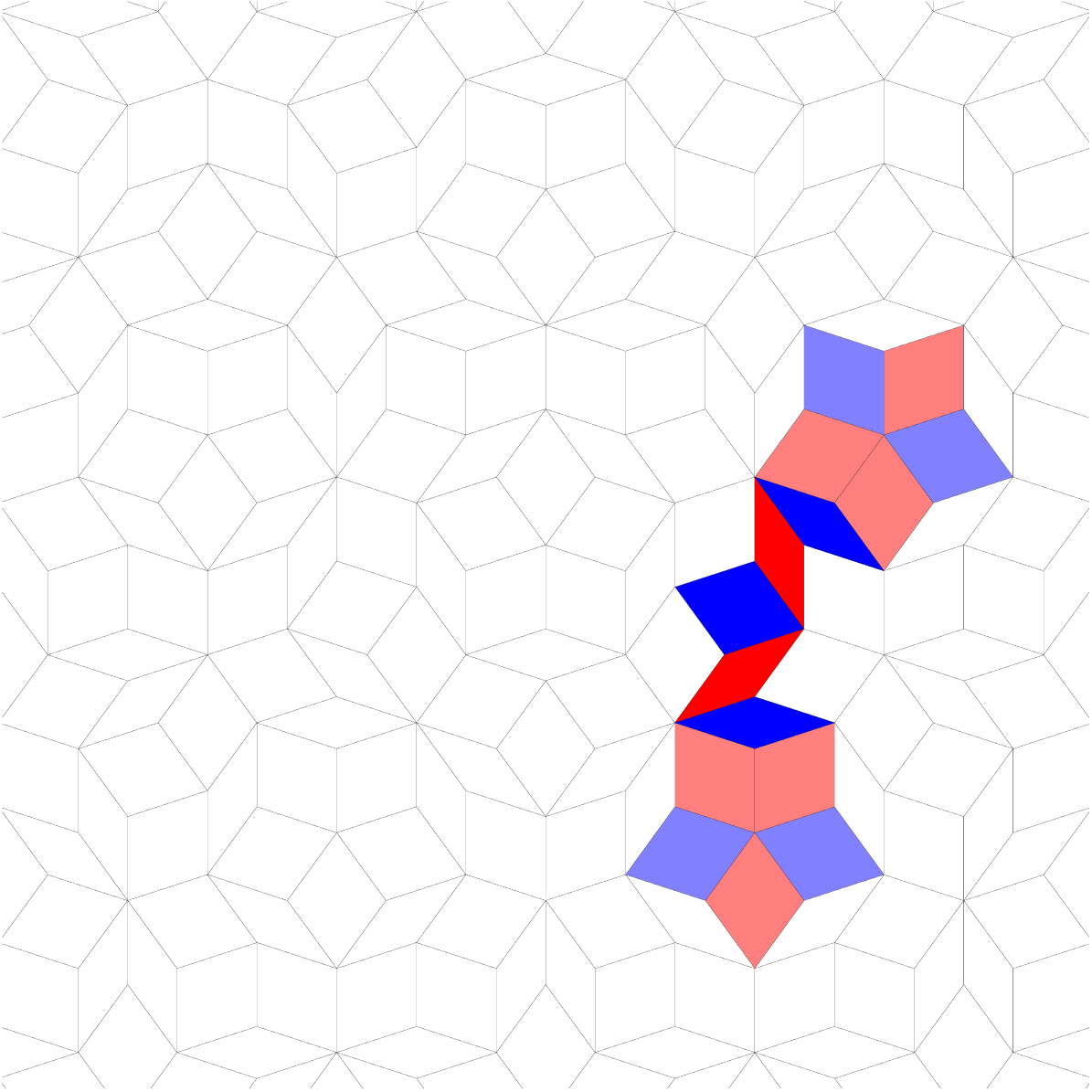}}
\end{picture}

\begin{picture}(420,93)
\put(0,0){\includegraphics[width=72pt]{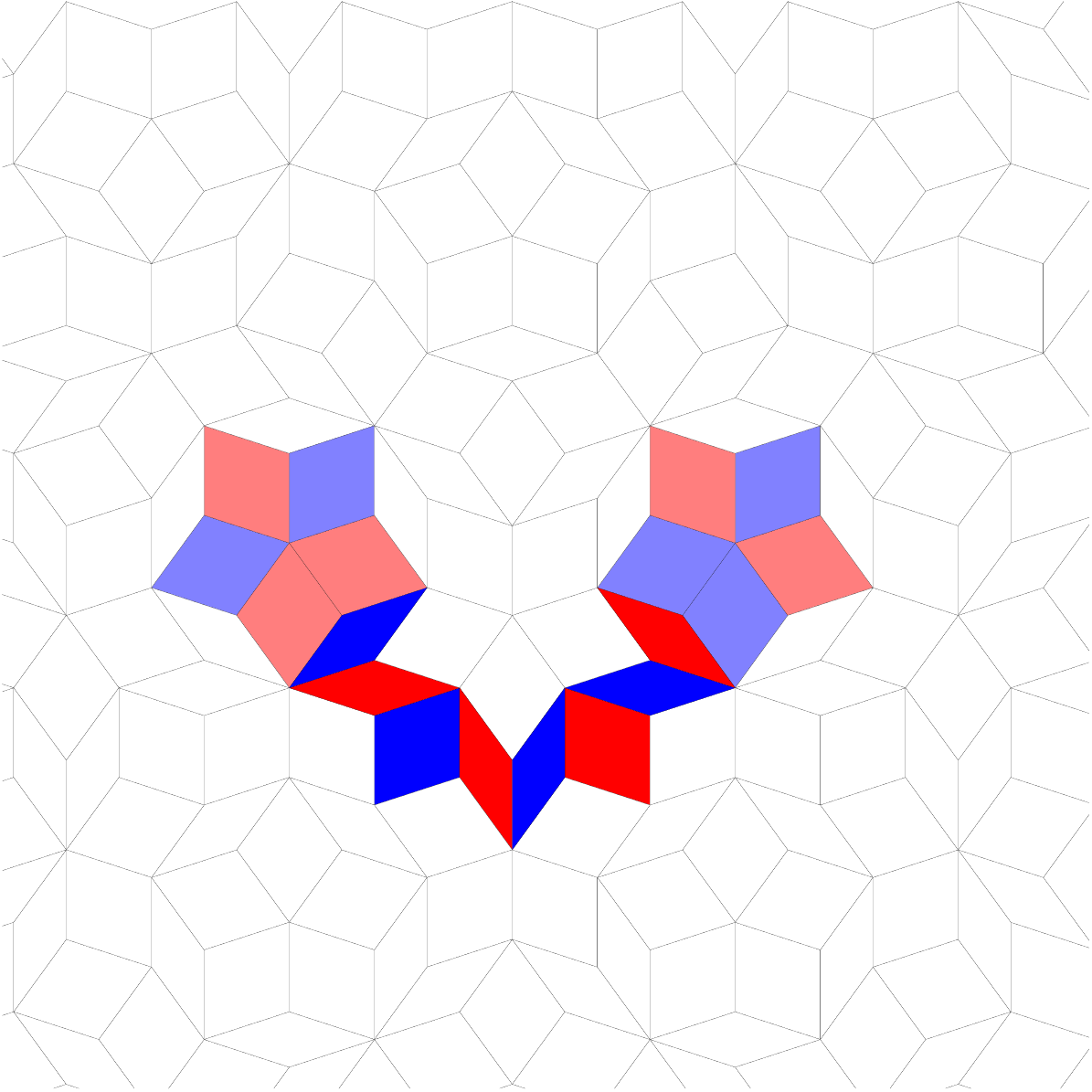}}
\put(73,33){\small ${} =\,2$}
\put(102,6){\includegraphics[width=60pt]{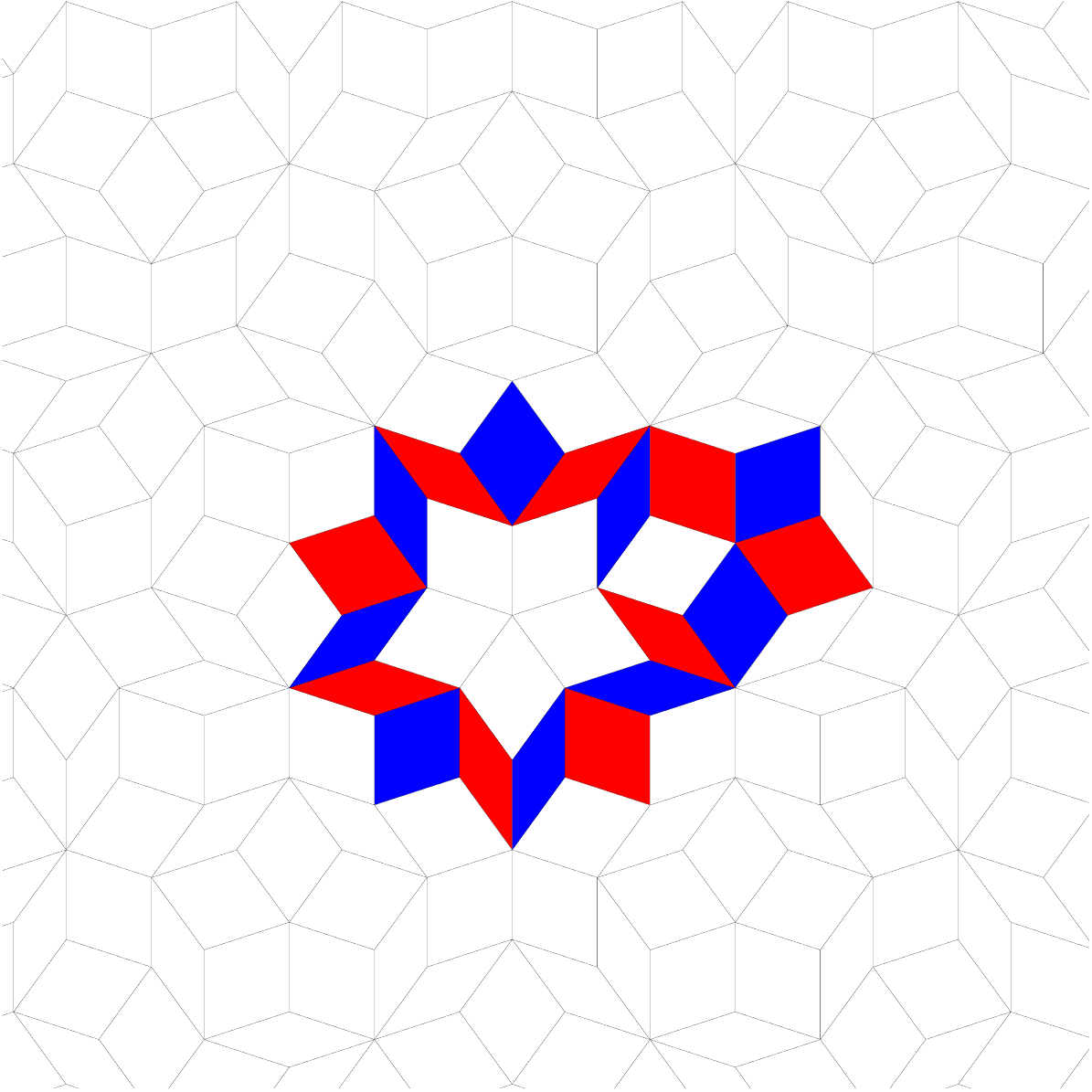}}
\put(165,33){\small ${} - 2$}
\put(190,6){\includegraphics[width=60pt]{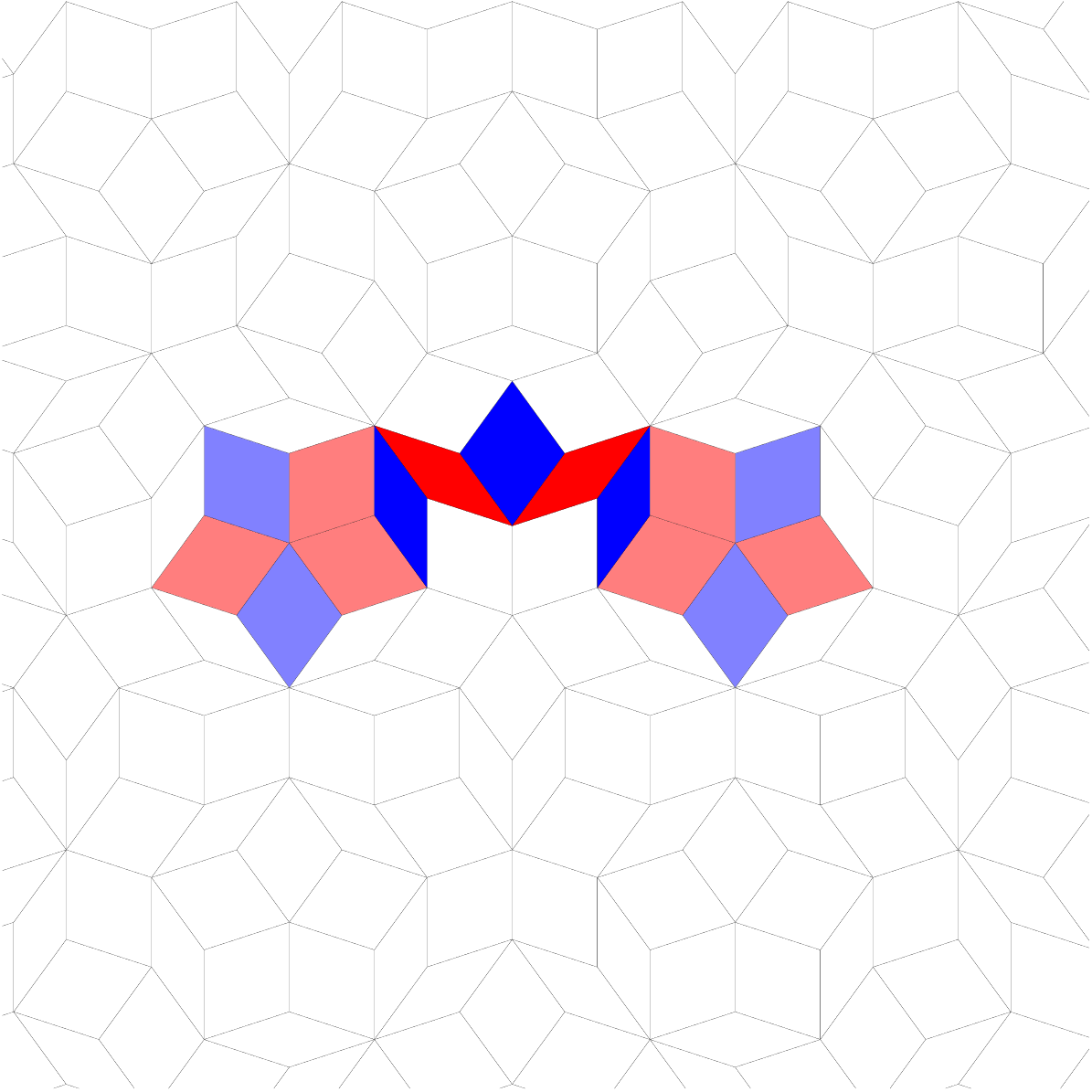}}
\put(263,33){\small $=$}
\put(285,6){\includegraphics[width=60pt]{Images/rhombus_fatk6_D1}}
\put(348,33){\small $+$}
\put(360,6){\includegraphics[width=60pt]{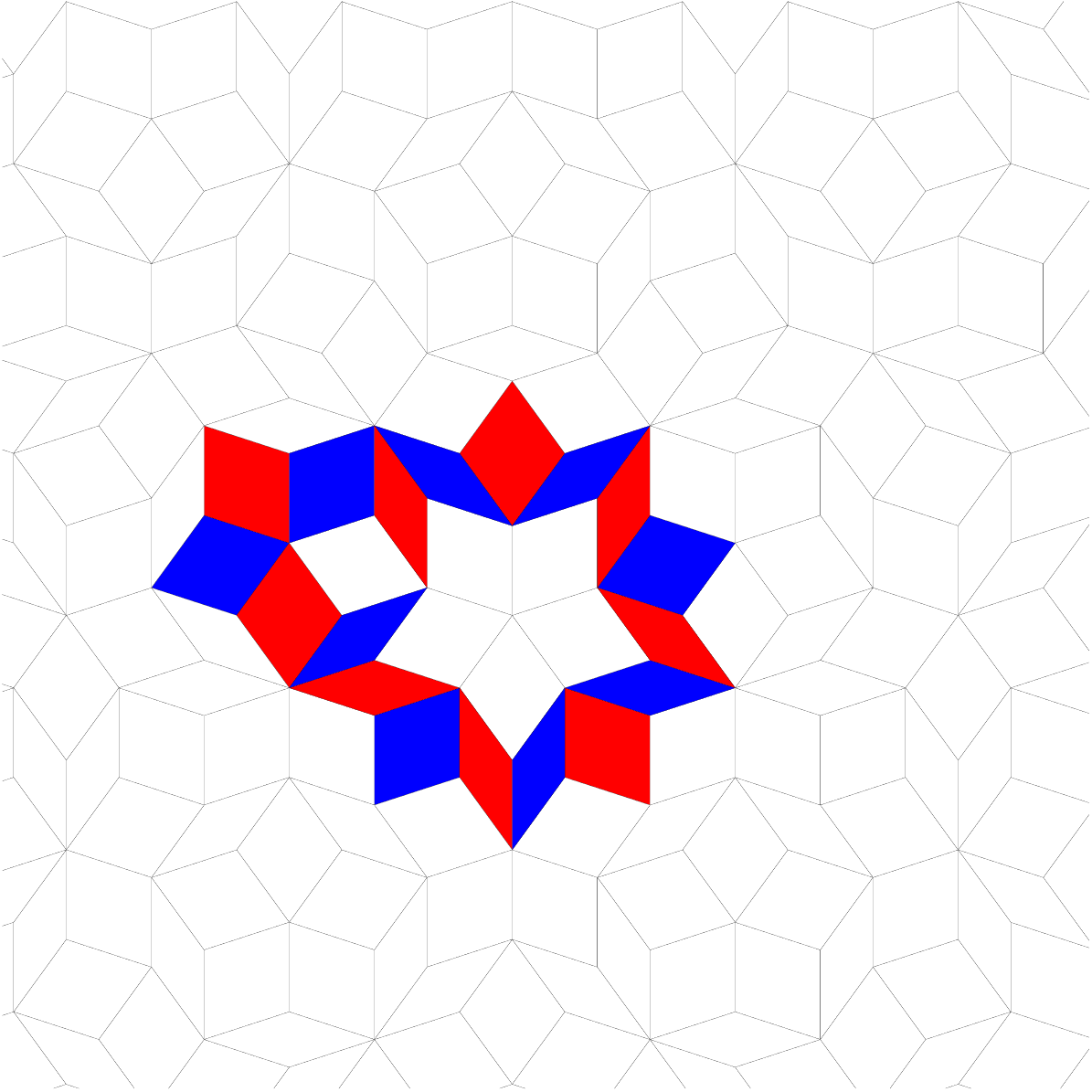}}
\end{picture}
\end{center}
\caption{\label{fig:fatk2}
Construction of eigenfunctions A1, A2, and C of Fujiwara--Arai--Rokihiro--Kohmoto~\cite{FATK1988PRB}
for the rhombus substitution as linear combinations of our basis of locally-supported eigenfunctions.
}
\end{figure}

\begin{figure}[t!]
\begin{center}
\includegraphics[width=2.4in]{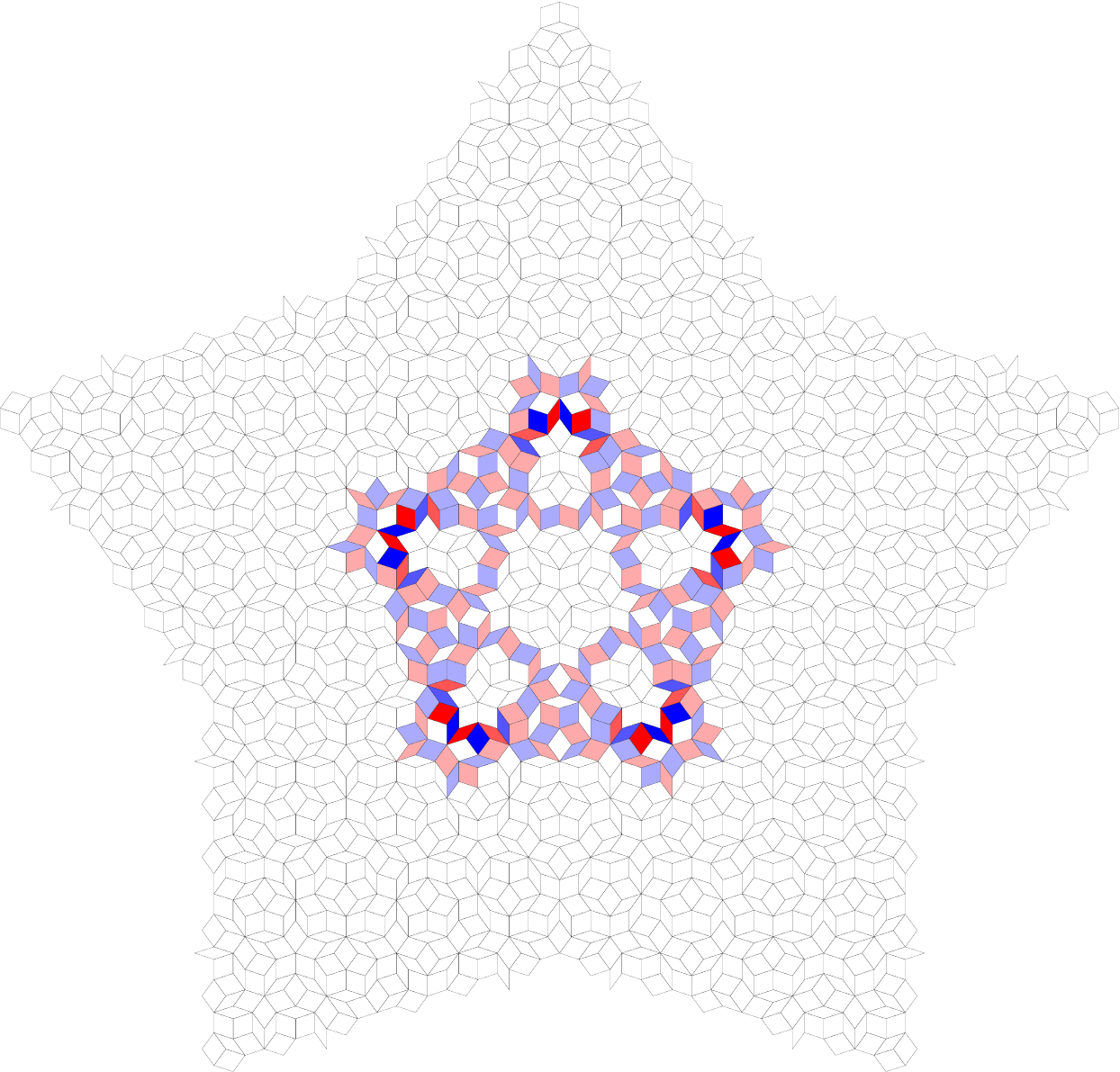}\qquad
\includegraphics[width=2.4in]{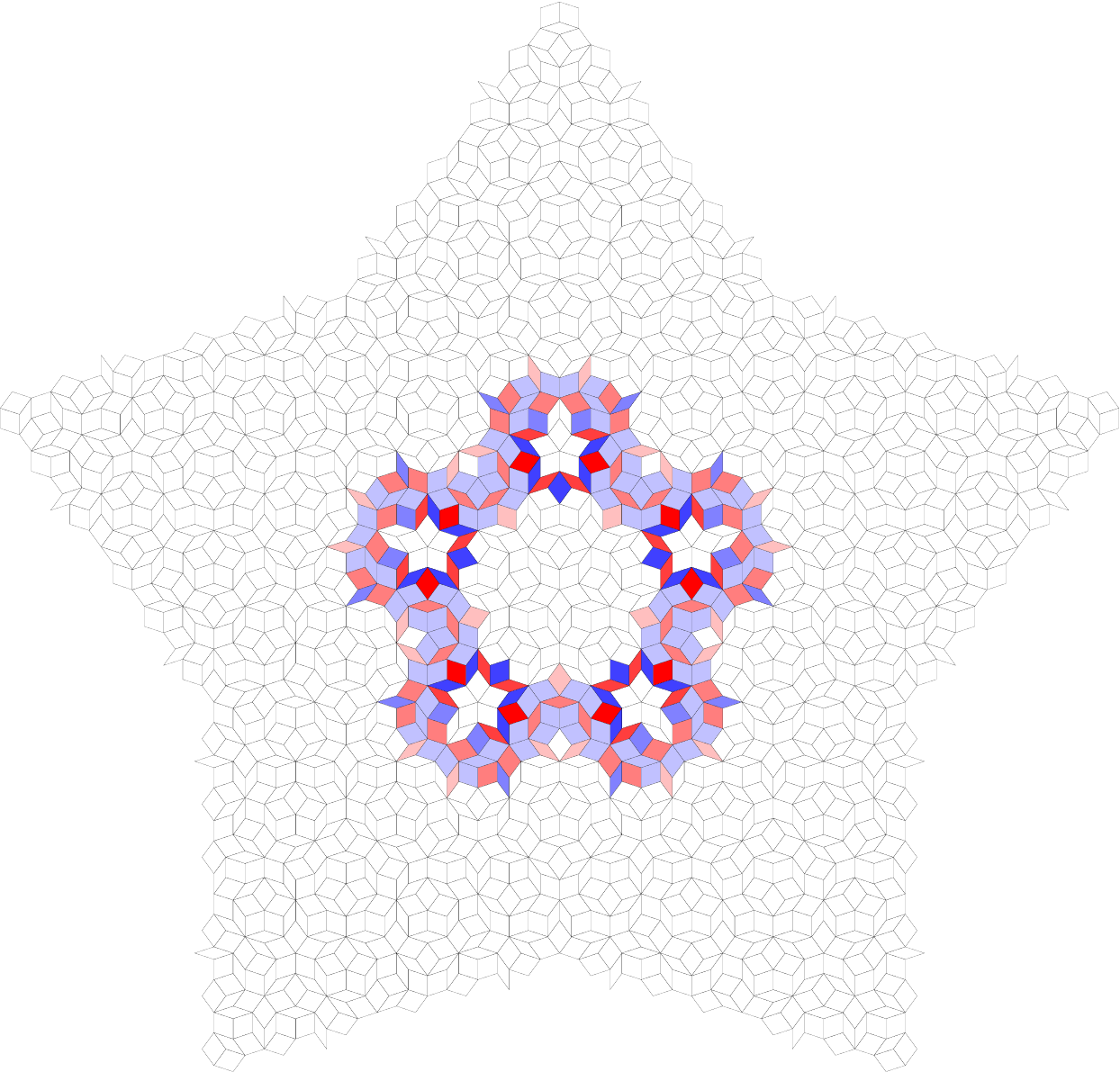}
\end{center}
\caption{\label{fig:rhombus_more_ev}
Two linearly independent eigenfunctions for $E=6$ at level~6. 
Though locally supported, neither is a linear combination of the other 100~modes of the types shown in Figure~\ref{fig:rhombus_ev}.   
(The vector on the left is supported on 200~tiles, with nonzero entries $\pm 1$, $\pm 2/3$, and $\pm 1/3$;
the vector on the right is supported on 245~tiles, with nonzero entries $1$, $\pm 3/4$, $\pm 1/2$, and $\pm 1/4$.)}
\end{figure}


\begin{table}[t!]
\caption{Rhombus tiling: the level of the tiling, the number of tiles, 
and the multiplicity of eigenvalue $E=6$ at levels $1$--$14$.
The final column shows the numerical approximation to the IDS jump.}\label{fig:rhombus:multTable}
\begin{center}
\begin{tabular}{crrc}
\emph{level} &
\multicolumn{1}{c}{\emph{tiles}} &
\multicolumn{1}{c}{$E = 6$} &
\multicolumn{1}{c}{$k_{\rhombus,n}(6+) - k_{\rhombus,n}(6-)$}
  \\ \hline
 1  & 20          & 0        & 0.00000000\ldots \\
 2  & 45          & 0        & 0.00000000\ldots \\
 3  & 115         & 2        & 0.01739130\ldots \\
 4  & 290         & 5        & 0.01724137\ldots \\
 5  & 745         & 27       & 0.03624161\ldots \\
 6  & 1\,925      & 102      & 0.05298701\ldots \\
 7  & 5\,000      & 287      & 0.05740000\ldots \\
 8  & 13\,025     & 797      & 0.06119001\ldots \\
 9  & 33\,995     & 2\,164   & 0.06365642\ldots \\
 10 & 88\,830     & 5\,792   & 0.06520319\ldots \\
 11 & 232\,285    & 15\,409  & 0.06633661\ldots \\
 12 & 607\,685    & 40\,744  & 0.06704789\ldots \\ 
 13 & 1\,590\,220 & 107\,289 & 0.06746802\ldots \\
 14 & 4\,162\,085 & 281\,939 & 0.06773984\ldots 
\end{tabular}
\end{center}
\end{table}

The rhombus tiling exhibits another intriguing property: the emergence of more complicated
locally-supported modes on larger tilings.  At level~6 the eigenvalue $E=6$
has multiplicity~102.  One can identify 10~filled circle modes, 10~big star modes, 20~diamond ring modes, 
and 60~two star modes, accounting for 100~linearly independent eigenfunctions.   
One can then find two additional linearly independent eigenfunctions, 
still having local support away from the boundary, but now involving many more tiles.  
Figure~\ref{fig:rhombus_more_ev} shows these two modes, one supported on 200~tiles, the other on 245~tiles.
(It does not appear that the modes in Figure~\ref{fig:rhombus_more_ev} were identified in~\cite{FATK1988PRB}.)
Like the simpler modes in Figure~\ref{fig:rhombus_ev}, these shapes must recur at higher levels;
moreover, yet more sophisticated locally-supported modes could also manifest at higher levels.
The emergence of such modes illustrates the challenge in explicitly calculating the jump in the
integrated density of states at $E=6$; moreover, the rarity of such modes (in comparison with the
more abundant mode shapes in Figure~\ref{fig:rhombus_ev}) indicates the challenge of precisely
estimating this jump numerically.
Table~\ref{fig:rhombus:multTable} shows numerical computations for this jump up through level~14
(4,162,085~tiles).



\section{Kites and Darts}
\begin{definition}\label{def:KiteDartRules}
The \emph{kite--dart} substitution\footnote{Illustration following \tt{https://tilings.math.uni-bielefeld.de/substitution/penrose-kite--dart/}} is given by

\begin{picture}(0,80)
\put(-10,0){\includegraphics[width=1.15in]{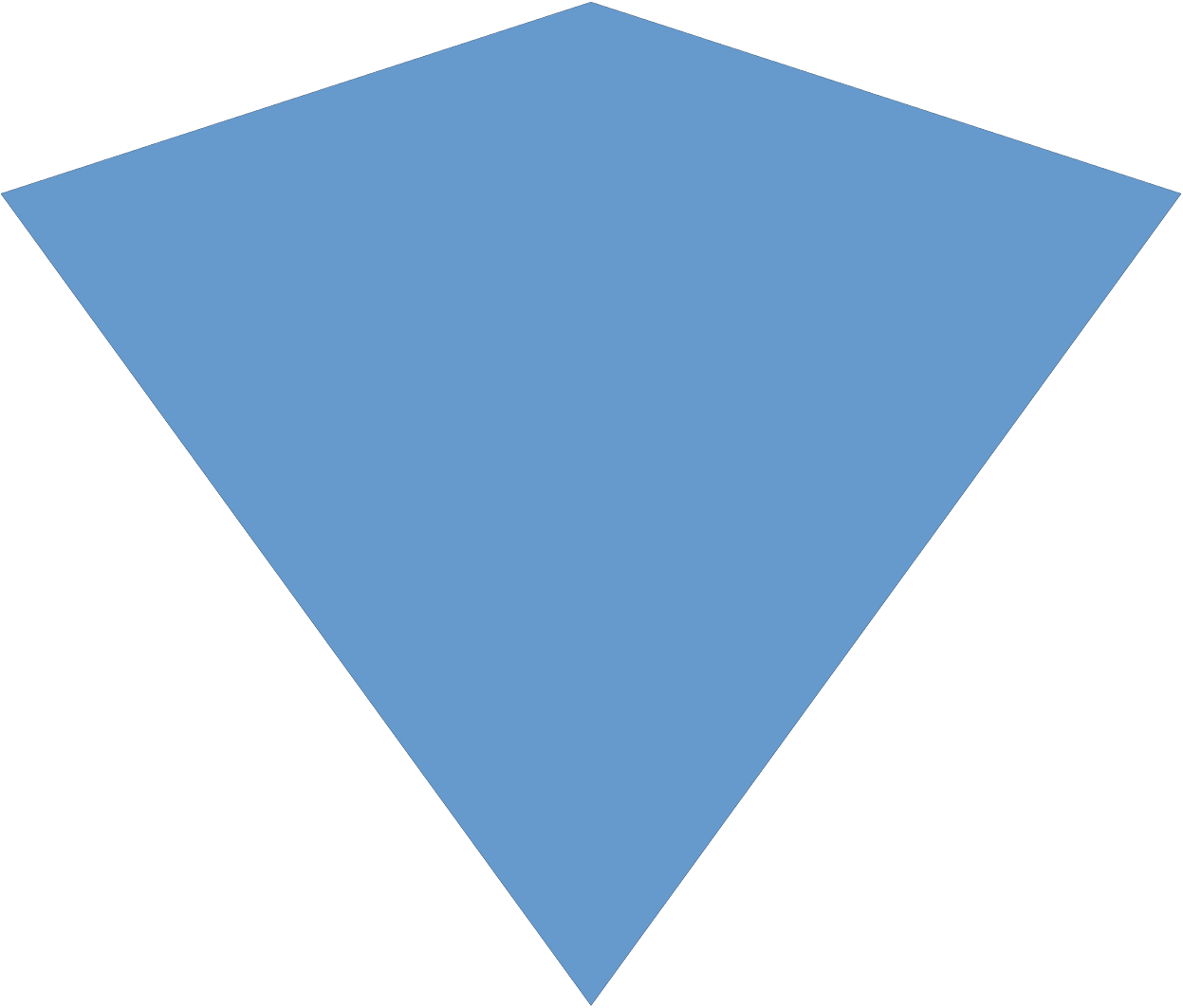}}
\put(100,0){\includegraphics[width=1.15in]{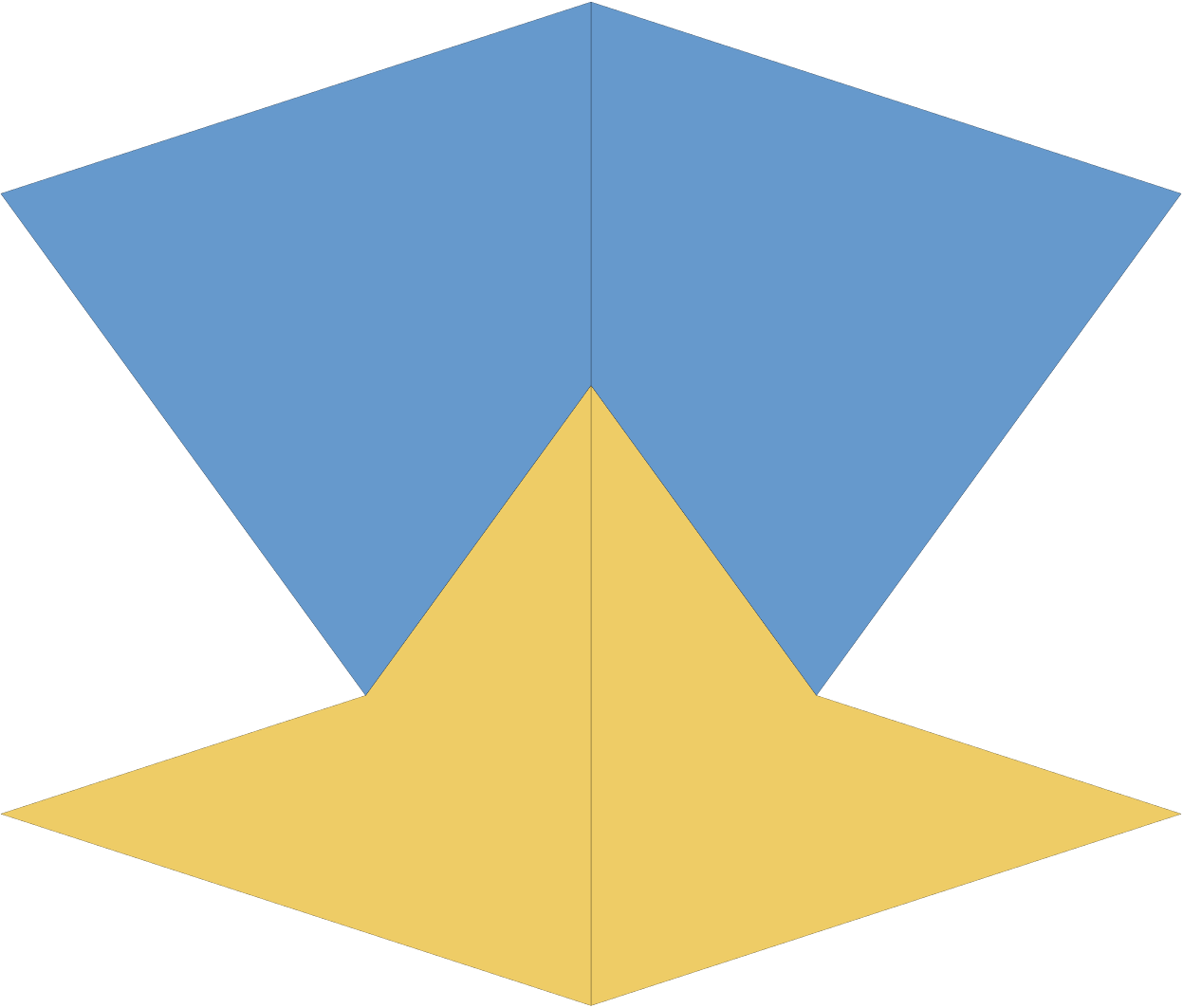}}
\put(217,0){\includegraphics[width=1.15in]{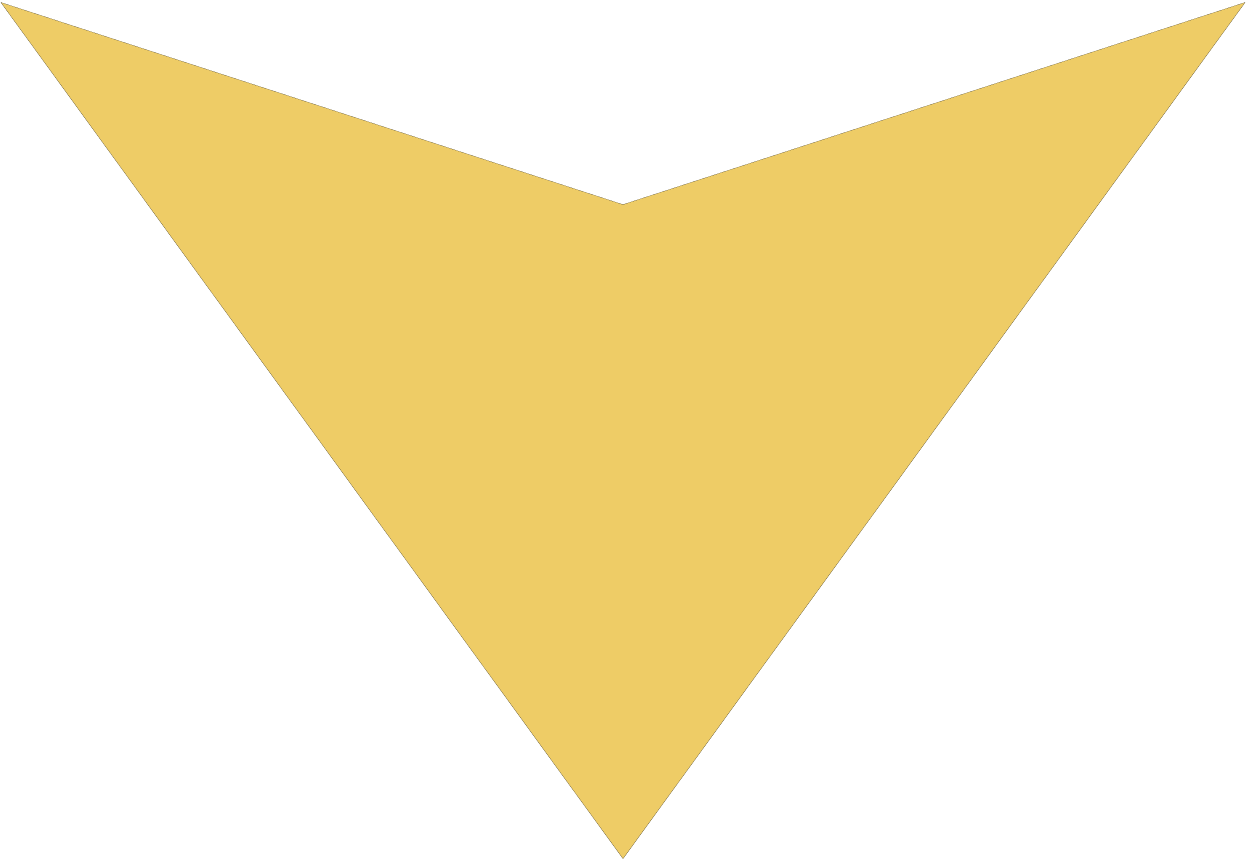}}
\put(327,0){\includegraphics[width=1.15in]{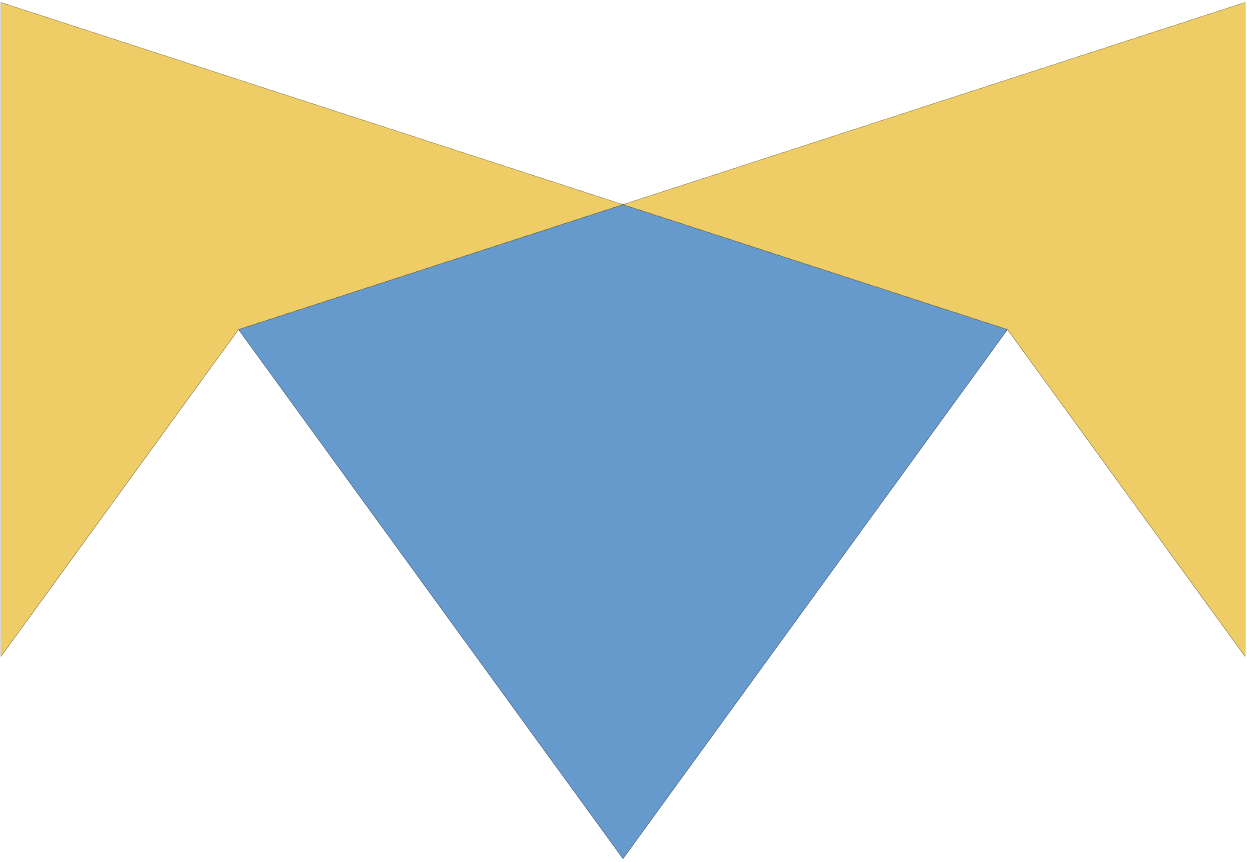}}
\put(81,30){\large $\to$}
\put(304,30){\large $\to$}
\end{picture}
\end{definition}
 
Analogous to the rhombus tiling,  we start at level~0 with a star-shaped 
configuration (comprising five darts), and then alternate between applications 
of the substitution rule and trimming back to maintain the star-shaped pattern. 
Figure~\ref{fig:kd_star} shows the first four steps of this process.

\begin{figure}[b!]
\begin{center}
\begin{minipage}{2.1in} \begin{center}
\includegraphics[width=2in]{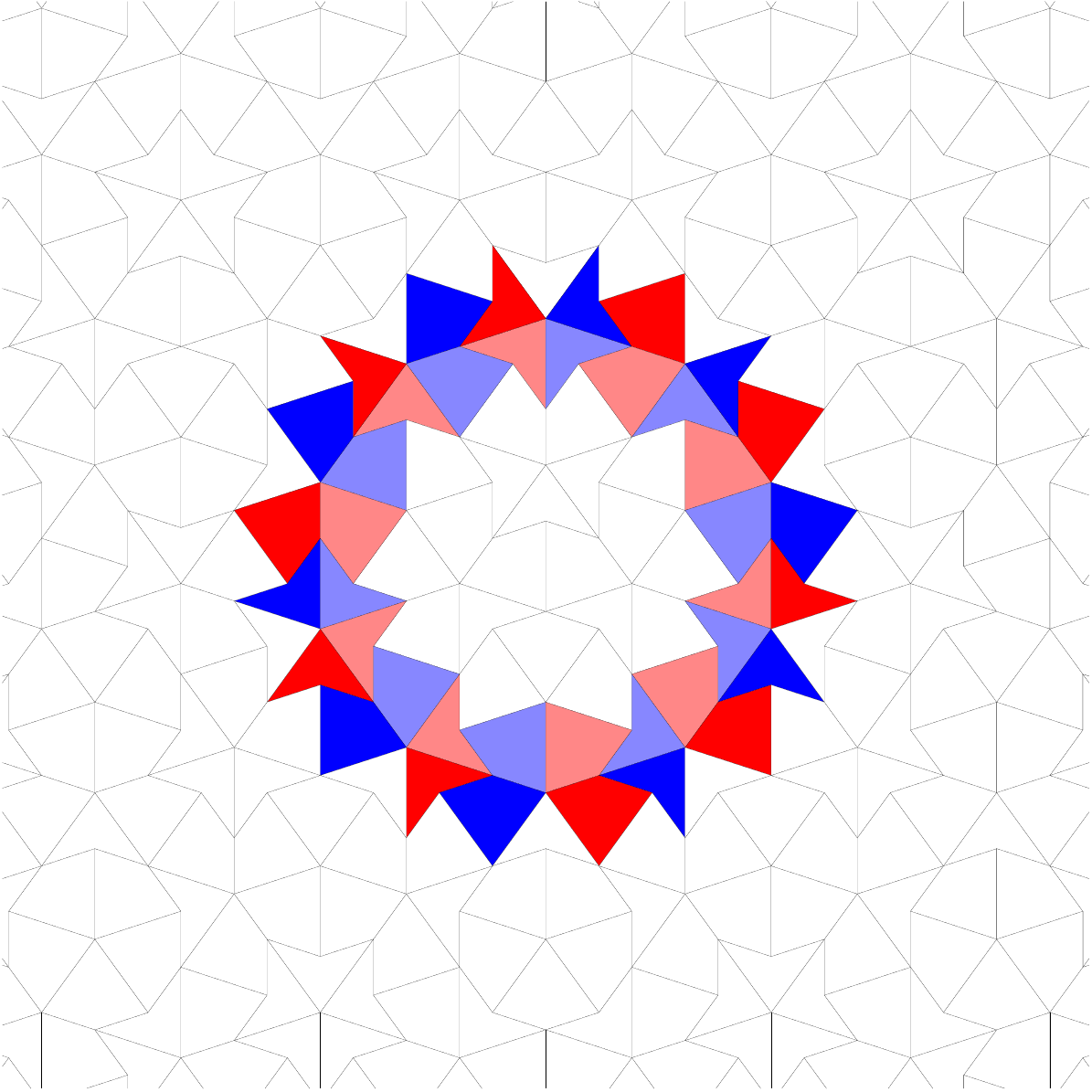}\\ \small $E=6-\varphi$
\end{center}\end{minipage}
\quad
\begin{minipage}{2.1in} \begin{center}
\includegraphics[width=2in]{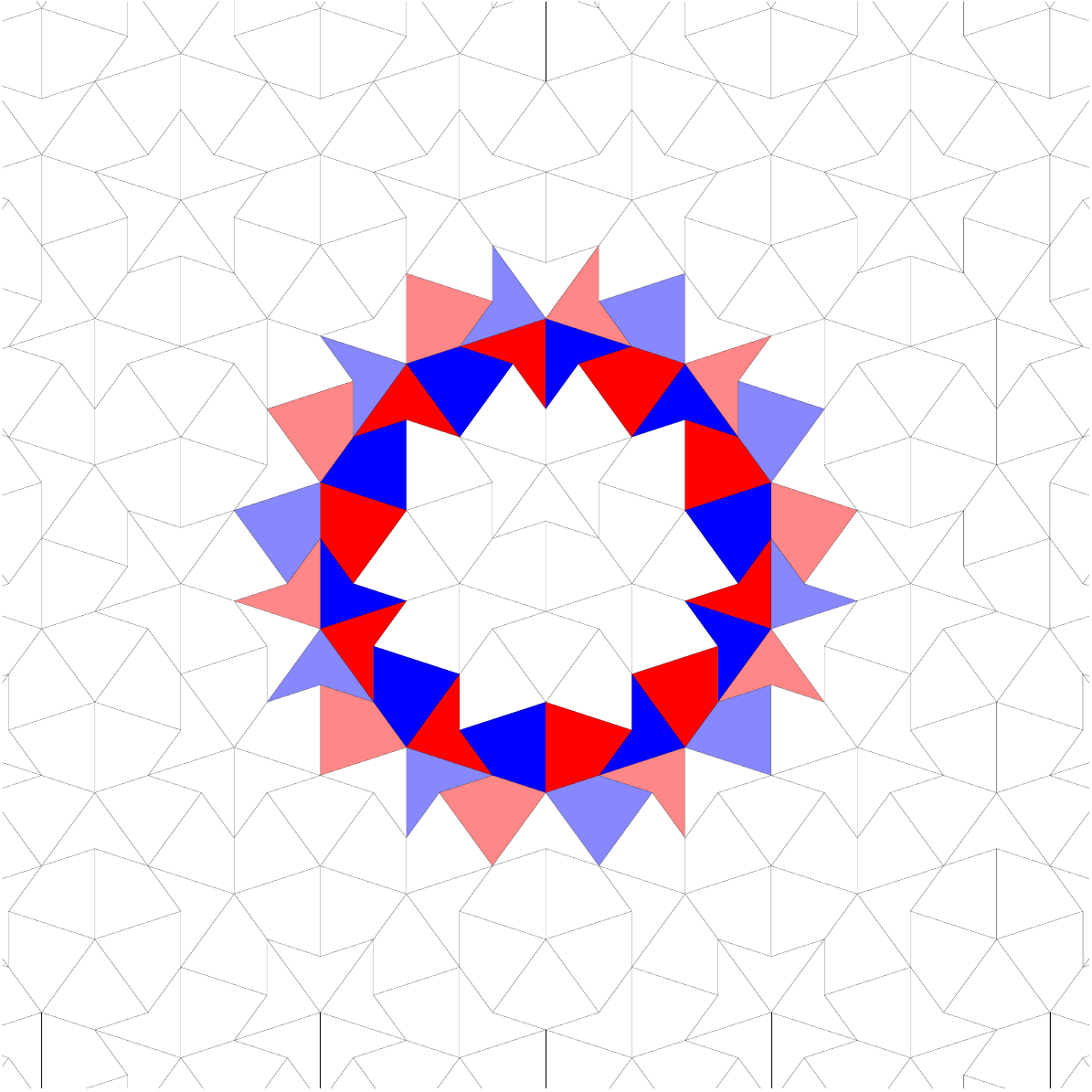}\\ \small $E=5+\varphi$
\end{center}\end{minipage}
\end{center}
\caption{ \label{fig:kd:ringmode}
Locally-supported eigenfunctions for the kite--dart substitution,
corresponding to $E=6-\varphi=4.381966\ldots$ and $E=5+\varphi=6.618033\ldots$.
The nonzero entries of these ring modes take the values $+1$ (dark blue), $+1/\varphi$ (light blue),
$-1/\varphi$ (light red), and $-1$ (dark red).} 
\end{figure}

\begin{figure}[t!]
\includegraphics[width=1.9in]{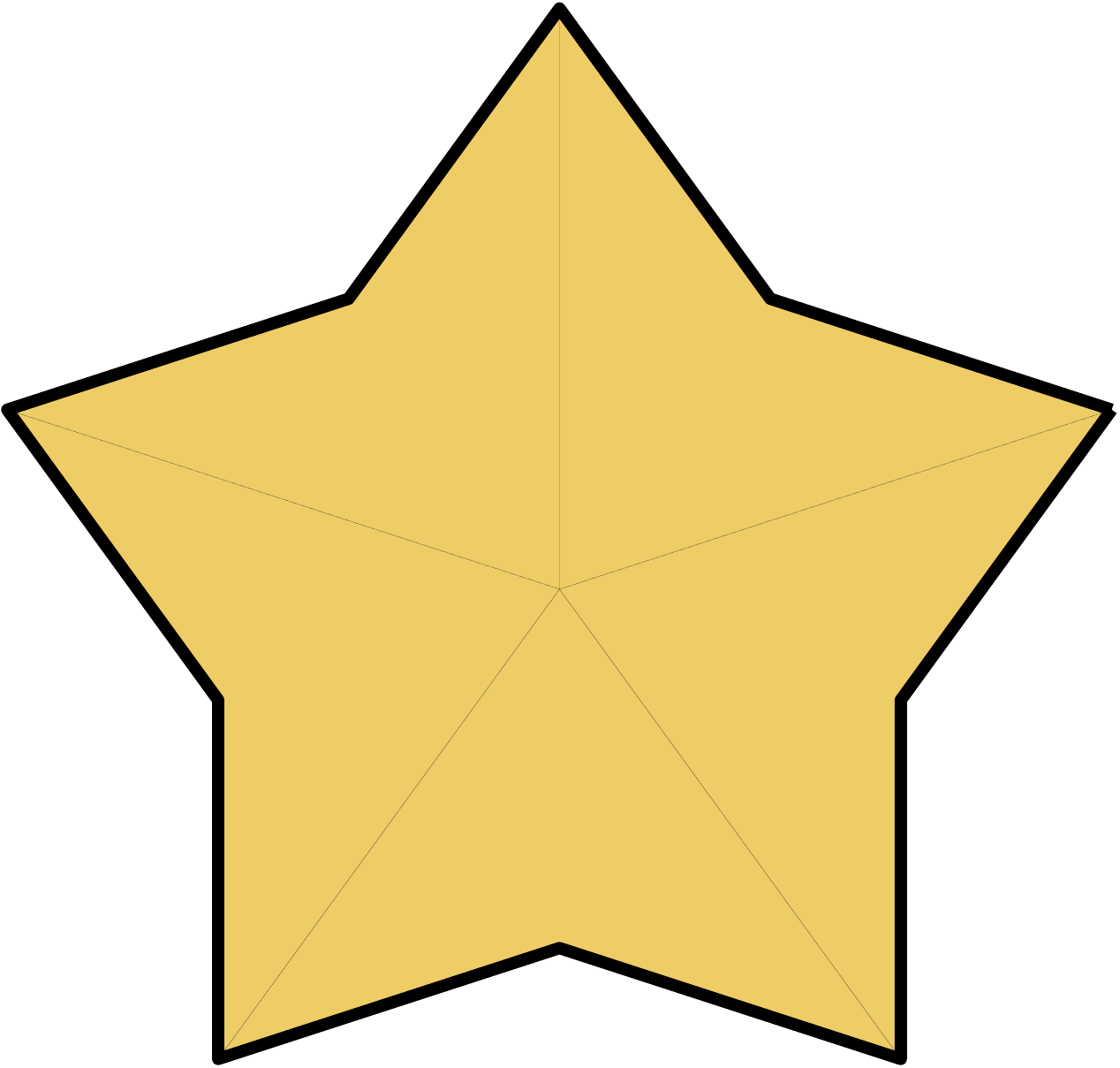}\qquad
\includegraphics[width=1.9in]{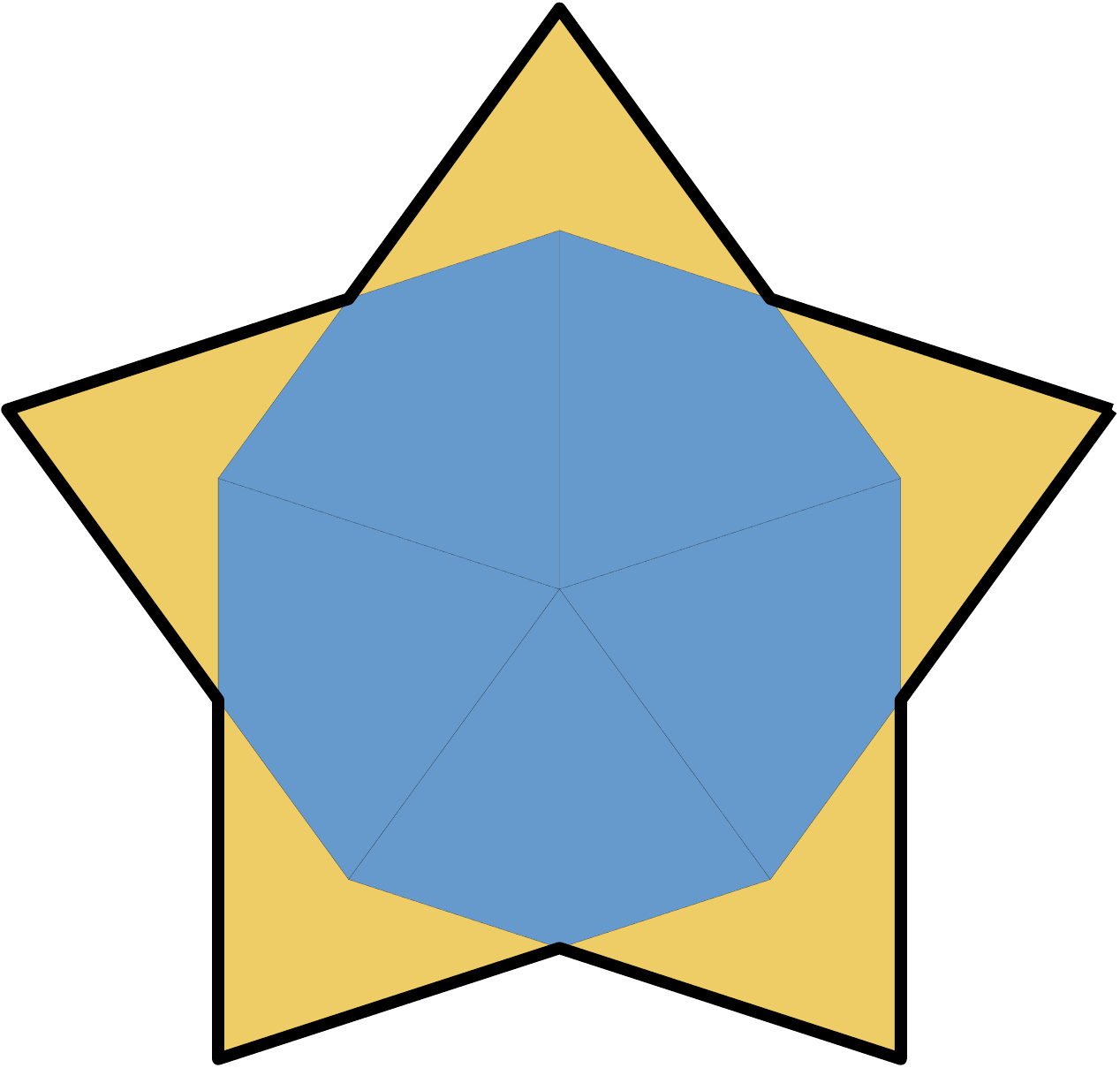}

\begin{picture}(0,0)
\put(-98,2){\small \textsl{Level 0}}
\put( 66,2){\small \textsl{Level 1}}
\end{picture}

\vspace*{1em}
\includegraphics[width=1.9in]{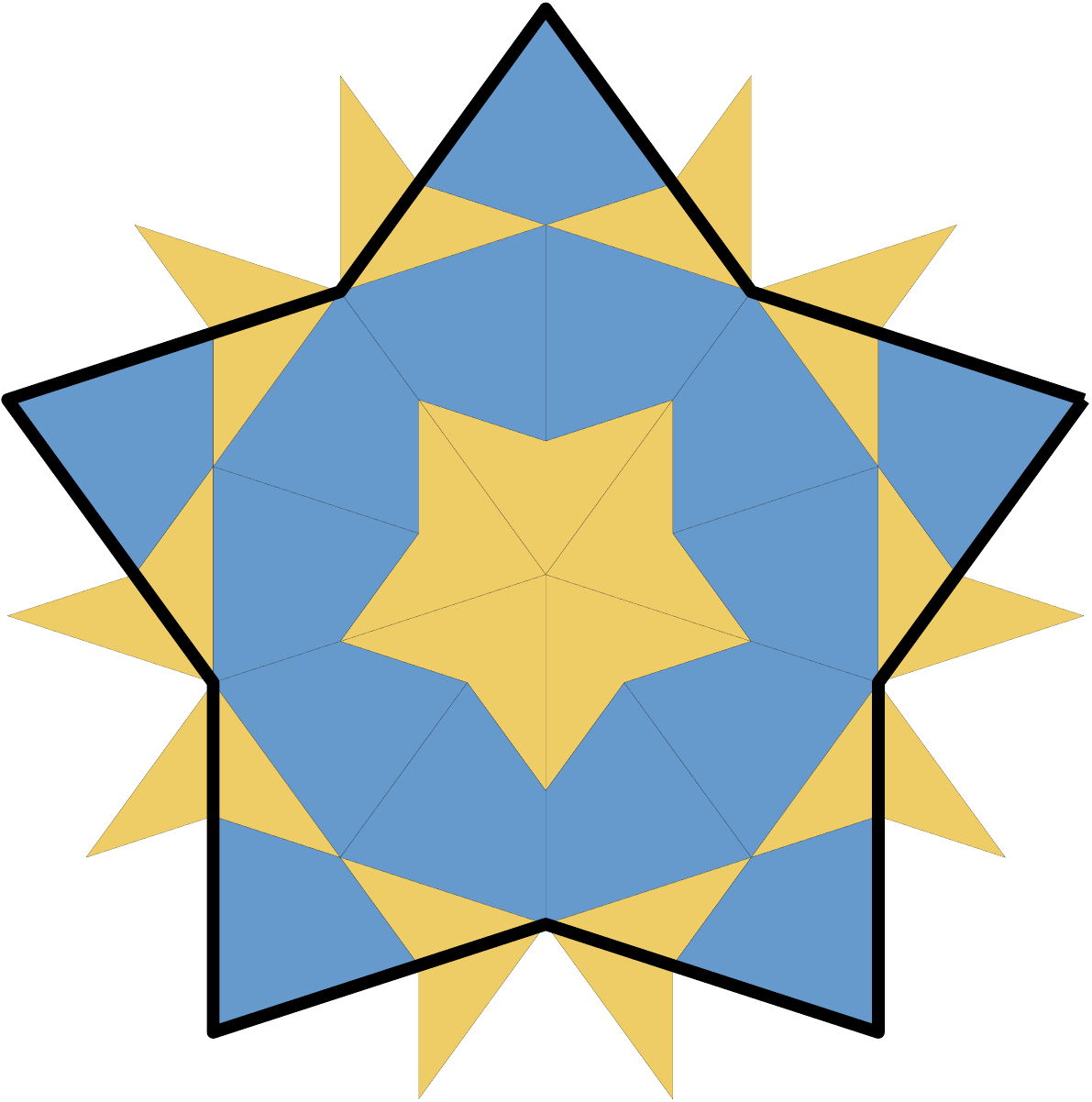}\qquad
\includegraphics[width=1.9in]{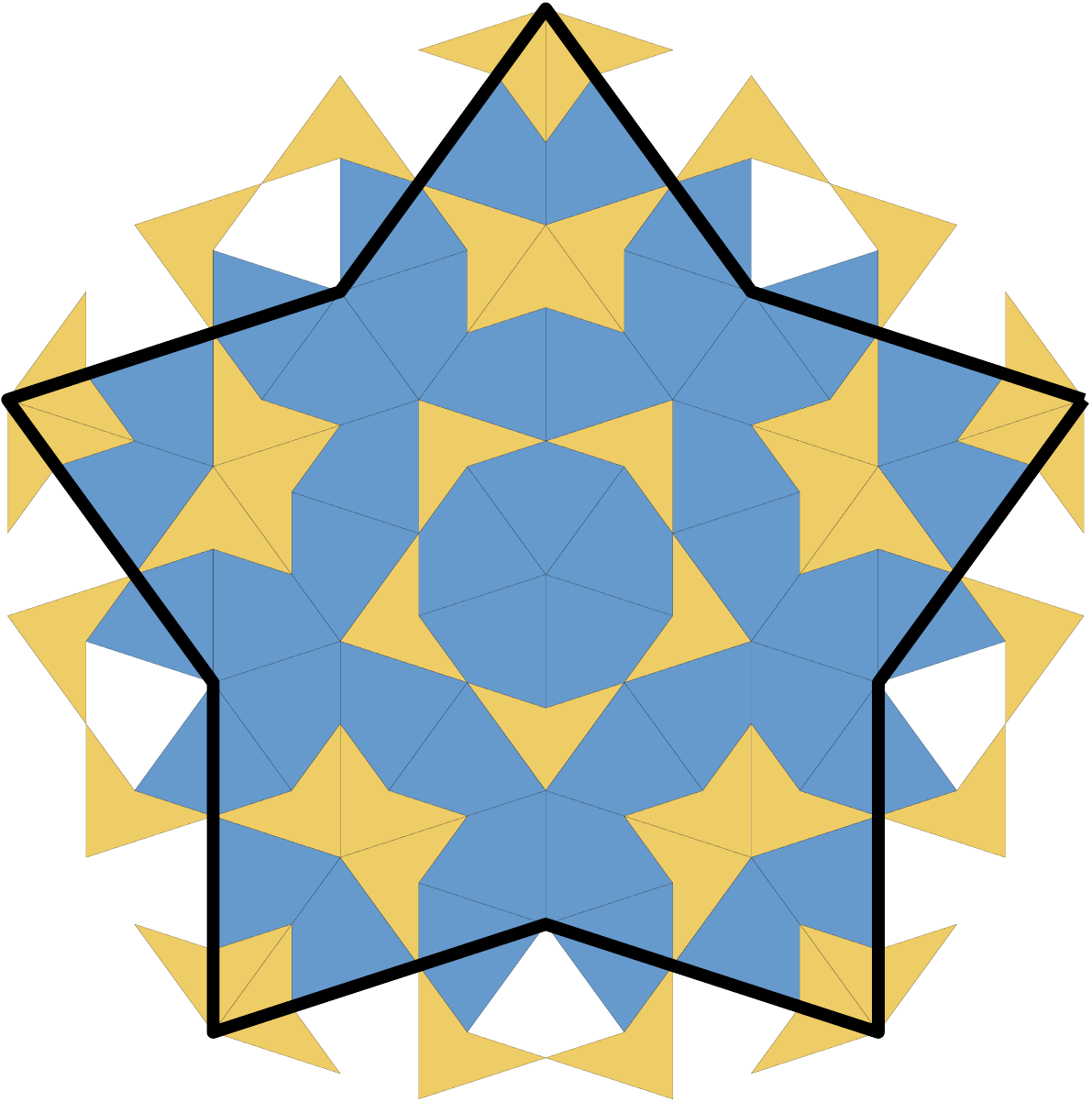}

\begin{picture}(0,0)
\put(-98,1){\small \textsl{Level 2}}
\put( 30,1){\small \textsl{Level 3} (untrimmed)}
\end{picture}

\vspace*{0.75em}
\includegraphics[width=1.9in]{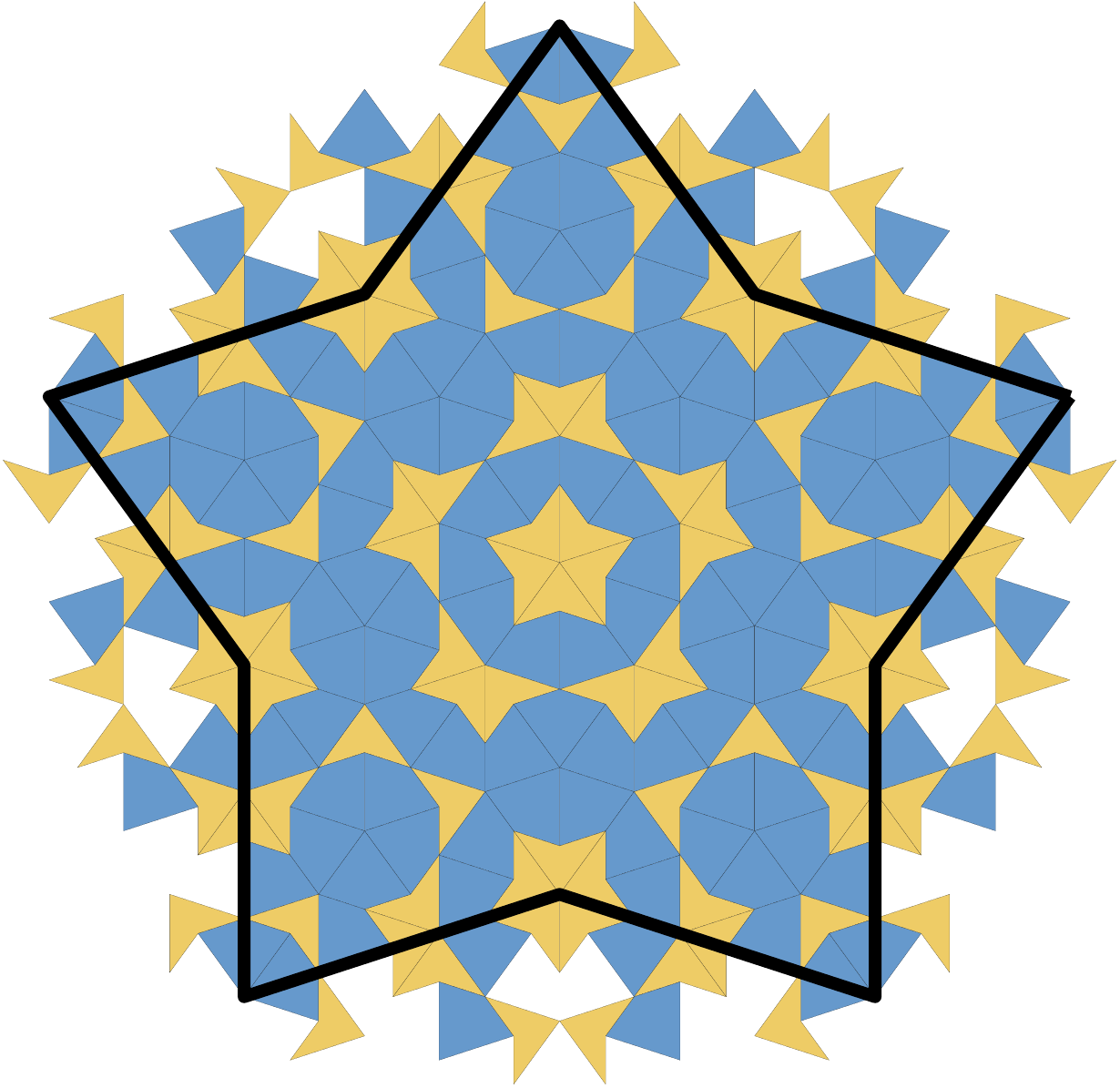}\qquad
\includegraphics[width=1.9in]{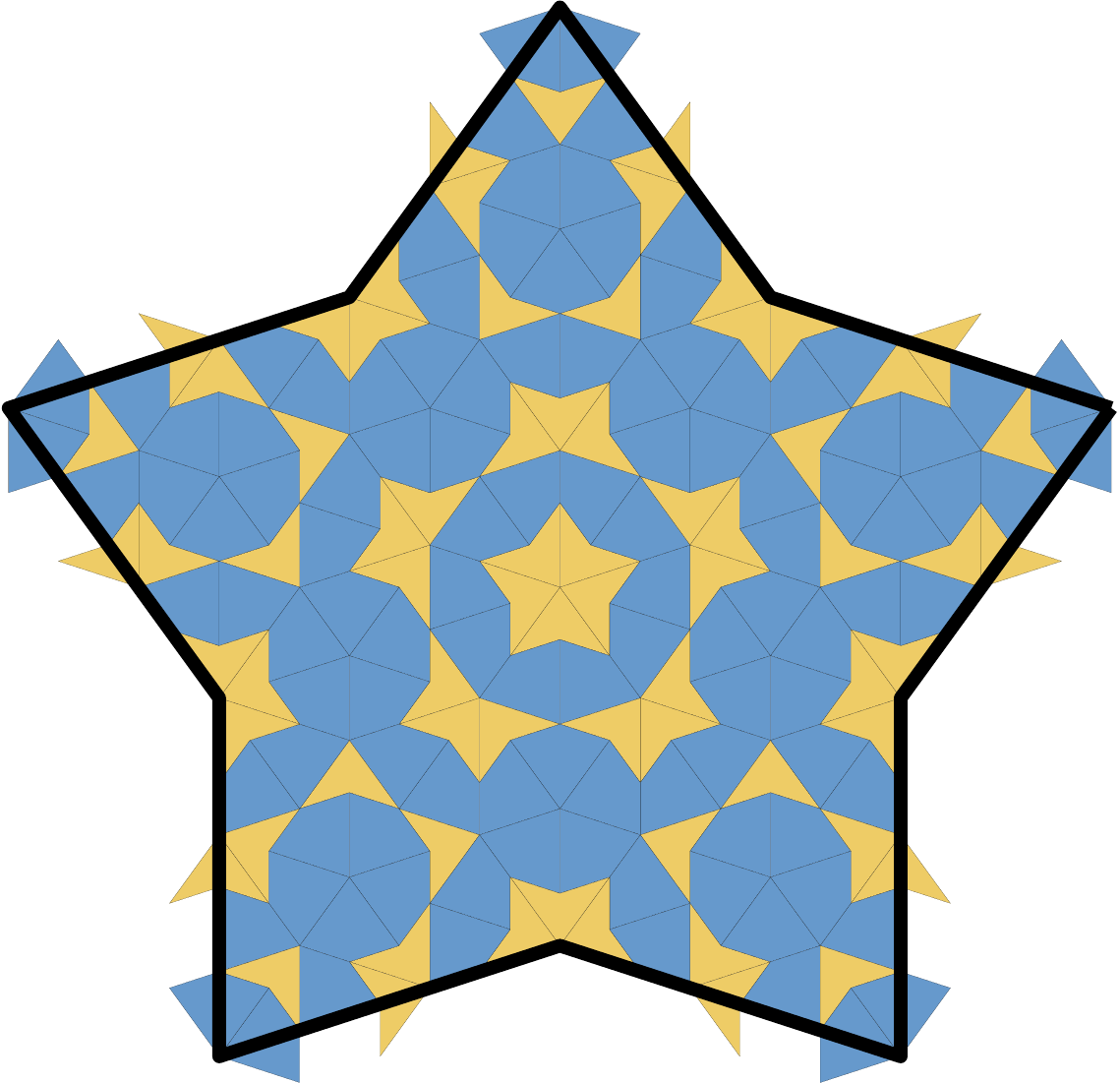}

\begin{picture}(0,0)
\put(-130,1){\small \textsl{Level 4} (untrimmed)}
\put( 38,1){\small \textsl{Level 4} (trimmed)}
\end{picture}
\caption{\label{fig:kd_star} Kite--dart rules applied to a starting configuration of five darts at level~0.}
\end{figure}

\begin{figure}[t!]
\begin{center}
\begin{minipage}{2.6in} \begin{center}
\includegraphics[width=2.5in]{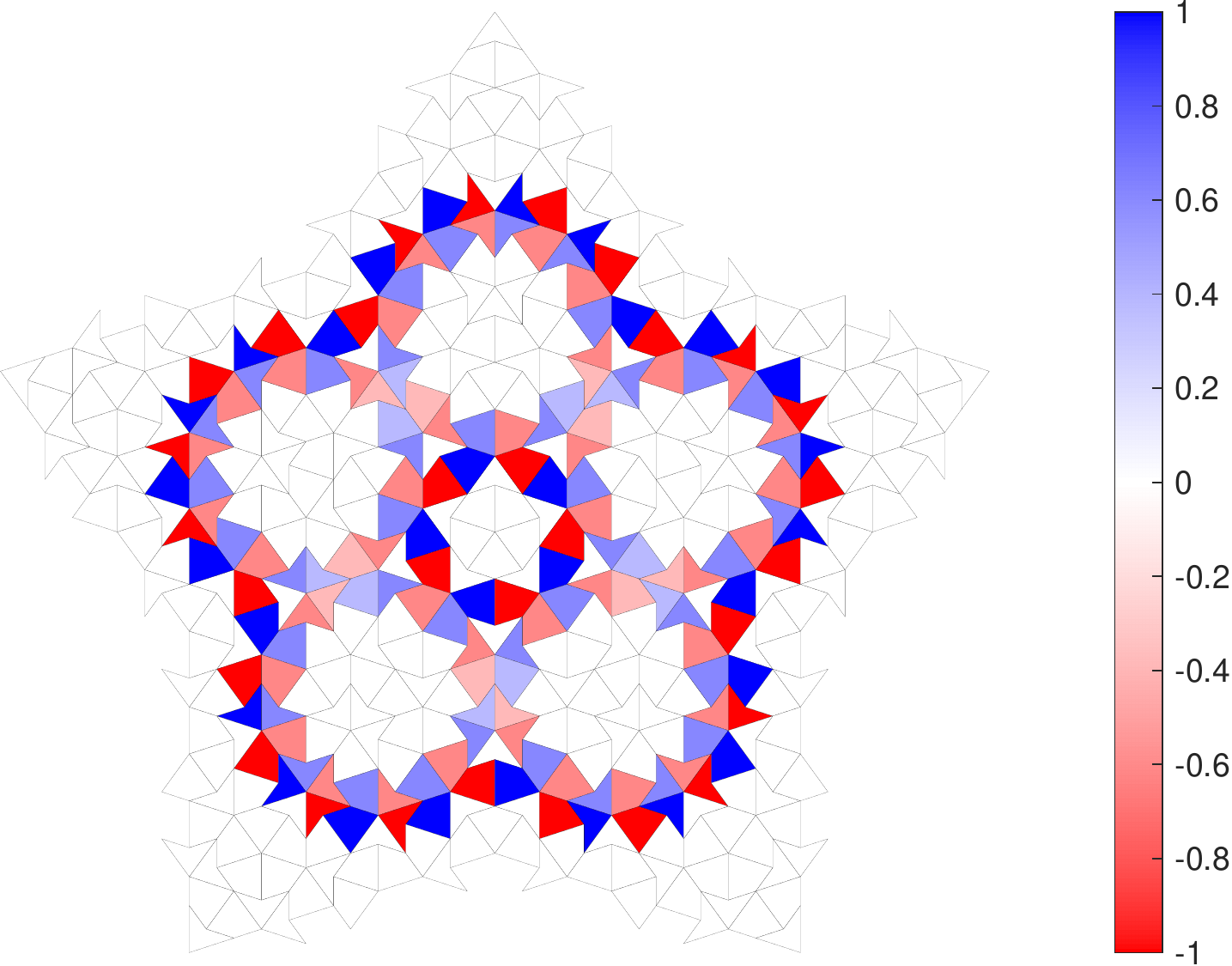}\\ \small \hspace*{-30pt}$E=6-\varphi$ 
\end{center}\end{minipage}
\quad
\begin{minipage}{2.6in} \begin{center}
\includegraphics[width=2.5in]{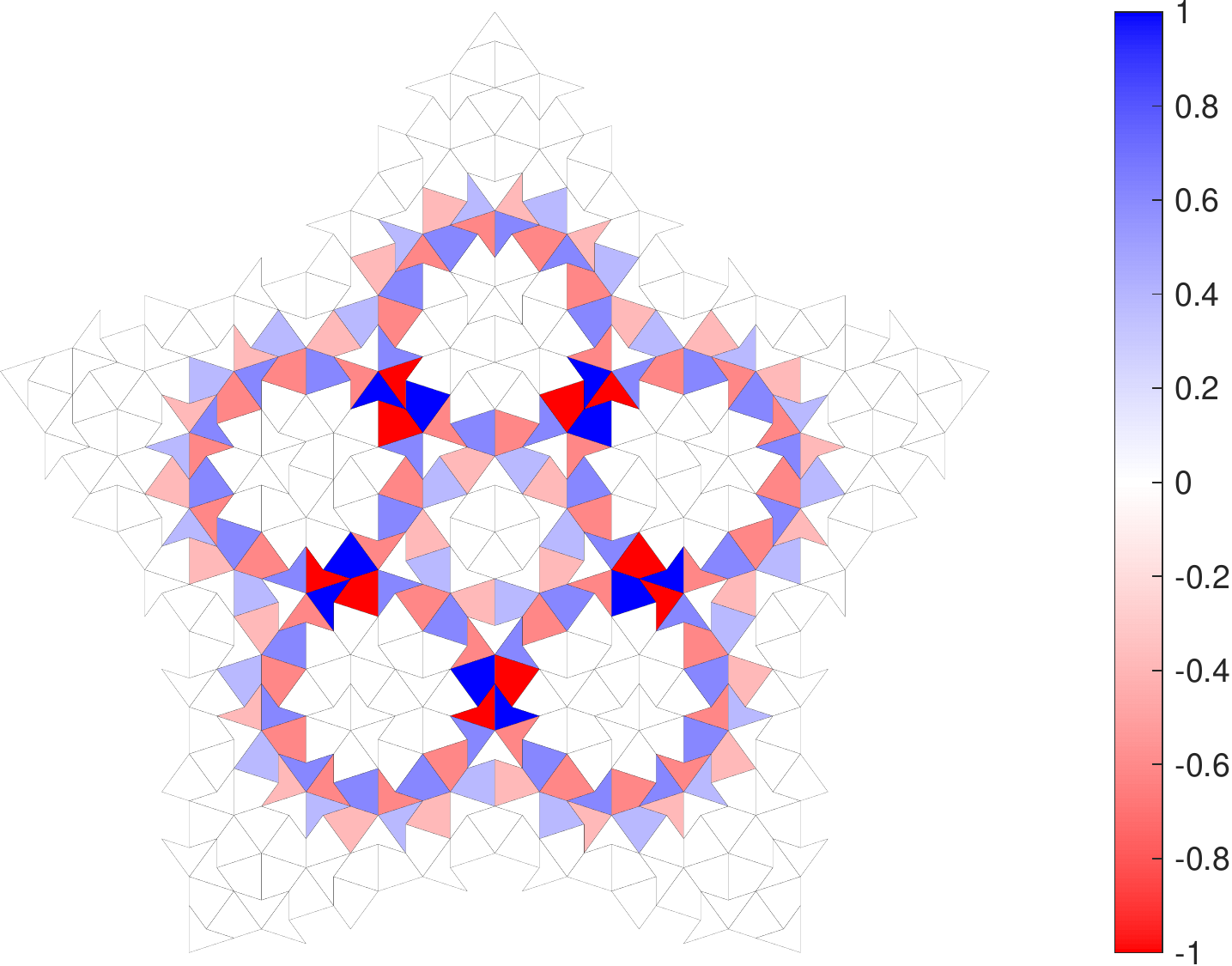}\\ \small \hspace*{-35pt}$E=5+\varphi$ 
\end{center}\end{minipage}
\end{center}
\caption{ \label{fig:kd:ringmode5}
Sum of five ``ring modes'' for level~5 of the kite--dart substitution
for $E=6-\varphi$ and $E=5+\varphi$, illustrating the
overlapping support of these modes. } 
\end{figure}

\begin{figure}[h]
\begin{center}
\includegraphics[width=4in]{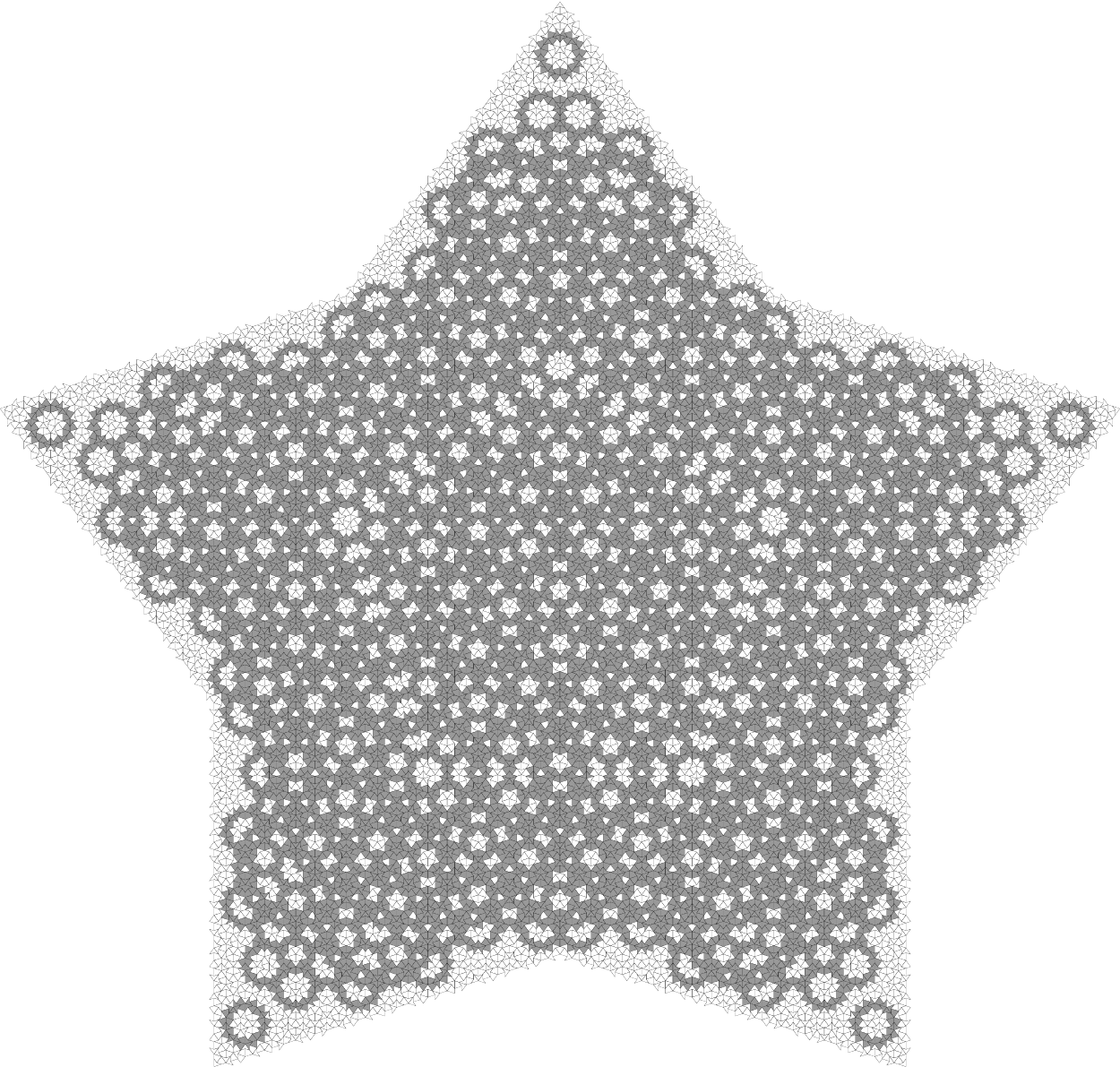}
\end{center}
\caption{\label{fig:kd:supp9}
The gray tiles show the support of the eigenfunctions at $E=5+\varphi$ and $E=6-\varphi$ for the kite--dart tiling at level~9.}
\end{figure}

In contrast to the previous examples, this tiling supports locally supported
eigenfunctions supported away from the boundary at \emph{irrational} energies.
At level~5, such eigenfunctions emerge at
$E=6-\varphi = 4.381966\ldots$ and $E=5+\varphi = 6.618033\ldots.$
(In contrast to the other tilings we consider, this latter energy appears to be at the top of the spectrum.) These eigenfunctions can be represented as rings of 40~tiles (20~kites and 20~darts) taking the values $\pm 1$ and $\pm1/\varphi$, 
as illustrated in Figure~\ref{fig:kd:ringmode}.
While these ring modes may superficially resemble those obtained
for the Robinson triangle (see Figure~\ref{fig:tri:ringmode}),
counting the frequency of these kite--dart modes is significantly
complicated by their overlapping support.  Figure~\ref{fig:kd:ringmode5}
shows the sum of the five ring modes that emerge at level~5
at $E=6-\varphi$ and $E=5+\varphi$.
(Contrast Figure~\ref{fig:kd:ringmode5} to the analogous 
illustration for the Robinson triangle tiling in Figure~\ref{fig:tri:ringmode5}.)
Figure~\ref{fig:kd:supp9} shows the support of the eigenfunctions for
these two energies at level~9, each of which has multiplicity~435.
The support covers~13,535 of the 21,025 tiles.  
The complement of this support exhibits interesting patterns,
including many ``short bow ties''~\cite{Gardner1997,GrunbaumShephard1987}.

Our numerical computations suggest that $E=6-\varphi$ and $E=5+\varphi$
have the same multiplicity (a multiple of~5) up through level~14 
(2,572,510 tiles).
Table~\ref{fig:kitedart:multTable} reports these frequencies,
along with the jump each induces in the integrated density of states.

\begin{table}[t!]
\caption{Kite--dart tiling: The level of the tiling, the number of tiles, 
the multiplicities of $E=5+\varphi$ and $E=6-\varphi$, 
and the jump in the corresponding approximant of the IDS at each of these energies.}
\label{fig:kitedart:multTable}
\begin{tabular}{crrrc}
\multicolumn{1}{c}{\emph{level}} & 
\multicolumn{1}{c}{\emph{tiles}} &
\multicolumn{1}{c}{$E = 6-\varphi$} &
\multicolumn{1}{c}{$E = 5+\varphi$} &
\multicolumn{1}{c}{$k_{\kitedart,n}(E_\kitedart+) - k_{\kitedart,n}(E_\kitedart+)$} \\ \hline
 1 & 10          & 0    & 0   & $0.00000000\ldots$\\
 2 & 30          & 0    & 0   & $0.00000000\ldots$\\
 3 & 75          & 0    & 0   & $0.00000000\ldots$\\
 4 & 180         & 0    & 0   & $0.00000000\ldots$\\
 5 & 460         & 5    & 5   & $0.01086956\ldots$ \\
 6 & 1\,195      & 10   & 10  & $0.00836820\ldots$ \\
 7 & 3\,100      & 50   & 50  & $0.01612903\ldots$ \\
 8 & 8\,060      & 135  & 135 & $0.01674938\ldots$ \\
 9 & 21\,025     & 435  & 435 & $0.02068965\ldots$ \\
10 & 54\,930     & 1\,185  & 1\,185 & $0.02157291\ldots$ \\
11 & 143\,610    & 3\,305  & 3\,305 & $0.02301371\ldots$ \\
12 & 375\,645    & 8\,875  & 8\,875 & $0.02362603\ldots$ \\
13 & 982\,930    & 23\,735 & 23\,735 & $0.02414719\ldots$ \\
14 & 2\,572\,510 & 62\,820 & 62\,820 & $0.02441973\ldots$
\end{tabular}
\end{table}




\section{Ammann--Beenker}

Thus far we have investigated four versions of the Penrose tiling.
In this section we explore related questions for the Ammann--Beenker tiling.
We begin by recalling the substitution rule.

\begin{definition}\label{def:ABRules}
The \emph{Ammann--Beenker} substitution\footnote{Illustration following
{\tt https://tilings.math.uni-bielefeld.de/substitution/ammann-beenker/}} is given by

\begin{center}
\hspace*{-2em}
\includegraphics[width=1.25in]{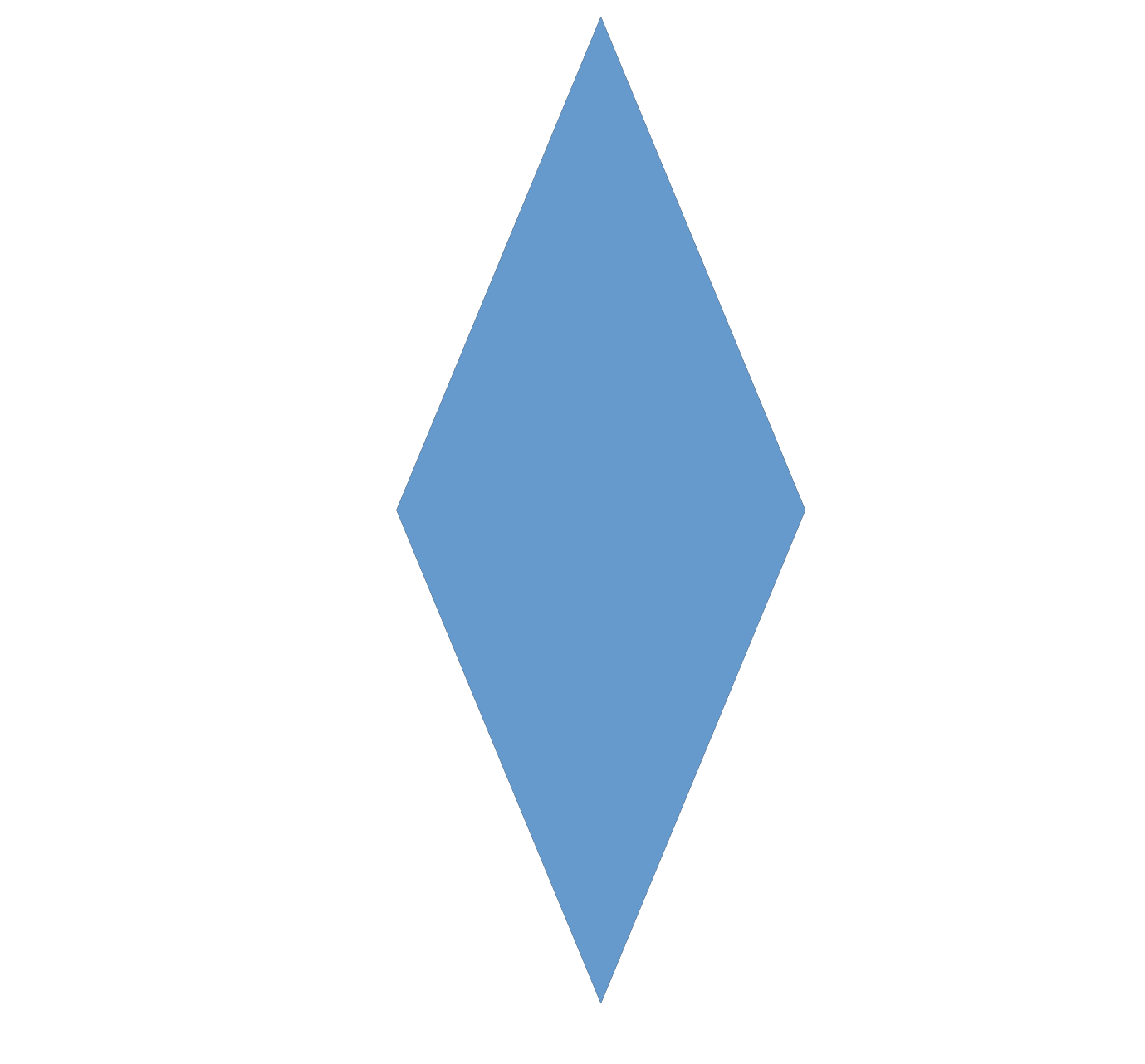}\quad
\includegraphics[width=1.25in]{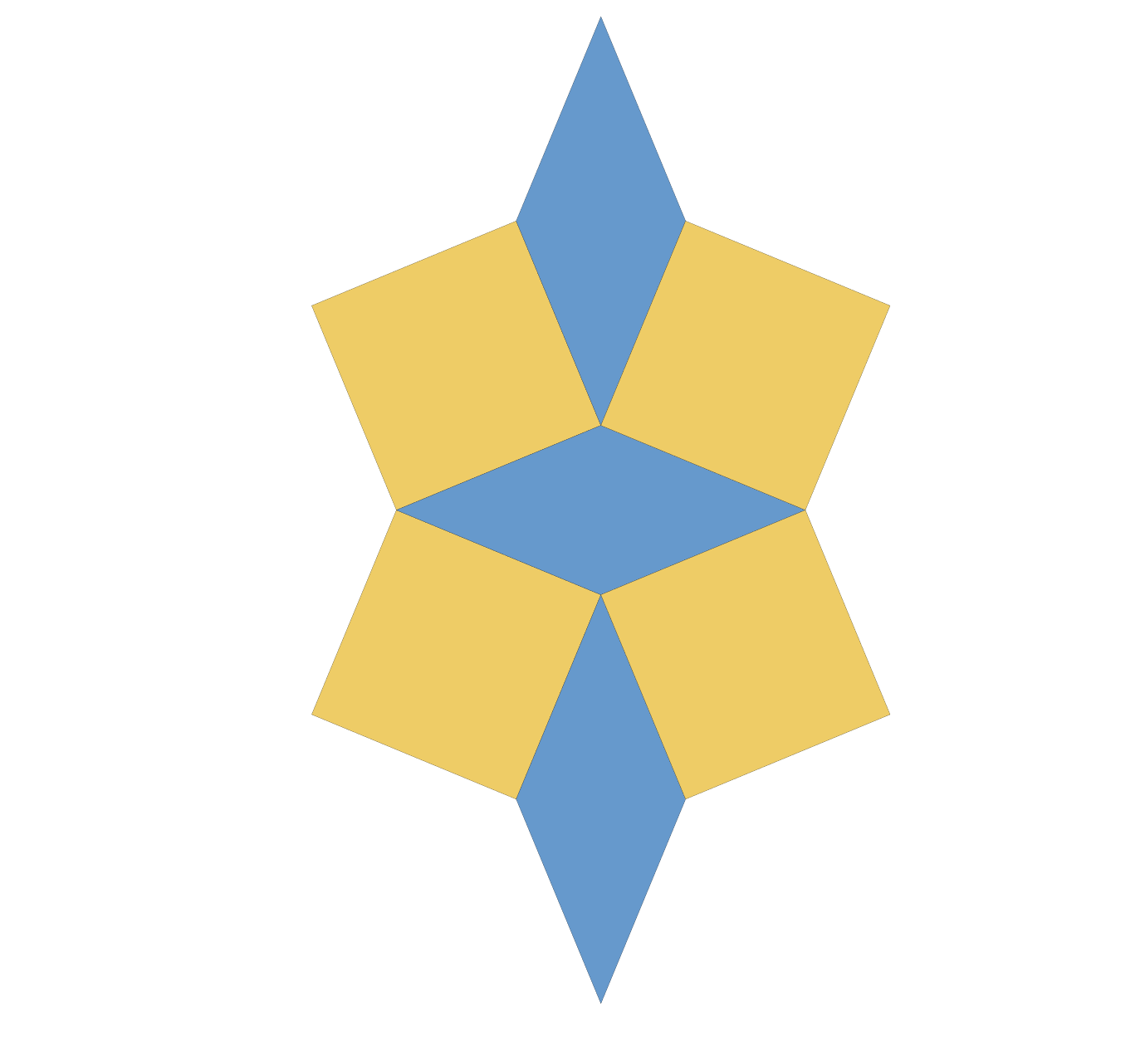}
\qquad\quad
\raisebox{8pt}{\includegraphics[width=1.25in]{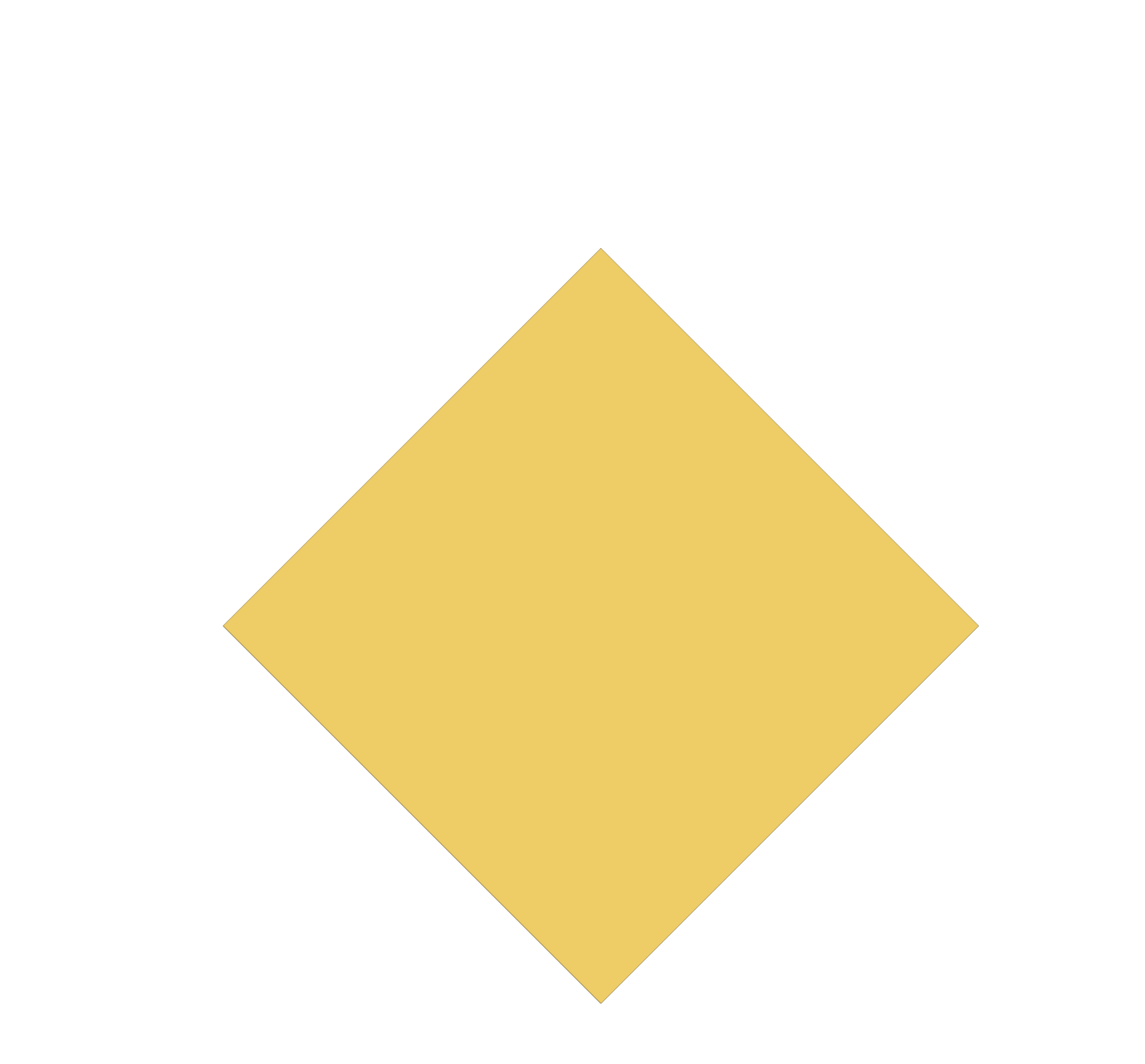}}\qquad
\raisebox{8pt}{\includegraphics[width=1.25in]{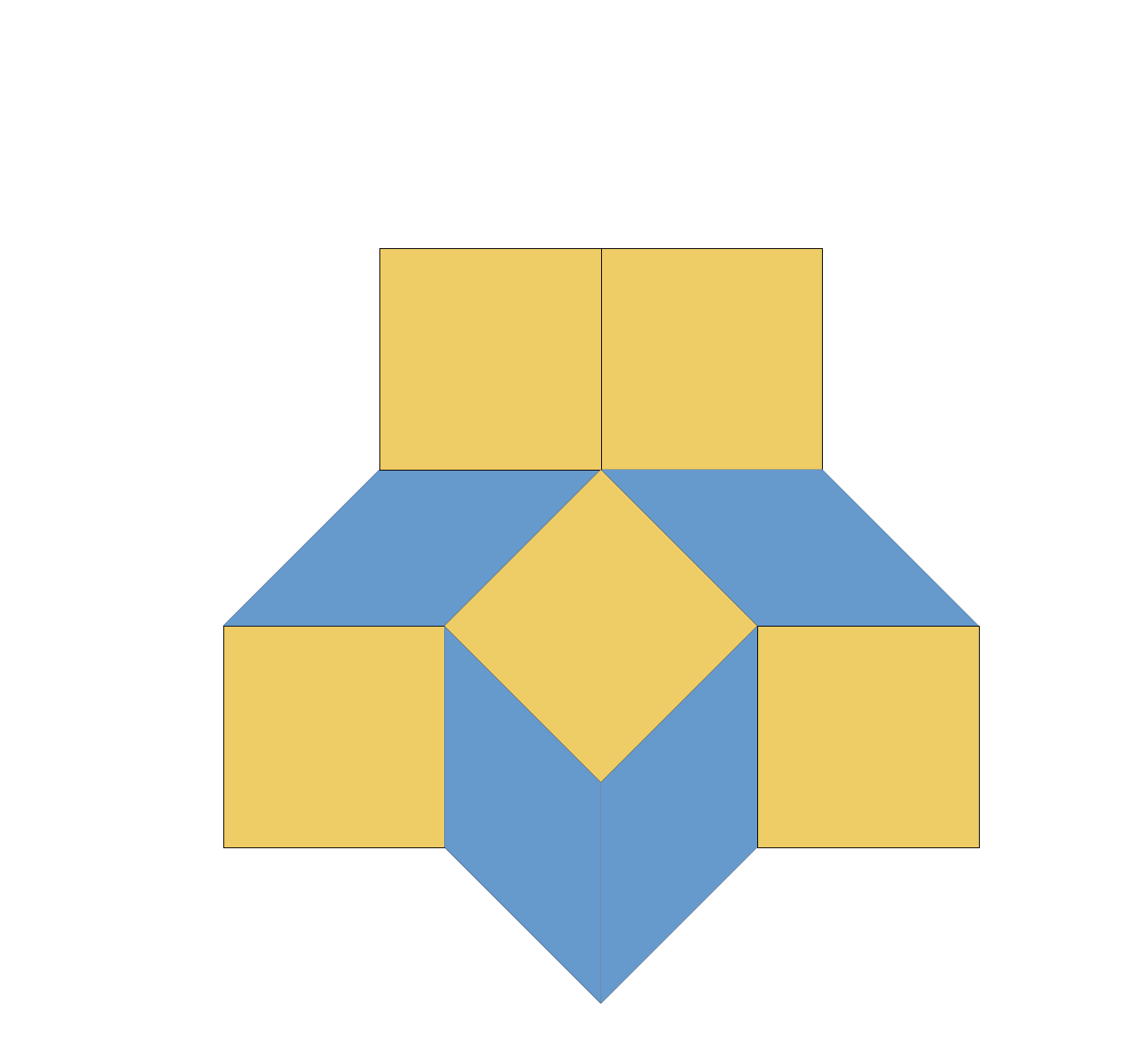}}
\begin{picture}(0,0)(210,13)
\put(-142,53){\large $\to$}
\put(98,53){\large $\to$}
\end{picture}
\end{center}
\end{definition}

As in previous sections, we generate a tiling by beginning with an initial seed and iteratively applying the substitution rule. Let $\calT_0^\AB$ denote the pattern consisting of eight thin rhombi arranged in an eight-point star shape as in Figure~\ref{fig:ABlevel0-4}. Analogous to the kite--dart substitution (see Figure~\ref{fig:kd_star}), we alternate between applications of the substitution rule and trimming back to an octagon.  We let $\calT_n^\AB$ denote $n$ steps of this process.

\begin{figure}[t!]

\includegraphics[width=2in]{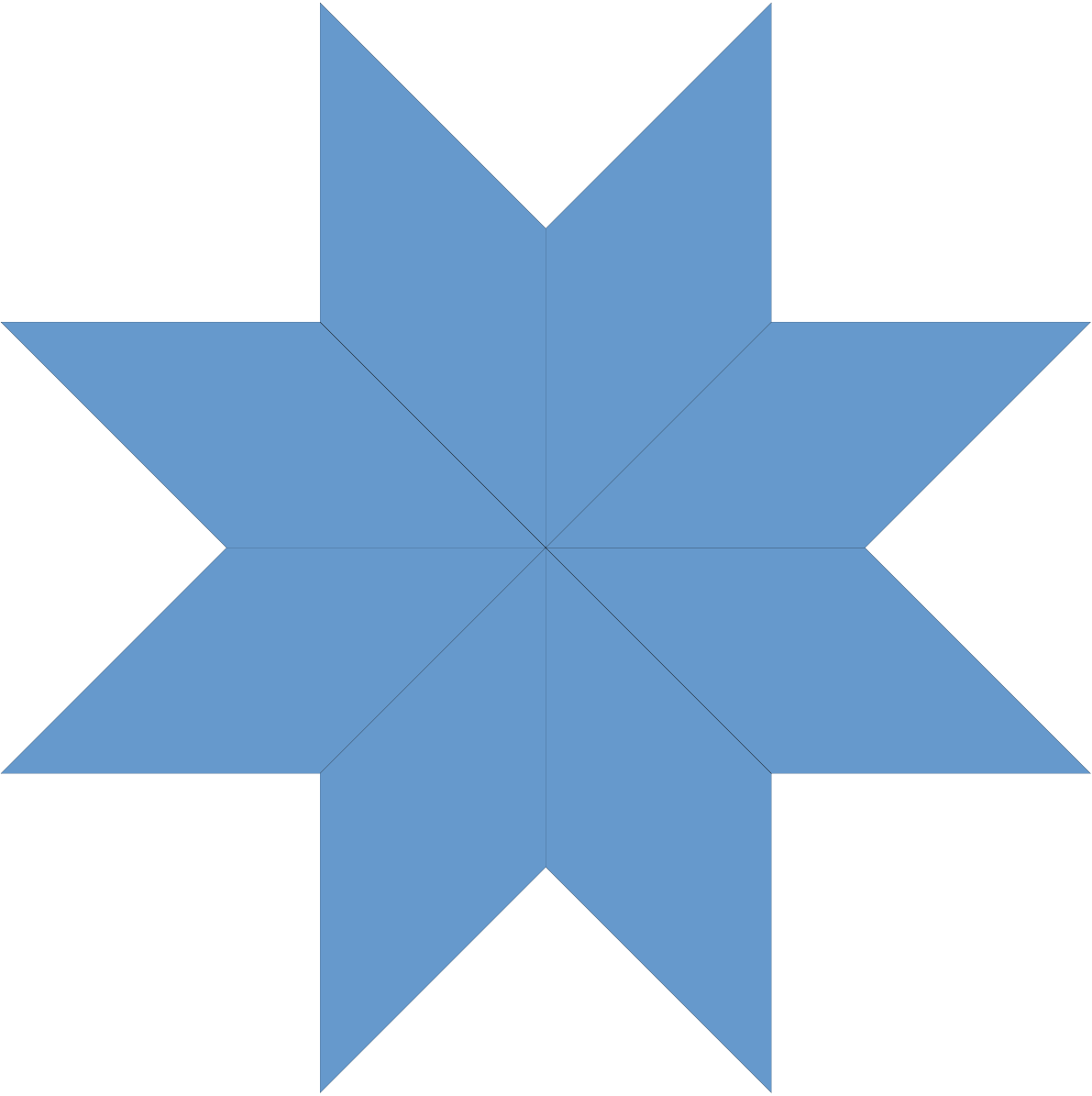}\qquad
\includegraphics[width=2in]{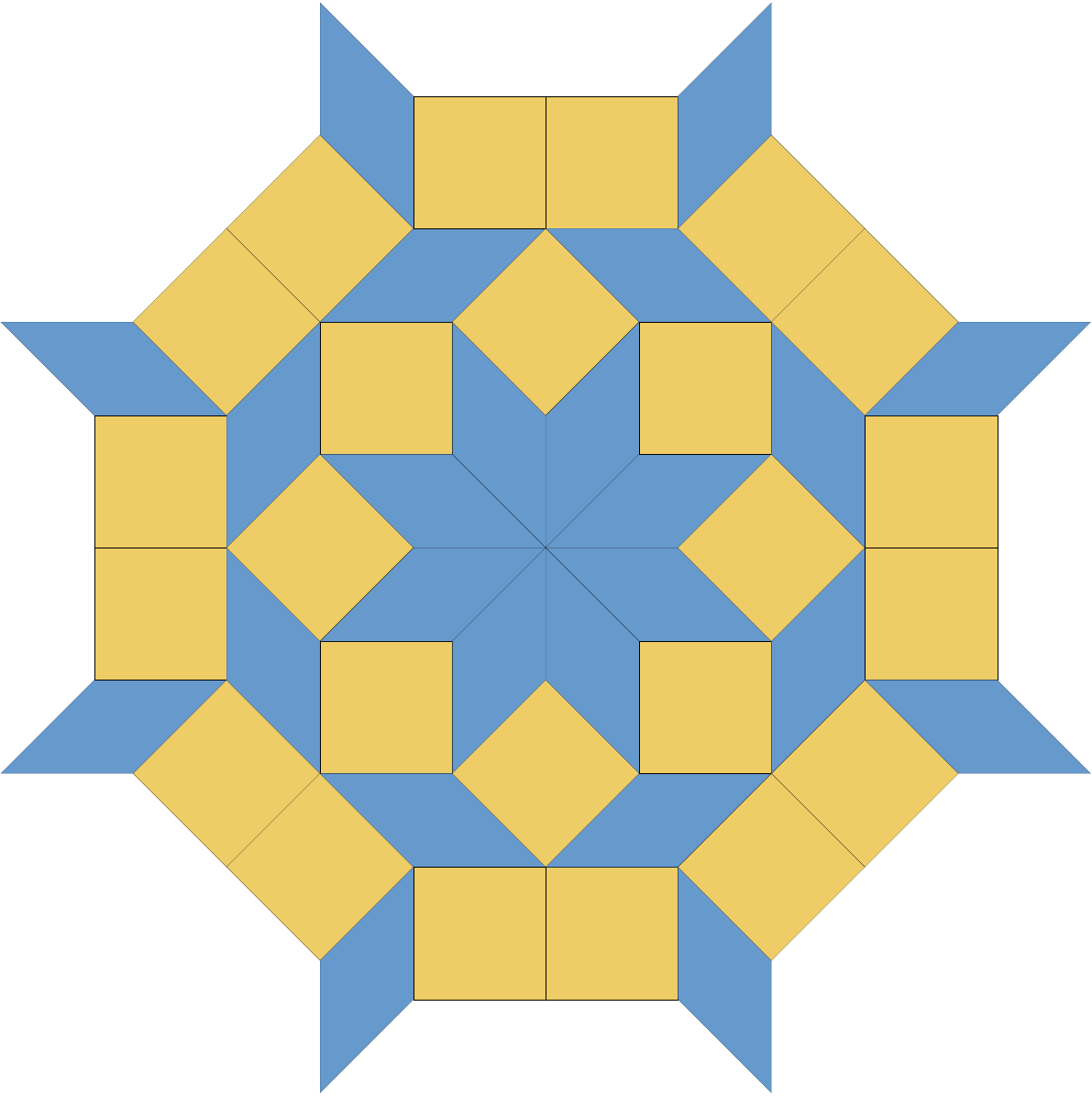}

\begin{picture}(0,0)
\put(-98,0){\small \textsl{Level 0}}
\put( 66,0){\small \textsl{Level 1}}
\end{picture}

\vspace*{1.1em}
\includegraphics[width=2in]{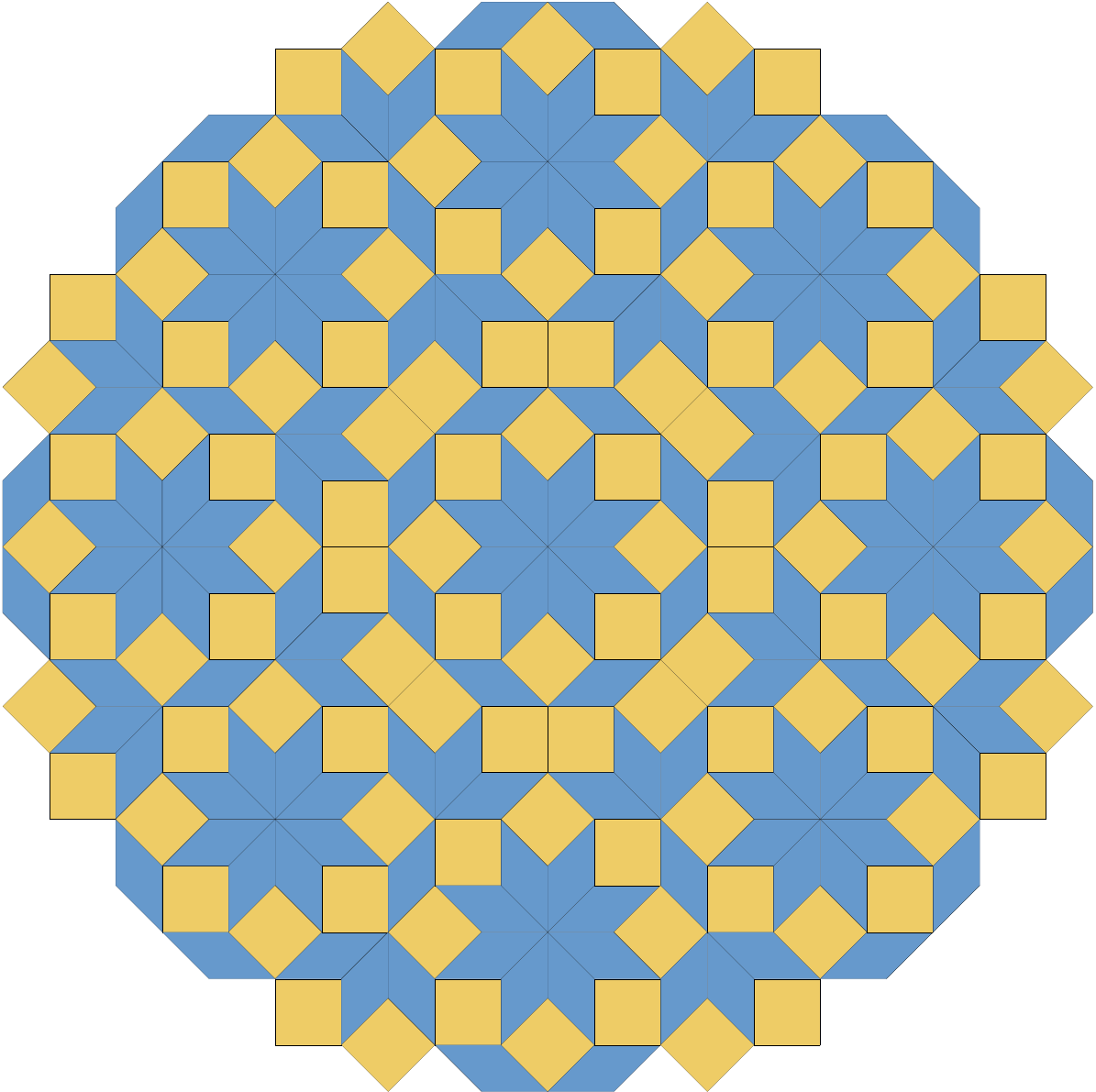}\qquad
\includegraphics[width=2in]{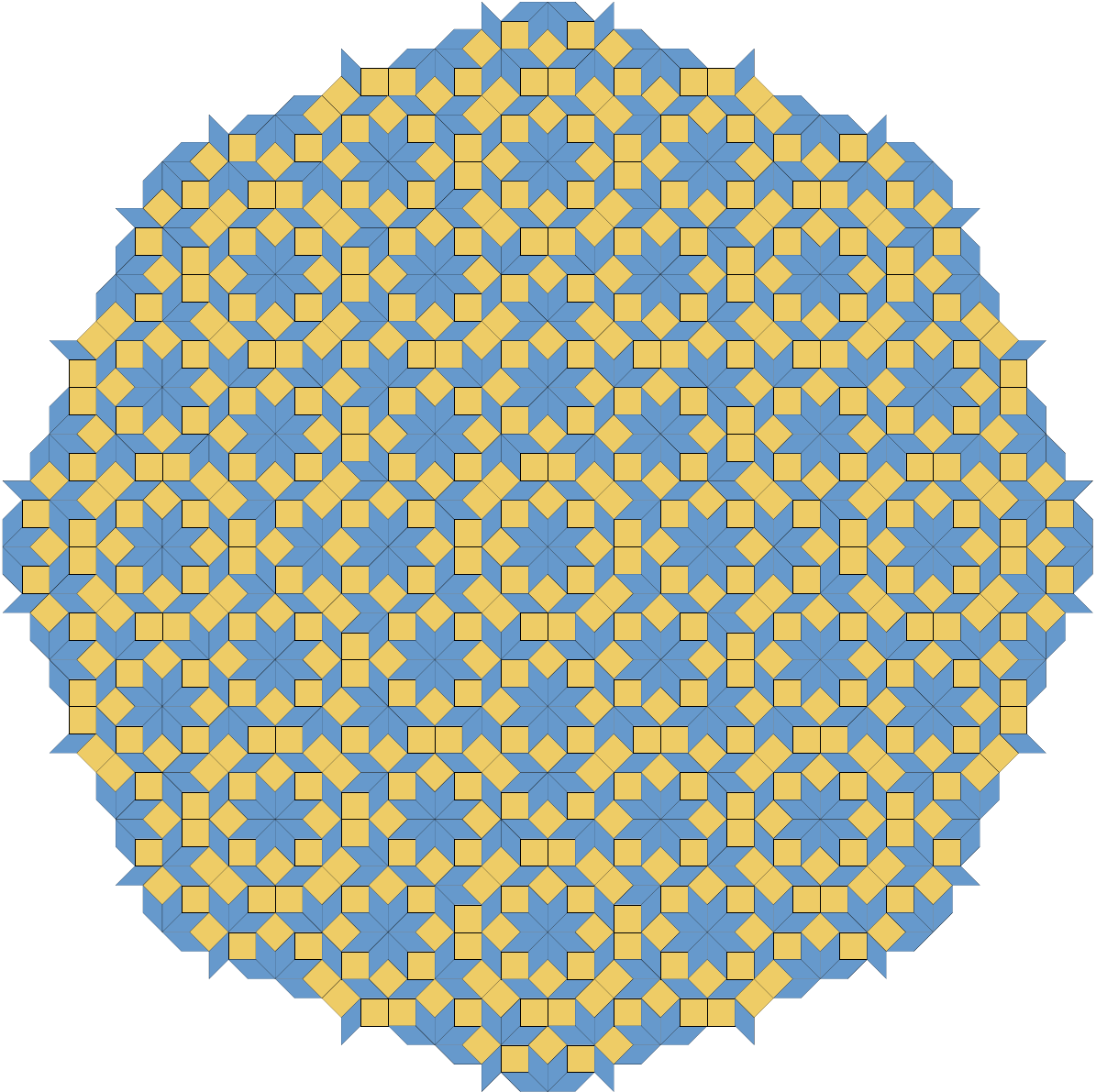}

\begin{picture}(0,0)
\put(-98,0){\small \textsl{Level 2}}
\put( 66,0){\small \textsl{Level 3}}
\end{picture}

\caption{\label{fig:ABlevel0-4}
Three iterations of the Ammann--Beenker rules applied to 
an initial seed of eight thin rhombi: $\calT_0^\AB,  \ldots,  \calT_3^\AB$.}
\end{figure}

\begin{theorem} \label{thm:AB}
The Ammann--Beenker tiling has eigenfunctions at energies $E_\AB = 4,6$.  Denoting $\lambda = \sqrt{2}-1$, we have
\begin{alignat*}{6}
k_\AB(4+)-k_\AB(4-) \;\geq&\  \lambda^4 + \lambda^6 + 2 \lambda^8 \ &=&\  1270-898\sqrt{2}  \ &=&\  0.0362209\ldots,\\
k_\AB(6+)-k_\AB(6-) \;\geq&\  \lambda^4 + \lambda^6               \ &=&\  116 - 82\sqrt{2}  \ &=&\  0.0344878\ldots. 
\end{alignat*}
\end{theorem}

\begin{proof}
As can be seen from Figures~\ref{fig:ABlevel0-4}, \ref{fig:ABlevel2_evec4}, \ref{fig:ABlevel2_evec6}, and \ref{fig:ABlevel3_bigmodes}, each occurrence of the eightfold vertex star produces an eigenfunction at both energies, each once-substituted version also gives an eigenfunction at both energies, and each twice-substituted vertex star produces an additional pair of eigenfunctions at energy $E=4$. One can check visually that the support of each eigenfunction is not contained in the union of the supports of the other eigenfunctions, and hence each occurrence of each patch contributes a linearly independent vector to the corresponding eigenspace.

Thus, the estimates contain three pieces that correspond to the frequencies of the eightfold vertex star, and the result of substituting it once and twice.  
The frequency of the eight-fold star is $\lambda^4$, as computed in \cite{BaakeGrimm2013:AOVol1}. 
The frequencies of the other patches can be seen to be bounded from below by $\lambda^2 \cdot \lambda^4 = \lambda^6$ (for the once-substituted eightfold vertex star) and $\lambda^4 \cdot \lambda^4 =\lambda^8$ (for the twice-substituted version), which can be seen by the multi-dimensional analog of the relevant material in Sections~5.3 and 5.4 of \cite{Queffelec2010}. To work out this analog, one needs to invoke uniform existence results for the limits defining the frequencies in question, which are contained, for example, in \cite{DamanikLenz2001} and \cite{GeerseHof1991}.
\end{proof}

\begin{figure}[b]
\begin{center}
\includegraphics[width=2in]{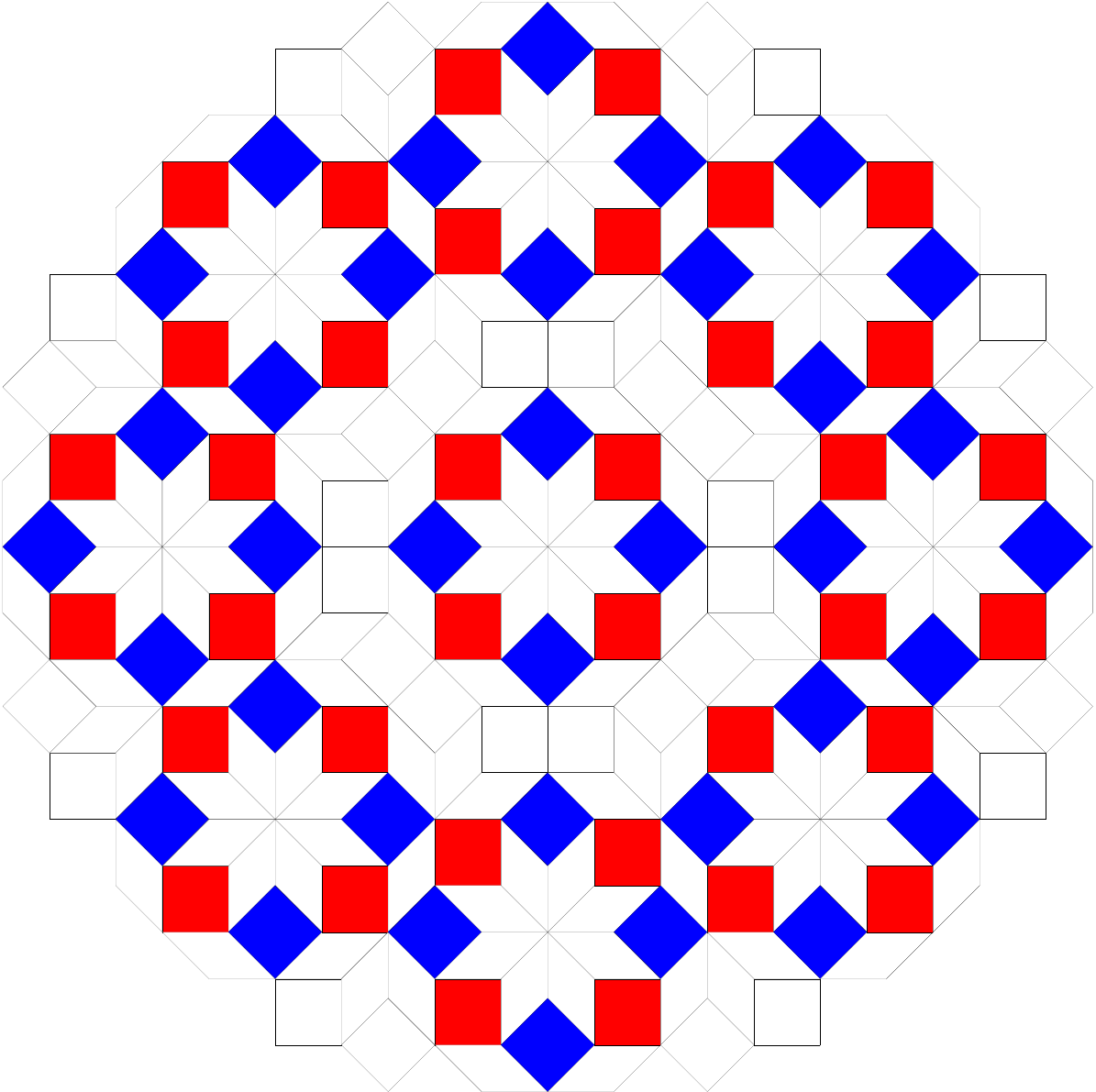}\qquad
\includegraphics[width=2in]{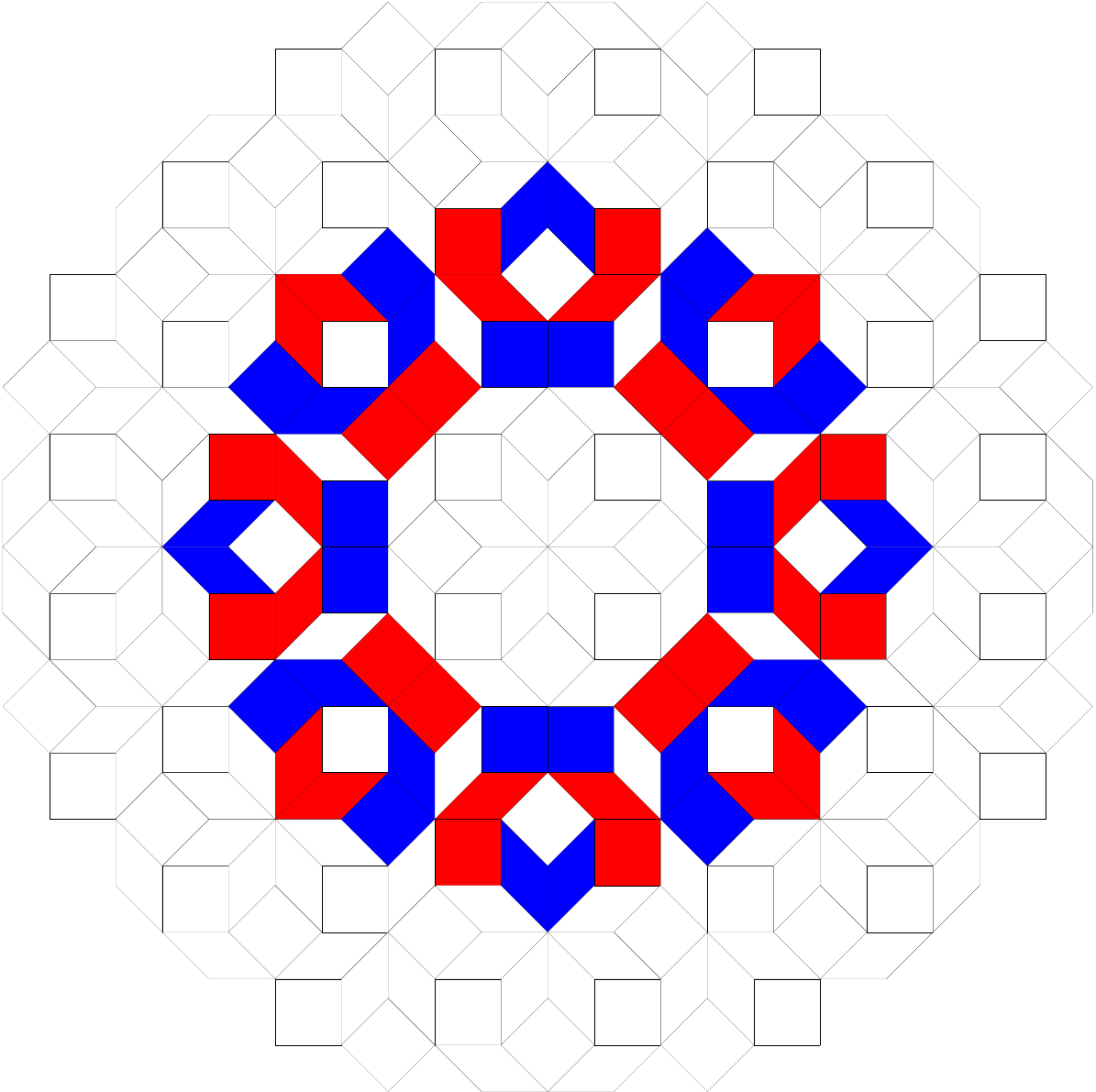}
\end{center}
\caption{\label{fig:ABlevel2_evec4}
Ten linearly independent eigenfunctions at energy $E=4$ (nine on the left, each supported on 8~tiles; one on the right, supported on 64~tiles) at level~2.}
\end{figure}

\begin{figure}[b!]
\begin{center}
\includegraphics[width=2in]{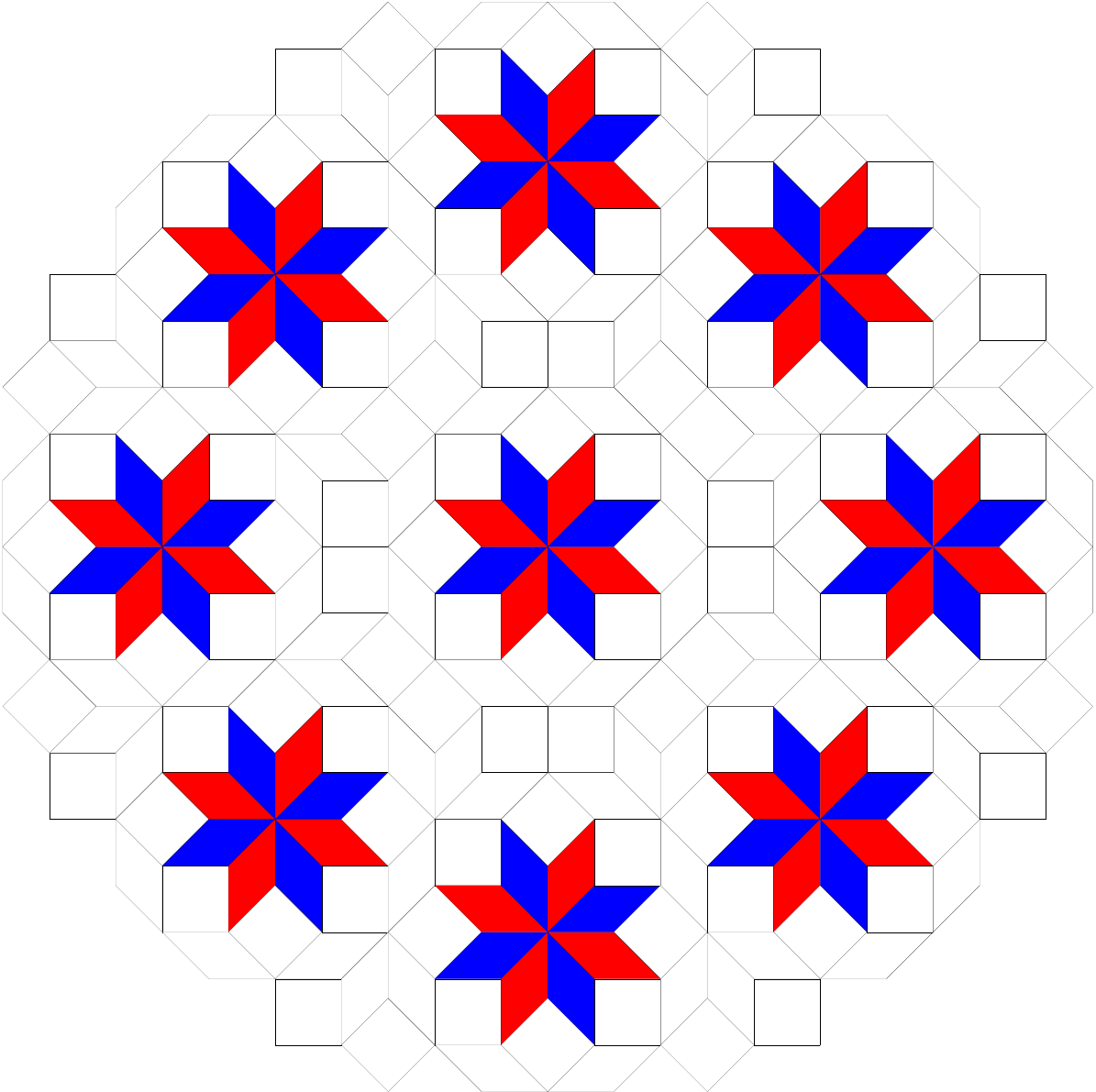}\qquad
\includegraphics[width=2in]{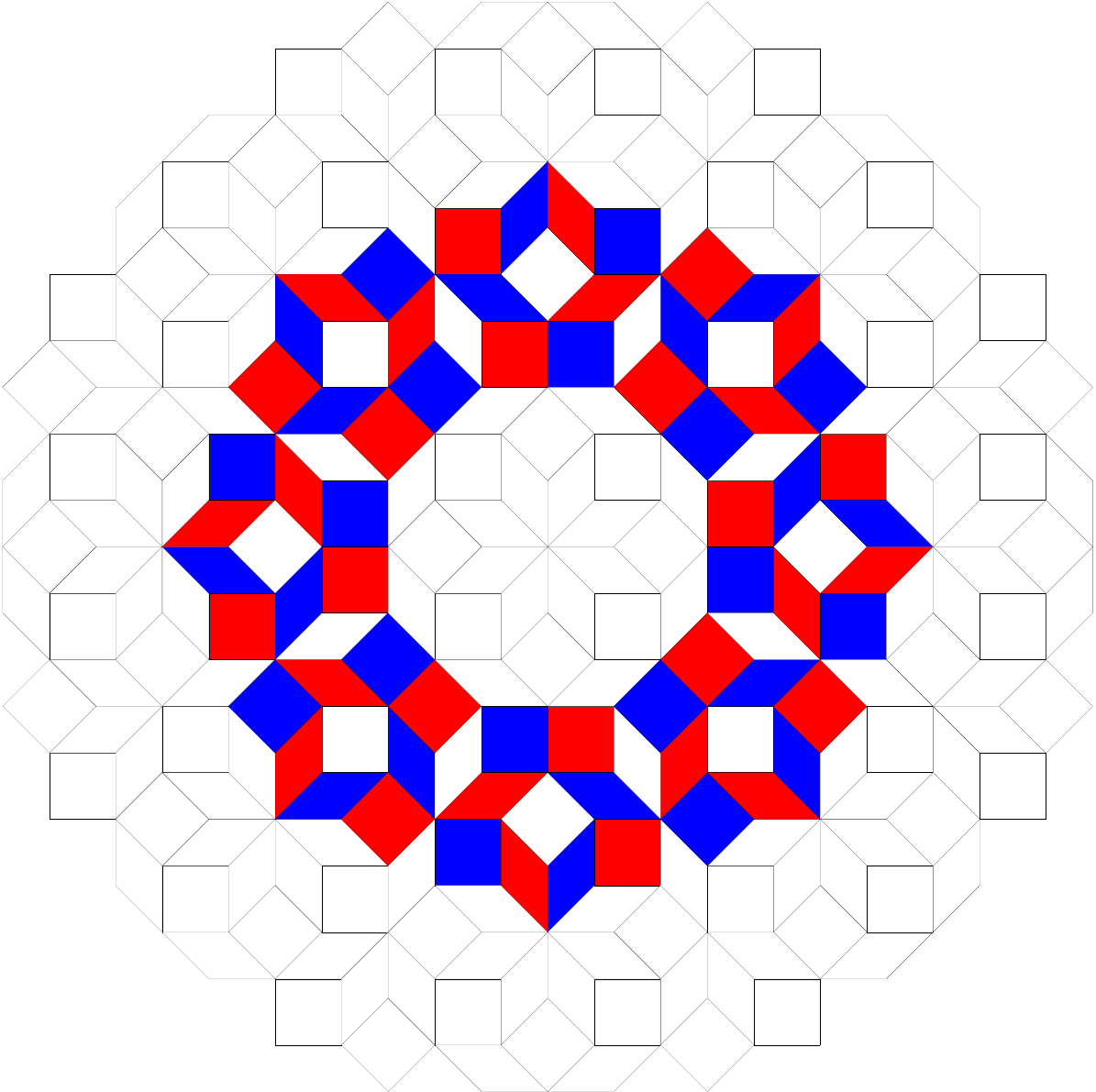}
\end{center}
\caption{\label{fig:ABlevel2_evec6}
Ten linearly independent eigenfunctions at energy $E=6$ (nine on the left, each supported on 8~tiles; one on the right, supported on 64~tiles) at level~2.}
\end{figure}

In Figures~\ref{fig:ABlevel2_evec4} and \ref{fig:ABlevel2_evec6} we present ten locally-supported eigenfunctions that correspond to energies $E=4$ and $E=6$, respectively. 
Table~\ref{tbl:AB} reports the numerically computed 
multiplicities of these eigenvalues up through level~8 (9,096,784 tiles).
An observant reader may notice that the level~2 row of Table~\ref{tbl:AB} 
indicates the existence of an eleventh eigenfunction not shown in 
Figures~\ref{fig:ABlevel2_evec4} and \ref{fig:ABlevel2_evec6}. 
This extra mode lies on the boundary and thus is an artifact of the finite-volume truncation.

\begin{table}[t!]
\caption{\label{tbl:AB}
Ammann--Beenker tiling: The level of tiling, number of tiles, multiplicities of eigenvalues $E=4$ and $E=6$,
and the jump in the IDS at $E=6$.}
\begin{center}
\begin{tabular}{crrrcc}
\emph{level} & 
\multicolumn{1}{c}{\emph{tiles}} & 
\multicolumn{1}{c}{$E = 4$} & 
\multicolumn{1}{c}{$E = 6$} &
\multicolumn{1}{c}{$k_{\AB,n}(4+) - k_{\AB,n}(4-)$}  &
\multicolumn{1}{c}{$k_{\AB,n}(6+) - k_{\AB,n}(6-)$} \\ \hline
 1 & 48          &       3  &        1   & 0.062500\ldots &  0.020833\ldots  \\
 2 & 256         &       11 &       11   & 0.042969\ldots &  0.042969\ldots  \\ 
 3 & 1\,392      &       44 &       42   & 0.031609\ldots &  0.030172\ldots  \\
 4 & 7\,984      &      276 &      258   & 0.034726\ldots &  0.032315\ldots  \\
 5 & 46\,160     &   1\,604 &   1\,538   & 0.034749\ldots &  0.033319\ldots  \\
 6 & 268\,256    &   9\,556 &   9\,106   & 0.035622\ldots &  0.033945\ldots  \\
 7 & 1\,561\,552 &  56\,116 &  53\,490   & 0.036256\ldots &  0.034254\ldots  \\
 8 & 9\,096\,784 & 328\,420 & 312\,834   & 0.036102\ldots &  0.034389\ldots 
\end{tabular}
\end{center}
\end{table}

\begin{figure}[b!]
\begin{center}
\includegraphics[width=2in]{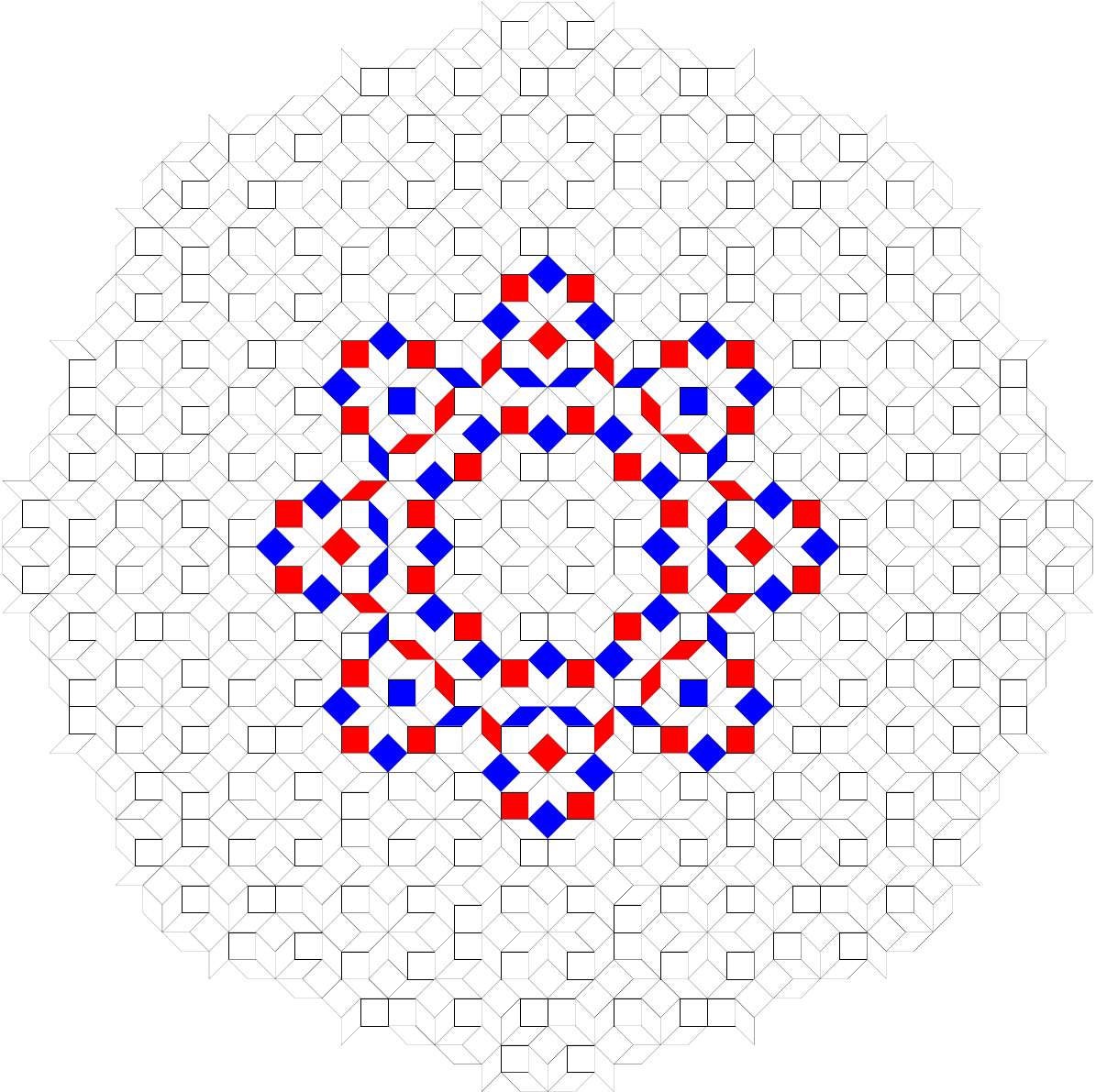}\qquad
\includegraphics[width=2in]{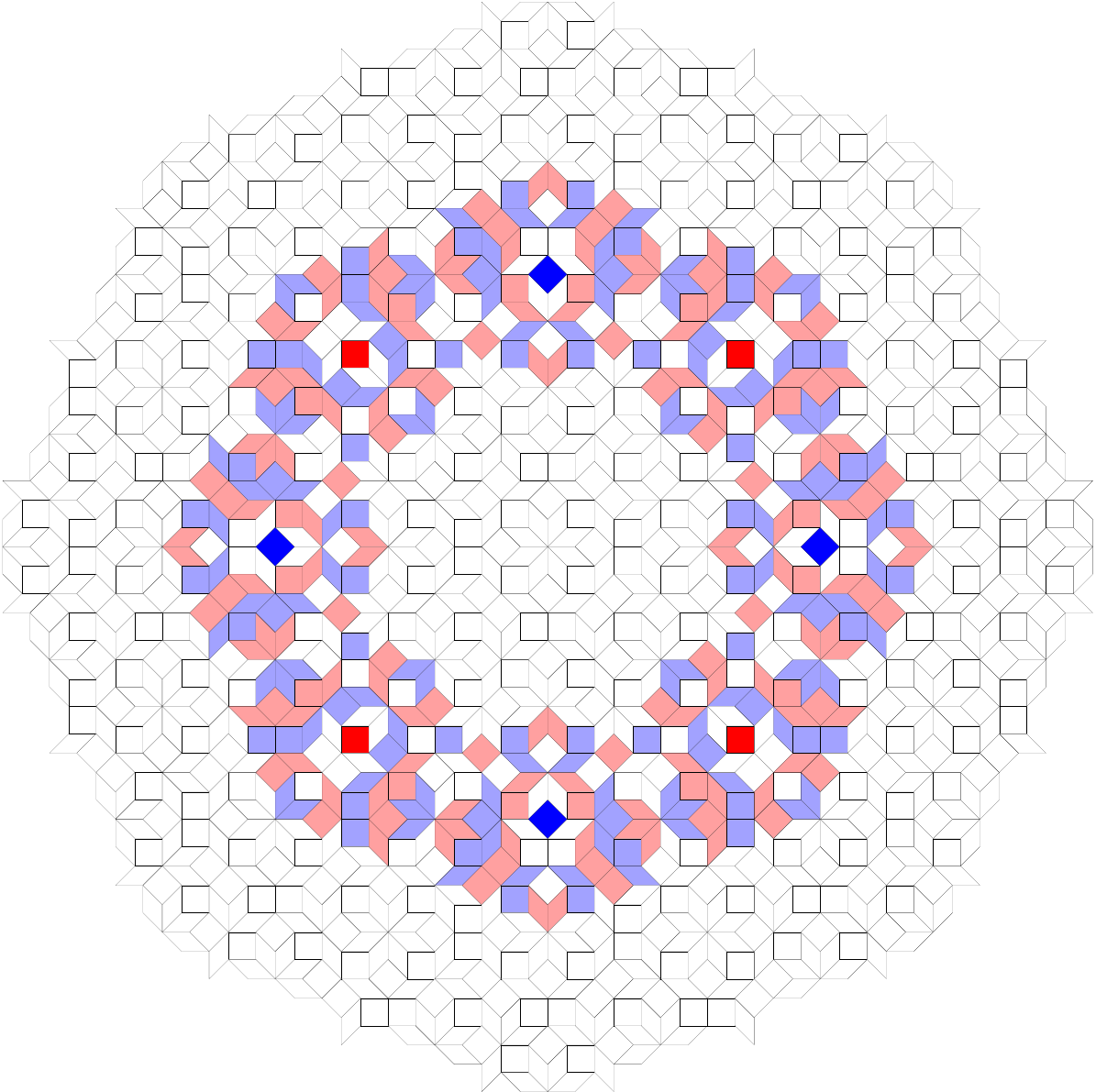}
\end{center}
\caption{Two locally-supported eigenfunctions at level~3 for energy $E=4$ that did not appear on lower levels of the substitution;
the one on the left is supported on 104~tiles with nonzero values $\pm1$;
the one on the right is supported on 328~tiles with nonzero values $\pm1$ and $\pm 1/2$.  }
\label{fig:ABlevel3_bigmodes}
\end{figure}

Theorem~\ref{thm:AB} gives different lower bounds on the 
jump in the integrated density of states for $E=4$ and $E=6$.
The discrepancy in the multiplicity of these eigenvalues emerges
at level~3, where $E=4$ admits two additional eigenfunctions that
are not in the span of the simple mode shapes in Figure~\ref{fig:ABlevel2_evec4}.
While still locally supported, these two modes involve many more tiles:
the eigenfunctions in Figure~\ref{fig:ABlevel3_bigmodes},
supported on 104~tiles and 328~tiles,
provide a basis for this extra two-dimensional eigenspace. 
Accounting for the recurrence of such mode shapes at higher levels 
explains the discrepancy of the bounds for $E=4$ and $E=6$ in Theorem~\ref{thm:AB}.
The situation is analogous to the rhombus tiling, where apparently new mode
shapes emerged at higher levels (see Figure~\ref{fig:rhombus_more_ev}).
Whether modes with additional complexity emerge at still higher levels
is an open question; the presence of such modes is difficult to tease 
out from numerical approximations to the jump in the integrated density
of states, given the relative rarity of those modes and the additional
complication of modes supported on the boundary.



\section{Questions and Open Problems}
Let us conclude by showing numerically-computed approximations to
the integrated density of states (IDS) for the five tilings we have discussed,
and posing some questions these plots suggest.
On one hand, the IDS is a fundamental spectral quantity. On the other hand, the shape of the graph of the IDS naturally suggests several possibilities. Concretely:
\begin{enumerate}
\item A sharp vertical jump suggests the presence of an eigenvalue corresponding to an eigenfunction. Indeed, this is precisely how many of the examples from the present work were observed. Of course, one must be careful here, since one is looking at eigenvalue counting functions associated to finite tilings, so \emph{every} jump is sharp. A simple eigenvalue causes a jump of size $1/(\mbox{\# tiles})$; higher multiplicities give bigger jumps.  One is looking for a jump that is stable, i.e., the size of the jump stays bounded from below as the level of the tiling is increased.
\item Since the IDS is constant on each connected component of the complement of the spectrum, a flat section in the plot of the approximations of the IDS that is stable upon iterating the substitution rule suggests the presence of a spectral gap.
\item Conversely, the spectrum is given by the set of points of increase of the IDS, so an interval on which the IDS is everywhere increasing corresponds to an interval that is completely contained in the spectrum.
\end{enumerate}

\begin{figure}[t!]
\begin{center}
\includegraphics[width=4.5in]{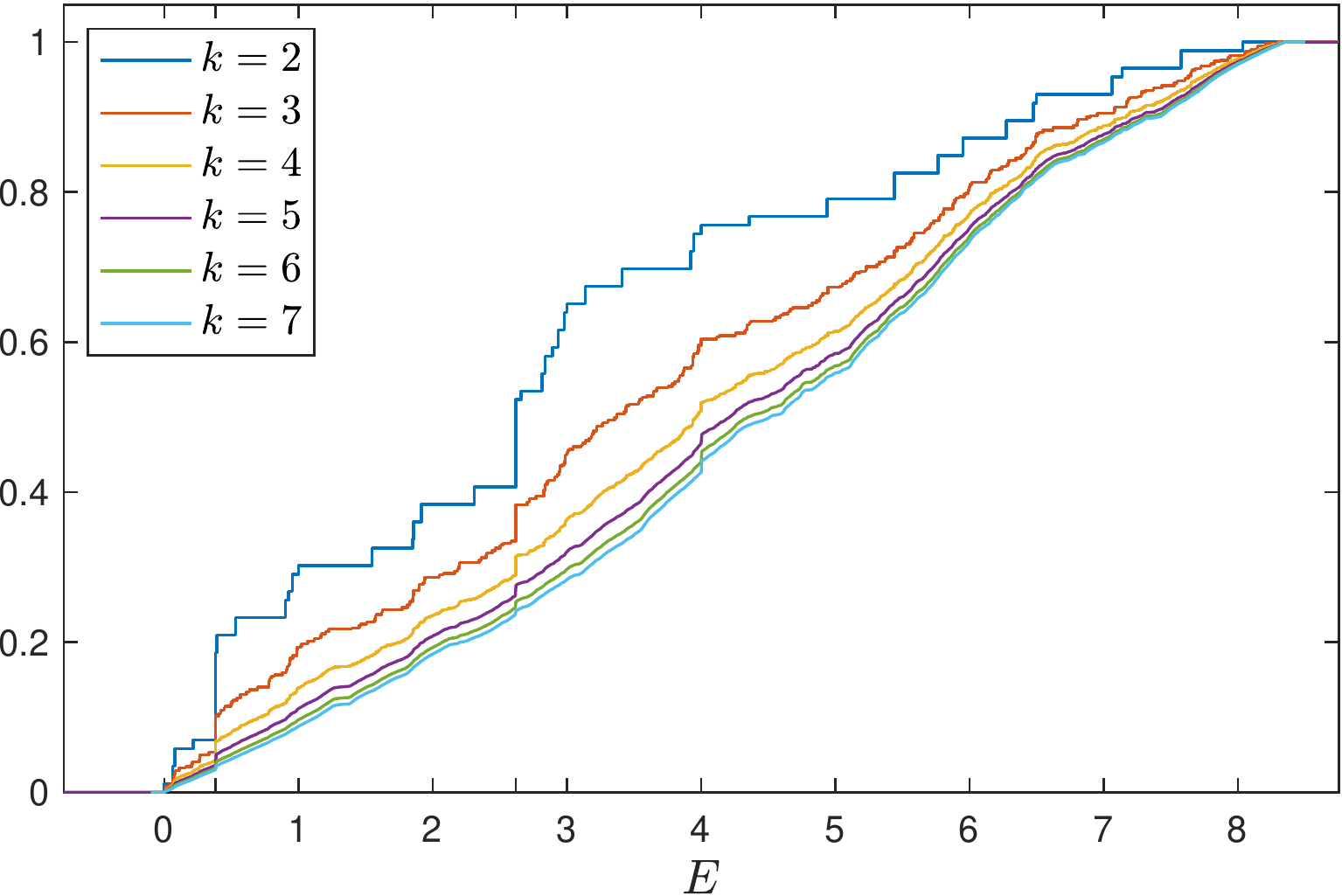}
\begin{picture}(0,0)
\put(-115,33){\makebox(100,10)[r]{\small Boat--Star}}
\end{picture}
\end{center}

\vspace*{-4pt}
\caption{\label{fig:ids_boatstar}
Computed integrated density of states for the boat--star tiling, levels~2 through~7.  Notice the (diminishing) jumps at $1/\varphi^2 \approx 0.381966$ and $\varphi^2 \approx 2.618034$ due to the boundary modes shown in Figure~\ref{fig:boatStar:bdy}; 
the jump at $E=4$ is due to ring modes of the form shown in Figure~\ref{fig:boatStar:pentSupport}.}
\end{figure}

Figure~\ref{fig:ids_boatstar} shows several finite-patch approximations to the 
IDS associated with the boat--star tiling.
To produce this plot (and the other IDS plots that follow), 
we prefer to compute all eigenvalues of $\Delta_n$ numerically (using {\tt eig} in MATLAB).\ \  
While expensive, this calculation allows one to evaluate the multiplicity of eigenvalues
(subject to rounding errors that are well understood for symmetric eigenvalue calculations).
In Figure~\ref{fig:ids_boatstar}, this {\tt eig} approach is feasible up through level~5 
(30,406 tiles).
For level~6 (210,181 tiles) and level~7 (1,447,691 tiles), we use 
a different strategy inspired by spectrum slicing~\cite[sect.~3.3]{Par98}).
The spectral interval is finely discretized with points $\{E_j\}$.
For each $E_j$, we use MATLAB's {\tt ldl} command to compute a factorization
$L - E_j = L_j D_j L_j^T$, where $L_j$ is a permuted unit lower-triangular matrix and $D_j$ is block-diagonal,
having 1-by-1 and 2-by-2 diagonal blocks~\cite[sect.~4.4]{GolubVanLoan2012}.  
By Sylvester's Law of Inertia, the congruent matrices $L-E_j$ and $D_j$ have 
the same number of negative eigenvalues; the block diagonal form of $D_j$ makes that number easy to count.  Since the negative eigenvalues of $L-E_j$ reveal the number of eigenvalues of $L$ smaller than $E_j$, these counts collectively give
an approximation to the IDS.\ \ 
(Indeed, we also use this approach to count the multiplicity of special energies known to have locally-supported eigenfunctions, as presented in the tables throughout this paper.
For these counts, we take slices just above and below the target energy; 
in most cases we vary the slice size to gain confidence in the presented numbers.)

One observes some interesting features that prompt the following questions. First, one is interested in the topological structure of the spectrum.

\begin{question} \label{q:boatstar:int}
Let $\Sigma_\boatstar$ denote the spectrum associated with the boat--star tiling. Is the interior of $\Sigma_\boatstar$ nonempty? If the interior is nonempty, is it dense in $\Sigma_\boatstar$?
\end{question}

We expect that establishing the presence, let alone density, of intervals in the spectrum to be quite challenging. The plots of the approximants to the IDS suggest that one may start looking for nonempty intervals near the extrema of the spectrum. Thus, we pose separately the following question, which may be approachable via perturbative methods.

\begin{question} \label{q:boatstar:extreme}
Does there exist $\delta>0$ such that 
\begin{equation}
[0,\delta) \cup (\max\Sigma_\boatstar-\delta,\max\Sigma_\boatstar] \subseteq\Sigma_\boatstar?
\end{equation}
\end{question}

The estimation and computation of extrema of the spectrum is a separately interesting question. The bottom is given by $E=0$ by elementary arguments, but the top of the spectrum is not always trivial to compute, so we also ask:

\begin{question} \label{q:boatstar:ground}
Can one compute $\max\Sigma_\boatstar$ in closed form?
\end{question}

It is clear from \cite{KlasLenzStol2003CMP} that one may ``insert'' a locally-supported eigenfunction into a Laplacian on any MLD class. Namely, the MLD class of any tiling contains a tiling whose nearest neighbor Laplacian has locally-supported eigenfunctions.
\begin{question}
Can one always remove a locally-supported eigenfunction from a Laplacian on any MLD class? Is it even true that for every MLD class, there exists a tiling in that class whose associated Laplacian does not have any locally-supported eigenfunctions?
\end{question}

On the one hand, as we just mentioned, one can always ensure the presence of some locally supported eigenfunction for a suitable choice of tiling in an MLD class, and then, assuming the model in question is linearly repetitive, this will always lead to a jump in the IDS. On the other hand, as we have seen in this paper, the IDS may in fact have multiple jumps in some cases. This naturally leads to the following question.

\begin{question} \label{q:boatstar:numjumps}
Is the number of jumps in the IDS always finite? Is there an effective way of bounding this number for a given linearly repetitive tiling?
\end{question}

\begin{figure}[b!]
\begin{center}
\includegraphics[width=4.5in]{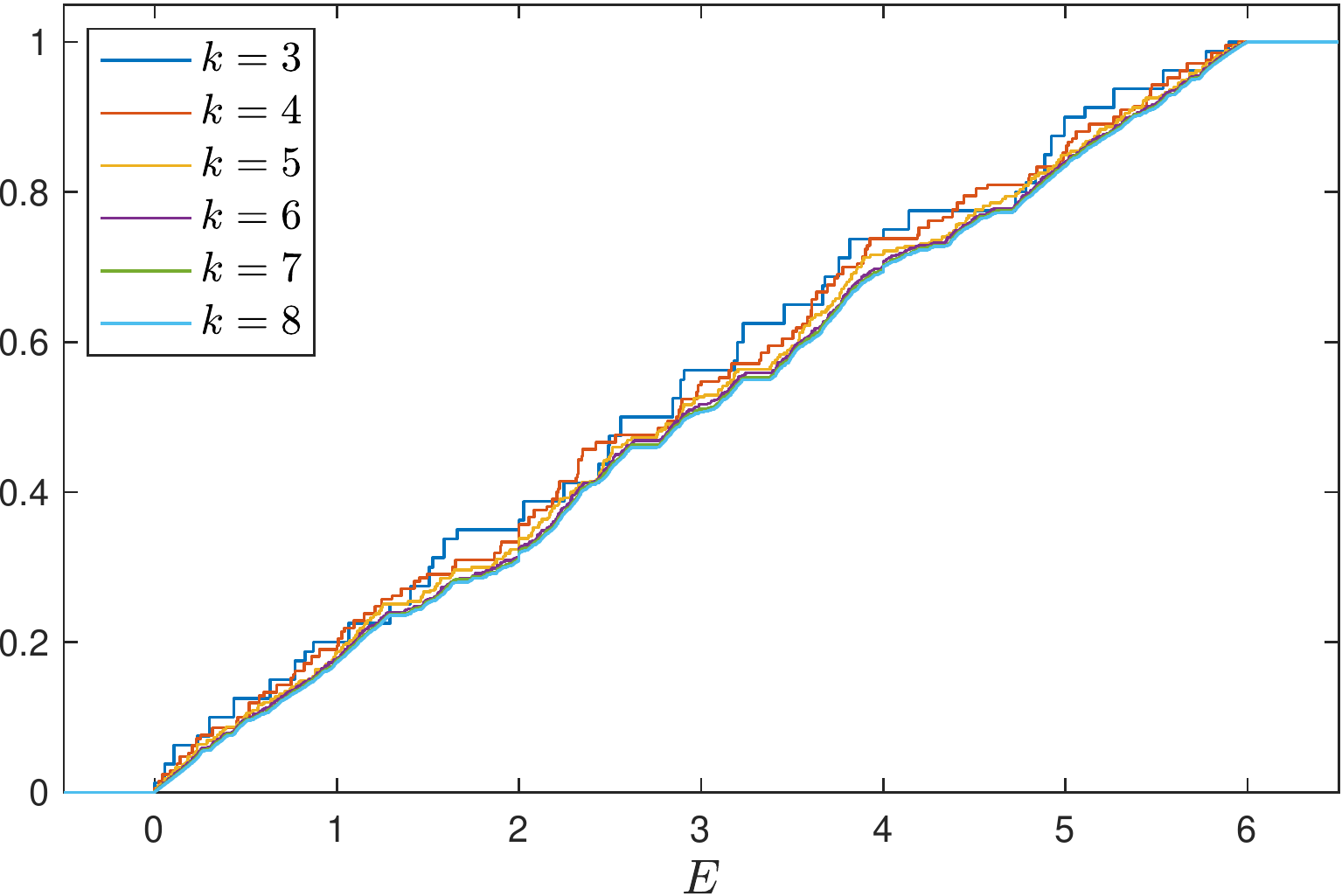}
\begin{picture}(0,0)
\put(-115,33){\makebox(100,10)[r]{\small Robinson triangle}}
\end{picture}

\vspace*{-4pt}
\caption{\label{fig:ids_triangle}
Computed integrated density of states for the Robinson triangle tiling, levels~3 through~8.
Note the jumps at $E=2$ and $E=4$.}
\end{center}
\end{figure}

\begin{figure}[b!]
\begin{center}
 \includegraphics[width=4.5in]{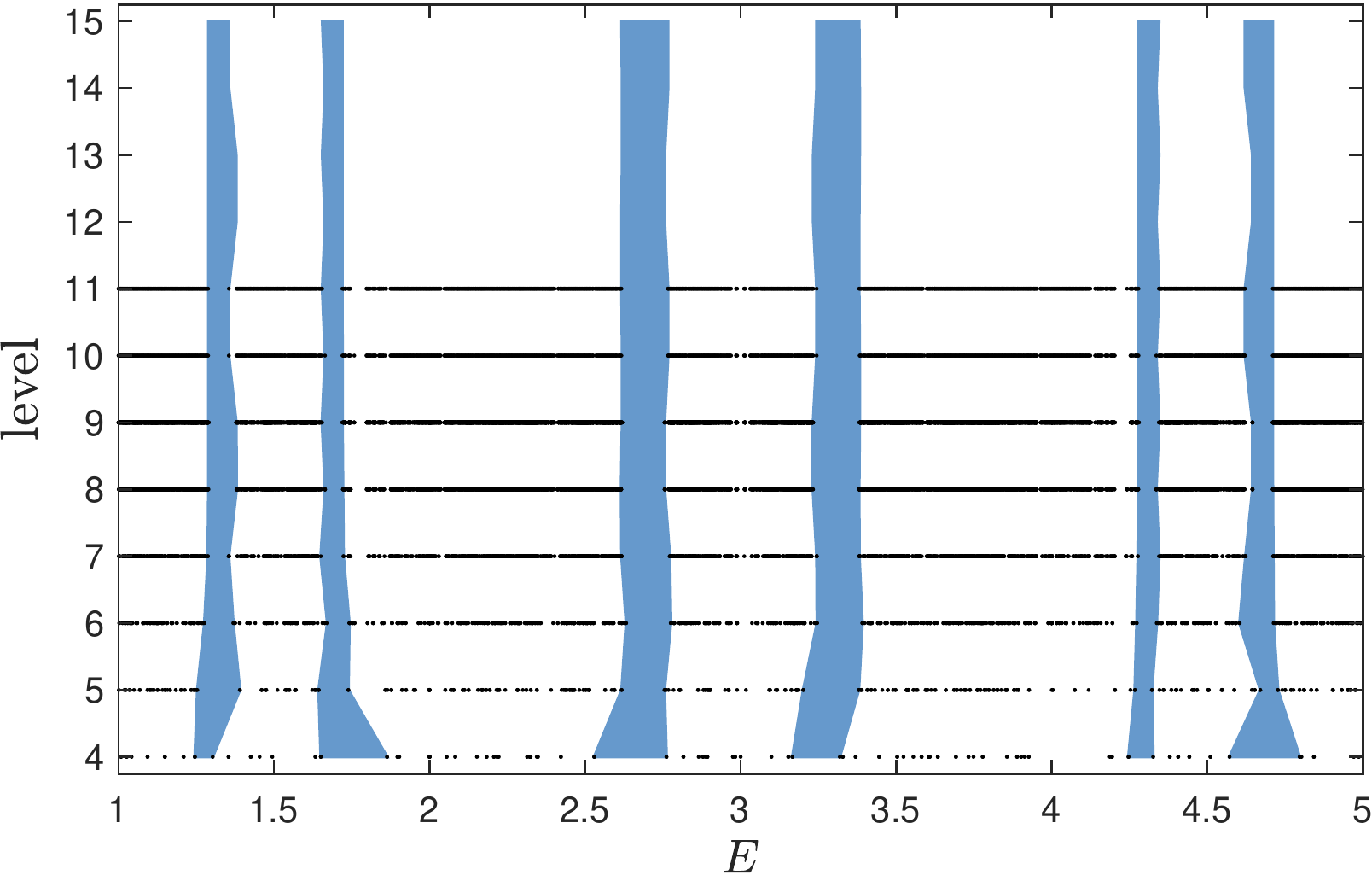}
\end{center}

\vspace*{-5pt}
\caption{\label{fig:trigaps}
The spectra of finite-patch approximations of the Robinson triangle tiling 
exhibit gaps that apparently persist as the level is increased,
corresponding to plateaus in the integrated density of states.
This image shows numerically computed interior bounds for six gaps (blue regions) 
as the level $k$ increases.
For lower levels, we also show all computed eigenvalues as black dots.
Several additional gaps are apparent.  (We suspect there are infinitely many such gaps.)
}
\end{figure}

Note that if the previous question has an affirmative answer, the IDS will be piecewise continuous, and then it is natural to ask for stronger regularity properties on these pieces. Specifically, based on the shape of the IDS plots we have exhibited, we ask the following: 

\begin{question} \label{q:boatstar:holder}
Near the top or bottom of the spectrum, is the IDS $\alpha$-H\"older continuous? Lipschitz continuous?
\end{question}

Figure~\ref{fig:ids_triangle} shows several finite-patch approximations to the IDS associated with the Robinson triangle tiling. The reader can observe the jumps at energies $E = 2$ and $E=4$. As mentioned before, the spectrum is given by the set of points of increase of the IDS.\ \ As such, the parts of this IDS plot that correspond to the bottom and top of the spectrum are somewhat suggestive.

\begin{question}
Investigate the analogs of Questions~\ref{q:boatstar:int}, \ref{q:boatstar:extreme}, and \ref{q:boatstar:holder} for the Robinson triangle tiling.
\end{question}

Notice an intriguing feature of the graph of the approximants to the IDS of the Robinson tiling in Figure~\ref{fig:ids_triangle}: the emergence of what appear to be relatively stable spectral gaps (e.g., a bit to the left and right of $E=3$).  Figure~\ref{fig:trigaps} examines this possibility in finer detail:  we compare finitely computed eigenvalues of $\Delta_n$ as $n$ grows, looking for persistent gaps.  Beyond the level at which we can compute all eigenvalues of $\Delta_n$, we use spectrum slicing to locate eigenvalues that define the edge of the intervals that were suggested at lower levels.
Figure~\ref{fig:trigaps} shows six such gaps in blue; several other potential gaps (e.g., to the right of the second blue gap, near $E=3$, and to the left of the fifth blue gap) are also apparent.
It would be interesting to verify this phenomenon rigorously.

\begin{figure}[b]
\begin{center}
\includegraphics[width=4.5in]{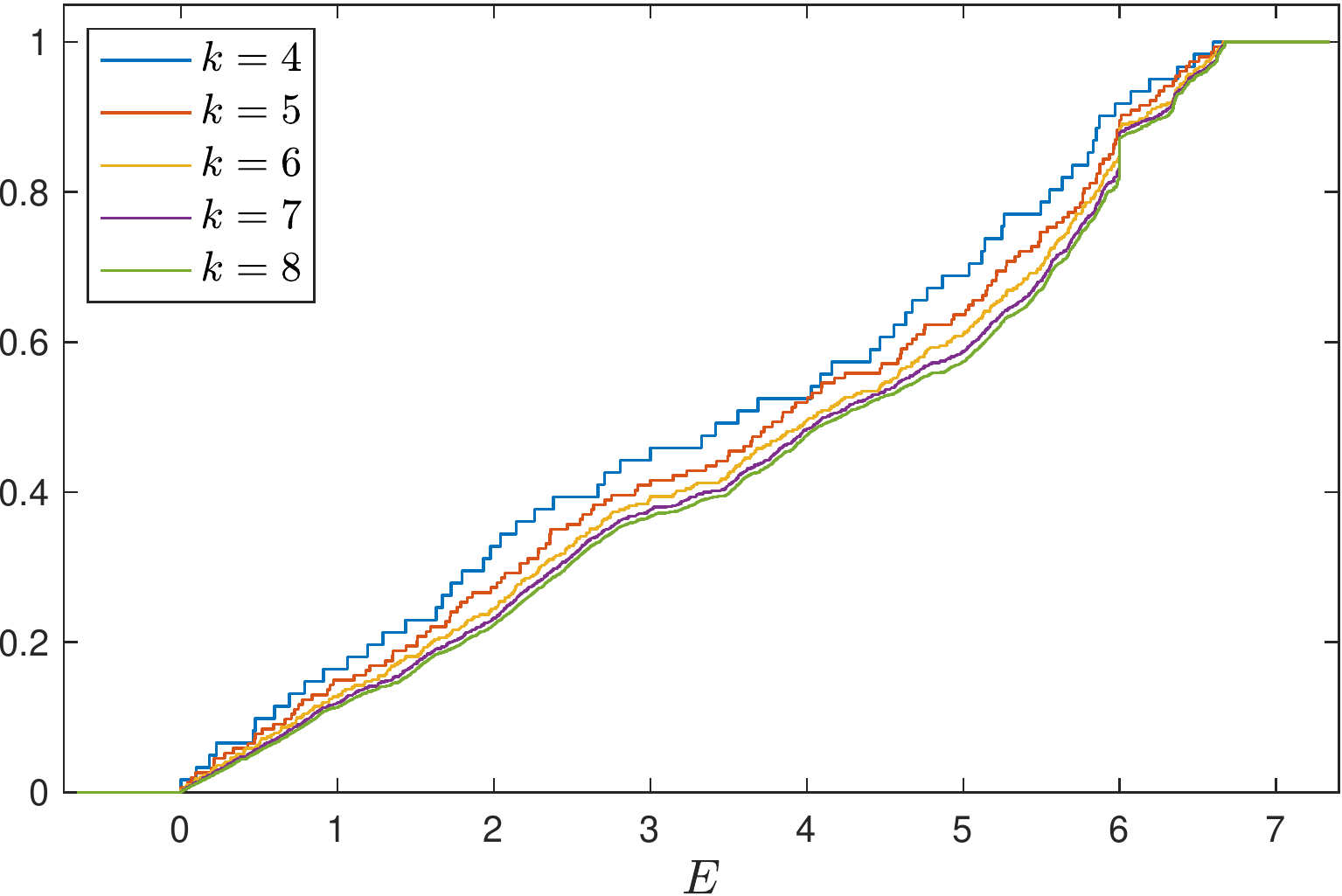}
\begin{picture}(0,0)
\put(-115,33){\makebox(100,10)[r]{\small Rhombus}}
\end{picture}

\vspace*{-4pt}
\caption{\label{fig:ids_rhombus}
Computed integrated density of states for the rhombus tiling, levels~4 through~8.
Note the jump at $E=6$.}
\end{center}
\end{figure}

\begin{figure}[b!]
\begin{center}
\includegraphics[width=4.5in]{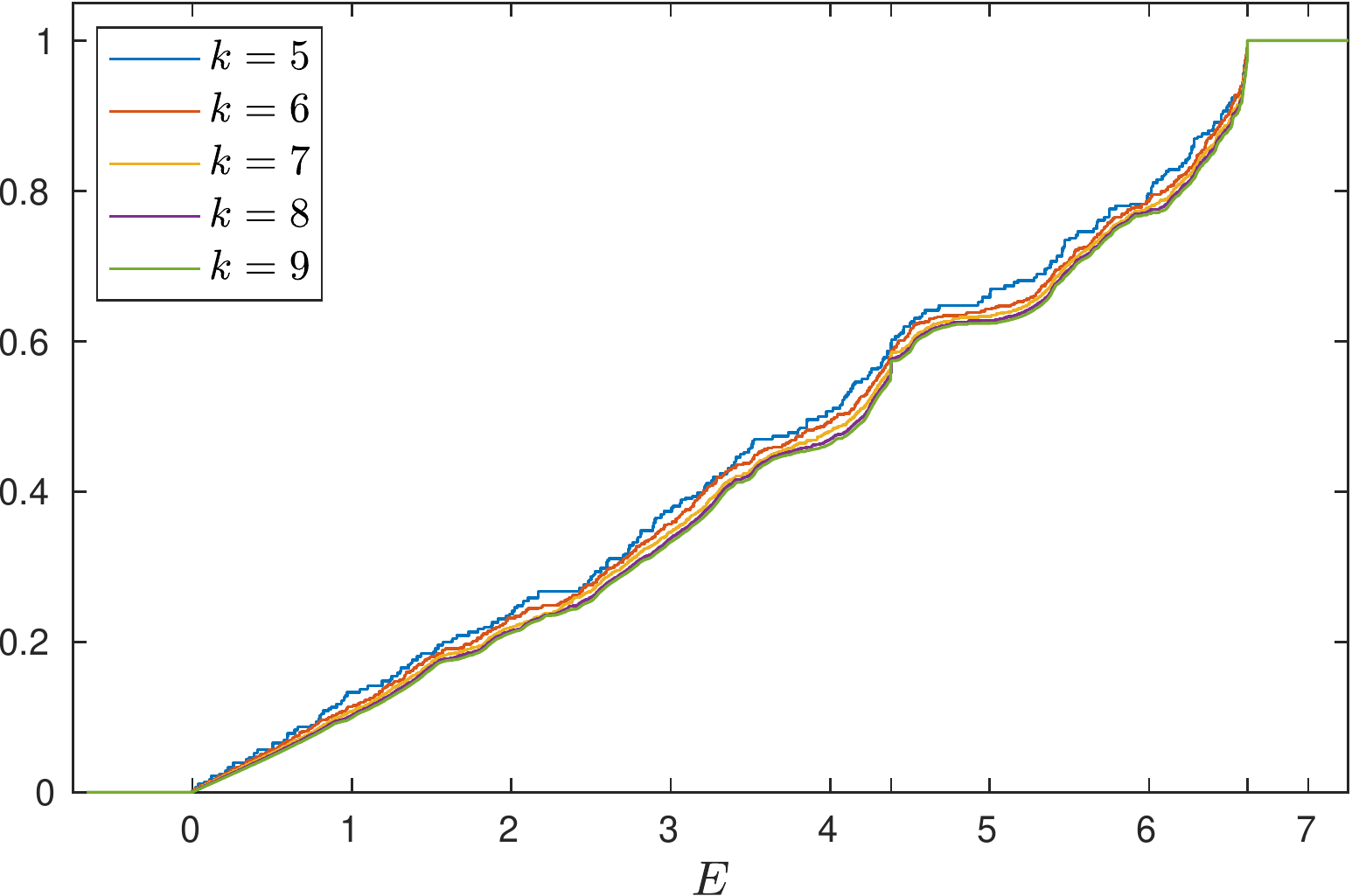}
\begin{picture}(0,0)
\put(-115,33){\makebox(100,10)[r]{\small Kite--Dart}}
\end{picture}

\vspace*{-4pt}
\caption{\label{fig:ids_kitedart}
Computed integrated density of states for the kite--dart tiling, levels~5 through~9.
Note the jumps at $E=6-\varphi=4.381966\ldots$ and $E=5+\varphi=6.618033\ldots$ (at the top of the spectrum).}
\end{center}
\end{figure}

\begin{question} \label{q:rob:gap}
Show that $\Sigma_\robinson$ has a nontrivial spectral gap.  Are there infinitely many?
\end{question}

We conclude with approximations to the IDS for the rhombus, kite--dart, and Ammann--Beenker tilings in Figures~\ref{fig:ids_rhombus}, \ref{fig:ids_kitedart}, and \ref{fig:ids_ab}. These plots have similar features to the first two, and hence one may ask similar questions to Questions~\ref{q:boatstar:int}--\ref{q:rob:gap}. 

\begin{figure}[h!]
\begin{center}
\includegraphics[width=4.5in]{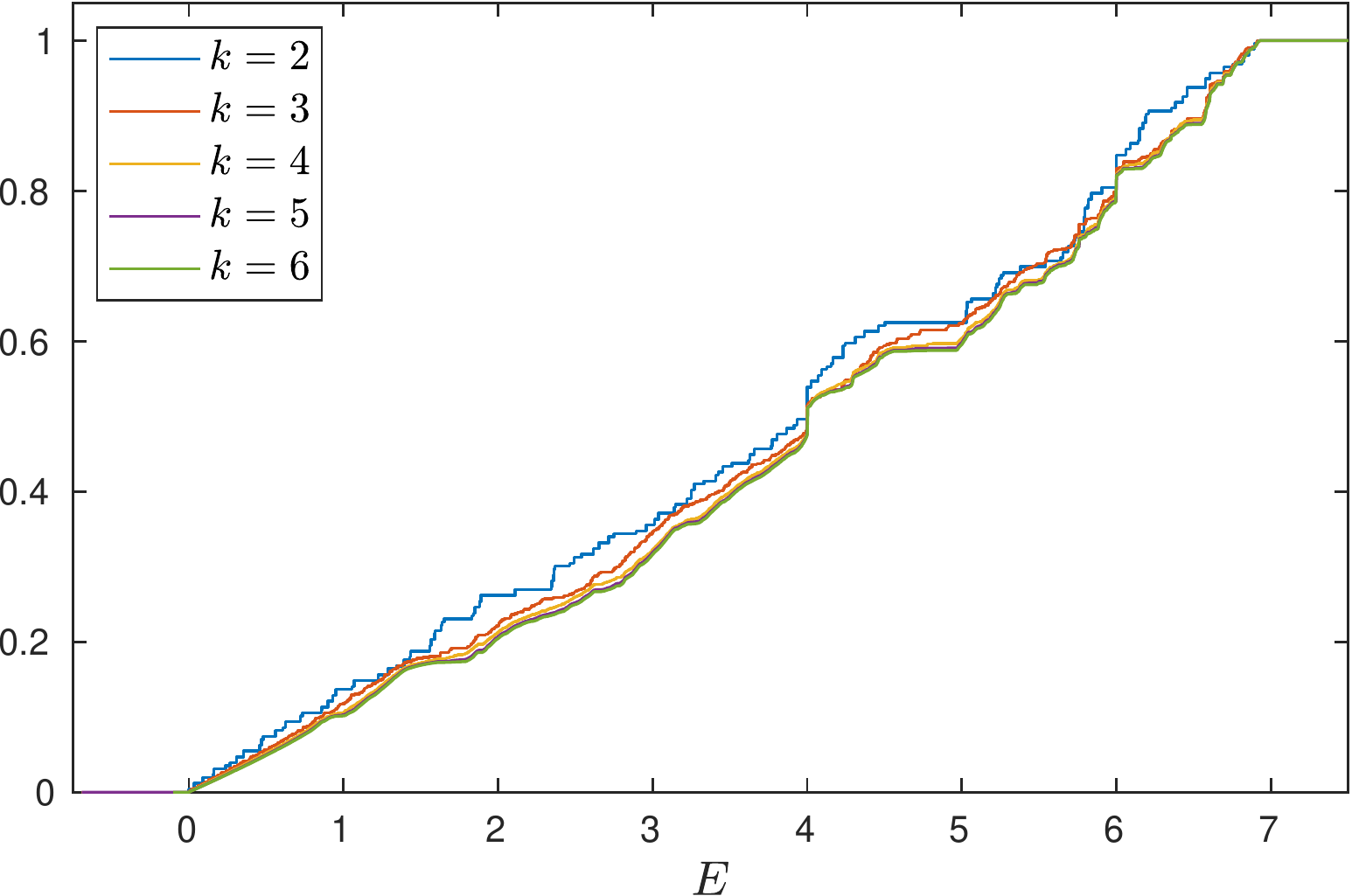}
\begin{picture}(0,0)
\put(-115,33){\makebox(100,10)[r]{\small Ammann--Beenker}}
\end{picture}

\vspace*{-4pt}
\caption{\label{fig:ids_ab}
Computed integrated density of states for the Ammann--Beenker tiling, levels~2 through~6.
Note the jumps at $E=4$ and $E=6$.}
\end{center}
\end{figure}



\bibliographystyle{abbrv}

\bibliography{PenroseBib}

\begin{thebibliography}{10}

\bibitem{ATFK1988PRB}
M.~Arai, T.~Tokihiro, T.~Fujiwara, and M.~Kohmoto.
\newblock Strictly localized states on a two-dimensional {P}enrose lattice.
\newblock {\em Phys. Rev. B (3)}, 38(3):1621--1626, 1988.

\bibitem{BaakeGrimm2013:AOVol1}
M.~Baake and U.~Grimm.
\newblock {\em Aperiodic order. {V}ol. 1}, volume 149 of {\em Encyclopedia of
  Mathematics and its Applications}.
\newblock Cambridge University Press, Cambridge, 2013.
\newblock A mathematical invitation, With a foreword by Roger Penrose.

\bibitem{BaakeGrimm2013:AOVol2}
M.~Baake and U.~Grimm.
\newblock {\em Aperiodic order. {V}ol. 2}, volume 166 of {\em Encyclopedia of
  Mathematics and its Applications}.
\newblock Cambridge University Press, Cambridge, 2017.

\bibitem{BaakeMoody2000CRM}
M.~Baake and R.~V. Moody, editors.
\newblock {\em Directions in mathematical quasicrystals}, volume~13 of {\em CRM
  Monograph Series}.
\newblock American Mathematical Society, Providence, RI, 2000.

\bibitem{DamGorYes2016Invent}
D.~Damanik, A.~Gorodetski, and W.~Yessen.
\newblock The {F}ibonacci {H}amiltonian.
\newblock {\em Invent. Math.}, 206(3):629--692, 2016.

\bibitem{DamanikLenz2001}
D.~Damanik and D.~Lenz.
\newblock Linear repetitivity. {I}. {U}niform subadditive ergodic theorems and
  applications.
\newblock {\em Discrete Comput. Geom.}, 26(3):411--428, 2001.

\bibitem{Frank2008Expo}
N.~P. Frank.
\newblock A primer of substitution tilings of the {E}uclidean plane.
\newblock {\em Expo. Math.}, 26(4):295--326, 2008.

\bibitem{TilingsEncyclopedia}
D.~Frettl\"oh, E.~Harriss, and F.~G\"ahler.
\newblock Tilings encyclopedia.
\newblock https://tilings.math.uni-bielefeld.de.

\bibitem{FATK1988PRB}
T.~Fujiwara, M.~Arai, T.~Tokihiro, and M.~Kohmoto.
\newblock Localized states and self-similar states of electrons on a
  two-dimensional {P}enrose lattice.
\newblock {\em Phys. Rev. B (3)}, 37(6):2797--2804, 1988.

\bibitem{Gardner1997}
M.~Gardner.
\newblock {\em Penrose Tiles to Trapdoor Ciphers}.
\newblock Mathematical Association of America, Washington, DC, revised edition,
  1997.

\bibitem{GeerseHof1991}
C.~P.~M. Geerse and A.~Hof.
\newblock Lattice gas models on self-similar aperiodic tilings.
\newblock {\em Rev. Math. Phys.}, 3(2):163--221, 1991.

\bibitem{GolubVanLoan2012}
G.~H. Golub and C.~F. Van~Loan.
\newblock {\em Matrix Computations}.
\newblock Johns Hopkins University Press, Baltimore, fourth edition, 2012.

\bibitem{GrunbaumShephard1987}
B.~Gr\"{u}nbaum and G.~C. Shephard.
\newblock {\em Tilings and Patterns}.
\newblock W. H. Freeman and Company, New York, 1987.

\bibitem{KellLenzSav2015}
J.~Kellendonk, D.~Lenz, and J.~Savinien, editors.
\newblock {\em Mathematics of aperiodic order}, volume 309 of {\em Progress in
  Mathematics}.
\newblock Birkh\"{a}user/Springer, Basel, 2015.

\bibitem{KlasLenzStol2003CMP}
S.~Klassert, D.~Lenz, and P.~Stollmann.
\newblock Discontinuities of the integrated density of states for random
  operators on {D}elone sets.
\newblock {\em Comm. Math. Phys.}, 241(2-3):235--243, 2003.

\bibitem{KohSut1986PRL}
M.~Kohmoto and B.~Sutherland.
\newblock Electronic states on a {P}enrose lattice.
\newblock {\em Phys. Rev. Lett.}, 56:2740--2743, 1986.

\bibitem{KohSutTan1987}
M.~Kohmoto, B.~Sutherland, and C.~Tang.
\newblock Critical wave functions and a {C}antor-set spectrum of a
  one-dimensional quasicrystal model.
\newblock {\em Phys. Rev. B (3)}, 35(3):1020--1033, 1987.

\bibitem{LenzStollmann2003MPAG}
D.~Lenz and P.~Stollmann.
\newblock Algebras of random operators associated to {D}elone dynamical
  systems.
\newblock {\em Math. Phys. Anal. Geom.}, 6(3):269--290, 2003.

\bibitem{LenzStollmann2001}
D.~Lenz and P.~Stollmann.
\newblock Delone dynamical systems and associated random operators.
\newblock In {\em Operator algebras and mathematical physics ({C}onstan\c{t}a,
  2001)}, pages 267--285. Theta, Bucharest, 2003.

\bibitem{LenzStollmann2005}
D.~Lenz and P.~Stollmann.
\newblock An ergodic theorem for {D}elone dynamical systems and existence of
  the integrated density of states.
\newblock {\em J. Anal. Math.}, 97:1--24, 2005.

\bibitem{MirzOkt2020PRB}
M.~Mirzhalilov and M.~{\"{O}}. Oktel.
\newblock Perpendicular space accounting of localized states in a quasicrystal.
\newblock {\em Phys. Rev. B}, 102, 2020.

\bibitem{Moody1997NATO}
R.~V. Moody, editor.
\newblock {\em The mathematics of long-range aperiodic order}, volume 489 of
  {\em NATO Advanced Science Institutes Series C: Mathematical and Physical
  Sciences}. Kluwer Academic Publishers Group, Dordrecht, 1997.

\bibitem{Oktel2021}
M.~{\"{O}}. Oktel.
\newblock Strictly localized states in the octagonal {A}mmann-{B}eenker
  quasicrystal.
\newblock {\em Phys. Rev. B}, 104, 2021.

\bibitem{Oktel2022}
M.~{\"{O}}. Oktel.
\newblock Localized states in local isomorphism classes of pentagonal
  quasicrystals.
\newblock preprint: arXiv:2203.09899, 2022.

\bibitem{OPRSS1983}
S.~Ostlund, R.~Pandit, D.~Rand, H.~J. Schellnhuber, and E.~D. Siggia.
\newblock One-dimensional {S}chr\"{o}dinger equation with an almost periodic
  potential.
\newblock {\em Phys. Rev. Lett.}, 50(23):1873--1876, 1983.

\bibitem{Par98}
B.~N. Parlett.
\newblock {\em The Symmetric Eigenvalue Problem}.
\newblock SIAM, Philadelphia, {SIAM} {Classics} edition, 1998.

\bibitem{Patera1998FIM}
J.~Patera, editor.
\newblock {\em Quasicrystals and discrete geometry}, volume~10 of {\em Fields
  Institute Monographs}. American Mathematical Society, Providence, RI, 1998.

\bibitem{Queffelec2010}
M.~Queff\'{e}lec.
\newblock {\em Substitution dynamical systems---spectral analysis}, volume 1294
  of {\em Lecture Notes in Mathematics}.
\newblock Springer-Verlag, Berlin, second edition, 2010.

\bibitem{SBGC1984PRL}
D.~Shechtman, I.~Blech, D.~Gratias, and J.~V. Cahn.
\newblock Metallic phase with long-range orientational order and no
  tranlational symmetry.
\newblock {\em Phys. Rev. Lett.}, 53:1951--1953, 1984.

\bibitem{Suto1987CMP}
A.~S\"{u}t\H{o}.
\newblock The spectrum of a quasiperiodic {S}chr\"{o}dinger operator.
\newblock {\em Comm. Math. Phys.}, 111(3):409--415, 1987.

\bibitem{Suto1989JSP}
A.~S\"{u}t\H{o}.
\newblock Singular continuous spectrum on a {C}antor set of zero {L}ebesgue
  measure for the {F}ibonacci {H}amiltonian.
\newblock {\em J. Statist. Phys.}, 56(3-4):525--531, 1989.

\end{thebibliography}

\end{document}